\documentclass{amsart}
\usepackage{amssymb,latexsym,amsmath,amscd,graphicx,graphics,epic,eepic,bm,color,array}
\usepackage{enumerate}
\usepackage[all,knot,poly]{xy}

\newcommand{\nat}{\mathbb{N}}
\newcommand{\zed}{\mathbb{Z}}

\newcommand{\C}{\mathbb{C}}
\newcommand{\fil}{\mathcal{F}}
\newcommand{\Hom}{\mathrm{Hom}}
\newcommand{\im}{\mathrm{Im}}
\newcommand{\Sym}{\mathrm{Sym}}
\newcommand{\qb}[2]{\genfrac{[}{]}{0pt}{}{#1}{#2}}
\newcommand{\ve}{\varepsilon}

\newcommand{\id}{\mathrm{id}}

\newcommand{\hmf}{\mathrm{hmf}}
\newcommand{\HMF}{\mathrm{HMF}}
\newcommand{\ch}{\mathsf{Ch}^{\mathsf{b}}}
\newcommand{\hch}{\mathsf{hCh}^{\mathsf{b}}}

\newcommand{\tc}{\mathrm{tc}}
\newcommand{\Mod}{\mathrm{Mod}^{\mathsf{gr}}}

\theoremstyle{plain}
\newtheorem{theorem}{Theorem}[section]
\newtheorem{lemma}[theorem]{Lemma}
\newtheorem{proposition}[theorem]{Proposition}
\newtheorem{corollary}[theorem]{Corollary}

\theoremstyle{definition}
\newtheorem{definition}[theorem]{Definition}

\newtheorem{acknowledgments}{Acknowledgments\ignorespaces}

\theoremstyle{remark}
\newtheorem{remark}[theorem]{Remark}

\numberwithin{equation}{section}

\begin{document}

\title{Equivariant colored $\mathfrak{sl}(N)$-homology for links}

\author{Hao Wu}

\address{Department of Mathematics, The George Washington University, Monroe Hall, Room 240, 2115 G Street, NW, Washington DC 20052}

\email{haowu@gwu.edu}



\begin{abstract}
In this sequel to \cite{Wu-color}, we construct an equivariant colored $\mathfrak{sl}(N)$-homology for links, which generalizes both the colored $\mathfrak{sl}(N)$-homology in \cite{Wu-color} and the equivariant $\mathfrak{sl}(N)$-homology in \cite{Krasner}. The construction is a straightforward generalization of \cite{Wu-color}. The proof of invariance is based on a simple observation which allows us to translate the proof in \cite{Wu-color} into the new setting.

As an application, we prove that deformations over $\C$ of the colored $\mathfrak{sl}(N)$-homology are link invariants. We also construct a spectral sequence connecting the colored $\mathfrak{sl}(N)$-homology to its deformations over $\C$, which generalizes the spectral sequence given in \cite{Gornik,Lee2}.
\end{abstract}

\maketitle

\section{Introduction}\label{sec-intro}

\subsection{Background} The $\mathfrak{sl}(N)$-Khovanov-Rozansky homology \cite{KR1} is a $\zed^{\oplus2}$-graded homological invariant for links. It categorifies the (uncolored) single variable $\mathfrak{sl}(N)$-HOMFLY-PT polynomial and generalizes the Khovanov homology \cite{K1}. Its construction in \cite{KR1} is based on graded matrix factorizations associated to special MOY graphs with potentials induced by $X^{N+1}$. One can perturb this construction by considering matrix factorizations with potentials induced by 
\begin{equation}\label{def-f}
f(X)=X^{N+1} + \sum_{k=1}^N (-1)^{k}\frac{N+1}{N+1-k}B_{k} X^{N+1-k}. 
\end{equation}
This idea has been explored by Lee \cite{Lee2}, Gornik \cite{Gornik}, Khovanov \cite{K2}, Mackaay, Vaz \cite{MackaayVaz}, the author \cite{Wu7} and, more recently, by Krasner \cite{Krasner}. This perturbed construction gives homological invariants for links. Their applications can be found in for example \cite{Lobb,Pl4,Ras1,Shumakovitch,Wu7}.

Recently, the author \cite{Wu-color} generalized the $\mathfrak{sl}(N)$-Khovanov-Rozansky homology to an $\mathfrak{sl}(N)$-homology for links colored by positive integers (or, equivalently, wedge products of the defining representation of $\mathfrak{sl}(N;\C)$.) The construction is based on matrix factorizations associated to general MOY graphs with potentials induced by $X^{N+1}$. 

\subsection{Main results} In the present paper, we consider the perturbed construction based on matrix factorizations with potentials induced by $f(X)$ in the colored situation. We take the view that $X$ is a homogeneous indeterminate of degree $2$, and $B_k$ is a homogeneous indeterminate of degree $2k$ for each $k=1,\dots,N$. We set $R_B=\C[B_1,\dots,B_N]$. (Unless otherwise specified, $N$ is a fixed integer greater than or equal to $2$ in the rest of this paper.)

For a link diagram $D$ colored by positive integers with a marking, this perturbed construction gives a bounded chain complex $(C_f(D), d)$ over the homotopy category $\hmf_{R_B,0}$ of graded matrix factorizations over $R_B$ with potential $0$. $C_f(D)$ is $\zed_2\oplus\zed\oplus\zed$-graded, where the $\zed_2$-grading is the $\zed_2$-grading of the underlying matrix factorizations, the first $\zed$-grading is the total polynomial grading\footnote{The total polynomial grading is the sum of the grading of $R_B$ and the grading from the alphabets marking the link diagram $D$.} of the underlying matrix factorizations, and the second $\zed$-grading is the homological grading. Moreover, the gradings of the alphabets marking $D$ induce a quantum filtration on $C_f(D)$.

Since $\hmf_{R_B,0}$ is not an abelian category, we can not directly define the homology of $C_f(D)$. As in \cite{KR1}, we note that objects of $\hmf_{R_B,0}$ are chain complexes and denote by $d_{mf}$ the differential map of the underlying matrix factorization of $C_f(D)$. Then $(H(C_f(D),d_{mf}), d_\ast)$ is a chain complex of $R_B$-modules, where $d_\ast$ is induced by the differential map $d$ of $C_f(D)$. Define
\[
H_f(D)= H(H(C_f(D),d_{mf}), d_\ast). 
\]
We call $H_f(D)$ the equivariant colored $\mathfrak{sl}(N)$-homology of $D$. It inherits the $\zed_2\oplus\zed\oplus\zed$-grading and the quantum filtration of $C_f(D)$ and, as we will later explain, is a finitely generated $R_B$-module.

The following theorem establishes the invariance of $H_f(D)$.

\begin{theorem}\label{thm-inv-main}
Let $D$ be a link diagram whose components are colored by positive integers, and $C_f(D)$ the chain complex associated to $D$.  Then the homotopy type of $C_f(D)$, with its $\zed_2\oplus\zed\oplus\zed$-grading, is independent of the choice of marking and is invariant under Reidemeister moves. 
If every component of $D$ is colored by $1$, then $C_f(D)$ is isomorphic to the chain complex defined by Krasner in \cite{Krasner}.

Consequently, the finitely generated $R_B$-module $H_f(D)$, with its $\zed_2\oplus\zed\oplus\zed$-grading, is an invariant for links colored by positive integers.
\end{theorem}

For $b_1,\dots, b_N\in \C$, let $\pi:R_B\rightarrow \C$ be the $\C$-algebra homomorphism given by $\pi(B_i)=b_i$. Then $\pi$ induces a functor $\varpi:\hmf_{R_B,0} \rightarrow \hmf_{\C,0}$. Of course, $\varpi$ does not preserve total polynomial grading. But, for a link diagram $D$ with a marking, the gradings of the alphabets marking $D$ also induce a quantum filtration on $C_{f,\pi}(D):=\varpi(C_f(D))$. And $\varpi$ preserves the quantum filtration structure. We define the homology $H_{f,\pi}(D)$ of $C_{f,\pi}(D)$ in a procedure similar to that of $H_f(D)$. We call $H_{f,\pi}$ a deformation of the colored $\mathfrak{sl}(N)$-homology over $\C$. As an application of Theorem \ref{thm-inv-main}, we prove the following theorem, which generalizes \cite[Theorems 1.1 and 1.2]{Wu7}.

\begin{theorem}\label{thm-inv-deformation}
Let $D$ be a link diagram whose components are colored by positive integers. Then the homotopy type of $C_{f,\pi}(D)$, with its $\zed_2$-grading, homological grading and quantum filtration, is independent of the choice of marking and is invariant under Reidemeister moves. If every component of $D$ is colored by $1$, then $C_{f,\pi}(D)$ is the chain complex defined in \cite{Wu7}.

Consequently, the $\C$-space $H_{f,\pi}(D)$, with its $\zed_2$-grading, homological grading and quantum filtration, is invariant under all Reidemeister moves.

Define the total color $\tc(D)$ of $D$ to be the sum of the colors of the components of $D$. Then the subspace of $H_{f,\pi}(D)$ of elements of $\zed_2$-degree $\tc(D)+1$ vanishes.

Moreover, the quantum filtration of $C_{f,\pi}(D)$ induces a spectral sequence converging to $H_{f,\pi}(D)$ with $E_1$-term isomorphic to the colored $\mathfrak{sl}(N)$-homology $H(D)$ defined in \cite{Wu-color}.
\end{theorem}

\subsection{Structure of the proof} Our proof of Theorem \ref{thm-inv-main} is based on a simple algebraic observation which allows us to translate the proof of invariance of the colored $\mathfrak{sl}(N)$-homology in \cite{Wu-color} into a proof of invariance of the equivariant colored $\mathfrak{sl}(N)$-homology. We assume the reader is somewhat familiar with \cite{Wu-color}.

In Section \ref{sec-algebra}, we review algebraic results necessary for the proof. In particular, we define in Subsection \ref{subsec-varpi-0} a functor $\varpi_0$ which serves as a translator between the colored $\mathfrak{sl}(N)$-homology in \cite{Wu-color} and the equivariant colored $\mathfrak{sl}(N)$-homology in the present paper. This functor is essentially an explicit generalization of the main technique used by Krasner in \cite{Krasner} to prove the invariance of the uncolored equivariant $\mathfrak{sl}(N)$-homology. 

In Section \ref{sec-mf-MOY}, we define the matrix factorization $C_f(\Gamma)$ associated to a MOY graph $\Gamma$. This definition is a straightforward generalization of the corresponding definition in \cite{Wu-color}. In the remainder of Section \ref{sec-mf-MOY} and Sections \ref{sec-some-morph}-\ref{sec-decomps}, we prove basic properties of $C_f(\Gamma)$ needed for the construction of the equivariant colored $\mathfrak{sl}(N)$-homology. These properties are established, for the most part, by using the functor $\varpi_0$ to translate the corresponding properties in \cite{Wu-color}.

In Sections \ref{sec-chain-complex-def}-\ref{sec-inv-reidemeister}, we construct the equivariant colored $\mathfrak{sl}(N)$-homology and prove its invariance under Reidemeister moves. This is again done by translating the corresponding work in \cite{Wu-color} using the functor $\varpi_0$.

Finally, we prove Theorem \ref{thm-inv-deformation} in Section \ref{sec-deform-over-C}. The invariance part of Theorem \ref{thm-inv-deformation} follows readily from Theorem \ref{thm-inv-main}. The purity of the $\zed_2$-grading and the spectral sequence are established using the methods developed in \cite{Gornik, Wu7}.

\subsection{Some remarks} To completely understand the present paper, the reader needs to understand the techniques developed in \cite{Wu-color}. On the other hand, the present paper is, in some sense, a ``quick" guide to \cite{Wu-color} and provides easier access to the ideas in \cite{Wu-color} with fewer technical details.

Three versions of the colored $\mathfrak{sl}(N)$-homology appear in this paper. The main body of the paper is about the equivariant colored $\mathfrak{sl}(N)$-homology $H_f$ defined using the chain complex $C_f$. In the construction and proof of invariance of $H_f$, we frequently compare it to the colored $\mathfrak{sl}(N)$-homology $H$ defined using the chain complex $C$. Finally, toward the end, we study the deformation $H_{f,\pi}$ of the colored $\mathfrak{sl}(N)$-homology, which is defined using the chain complex $C_{f,\pi}$. The author hopes the reader will not be confused by these notations.

The applications in \cite{Lobb,Pl4,Ras1,Shumakovitch,Wu7} are mostly based on generic deformations of the $\mathfrak{sl}(N)$-Khovanov-Rozansky homology. Here, ``generic" means that 
\[
\frac{d}{dX} (\pi(f(X)))= (N+1)(X^N+ \sum_{k=1}^{N} (-1)^{k} b_{k}X^{N-k})
\] 
has $N$ distinct roots over $\C$. This is because Lee \cite{Lee2} and Gornik \cite{Gornik} constructed explicit bases for generic deformations of the $\mathfrak{sl}(N)$-Khovanov-Rozansky homology. The author is working to generalize their construction to generic deformations of the colored $\mathfrak{sl}(N)$-homology.

\begin{acknowledgments}
I would like to thank Mikhail Khovanov and Daniel Krasner for helpful discussions.
\end{acknowledgments}

\section{A Little Algebra}\label{sec-algebra}

We refer the reader to  \cite[Sections 2-4]{Wu-color} and \cite[Subsection 2.1-2.2]{Wu7} for the definitions and properties of the algebraic structures used in the construction. Here, we just add a few things that are not explicitly given in those papers.

\subsection{A new contraction lemma for Koszul matrix factorizations} Several versions of the contraction lemma for Koszul matrix factorizations are given in \cite[Propositions 2.19, 2.20, 2.22]{Wu-color}. Here, we give a new version of this lemma that replaces all the versions in \cite{Wu-color}.

\begin{definition}\label{def-nice-pair}
Let $\hat{R}=\C[X_1,\dots,X_n]$, where $X_1,\dots,X_n$ are homogeneous indeterminates with positive integer degrees. Assume that $R$ is a commutative graded unital $\hat{R}$-algebra\footnote{A ``commutative graded unital $\hat{R}$-algebra" is a graded commutative ring with $1$ equipped with a grading preserving injective ring homomorphism $\jmath: \hat{R} \rightarrow R$ such that $\jmath(1)=1$. We do not distinguish between $\hat{R}$ and its image in $R$ under $\jmath$.}. For a sequence $\{b_1,\dots,b_l\}$ of homogeneous elements of $R$, we say that $(R,\{b_1,\dots,b_l\})$ is a nice pair over $\hat{R}$ if
\begin{itemize}
	\item $R$ is a free $\hat{R}$-module and the grading on $R$ is bounded below,
	\item $R/(b_1,\dots,b_l)$ is also a commutative graded unital $\hat{R}$-algebra and a free $\hat{R}$-module,
	\item $\{b_1,\dots,b_l\}$ is regular sequence in $R/(X_1,\dots,X_n)$ (see \cite[Definition 2.17]{Wu-color}),
\end{itemize}
where $(b_1,\dots,b_l)$ and $(X_1,\dots,X_n)$ are the ideals of $R$ generated by $b_1,\dots,b_l$ and $X_1,\dots,X_n$.

Here, we allow $n=0$. That is, $\hat{R}=\C$, $R$ is a graded commutative unital $\C$-algebra. In this case, $(R,\{b_1,\dots,b_l\})$ is a nice pair over $\C$ if $\{b_1,\dots,b_l\}$ is regular sequence in $R$.
\end{definition}

Before stating the lemma, let us recall the definition of homotopy equivalences of matrix factorizations.

\begin{definition}\label{def-homotopy-equivalence}
A morphism of matrix factorizations $\varphi:M\rightarrow M'$ is called a homotopy equivalence of matrix factorizations if there exists a morphism of matrix factorizations $\overline{\varphi}:M'\rightarrow M$ such that such that $\overline{\varphi} \circ \varphi \simeq \id_M$ and $\varphi \circ \overline{\varphi} \simeq \id_{M'}$. 

$M$ and $M'$ are called homotopic if there is a homotopy equivalence from $M$ to $M'$.
\end{definition}

\begin{lemma}\label{nice-b-contract}
Let $\hat{R}=\C[X_1,\dots,X_n]$, where $X_1,\dots,X_n$ are homogeneous indeterminates with positive integer degrees, and $R$ a graded commutative unital $\hat{R}$-algebra. Assume that $a_1,\dots,a_k,b_1,\dots,b_k$ are homogeneous elements of $R$ satisfying: 
\begin{itemize}
	\item $\deg a_j +\deg b_j = 2N+2$ for all $j$;
	\item $\sum_{j=1}^k a_jb_j =w \in \hat{R}$;
	\item For some $1\leq l \leq k$, $(R,\{b_1,\dots,b_l\})$ is a nice pair over $\hat{R}$.
\end{itemize}
Let $R'=R/(b_1,\dots,b_l)$ and $P:R\rightarrow R'$ be the standard quotient map. Then the Koszul matrix factorizations 
\[
M=\left(%
\begin{array}{cc}
  a_1 & b_1 \\
  a_2 & b_2 \\
  \dots & \dots \\
  a_k & b_k
\end{array}%
\right)_R
\hspace{.5cm}
\text{and}
\hspace{.5cm}
M'=\left(%
\begin{array}{cc}
  P(a_{l+1}) & P(b_{l+1}) \\
  P(a_{l+2}) & P(b_{l+2}) \\
  \dots & \dots \\
  P(a_k) & P(b_k)
\end{array}%
\right)_{R'}
\]
are homotopic as graded matrix factorizations over $\hat{R}$.
\end{lemma}

\begin{proof} 
We prove the lemma by modifying the proof of \cite[Proposition 2.2]{Wu-color}. First, we prove the case when $l=1$. In this case $R'=R/(b_1)$ and
\[
M'=\left(%
\begin{array}{cc}
  P(a_{2}) & P(b_{2}) \\
  \dots & \dots \\
  P(a_k) & P(b_k)
\end{array}%
\right)_{R'}.
\]
We invoke the ``$1_\ve$" notation
introduced in \cite[Definition 2.4]{Wu-color} and define an $R$-module homomorphism $G:M\rightarrow M'$ by
\[
G(r1_\ve) = \left\{%
\begin{array}{ll}
    P(r) 1_{(\ve_{2},\dots,\ve_k)} & \text{if } \ve_1=0, \\
    0 & \text{if } \ve_1 = 1, 
\end{array}%
\right.
\]
for $r\in R$ and $\ve=(\ve_1,\dots,\ve_k)\in I^k$. It is easy to check that $G$ is surjective, commutes with the differential maps and preserves both the $\zed_2$ and the total polynomial gradings. So, in particular, $G$ is a morphism of matrix factorizations over $\hat{R}$.

Let $S=R/(X_1,\dots,X_n)$ and $S'=R/(X_1,\dots,X_n,b_1)$, both of which are commutative graded unital $\C$-algebras. Then $P$ induces a projection $p:S\rightarrow S'$. Denote by $\pi:R\rightarrow S$ the standard projection and by $\hat{\mathfrak{I}}$ the ideal of $\hat{R}$ generated by $X_1,\dots,X_n$. Then $w\in \hat{\mathfrak{I}}$, and $M/\hat{\mathfrak{I}}M$ and $M'/\hat{\mathfrak{I}}M'$ are graded matrix factorizations over $\C$ of potential $0$, that is, chain complexes. $M/\hat{\mathfrak{I}}M$ and $M'/\hat{\mathfrak{I}}M'$ are given by 
\[
M/\hat{\mathfrak{I}}M=\left(%
\begin{array}{cc}
  \pi(a_1) & \pi(b_1) \\
  \pi(a_2) & \pi(b_2) \\
  \dots & \dots \\
  \pi(a_k) & \pi(b_k)
\end{array}%
\right)_S
~
\text{and}
~
M'/\hat{\mathfrak{I}}M'=\left(%
\begin{array}{cc}
  (p\circ\pi)(a_{2}) & (p\circ\pi)(b_{2}) \\
  \dots & \dots \\
  (p\circ\pi)(a_k) & (p\circ\pi)(b_k)
\end{array}%
\right)_{S'}.
\] 

$G$ induces a surjective $S$-module homomorphism $g: M/\hat{\mathfrak{I}}M \rightarrow M'/\hat{\mathfrak{I}}M'$ given by 
\[
g(s1_\ve) = \left\{%
\begin{array}{ll}
    p(s) 1_{(\ve_{2},\dots,\ve_k)} & \text{if } \ve_1=0, \\
    0 & \text{if } \ve_1 = 1, 
\end{array}%
\right.
\]
It is easy to see that $g$ is a chain map that preserves both gradings. The kernel of $g$ is the subcomplex
\[
\ker g = \bigoplus_{\ve_{2},\dots,\ve_k \in I}(S \cdot 1_{(1,\ve_{2},\dots,\ve_k)} \oplus \pi(b_1)S \cdot 1_{(0,\ve_{2},\dots,\ve_k)}).
\] 

Recall that, by the definition of nice pairs, $\pi(b_1)$ is not a zero-divisor in $S$. So the division map $\varphi:\pi(b_1)S\rightarrow S$ given by $\varphi(\pi(b_1)s)=s$ is a well defined $S$-module homomorphism. Define an $S$-module homomorphism $h:\ker g \rightarrow \ker g$ by 
\begin{eqnarray*}
h(1_{(1,\ve_{2},\dots,\ve_k)}) & = & 0, \\
h(\pi(b_1) 1_{(0,\ve_{2},\dots,\ve_k)}) & = & 1_{(1,\ve_{2},\dots,\ve_k)}.
\end{eqnarray*}
Then 
\[
d|_{\ker g} \circ h+h \circ d|_{\ker g}=\id_{\ker g},
\] 
where $d$ is the differential map of $M/\hat{\mathfrak{I}}M$. In particular, this means that $H(\ker g)=0$. Then, using the long exact sequence induced by
\[
0\rightarrow \ker g \rightarrow M/\hat{\mathfrak{I}}M \xrightarrow{g} M'/\hat{\mathfrak{I}}M' \rightarrow 0,
\]
it is easy to see that $g$ is a quasi-isomorphism. Thus, by \cite[Proposition 8]{KR1} (see \cite[Corollary 3.8]{Wu-color}), $G$ is a homotopy equivalence of matrix factorizations over $\hat{R}$.

Now we prove the lemma for general $1 \leq l \leq k$. For $1\leq i \leq l$, let $R_i = R/(b_1,\dots,b_i)$ and $P_i:R \rightarrow R_i$ the standard projection. Then $R_l=R'$ and $P_l=P$.

Consider the Koszul matrix factorizations $M_0=M$ and $M_i$ over $R_i$ defined by
\[
M_i = \left(%
\begin{array}{cc}
  P_i(a_{i+1}) & P_i(b_{i+1}) \\
  P_i(a_{i+2}) & P_i(b_{i+2}) \\
  \dots & \dots \\
  P_i(a_k) & P_i(b_k)
\end{array}%
\right)_{R_i} \text{ for } 1\leq i \leq l.
\]
Then $M'=M_l$. Denote by $\hat{\mathfrak{I}}$ again the ideal of $\hat{R}$ generated by $X_1,\dots,X_n$. By the proof of the case $l=1$, one can see that, for $i=0,1,\dots,l-1$, there is an $R$-module homomorphism $G_i: M_i \rightarrow M_{i+1}$ that commutes with the differential maps, preserves both gradings and induces a quasi-isomorphism $g_i: M_i/\hat{\mathfrak{I}} M_i \rightarrow M_{i+1}/\hat{\mathfrak{I}} M_{i+1}$ of chain complexes. Define $G:M \rightarrow M'$ by $G=G_{l-1}\circ \cdots \circ G_1 \circ G_0$. Then $G$ is a morphism of matrix factorizations over $\hat{R}$ preserving both gradings and induces a quasi-isomorphism $g=g_{l-1}\circ \cdots \circ g_1 \circ g_0: M/\hat{\mathfrak{I}} M \rightarrow M'/\hat{\mathfrak{I}} M'$ of chain complexes. Thus, by \cite[Proposition 8]{KR1} (see \cite[Corollary 3.8]{Wu-color}), $G$ is a homotopy equivalence of matrix factorizations over $\hat{R}$.
\end{proof}

Applying the same arguments, we get the following lemma.

\begin{lemma}\label{nice-a-contract}
Let $\hat{R}=\C[X_1,\dots,X_n]$, where $X_1,\dots,X_n$ are homogeneous indeterminates with positive integer degrees, and $R$ a graded commutative unital $\hat{R}$-algebra. Assume that $a_1,\dots,a_k,b_1,\dots,b_k$ are homogeneous elements of $R$ satisfying: 
\begin{itemize}
	\item $\deg a_j +\deg b_j = 2N+2$ for all $j$;
	\item $\sum_{j=1}^k a_jb_j =w \in \hat{R}$;
	\item For some $1\leq l \leq k$, $(R,\{a_1,\dots,a_l\})$ is a nice pair over $\hat{R}$.
\end{itemize}
Let $R'=R/(a_1,\dots,a_l)$ and $P:R\rightarrow R'$ be the standard quotient map. Then the Koszul matrix factorizations 
\[
\left(%
\begin{array}{cc}
  a_1 & b_1 \\
  a_2 & b_2 \\
  \dots & \dots \\
  a_k & b_k
\end{array}%
\right)_R
\simeq \left(%
\begin{array}{cc}
  P(a_{l+1}) & P(b_{l+1}) \\
  P(a_{l+2}) & P(b_{l+2}) \\
  \dots & \dots \\
  P(a_k) & P(b_k)
\end{array}%
\right)_{R'} \{q^{l(N+1)-\sum_{i=1}^l\deg a_i}\}\left\langle l\right\rangle
\]
as graded matrix factorizations over $\hat{R}$.
\end{lemma}

\subsection{Graded matrix factorizations over $R_B \otimes \C[X_1,\dots,X_m]$}\label{subsec-varpi-0} Recall that $R_B=\C[B_1,\dots,B_N]$, where $B_i$ is a homogeneous indeterminate of degree $2i$. Let $X_1,\dots, X_m$ be homogeneous indeterminates of positive integer degrees. Set $R=\C[X_1,\dots,X_m]$ and $\tilde{R} = R \otimes_\C R_B = \C[X_1,\dots,X_m,B_1,\dots,B_N]$. Denote by $\mathfrak{I}_B$ the homogeneous ideal of $\tilde{R}$ generated by $B_1,\dots,B_N$ and by $\pi_0: \tilde{R} \rightarrow \tilde{R}/\mathfrak{I}_B = R$ the standard projection map. Let $w$ be a homogeneous element of $\tilde{R}$ of degree $2N+2$. If $\tilde{M}$ is a graded matrix factorization over $\tilde{R}$ with potential $w$, then $\tilde{M}/\mathfrak{I}_B \tilde{M}$ is a graded matrix factorization over $R$ with potential $\pi_0(w)$. 

We call the grading of $\tilde{R}$ and the grading of any graded $\tilde{R}$-module the total polynomial grading. Also, recall that every matrix factorization comes with a $\zed_2$-grading.

\begin{lemma}\label{reduce-base-homotopic-equivalence}
Let $\tilde{M}$, $\tilde{M}'$ be graded matrix factorizations over $\tilde{R}$ with potential $w$, and $\tilde{\psi}:\tilde{M}\rightarrow\tilde{M}'$ a morphism of matrix factorizations over $\tilde{R}$ preserving both gradings. Assume that the total polynomial gradings of $\tilde{M}$ and $\tilde{M}'$ are bounded below. Then $\tilde{\psi}$ is a homotopy equivalence of matrix factorizations if and only if the induced morphism $\psi:\tilde{M}/\mathfrak{I}_B \tilde{M}\rightarrow\tilde{M}/\mathfrak{I}_B \tilde{M}$ is a homotopy equivalence of matrix factorizations.
\end{lemma}
\begin{proof}
Write $M=\tilde{M}/\mathfrak{I}_B \tilde{M}$ and $M'=\tilde{M}'/\mathfrak{I}_B \tilde{M}'$. Let $\tilde{\mathfrak{I}}$ be the maximal homogeneous ideal of $\tilde{R}$ generated by $X_1,\dots,X_m,B_1,\dots,B_N$, and $\mathfrak{I}$ the maximal homogeneous ideal of $R$ generated by $X_1,\dots,X_m$. 

By \cite[Proposition 8]{KR1} (see also \cite[Corollary 3.8]{Wu-color}), $\tilde{\psi}$ is a homotopy equivalence of matrix factorizations if and only if it induces a quasi-isomorphism $\tilde{M}/ \tilde{\mathfrak{I}} \tilde{M} \rightarrow \tilde{M}'/ \tilde{\mathfrak{I}} \tilde{M}'$, and $\psi$ is a homotopy equivalence of matrix factorizations if and only if it induces a quasi-isomorphism $M/ \mathfrak{I}M \rightarrow M'/ \mathfrak{I} M'$. It is clear that $\tilde{M}/ \tilde{\mathfrak{I}} \tilde{M} = M/ \mathfrak{I}M$, $\tilde{M}'/ \tilde{\mathfrak{I}} \tilde{M}' = M'/ \mathfrak{I}M'$ and $\tilde{\psi}$, $\psi$ induce the same chain map. And the lemma follows.
\end{proof}

\begin{lemma}\label{homotopy-finite-reduce-base}
Let $\tilde{M}$ be a graded matrix factorizations over $\tilde{R}$ with potential $w$. Assume that the total polynomial grading of $\tilde{M}$ is bounded below. then we have:
\begin{enumerate}[(a)]
	\item $\tilde{M}$ is homotopically finite over $\tilde{R}$ if and only if $M=\tilde{M}/\mathfrak{I}_B \tilde{M}$ is homotopically finite over $R$.
	\item $\tilde{M}$ is null-homotopic over $\tilde{R}$ if and only if $M=\tilde{M}/\mathfrak{I}_B \tilde{M}$ is null-homotopic over $R$.
\end{enumerate}

\end{lemma}
\begin{proof}
Recall that $\tilde{M}$ is homotopically finite over $\tilde{R}$ if and only if it is homotopic over $\tilde{R}$ to a finitely generated matrix factorization over $\tilde{R}$. Then, by \cite[Proposition 7]{KR1} (see also \cite[Corollary 3.9]{Wu-color}), $\tilde{M}$ is homotopically finite (resp. null-homotopic) over $\tilde{R}$ if and only if the homology of $\tilde{M}/ \tilde{\mathfrak{I}} \tilde{M}$ is finite dimensional over $\C$ (resp. $0$), and $M$ is homotopically finite (resp. null-homotopic) over $R$ if and only if the homology of $M/\mathfrak{I}M$ is finite dimensional over $\C$ (resp. $0$). It is clear that $\tilde{M}/ \tilde{\mathfrak{I}} \tilde{M} \cong M/\mathfrak{I}M$. And lemma follows.
\end{proof}

\begin{corollary}\label{reduce-base-functor}
$\pi_0$ induces a functor $\varpi_0: \hmf_{\tilde{R},w} \rightarrow \hmf_{R,\pi_0(w)}$, which maps each object $\tilde{M}$ of $\hmf_{\tilde{R},w}$ to the object $\varpi_0(\tilde{M})=\tilde{M}/\mathfrak{I}_B \tilde{M}$ of $\hmf_{R,\pi_0(w)}$. 
\end{corollary}
\begin{proof}
This follows easily from Lemma \ref{homotopy-finite-reduce-base}, and the fact that $\tilde{M}/\mathfrak{I}_B \tilde{M}$ inherits the total polynomial grading of $\tilde{M}$, which is bounded below. (See \cite[Definition 2.30]{Wu-color}.) 
\end{proof}

\subsection{Schur polynomials of the difference of two alphabets} Now we review the notion of the Schur polynomials associated to the difference of two alphabets. For more details, see for example \cite{Lascoux-notes}.

Recall that, for an alphabet $\mathbb{X}$, the complete symmetric polynomials $\{h_k(\mathbb{X})\}$ are the unique symmetric polynomials in $\mathbb{X}$ such that
\[
 \sum_{k=0}^{\infty} h_k(\mathbb{X}) t^k = H_{\mathbb{X}}(t) := \prod_{x\in \mathbb{X}} (1-xt)^{-1}.
\]
For two alphabets $\mathbb{X}$ and $\mathbb{Y}$, not necessarily disjoint, define the complete symmetric polynomials $\{h_k(\mathbb{X}-\mathbb{Y})\}$ to be the unique elements of $\Sym(\mathbb{X}|\mathbb{Y})$ such that
\begin{equation}\label{def-complete-poly-diff-0}
 \sum_{k=0}^{\infty} h_k(\mathbb{X}-\mathbb{Y}) t^k = \frac{H_{\mathbb{X}}(t)}{H_{\mathbb{Y}}(t)} = (\prod_{x\in \mathbb{X}} (1-xt)^{-1}) \cdot (\prod_{y\in \mathbb{Y}} (1-yt)).
\end{equation}
If $\mathbb{X}=\emptyset$, then $h_k(\emptyset-\mathbb{Y})= h_k(-\mathbb{Y})=(-1)^kY_k$ as in \cite[Subsection 4.2]{Wu-color}, where $Y_k$ is the $k$-th elementary symmetric polynomial in $\mathbb{Y}$. In general,
\begin{equation}\label{def-complete-poly-diff}
h_k(\mathbb{X}-\mathbb{Y}) = \sum_{j=0}^{k} h_{k-j}(\mathbb{X})h_j(-\mathbb{Y}) = \sum_{j=0}^{k} (-1)^j Y_j h_{k-j}(\mathbb{X}).
\end{equation}
More generally, for a partition $\lambda=(\lambda_1\geq\dots\geq\lambda_m)$, we have the Schur polynomial $S_{\lambda}(\mathbb{X}-\mathbb{Y})$ given by
\begin{eqnarray*}
&& S_\lambda(\mathbb{X}-\mathbb{Y}) \\
& = & \det (h_{\lambda_i -i +j}(\mathbb{X}-\mathbb{Y})) \\
& = & \left|%
\begin{array}{llll}
h_{\lambda_1}(\mathbb{X}-\mathbb{Y}) & h_{\lambda_1+1}(\mathbb{X}-\mathbb{Y}) & \dots & h_{\lambda_1+m-1}(\mathbb{X}-\mathbb{Y}) \\
h_{\lambda_2-1}(\mathbb{X}-\mathbb{Y}) & h_{\lambda_2}(\mathbb{X}-\mathbb{Y}) & \dots & h_{\lambda_2+m-2}(\mathbb{X}-\mathbb{Y}) \\
\dots & \dots & \dots & \dots \\
h_{\lambda_m-m+1}(\mathbb{X}-\mathbb{Y}) & h_{\lambda_m-m+2}(\mathbb{X}-\mathbb{Y}) & \dots & h_{\lambda_m}(\mathbb{X}-\mathbb{Y}) 
\end{array}%
\right|.
\end{eqnarray*}

\begin{lemma}\label{schur-poly-diff-cancel}
Let $\mathbb{X}$, $\mathbb{Y}$ and $\mathbb{W}$ be disjoint alphabets. Then, for any partition $\lambda=(\lambda_1\geq\dots\geq\lambda_m)$,
\[
S_\lambda(\mathbb{X}-\mathbb{Y}) = S_\lambda((\mathbb{X}\cup\mathbb{W})-(\mathbb{Y}\cup\mathbb{W})).
\]
\end{lemma}

\begin{proof}
If $\lambda=(\lambda_1)$, that is, $\lambda$ has only one part, then $S_\lambda(\mathbb{X}-\mathbb{Y})=h_{\lambda_1}(\mathbb{X}-\mathbb{Y})$ and $S_\lambda((\mathbb{X}\cup\mathbb{W})-(\mathbb{Y}\cup\mathbb{W}))=h_{\lambda_1}((\mathbb{X}\cup\mathbb{W})-(\mathbb{Y}\cup\mathbb{W}))$. In this case, the lemma follows easily from the definition of $h_{k}(\mathbb{X}-\mathbb{Y})$ by \eqref{def-complete-poly-diff-0}. Then, for general partitions, the lemma follows from the above definition of $S_\lambda(\mathbb{X}-\mathbb{Y})$.
\end{proof}

Recall that the power sum symmetric polynomial $p_k(\mathbb{X})$ is defined by
\[
p_k(\mathbb{X})=
\begin{cases}    
\sum_{x \in \mathbb{X}} x^k & \text{if } k\geq0, \\
0 & \text{if } k<0. 
\end{cases}
\]

The next lemma follows easily from \cite[Lemma 4.1]{Wu-color}.

\begin{lemma}\label{power-derive-complete}
Let $B_1,\dots,B_N$ be as above. Define 
\[
f(\mathbb{X})= p_{N+1}(\mathbb{X}) + \sum_{k=1}^N (-1)^{N+1-k}\frac{N+1}{k}B_{N+1-k} p_{k}(\mathbb{X}).
\]
Denote by $X_j$ the $j$-th elementary symmetric polynomial in $\mathbb{X}$. Then
\[
\frac{\partial f(\mathbb{X})} {\partial X_j} = (-1)^{j+1} (N+1) (h_{N+1-j}(\mathbb{X}) + \sum_{k=1}^N (-1)^{N+1-k}B_{N+1-k} h_{k-j}(\mathbb{X})).
\]
If we identify $B_k$ with the $k$-th elementary symmetric polynomial of an alphabet $\mathbb{B}$ of $N$ indeterminates disjoint from $\mathbb{X}$, then 
\[
\frac{\partial f(\mathbb{X})} {\partial X_j} = (-1)^{j+1} (N+1) h_{N+1-j}(\mathbb{X}-\mathbb{B}).
\]
\end{lemma}

The following is a slight variant of \cite[Proposition \emph{Gr}5]{Lascoux-notes}. 

\begin{theorem}\cite[Proposition \emph{Gr}5]{Lascoux-notes}\label{part-symm-str-alter}
Let $\mathbb{X}$ and $\mathbb{B}$ be two disjoint alphabets with $m$ and $N$ indeterminates, where $m \leq N$. Then the quotient ring
\[
\hat{R}:=\Sym(\mathbb{X}|\mathbb{B})/(h_{N}(\mathbb{X}-\mathbb{B}), h_{N-1}(\mathbb{X}-\mathbb{B}),\dots, h_{N+1-m}(\mathbb{X}-\mathbb{B}))
\] 
is a graded-free $\Sym(\mathbb{B})$-module. 

Denote by $\Lambda_{m,N-m}$ the set of partitions $\Lambda_{m,N-m}=\{\lambda~|~l(\lambda)\leq m, ~\lambda_1\leq N-m\}$. Then 
\[\{S_\lambda(\mathbb{X})~|~ \lambda \in \Lambda_{m,N-m}\} \text{ and } \{S_\lambda(\mathbb{X}-\mathbb{B})~|~ \lambda \in \Lambda_{m,N-m}\}
\] 
are two homogeneous bases for the $\Sym(\mathbb{B})$-module $\hat{R}$. In particular,
\[
\hat{R} \cong \Sym(\mathbb{B}) \{\qb{N}{m}\}
\]
as $\Sym(\mathbb{B})$-modules.

Moreover, there is a unique $\Sym(\mathbb{B})$-module homomorphism 
\[
\zeta:\hat{R} \rightarrow \Sym(\mathbb{B}),
\] 
called the Sylvester operator, such that, for $\lambda,\mu \in \Lambda_{m,N-m}$,
\[
\zeta(S_\lambda(\mathbb{X}) \cdot S_\mu(\mathbb{X}-\mathbb{B})) = \left\{%
\begin{array}{ll}
    1 & \text{if } \lambda_j + \mu_{m+1-j} =N-m ~\forall j=1,\dots,m, \\
    0 & \text{otherwise.}  \\
\end{array}%
\right.
\]
\end{theorem}
\begin{proof}
Compare \cite[Proposition \emph{Gr}5]{Lascoux-notes} (see \cite[Theorem 4.3]{Wu-color}) to \cite[Lemma 11.5]{Wu-color}.
\end{proof}

\begin{remark}
Note that $\hat{R}$ is isomorphic to the $GL(N)$-equivariant cohomology ring of the complex $(m,N)$-Grassmannian $G_{m,N}$ and $\Sym(\mathbb{B})$ is isomorphic to the base ring of this equivariant cohomology. Also, the Sylvester operator gives the corresponding Poincar\'e duality. See for example \cite[Lecture 6]{Fulton-notes} for more details.
\end{remark}

\section{Matrix Factorizations Associated to MOY Graphs}\label{sec-mf-MOY}

\subsection{MOY graphs and their markings} First, let us recall the definitions of MOY graphs and markings of MOY graphs in \cite{Wu-color}.
\begin{definition}\label{MOY-graph-def}
An abstract MOY graph is an oriented graph with each edge colored by a non-negative integer such that, for every vertex $v$ with valence at least $2$, the sum of integers coloring the edges entering $v$ is equal to the sum of integers coloring the edges leaving $v$. We call this common sum the width of $v$.

A vertex of valence $1$ in an abstract MOY graph is called an end point. An abstract MOY graph $\Gamma$ is said to be closed if it has no end points.

An embedded MOY graph, or simply an MOY graph, $\Gamma$ is an embedding of an abstract MOY graph into $\mathbb{R}^2$ such that, through each vertex $v$ of $\Gamma$, there is a straight line $L_v$ so that all the edges entering $v$ enter through one side of $L_v$ and all edges leaving $v$ leave through the other side of $L_v$.
\end{definition}

\begin{remark}
Before moving on, we should emphasize the following two points:
\begin{enumerate}[(i)]
	\item In this paper, an MOY graph means an embedded MOY graph.
	\item Some abstract MOY graphs can not be realized as an (embedded) MOY graph.
\end{enumerate}
\end{remark}

\begin{definition}\label{MOY-marking-def}
A marking of an MOY graph $\Gamma$ consists of the following:
\begin{enumerate}
	\item A finite collection of marked points on $\Gamma$ such that
	\begin{itemize}
	\item every edge of $\Gamma$ has at least one marked point;
	\item all the end points (vertices of valence $1$) are marked;
	\item none of the interior vertices (vertices of valence at least $2$) is marked.
  \end{itemize}
  \item An assignment of pairwise disjoint alphabets to the marked points such that the alphabet associated to a marked point on an edge of color $m$ has $m$ independent indeterminates. (Recall that an alphabet is a finite collection of homogeneous indeterminates of degree $2$.)
\end{enumerate}
\end{definition}

\begin{figure}[ht]

\setlength{\unitlength}{1pt}

\begin{picture}(360,95)(-180,-55)


\put(0,0){\vector(-1,1){15}}

\put(-15,15){\line(-1,1){15}}

\put(-23,25){\tiny{$i_1$}}

\put(-33,32){\small{$\mathbb{X}_1$}}

\put(0,0){\vector(-1,2){7.5}}

\put(-7.5,15){\line(-1,2){7.5}}

\put(-11,25){\tiny{$i_2$}}

\put(-18,32){\small{$\mathbb{X}_2$}}

\put(3,25){$\cdots$}

\put(0,0){\vector(1,1){15}}

\put(15,15){\line(1,1){15}}

\put(31,25){\tiny{$i_k$}}

\put(27,32){\small{$\mathbb{X}_k$}}


\put(4,-2){$v$}

\multiput(-50,0)(5,0){19}{\line(1,0){3}}

\put(-70,0){$L_v$}

\put(45,0){\tiny{$i_1+i_2+\cdots +i_k = j_1+j_2+\cdots +j_l$}}


\put(-30,-30){\vector(1,1){15}}

\put(-15,-15){\line(1,1){15}}

\put(-26,-30){\tiny{$j_1$}}

\put(-33,-40){\small{$\mathbb{Y}_1$}}

\put(-15,-30){\vector(1,2){7.5}}

\put(-7.5,-15){\line(1,2){7.5}}

\put(-13,-30){\tiny{$j_2$}}

\put(-18,-40){\small{$\mathbb{Y}_2$}}

\put(3,-30){$\cdots$}

\put(30,-30){\vector(-1,1){15}}

\put(15,-15){\line(-1,1){15}}

\put(31,-30){\tiny{$j_l$}}

\put(27,-40){\small{$\mathbb{Y}_l$}}

\put(-4,-55){$\Gamma_v$}

\end{picture}

\caption{}\label{general-MOY-vertex}

\end{figure}

\subsection{The matrix factorization associated to an MOY graph} For an MOY graph $\Gamma$ with a marking, cut it open at the marked points. This gives a collection of marked MOY graphs, each of which is a star-shaped neighborhood of a vertex in $G$ and is marked only at the endpoints. We call these the pieces of $\Gamma$. (If an edge of $\Gamma$ has two or more marked points, then some of these pieces may be oriented arcs from one marked point to another. In this case, we consider such an arc as a neighborhood of an additional vertex of valence $2$ in the middle of that arc.)

Let $\Gamma_v$ in Figure \ref{general-MOY-vertex} be a piece of $\Gamma$. Set $m=i_1+i_2+\cdots +i_k = j_1+j_2+\cdots +j_l$ (the width of $v$.) Define $R=\Sym(\mathbb{X}_1|\dots|\mathbb{X}_k|\mathbb{Y}_1|\dots|\mathbb{Y}_l)$. Let $\tilde{R}=R\otimes_{\C} R_B$ with the total grading induced by the gradings of $R$ and $R_B$. Set $\mathbb{X}=\mathbb{X}_1\cup\cdots\cup \mathbb{X}_k$ and $\mathbb{Y}=\mathbb{Y}_1\cup\cdots\cup \mathbb{Y}_l$. Denote by $X_j$ the $j$-th elementary symmetric polynomial in $\mathbb{X}$ and by $Y_j$ the $j$-th elementary symmetric polynomial in $\mathbb{Y}$. Let $U_1, \dots,U_m$ be homogeneous elements of $\tilde{R}$ satisfying
\begin{enumerate} [(i)]
	\item $\deg U_i = 2N+2-2i$,
	\item $\sum_{i=1}^m U_i \cdot (X_i-Y_i) = \sum_{\alpha=1}^k f(\mathbb{X}_\alpha) - \sum_{\beta=1}^l f(\mathbb{Y}_\beta)$,
\end{enumerate}
where $f(\mathbb{X}_\alpha)$ and $f(\mathbb{Y}_\beta)$ are defined as in Lemma \ref{power-derive-complete}.

The matrix factorization associated to the vertex $\Gamma_v$ is defined to be
\[
C_f(\Gamma_v)=\left(%
\begin{array}{cc}
  U_1 & X_1-Y_1 \\
  U_2 & X_2-Y_2 \\
  \dots & \dots \\
  U_m & X_m-Y_m
\end{array}%
\right)_{\tilde{R}}
\{q^{-\sum_{1\leq s<t \leq k} i_si_t}\},
\]
which is a graded matrix factorization over $\tilde{R}$ with potential $\sum_{\alpha=1}^k f(\mathbb{X}_\alpha) - \sum_{\beta=1}^l f(\mathbb{Y}_\beta)$. As in \cite[Subsection 5.2]{Wu-color}, one can easily check that $\{X_1-Y_1,\dots,X_m-Y_m\}$ is a regular sequence in $\tilde{R}$. (See \cite[Definition 2.17]{Wu-color}.) So, by \cite[Lemma 2.18]{Wu-color}, the isomorphism type of $C_f(\Gamma_v)$ does not depend on the choice of $U_1, \dots,U_m$. Also, note that the grading of $R$ induces a filtration on $C_f(\Gamma_v)$, which we call the quantum filtration. (See \cite[Subsection 2.2]{Wu7} for definition.) It is easy to check that the isomorphism induced by changing the choice of $U_1, \dots,U_m$ preserves the quantum filtration too.

From now on, we will only specify our choice for $U_1,\dots,U_m$ when it is used in the computation. Otherwise, we will simply denote them by $\ast$'s. 

\begin{definition}\label{MOY-mf-def}
\[
C_f(\Gamma) := \bigotimes_{\Gamma_v} C_f(\Gamma_v),
\]
where $\Gamma_v$ runs through all pieces of $\Gamma$. Here, the tensor product is done over the common end points. More precisely, for two sub-MOY graphs $\Gamma_1$ and $\Gamma_2$ of $\Gamma$ intersecting only at (some of) their open end points, let $\mathbb{W}_1,\dots,\mathbb{W}_n$ be the alphabets associated to these common end points. Then, in the above tensor product, $C_f(\Gamma_1)\otimes C_f(\Gamma_2)$ is the tensor product 
\[
C_f(\Gamma_1)\otimes_{\Sym(\mathbb{W}_1|\dots|\mathbb{W}_n)\otimes_\C R_b} C_f(\Gamma_2).
\]

$C_f(\Gamma)$ has a $\zed_2$-grading, a total polynomial grading and a quantum filtration.

Assume $\Gamma$ has end points. Let $\mathbb{E}_1,\dots,\mathbb{E}_n$ be the alphabets assigned to all end points of $\Gamma$, among which $\mathbb{E}_1,\dots,\mathbb{E}_k$ are assigned to exits and $\mathbb{E}_{k+1},\dots,\mathbb{E}_n$ are assigned to entrances. Then the potential of $C_f(\Gamma)$ is  
\[
w= \sum_{i=1}^k f(\mathbb{E}_i) - \sum_{j=k+1}^n f(\mathbb{E}_j).
\]
Let $R_\partial=\Sym(\mathbb{E}_1|\cdots|\mathbb{E}_n)$ and $\tilde{R}_\partial=R_\partial \otimes_\C R_B$. Although the alphabets assigned to all marked points on $\Gamma$ are used in its construction, $C_f(\Gamma)$ is viewed as a matrix factorization over $\tilde{R}_\partial$ with potential $w$.

If $\Gamma$ is closed, i.e. has no end points, then $R_\partial=\C$, $\tilde{R}_\partial=R_B$, and $C_f(\Gamma)$ is a matrix factorization over $\tilde{R}_\partial=R_B$ with potential $0$. 

We allow the MOY graph to be empty. In this case, we define 
\[
C_f(\emptyset)=R_B\rightarrow 0 \rightarrow R_B,
\]
where the $\zed_2$-grading $R_B$ is $0$ and the quantum filtration on $R_B$ is given by $\fil^{-1}R_B=0$, $\fil^{0}R_B=R_B$.
\end{definition}

\begin{remark}\label{functor-varpi-0}
Note that the projection $\pi_0: R_B \rightarrow \C$ given by $\pi_0(B_i)=0$ induces a functor $\varpi_0$ such that $\varpi_0(C_f(\Gamma)) = C(\Gamma)$, where $C(\Gamma)$ is the the matrix factorization associated to $\Gamma$ defined in \cite[Definition 5.5]{Wu-color}. This functor allows us to translate the proofs of most of the results about $C(\Gamma)$ in \cite{Wu-color} to $C_f(\Gamma)$. 
\end{remark}

The lemmas in the rest of this subsection correspond to those in \cite[Subsection 5.2]{Wu-color}. Their proofs remain more or less unchanged.

\begin{lemma}\label{marking-independence}
If $\Gamma$ is an MOY graph, then the homotopy type of $C_f(\Gamma)$ does not depend on the choice of the marking.
\end{lemma}
\begin{proof}
We only need to show that adding or removing an extra marked point corresponds to a homotopy of matrix factorizations preserving the $\zed_2$-grading and the total polynomial grading. This follows easily from Lemma \ref{nice-b-contract}.
\end{proof}

\begin{definition}\label{homology-MOY-def}
Let $\Gamma$ be an MOY graph with a marking. 
\begin{enumerate}[(i)]
	\item If $\Gamma$ is closed, i.e. has no open end points, then $C_f(\Gamma)$ is a chain complex. Denote by $H_f(\Gamma)$ the homology of $C_f(\Gamma)$. Note that $H_f(\Gamma)$ inherits both gradings of $C_f(\Gamma)$.
	\item If $\Gamma$ has end points, let $\mathbb{E}_1,\dots,\mathbb{E}_n$ be the alphabets assigned to all end points of $\Gamma$, and $R_\partial=\Sym(\mathbb{E}_1|\cdots|\mathbb{E}_n)$. Denote by $E_{i,j}$ the $j$-th elementary symmetric polynomial in $\mathbb{E}_i$ and by $\mathfrak{I}$ the homogeneous ideal of $\tilde{R}_\partial=R_\partial \otimes_\C R_B$ generated by $\{E_{i,j}\}$. Then $H_f(\Gamma)$ is defined to be the homology of the chain complex $C_f(\Gamma)/\mathfrak{I} C_f(\Gamma)$. Clearly, $H_f(\Gamma)$ inherits both gradings of $C_f(\Gamma)$.
\end{enumerate}
Note that (i) is a special case of (ii).
\end{definition}

\begin{lemma}\label{width-cap}
If $\Gamma$ is an MOY graph with a vertex of width greater than $N$, then $C_f(\Gamma)\simeq 0$.
\end{lemma}
\begin{proof}
Follows from \cite[Lemma 5.8]{Wu-color} and Lemma \ref{homotopy-finite-reduce-base}.
\end{proof}

\begin{lemma}\label{MOY-object-of-hmf} 
Let $\Gamma$ be an MOY graph, and $\mathbb{E}_1,\dots,\mathbb{E}_n$ the alphabets assigned to all end points of $\Gamma$, among which $\mathbb{E}_1,\dots,\mathbb{E}_k$ are assigned to exits and $\mathbb{E}_{k+1},\dots,\mathbb{E}_n$ are assigned to entrances. (Here we allow $n=0$, i.e. $\Gamma$ to be closed.) Write $R_\partial=\Sym(\mathbb{E}_1|\cdots|\mathbb{E}_n)$, $\tilde{R}_\partial=R_\partial\otimes_\C R_B$ and $w= \sum_{i=1}^k f(\mathbb{E}_i) - \sum_{j=k+1}^n f(\mathbb{E}_j)$. Then $C_f(\Gamma)$ is an object of $\hmf_{\tilde{R}_\partial,w}$.

Moreover, the projection $\pi_0: \tilde{R}_\partial \rightarrow R_\partial$ given by $\pi_0(B_i)=0$ induces a functor $\varpi_0:\hmf_{\tilde{R}_\partial,w}\rightarrow \hmf_{R_\partial,\pi_0(w)}$ such that $\varpi_0(C_f(\Gamma)) = C(\Gamma)$, where $C(\Gamma)$ is the the matrix factorization associated to $\Gamma$ defined in \cite[Definition 5.5]{Wu-color}. 
\end{lemma}
\begin{proof}
Follows from \cite[Lemma 5.11]{Wu-color}, Lemma \ref{homotopy-finite-reduce-base} and Corollary \ref{reduce-base-functor}.
\end{proof}

\begin{figure}[ht]

\begin{picture}(360,90)(-180,-50)


\put(-130,0){\vector(-3,2){22.5}}

\put(-152.5,15){\line(-3,2){22.5}}

\put(-178,25){\tiny{$i_1$}}

\put(-178,32){\small{$\mathbb{X}_1$}}

\put(-152,15){$\cdots$}

\put(-150,25){\tiny{$i_s$}}

\put(-148,32){\small{$\mathbb{X}_s$}}

\put(-115,25){\tiny{$i_{s+1}$}}

\put(-117,32){\small{$\mathbb{X}_{s+1}$}}

\put(-130,0){\vector(0,1){7.5}}

\put(-130,7.5){\line(0,1){7.5}}

\put(-130,15){\vector(-1,1){7.5}}

\put(-137.5,22.5){\line(-1,1){7.5}}

\put(-130,15){\vector(1,1){7.5}}

\put(-122.5,22.5){\line(1,1){7.5}}

\put(-131,10){\line(1,0){2}}

\put(-129,6){\small{$\mathbb{A}$}}

\put(-119,15){$\cdots$}

\put(-130,0){\vector(3,2){22.5}}

\put(-107.5,15){\line(3,2){22.5}}

\put(-84,25){\tiny{$i_k$}}

\put(-87,32){\small{$\mathbb{X}_k$}}


\put(-123,-2){$v$}


\put(-152.5,-15){\line(3,2){22.5}}

\put(-175,-30){\vector(3,2){22.5}}

\put(-178,-27){\tiny{$j_1$}}

\put(-178,-38){\small{$\mathbb{Y}_1$}}

\put(-152,-20){$\cdots$}

\put(-137.5,-15){\line(1,2){7.5}}

\put(-145,-30){\vector(1,2){7.5}}

\put(-150,-27){\tiny{$j_t$}}

\put(-148,-38){\small{$\mathbb{Y}_t$}}

\put(-122.5,-15){\line(-1,2){7.5}}

\put(-115,-30){\vector(-1,2){7.5}}

\put(-115,-27){\tiny{$j_{t+1}$}}

\put(-117,-38){\small{$\mathbb{Y}_{t+1}$}}

\put(-119,-20){$\cdots$}

\put(-107.5,-15){\line(-3,2){22.5}}

\put(-85,-30){\vector(-3,2){22.5}}

\put(-84,-27){\tiny{$j_l$}}

\put(-87,-38){\small{$\mathbb{Y}_l$}}

\put(-130,-50){$\Gamma_1$}


\put(0,0){\vector(-3,2){22.5}}

\put(-22.5,15){\line(-3,2){22.5}}

\put(-48,25){\tiny{$i_1$}}

\put(-48,32){\small{$\mathbb{X}_1$}}

\put(-22,15){$\cdots$}

\put(0,0){\vector(-1,2){7.5}}

\put(-7.5,15){\line(-1,2){7.5}}

\put(-20,25){\tiny{$i_s$}}

\put(-18,32){\small{$\mathbb{X}_s$}}

\put(0,0){\vector(1,2){7.5}}

\put(7.5,15){\line(1,2){7.5}}

\put(15,25){\tiny{$i_{s+1}$}}

\put(13,32){\small{$\mathbb{X}_{s+1}$}}

\put(11,15){$\cdots$}

\put(0,0){\vector(3,2){22.5}}

\put(22.5,15){\line(3,2){22.5}}

\put(46,25){\tiny{$i_k$}}

\put(43,32){\small{$\mathbb{X}_k$}}


\put(7,-2){$v$}


\put(-22.5,-15){\line(3,2){22.5}}

\put(-45,-30){\vector(3,2){22.5}}

\put(-48,-27){\tiny{$j_1$}}

\put(-48,-38){\small{$\mathbb{Y}_1$}}

\put(-22,-20){$\cdots$}

\put(-7.5,-15){\line(1,2){7.5}}

\put(-15,-30){\vector(1,2){7.5}}

\put(-20,-27){\tiny{$j_t$}}

\put(-18,-38){\small{$\mathbb{Y}_t$}}

\put(7.5,-15){\line(-1,2){7.5}}

\put(15,-30){\vector(-1,2){7.5}}

\put(15,-27){\tiny{$j_{t+1}$}}

\put(13,-38){\small{$\mathbb{Y}_{t+1}$}}

\put(11,-20){$\cdots$}

\put(22.5,-15){\line(-3,2){22.5}}

\put(45,-30){\vector(-3,2){22.5}}

\put(46,-27){\tiny{$j_l$}}

\put(43,-38){\small{$\mathbb{Y}_l$}}

\put(0,-50){$\Gamma$}


\put(130,0){\vector(-3,2){22.5}}

\put(107.5,15){\line(-3,2){22.5}}

\put(82,25){\tiny{$i_1$}}

\put(82,32){\small{$\mathbb{X}_1$}}

\put(108,15){$\cdots$}

\put(130,0){\vector(-1,2){7.5}}

\put(122.5,15){\line(-1,2){7.5}}

\put(110,25){\tiny{$i_s$}}

\put(112,32){\small{$\mathbb{X}_s$}}

\put(130,0){\vector(1,2){7.5}}

\put(137.5,15){\line(1,2){7.5}}

\put(145,25){\tiny{$i_{s+1}$}}

\put(143,32){\small{$\mathbb{X}_{s+1}$}}

\put(141,15){$\cdots$}

\put(130,0){\vector(3,2){22.5}}

\put(152.5,15){\line(3,2){22.5}}

\put(176,25){\tiny{$i_k$}}

\put(173,32){\small{$\mathbb{X}_k$}}


\put(137,-2){$v$}


\put(107.5,-15){\line(3,2){22.5}}

\put(85,-30){\vector(3,2){22.5}}

\put(82,-27){\tiny{$j_1$}}

\put(82,-38){\small{$\mathbb{Y}_1$}}

\put(108,-20){$\cdots$}

\put(115,-30){\vector(1,1){7.5}}

\put(122.5,-22.5){\line(1,1){7.5}}

\put(110,-27){\tiny{$j_t$}}

\put(112,-38){\small{$\mathbb{Y}_t$}}

\put(145,-30){\vector(-1,1){7.5}}

\put(137.5,-22.5){\line(-1,1){7.5}}

\put(145,-27){\tiny{$j_{t+1}$}}

\put(143,-38){\small{$\mathbb{Y}_{t+1}$}}

\put(130,-7.5){\line(0,1){7.5}}

\put(130,-15){\vector(0,1){7.5}}

\put(129,-5){\line(1,0){2}}

\put(132,-12){\small{$\mathbb{B}$}}

\put(141,-20){$\cdots$}

\put(152.5,-15){\line(-3,2){22.5}}

\put(175,-30){\vector(-3,2){22.5}}

\put(176,-27){\tiny{$j_l$}}

\put(173,-38){\small{$\mathbb{Y}_l$}}

\put(130,-50){$\Gamma_2$}
\end{picture}

\caption{}\label{edge-contraction-figure}

\end{figure}

\begin{lemma}\label{edge-contraction}
Let $\Gamma$, $\Gamma_1$ and $\Gamma_2$ be MOY graphs shown in Figure \ref{edge-contraction-figure}. Then $C_f(\Gamma_1) \simeq C_f(\Gamma_2) \simeq C_f(\Gamma)$.
\end{lemma}
\begin{proof}
Follows from Lemma \ref{nice-b-contract}. See the proof of \cite[Lemma 5.12]{Wu-color} for full details.
\end{proof}

\begin{figure}[ht]

\setlength{\unitlength}{1pt}

\begin{picture}(360,100)(-180,-50)


\put(-100,25){$\Gamma_1$:}

\put(-60,10){\vector(0,1){10}}

\put(-60,20){\vector(-1,1){20}}

\put(-60,20){\vector(1,1){10}}

\put(-50,30){\vector(-1,1){10}}

\put(-50,30){\vector(1,1){10}}

\put(-75,3){\tiny{$i+j+k$}}

\put(-55,21){\tiny{$j+k$}}

\put(-80,42){\tiny{$i$}}

\put(-60,42){\tiny{$j$}}

\put(-40,42){\tiny{$k$}}


\put(20,25){$\Gamma'_1$:}

\put(60,10){\vector(0,1){10}}

\put(60,20){\vector(1,1){20}}

\put(60,20){\vector(-1,1){10}}

\put(50,30){\vector(1,1){10}}

\put(50,30){\vector(-1,1){10}}

\put(45,3){\tiny{$i+j+k$}}

\put(38,21){\tiny{$i+j$}}

\put(80,42){\tiny{$k$}}

\put(60,42){\tiny{$j$}}

\put(40,42){\tiny{$i$}}


\put(-100,-25){$\Gamma_2$:}

\put(-60,-30){\vector(0,-1){10}}

\put(-80,-10){\vector(1,-1){20}}

\put(-50,-20){\vector(-1,-1){10}}

\put(-60,-10){\vector(1,-1){10}}

\put(-40,-10){\vector(-1,-1){10}}

\put(-75,-47){\tiny{$i+j+k$}}

\put(-55,-29){\tiny{$j+k$}}

\put(-80,-8){\tiny{$i$}}

\put(-60,-8){\tiny{$j$}}

\put(-40,-8){\tiny{$k$}}


\put(20,-25){$\Gamma'_2$:}

\put(60,-30){\vector(0,-1){10}}

\put(80,-10){\vector(-1,-1){20}}

\put(50,-20){\vector(1,-1){10}}

\put(60,-10){\vector(-1,-1){10}}

\put(40,-10){\vector(1,-1){10}}

\put(45,-47){\tiny{$i+j+k$}}

\put(38,-29){\tiny{$i+j$}}

\put(80,-8){\tiny{$k$}}

\put(60,-8){\tiny{$j$}}

\put(40,-8){\tiny{$i$}}

\end{picture}

\caption{}\label{contract-expand-figure}

\end{figure}

\begin{corollary}\label{contract-expand}
Suppose that $\Gamma_1$, $\Gamma'_1$, $\Gamma_2$ and $\Gamma'_2$ are MOY graphs shown in Figure \ref{contract-expand-figure}. Then $C_f(\Gamma_1) \simeq C_f(\Gamma'_1)$ and $C_f(\Gamma_2) \simeq C_f(\Gamma'_2)$.
\end{corollary}
\begin{proof}
This is a special case of Lemma \ref{edge-contraction}.
\end{proof}

\subsection{Direct sum decomposition (II)} The proof of direct sum decomposition (II) in \cite{Wu-color} applies to $C_f$ without change. 

\begin{figure}[ht]

\setlength{\unitlength}{1pt}

\begin{picture}(360,70)(-180,-10)


\put(-60,0){\vector(0,1){15}}

\qbezier(-60,15)(-70,15)(-70,25)

\put(-70,25){\vector(0,1){10}}

\qbezier(-70,35)(-70,45)(-60,45)

\put(-71,26){\line(1,0){2}}

\qbezier(-60,15)(-50,15)(-50,25)

\put(-50,25){\vector(0,1){10}}

\qbezier(-50,35)(-50,45)(-60,45)

\put(-51,26){\line(1,0){2}}

\put(-60,45){\vector(0,1){15}}

\put(-65,55){\tiny{$n$}}

\put(-58,54){\small{$\mathbb{Y}$}}

\put(-65,0){\tiny{$n$}}

\put(-58,0){\small{$\mathbb{X}$}}

\put(-77,30){\tiny{$m$}}

\put(-77,22){\small{$\mathbb{A}$}}

\put(-47,30){\tiny{$n-m$}}

\put(-48,22){\small{$\mathbb{B}$}}

\put(-63,-10){$\Gamma$}


\put(60,0){\vector(0,1){60}}

\put(62,54){\small{$\mathbb{Y}$}}

\put(55,30){\tiny{$n$}}

\put(62,0){\small{$\mathbb{X}$}}

\put(57,-10){$\Gamma_1$}

\end{picture}

\caption{}\label{decomp-II-figure}

\end{figure}

\begin{theorem}[Direction Sum Decomposition (II)]\label{decomp-II}
Suppose that $\Gamma$ and $\Gamma_1$ are MOY graphs shown in Figure \ref{decomp-II-figure}, where $n\geq m \geq 0$. Then 
\[
C_f(\Gamma) \simeq C_f(\Gamma_1)\{\qb{n}{m}\}.
\]
\end{theorem}
\begin{proof}
As in the proof of \cite[Theorem 5.14]{Wu-color}, we observe that
\[
C_f(\Gamma) \simeq C_f(\Gamma_1) \otimes_{\Sym(\mathbb{A}\cup\mathbb{B})} \Sym(\mathbb{A}|\mathbb{B}) \cong C_f(\Gamma_1)\{\qb{n}{m}\}.
\]
See \cite[Theorem 5.14]{Wu-color} for more details.
\end{proof}

\subsection{Colored circles}

\begin{figure}[ht]

\setlength{\unitlength}{1pt}

\begin{picture}(360,60)(-180,0)


\qbezier(0,60)(-20,60)(-20,45)

\qbezier(0,60)(20,60)(20,45)

\put(-20,15){\vector(0,1){30}}

\put(20,15){\line(0,1){30}}

\qbezier(0,0)(-20,0)(-20,15)

\qbezier(0,0)(20,0)(20,15)

\put(19,30){\line(1,0){2}}

\put(-15,30){\tiny{$m$}}

\put(25,25){\small{$\mathbb{X}$}}
 
\end{picture}

\caption{}\label{circle-module-figure}

\end{figure}

\begin{proposition}\label{circle-module}
Assume $\bigcirc_m$ is the circle colored by $m~(\leq N)$ in Figure \ref{circle-module-figure}. Let $\mathbb{B}$ be an alphabet of $N$ indeterminates. Identify $R_B$ and $\Sym(\mathbb{B})$ by identifying $B_j$ with the $j$-th elementary symmetric polynomial in $\mathbb{B}$. Denote by $\mathcal{H}$ the ideal of $\Sym(\mathbb{X}|\mathbb{B})$ generated by $\{h_N(\mathbb{X}-\mathbb{B}),h_{N-1}(\mathbb{X}-\mathbb{B}),\dots,h_{N+1-m}(\mathbb{X}-\mathbb{B})\}$. Then, as graded matrix factorizations over $R_B$, 
\begin{eqnarray*}
C_f(\bigcirc_m)
& \simeq & C_f(\emptyset) \otimes_{R_B=\Sym(\mathbb{B})} \Sym(\mathbb{X}|\mathbb{B})/\mathcal{H} ~\{q^{-m(N-m)}\} \left\langle m \right\rangle \\
& \cong & C_f(\emptyset) \{\qb{N}{m}\}\left\langle m \right\rangle,
\end{eqnarray*}
where $\mathbb{X}$ is an alphabet of $m$ indeterminates. 
\end{proposition}
\begin{proof}
By definition,
\[
C(\bigcirc_m) =
\left(%
\begin{array}{cc}
  U_1 & 0 \\
  \dots & \dots \\
  U_m & 0 
\end{array}%
\right)_{\Sym(\mathbb{X})},
\]
where we choose $U_j=\frac{\partial f(\mathbb{X})}{\partial X_j}$. By Lemma \ref{power-derive-complete}, we know 
\[
U_j=(-1)^{j+1} (N+1) h_{N+1-j}(\mathbb{X}-\mathbb{B}).
\]
By Theorem \ref{part-symm-str-alter} and \cite[Proposition 6.2]{Wu-color}, $(\Sym(\mathbb{X}|\mathbb{B}), \{U_1,\dots,U_m\})$ is a nice pair over $R_B$. So by Lemma \ref{nice-a-contract}, 
\[
C_f(\bigcirc_m) \simeq C_f(\emptyset) \otimes_{R_B=\Sym(\mathbb{B})} \Sym(\mathbb{X}|\mathbb{B})/\mathcal{H} ~\{q^{-m(N-m)}\} \left\langle m \right\rangle.
\]
By Theorem \ref{part-symm-str-alter}, one can easily see that
\[
C_f(\emptyset) \otimes_{R_B=\Sym(\mathbb{B})} \Sym(\mathbb{X}|\mathbb{B})/\mathcal{H} ~\{q^{-m(N-m)}\} \left\langle m \right\rangle \cong C_f(\emptyset) \{\qb{N}{m}\}\left\langle m \right\rangle.
\]
\end{proof}

\section{Morphisms Induced by Local Changes of MOY Graphs}\label{sec-some-morph}

First, we recall some terminology.

\begin{definition}
If $M,~M'$ are matrix factorizations of the same potential over a graded commutative unital $\C$-algebra and $f,g:M\rightarrow M'$ are morphisms of matrix factorizations, we write $f \approx g$ if $\exists ~c\in \C\setminus \{0\}$ such that $f \simeq c \cdot g$.
\end{definition}

We say that two MOY graphs $\Gamma_1$ and $\Gamma_2$ have the same boundary condition if there is a one-to-one correspondence between their end points such that
\begin{itemize}
	\item every exit corresponds to an exit, and every entrance corresponds to an entrance,
	\item edges adjacent to corresponding end points have the same color.
\end{itemize}

Suppose MOY graphs $\Gamma_1$ and $\Gamma_2$ have the same boundary condition. Mark $\Gamma_1,\Gamma_2$ so that every pair of corresponding end points are assigned the same alphabet, and alphabets associated to internal marked points are pairwise disjoint. Let $\mathbb{E}_1,\mathbb{E}_2,\dots,\mathbb{E}_n$ be the alphabets assigned to the end points of $\Gamma_1,\Gamma_2$. Define $R_{\partial}=\Sym(\mathbb{E}_1|\mathbb{E}_2|\dots|\mathbb{E}_n)$ and $\tilde{R}_\partial = R_\partial \otimes_\C R_B$. Note that $C_f(\Gamma_1)$ and $C_f(\Gamma_2)$ are both objects of the category $\hmf_{\tilde{R}_\partial, w}$ (and the category $\HMF_{\tilde{R}_\partial, w}$), where $w = \sum \pm f(\mathbb{E}_i) \in \tilde{R}_\partial$, and the sign depends on whether the end point is an entrance or an exit.

Recall that the morphism space $\Hom_{\HMF_{\tilde{R}_\partial, w}}(C_f(\Gamma_1),C_f(\Gamma_2))$ is isomorphic to the homology of the chain complex $\Hom_{\tilde{R}_\partial}(C_f(\Gamma_1),C_f(\Gamma_2))$. Also, the morphism space $\Hom_{\hmf_{\tilde{R}_\partial, w}}(C_f(\Gamma_1),C_f(\Gamma_2))$ is the subspace of $\Hom_{\HMF_{\tilde{R}_\partial, w}}(C_f(\Gamma_1),C_f(\Gamma_2))$ of homogeneous elements that preserves both the $\zed_2$-grading and the total polynomial grading. When the set up is clear from context, we drop $\tilde{R}_\partial$ and $w$ from the notation and simply write $\Hom_{\HMF}(C_f(\Gamma_1),C_f(\Gamma_2))$ and $\Hom_{\hmf}(C_f(\Gamma_1),C_f(\Gamma_2))$. 

From Lemma \ref{nice-b-contract}, it is easy to see that $\Hom_{\HMF}(C_f(\Gamma_1),C_f(\Gamma_2))$ and, therefore, $\Hom_{\hmf}(C_f(\Gamma_1),C_f(\Gamma_2))$ do not depend on the choice of the marking.

In this section, we show that certain local changes of MOY graphs induce morphisms of matrix factorizations. The constructions of these morphisms are very similar to those in \cite[Section 7]{Wu-color}. Moreover, we will show that the functor $\varpi_0$ (see Lemma \ref{MOY-object-of-hmf}) changes the morphisms defined here to the corresponding morphism defined in \cite[Section 7]{Wu-color}.

\subsection{Bouquet move} We call the moves in Figure \ref{bouquet-move-figure} bouquet moves. From Corollary \ref{contract-expand}, we know bouquet moves induce homotopy equivalence of matrix factorizations. Next, we show that, up to homotopy and scaling, a bouquet move induces a unique homotopy equivalence  of matrix factorizations.

\begin{figure}[ht]

\setlength{\unitlength}{1pt}

\begin{picture}(360,100)(-180,-50)


\put(-100,25){$\Gamma_1$:}

\put(-60,10){\vector(0,1){10}}

\put(-60,20){\vector(-1,1){20}}

\put(-60,20){\vector(1,1){10}}

\put(-50,30){\vector(-1,1){10}}

\put(-50,30){\vector(1,1){10}}

\put(-75,3){\tiny{$i+j+k$}}

\put(-55,21){\tiny{$j+k$}}

\put(-80,42){\tiny{$i$}}

\put(-60,42){\tiny{$j$}}

\put(-40,42){\tiny{$k$}}


\put(-15,25){$\longleftrightarrow$}


\put(20,25){$\Gamma'_1$:}

\put(60,10){\vector(0,1){10}}

\put(60,20){\vector(1,1){20}}

\put(60,20){\vector(-1,1){10}}

\put(50,30){\vector(1,1){10}}

\put(50,30){\vector(-1,1){10}}

\put(45,3){\tiny{$i+j+k$}}

\put(38,21){\tiny{$i+j$}}

\put(80,42){\tiny{$k$}}

\put(60,42){\tiny{$j$}}

\put(40,42){\tiny{$i$}}


\put(-100,-25){$\Gamma_2$:}

\put(-60,-30){\vector(0,-1){10}}

\put(-80,-10){\vector(1,-1){20}}

\put(-50,-20){\vector(-1,-1){10}}

\put(-60,-10){\vector(1,-1){10}}

\put(-40,-10){\vector(-1,-1){10}}

\put(-75,-47){\tiny{$i+j+k$}}

\put(-55,-29){\tiny{$j+k$}}

\put(-80,-8){\tiny{$i$}}

\put(-60,-8){\tiny{$j$}}

\put(-40,-8){\tiny{$k$}}


\put(-15,-25){$\longleftrightarrow$}


\put(20,-25){$\Gamma'_2$:}

\put(60,-30){\vector(0,-1){10}}

\put(80,-10){\vector(-1,-1){20}}

\put(50,-20){\vector(1,-1){10}}

\put(60,-10){\vector(-1,-1){10}}

\put(40,-10){\vector(1,-1){10}}

\put(45,-47){\tiny{$i+j+k$}}

\put(38,-29){\tiny{$i+j$}}

\put(80,-8){\tiny{$k$}}

\put(60,-8){\tiny{$j$}}

\put(40,-8){\tiny{$i$}}

\end{picture}

\caption{}\label{bouquet-move-figure}

\end{figure}

\begin{lemma}\label{bouquet-move-lemma}
Suppose that $\Gamma_1$, $\Gamma'_1$, $\Gamma_2$ and $\Gamma'_2$ are MOY graphs shown in Figure \ref{bouquet-move-figure}. Then, as $\zed_2\oplus\zed$-graded vector spaces over $\C$, 
\begin{eqnarray*}
& & \Hom_\HMF (C_f(\Gamma_1), C_f(\Gamma'_1)) \cong \Hom_\HMF (C_f(\Gamma_2),C_f(\Gamma'_2)) \\
& \cong &  C_f(\emptyset)\{\qb{N}{i+j+k}\qb{i+j+k}{k}\qb{i+j}{j}q^{(i+j+k)(N-i-j-k)+ij+jk+ki}\}.
\end{eqnarray*}
In particular, $\Hom_\hmf (C_f(\Gamma_1), C_f(\Gamma'_1)) \cong \Hom_\hmf (C_f(\Gamma_2),C_f(\Gamma'_2)) \cong \C$. 

Therefore, the homotopy equivalence of matrix factorizations induced by the bouquet move is unique up to homotopy and scaling.
\end{lemma}

\begin{figure}[ht]

\setlength{\unitlength}{1pt}

\begin{picture}(360,75)(-180,-15)

\qbezier(-40,50)(-40,60)(0,60)

\qbezier(40,50)(40,60)(0,60)

\qbezier(-40,10)(-40,0)(0,0)

\qbezier(40,10)(40,0)(0,0)

\put(40,50){\vector(0,-1){40}}

\qbezier(-40,10)(0,15)(0,30)

\qbezier(-40,50)(0,45)(0,30)

\put(0,30){\vector(0,1){0}}

\qbezier(-40,10)(-60,10)(-60,20)

\qbezier(-40,50)(-60,50)(-60,40)

\put(-60,20){\vector(0,1){0}}

\qbezier(-60,20)(-80,20)(-80,30)

\qbezier(-60,20)(-40,20)(-40,30)

\qbezier(-60,40)(-80,40)(-80,30)

\qbezier(-60,40)(-40,40)(-40,30)

\put(-40,50){\vector(1,0){0}}

\put(-80,30){\vector(0,1){0}}

\put(-40,30){\vector(0,1){0}}

\put(45,30){\tiny{$i+j+k$}}

\put(-75,10){\tiny{$i+j$}}

\put(-75,45){\tiny{$i+j$}}

\put(5,30){\tiny{$k$}}

\put(-35,30){\tiny{$j$}}

\put(-90,30){\tiny{$i$}}

\put(-5,-15){$\Gamma$}

\end{picture}

\caption{}\label{bouquet-move-figure-2}

\end{figure}

\begin{proof}
Similar to the proof of \cite[Lemma 7.4]{Wu-color}, we only compute $\Hom_\HMF (C_f(\Gamma_1), C_f(\Gamma'_1))$. The computation of $\Hom_\HMF (C_f(\Gamma_2),C_f(\Gamma'_2))$ is similar. By Corollary \ref{contract-expand}, one can see that 
\begin{eqnarray*}
\Hom_\HMF (C_f(\Gamma_1), C_f(\Gamma'_1)) & \cong & \Hom_\HMF (C_f(\Gamma'_1), C_f(\Gamma'_1)) \\
& \cong & H_f(\Gamma)\left\langle i+j+k \right\rangle\{q^{(i+j+k)(N-i-j-k)+ij+jk+ki}\},
\end{eqnarray*}
where $\Gamma$ is the MOY graph in Figure \ref{bouquet-move-figure-2}. Using Decomposition (II) (Theorems \ref{decomp-II}) and Proposition \ref{circle-module}, we have that
\[
H_f(\Gamma) \cong C_f(\emptyset)\left\langle i+j+k \right\rangle\{\qb{N}{i+j+k}\qb{i+j+k}{k}\qb{i+j}{j}\}.
\]
And the lemma follows.
\end{proof}

\begin{lemma}\label{varpi-0-reduce-bouquet}
Let $\varpi_0$ be the functor given in Lemma \ref{MOY-object-of-hmf}. For the MOY graphs $\Gamma_1$, $\Gamma'_1$, $\Gamma_2$ and $\Gamma'_2$ Figure \ref{bouquet-move-figure}, denote by $h_i: C_f(\Gamma_i) \rightarrow C_f(\Gamma_i')$ the homotopy equivalence induced by the bouquet move. Then, up to homotopy and scaling, $\varpi_0(h_i): C(\Gamma_i) \rightarrow C(\Gamma_i')$ is the homotopy equivalence induced by bouquet move given in \cite[Subsection 7.2]{Wu-color}.
\end{lemma}

\begin{proof}
By Lemma \ref{reduce-base-homotopic-equivalence}, $\varpi_0(h_i)$ is a homotopy equivalence. By \cite[Lemma 7.4]{Wu-color}, up to homotopy and scaling, there is only one morphism $C(\Gamma_i) \rightarrow C(\Gamma_i')$ that preserves both gradings and is not null-homotopic. And the lemma follows.
\end{proof}

\subsection{Circle creation and annihilation}

\begin{lemma}\label{circle-empty-hmf}
Let $\bigcirc_m$ be a circle colored by $m$. Then, as $\zed_2\oplus\zed$-graded $R_B$-modules,
\[
\Hom_{HMF}(C_f(\bigcirc_m),C_f(\emptyset)) \cong \Hom_{HMF}(C_f(\emptyset),C_f(\bigcirc_m)) \cong C_f(\emptyset)\{\qb{N}{m}\}\left\langle m \right\rangle,
\]
where $C_f(\emptyset)$ is the matrix factorization $R_B\rightarrow 0 \rightarrow R_B$.

In particular, the subspaces of $\Hom_{HMF}(C_f(\emptyset),C_f(\bigcirc_m))$ and $\Hom_{HMF}(C_f(\bigcirc_m),C_f(\emptyset))$ of homogeneous elements of total polynomial degree $-m(N-m)$ are $1$-dimensional. This leads to the following definitions.
\end{lemma}

\begin{proof}
It is clear that, as $\zed_2\oplus\zed$-graded $R_B$-modules, $\Hom_{HMF}(C_f(\emptyset),C_f(\emptyset)) \cong C_f(\emptyset)$. By Proposition \ref{circle-module}, $C_f(\bigcirc_m) \simeq C_f(\emptyset)\{\qb{N}{m}\}\left\langle m \right\rangle$. So
\begin{eqnarray*}
\Hom_{HMF}(C_f(\bigcirc_m),C_f(\emptyset)) & \cong & \Hom_{HMF}(C_f(\emptyset),C_f(\emptyset)) \{\qb{N}{m}\}\left\langle m \right\rangle \\
& \cong & C_f(\emptyset)\{\qb{N}{m}\}\left\langle m \right\rangle
\end{eqnarray*}
and, similarly,
\[
\Hom_{HMF}(C_f(\emptyset),C_f(\bigcirc_m)) \cong C_f(\emptyset)\{\qb{N}{m}\}\left\langle m \right\rangle.
\]
\end{proof}

\begin{definition}
Let $\bigcirc_m$ be a circle colored by $m$. Associate to the circle creation a homogeneous morphism 
\[
\iota: C_f(\emptyset)(\cong R_B) \rightarrow C_f(\bigcirc_m)
\] 
of total polynomial degree $-m(N-m)$ not homotopic to $0$.

Associate to the circle annihilation a homogeneous morphism 
\[
\epsilon:C_f(\bigcirc_m) \rightarrow C_f(\emptyset) (\cong R_B)
\] 
of total polynomial degree $-m(N-m)$ not homotopic to $0$.

By Lemma \ref{circle-empty-hmf}, $\iota$ and $\epsilon$ are unique up to homotopy and scaling. Both of them have $\zed_2$-degree $m$.
\end{definition}

Mark $\bigcirc_m$ by a single alphabet $\mathbb{X}$. Let $\mathbb{B}$ be an alphabet of $N$ indeterminates. Identify $R_B$ and $\Sym(\mathbb{B})$ as in Proposition \ref{circle-module}. By Proposition \ref{circle-module}, we know that there is an ($\Sym(\mathbb{X}|\mathbb{B})$-linear) homotopy equivalence of matrix factorizations over $R_B$
\[
P: C_f(\bigcirc_m) \rightarrow C_f(\emptyset) \otimes_{R_B} \Sym(\mathbb{X}|\mathbb{B})/ \mathcal{H} \{q^{-m(N-m)}\} \left\langle m \right\rangle,
\]
where $\mathcal{H}$ is the ideal of $\Sym(\mathbb{X}|\mathbb{B})$ generated by 
\[
\{h_N(\mathbb{X}-\mathbb{B}), h_{N-1}(\mathbb{X}-\mathbb{B}), \dots, h_{N+1-m}(\mathbb{X}-\mathbb{B})\}.
\]
Let $Q$ be an ($R_B$-linear) homotopic inverse of $P$. $P$ and $Q$ induce quasi-isomorphisms of $\zed_2$-graded chain complexes
\[
\Hom_{R_B}(\Sym(\mathbb{X}|\mathbb{B})/ \mathcal{H} \{q^{-m(N-m)}\} \left\langle m \right\rangle, C_f(\emptyset)) \xrightarrow{P^\sharp} \Hom_{R_B} (C_f(\bigcirc_m), C_f(\emptyset)).
\]
and 
\[
\Hom_{R_B}(C_f(\emptyset),\Sym(\mathbb{X}|\mathbb{B})/ \mathcal{H} \{q^{-m(N-m)}\} \left\langle m \right\rangle) \xrightarrow{Q_\sharp} \Hom_{R_B} (C_f(\emptyset),C_f(\bigcirc_m)).
\]

Let $\jmath: R_B \rightarrow \Sym(\mathbb{X}|\mathbb{B})/ \mathcal{H}$ be the standard $R_B$-linear inclusion given by $\jmath(1)=1$, and $\zeta$ the Sylvester operator given in Theorem \ref{part-symm-str-alter}. Then $P^\sharp(\zeta) = \zeta \circ P$ and $Q_\sharp (\jmath)= Q \circ \jmath$ are $R_B$-linear homogeneous morphisms of total polynomial degree $-m(N-m)$. Denote by $\mathfrak{m}(\ast)$ the morphism $C_f(\bigcirc_m) \rightarrow C_f(\bigcirc_m)$ given by the multiplication of $\ast$. Then, by Theorem \ref{part-symm-str-alter}, we have
\begin{eqnarray*}
&& P^\sharp(\zeta) \circ \mathfrak{m}(S_\lambda(\mathbb{X}) \cdot S_\mu(\mathbb{X}-\mathbb{B})) \circ Q_\sharp (\jmath) \\
& = & \zeta \circ P \circ \mathfrak{m}(S_\lambda(\mathbb{X}) \cdot S_\mu(\mathbb{X}-\mathbb{B})) \circ Q \circ \jmath \\
& = & \zeta \circ \mathfrak{m}(S_\lambda(\mathbb{X}) \cdot S_\mu(\mathbb{X}-\mathbb{B})) \circ P \circ Q \circ \jmath \\
& \simeq & \zeta \circ \mathfrak{m}(S_\lambda(\mathbb{X}) \cdot S_\mu(\mathbb{X}-\mathbb{B})) \circ \jmath \\
& = & 
\begin{cases}
    \id_{R_B} & \text{if } \lambda_j + \mu_{m+1-j} =N-m ~\forall j=1,\dots,m, \\
    0 & \text{otherwise.}
\end{cases}
\end{eqnarray*}
This implies that $P^\sharp(\zeta)$ and $Q_\sharp (\jmath)$ are not null-homotopic. Therefore, $\epsilon \approx P^\sharp(\zeta)$ and $\iota \approx Q_\sharp (\jmath)$. And we have the following corollary.

\begin{corollary}\label{iota-epsilon-composition}
Denote by $\mathfrak{m}(\ast)$ the morphism $C_f(\bigcirc_m)\rightarrow C_f(\bigcirc_m)$ induced by multiplication by $\ast$. Then, for any $\lambda, \mu \in \Lambda_{m,N-m}$,
\begin{equation}\label{iota-epsilon-composition-eq}
\epsilon \circ \mathfrak{m}(S_\lambda(\mathbb{X}) \cdot S_\mu(\mathbb{X}-\mathbb{B})) \circ \iota \approx
\begin{cases}
\id_{C_f(\emptyset)} & \text{if } \lambda_j + \mu_{m+1-j} =N-m ~\forall j=1,\dots,m, \\
0 & \text{otherwise.} 
\end{cases}
\end{equation}
\end{corollary}

\begin{lemma}\label{varpi-0-reduce-circle}
Let $\varpi_0$ be the functor given in Lemma \ref{MOY-object-of-hmf}. Then $\varpi_0(\iota): C(\emptyset) \rightarrow C(\bigcirc_m)$ and $\varpi_0(\epsilon): C(\bigcirc_m) \rightarrow C(\emptyset)$ are the morphisms associated to circle creation and annihilation defined in \cite[Definition 7.7]{Wu-color}.
\end{lemma}
\begin{proof}
By Corollary \ref{iota-epsilon-composition}, one can see that 
\[
\varpi_0(\epsilon) \circ \mathfrak{m}(S_\lambda(\mathbb{X}) \cdot S_\mu(\mathbb{X})) \circ \varpi_0(\iota) \approx
\begin{cases}
\id_{C(\emptyset)} & \text{if } \lambda_j + \mu_{m+1-j} =N-m ~\forall j=1,\dots,m, \\
0 & \text{otherwise.} 
\end{cases}
\]
This implies that $\varpi_0(\iota)$ and $\varpi_0(\epsilon)$ are not null-homotopic. Note that they are both homogeneous of quantum degree $-m(N-m)$. The lemma then follows from \cite[Lemma 7.6 and Definition 7.7]{Wu-color}.
\end{proof}

\subsection{Edge splitting and merging} Let $\Gamma_0$ and $\Gamma_1$ be the MOY graphs in Figure \ref{edge-splitting}. We call the change $\Gamma_0\leadsto\Gamma_1$ an edge splitting and the change $\Gamma_1\leadsto\Gamma_0$ an edge merging. In this subsection, we define morphisms $\phi$ and $\overline{\phi}$ associated to edge splitting and merging.

\begin{figure}[ht]

\setlength{\unitlength}{1pt}

\begin{picture}(360,75)(-180,-90)


\put(-67,-45){\tiny{$m+n$}}

\put(-70,-75){\vector(0,1){50}}

\put(-71,-50){\line(1,0){2}}

\put(-67,-30){\small{$\mathbb{X}$}}

\put(-95,-53){\small{$\mathbb{A}\cup\mathbb{E}$}}

\put(-67,-75){\small{$\mathbb{Y}$}}

\put(-75,-90){$\Gamma_0$}


\put(-25,-50){\vector(1,0){50}}

\put(25,-60){\vector(-1,0){50}}

\put(-5,-47){\small{$\phi$}}

\put(-5,-70){\small{$\overline{\phi}$}}


\put(70,-75){\vector(0,1){10}}

\put(70,-35){\vector(0,1){10}}

\qbezier(70,-65)(60,-65)(60,-55)

\qbezier(70,-35)(60,-35)(60,-45)

\put(60,-55){\vector(0,1){10}}

\put(59,-55){\line(1,0){2}}

\qbezier(70,-65)(80,-65)(80,-55)

\qbezier(70,-35)(80,-35)(80,-45)

\put(80,-55){\vector(0,1){10}}

\put(79,-55){\line(1,0){2}}

\put(73,-30){\tiny{$m+n$}}

\put(73,-70){\tiny{$m+n$}}

\put(83,-50){\tiny{$n$}}

\put(51,-50){\tiny{$m$}}

\put(60,-30){\small{$\mathbb{X}$}}

\put(60,-75){\small{$\mathbb{Y}$}}

\put(50,-58){\small{$\mathbb{A}$}}

\put(83,-58){\small{$\mathbb{E}$}}

\put(65,-90){$\Gamma_1$}

\end{picture}

\caption{}\label{edge-splitting}

\end{figure}

\begin{lemma}\label{edge-splitting-lemma}
Let $\Gamma_0$ and $\Gamma_1$ be the colored MOY graphs in Figure \ref{edge-splitting}. Then, as $\zed_2\oplus \zed$-graded $R_B$-modules, 
\[
\Hom_\HMF(C_f(\Gamma_0),C_f(\Gamma_0 )) \cong C_f(\emptyset) \{q^{(N-m-n)(m+n)}\qb{N}{m+n}\}
\]
and
\begin{eqnarray*}
\Hom_{HMF}(C_f(\Gamma_0),C_f(\Gamma_1)) & \cong & \Hom_{HMF}(C_f(\Gamma_1),C_f(\Gamma_0)) \\
& \cong & C_f(\emptyset) \{q^{(N-m-n)(m+n)}\qb{N}{m+n} \qb{m+n}{m}\}.
\end{eqnarray*}
In particular, the lowest total polynomial gradings of the above spaces are $-mn$, and the subspaces of these spaces of homogeneous elements of total polynomial grading $-mn$ are all $1$-dimensional.
\end{lemma}

\begin{proof}
By Theorem \ref{decomp-II}, $C_f(\Gamma_1) \simeq C_f(\Gamma_0)\{\qb{m+n}{m}\}$. So 
\[
\Hom(C_f(\Gamma_0),C_f(\Gamma_1)) \simeq \Hom(C_f(\Gamma_0),C_f(\Gamma_0 ))\{\qb{m+n}{m}\} \simeq \Hom(C_f(\Gamma_1),C_f(\Gamma_0)).
\]
Denote by $\bigcirc_{m+n}$ the circle colored by $m+n$. Then, by Proposition \ref{circle-module}, 
\begin{eqnarray*}
\Hom_{\Sym(\mathbb{X}|\mathbb{Y})}(C_f(\Gamma_0),C_f(\Gamma_0 )) & \cong & C_f(\bigcirc_{m+n})\{q^{(N-m-n)(m+n)}\} \left\langle m+n \right\rangle \\
& \simeq & C_f(\emptyset) \{q^{(N-m-n)(m+n)}\qb{N}{m+n}\},
\end{eqnarray*}
and the lemma follows.
\end{proof}

\begin{definition}\label{morphism-edge-splitting-merging-def}
Let $\Gamma_0$ and $\Gamma_1$ be the colored MOY graphs in Figure \ref{edge-splitting}. Associate to the edge splitting a homogeneous morphism 
\[
\phi: C_f(\Gamma_0) \rightarrow C_f(\Gamma_1)
\] 
of total polynomial degree $-mn$ not homotopic to $0$.

Associate to the edge merging a homogeneous morphism 
\[
\overline{\phi}:C_f(\Gamma_1) \rightarrow C_f(\Gamma_0)
\] 
of total polynomial degree $-mn$ not homotopic to $0$.
\end{definition}

\begin{lemma}\label{phibar-compose-phi}
Let $\Gamma_0$ and $\Gamma_1$ be the MOY graphs in Figure \ref{edge-splitting}. Then
\begin{equation}\label{phibar-compose-phi-eq}
\overline{\phi} \circ \mathfrak{m}(S_{\lambda}(\mathbb{A})\cdot S_{\mu}(-\mathbb{E})) \circ \phi \approx 
\begin{cases}
    \id_{C_f(\Gamma_0)} & \text{if } \lambda_j + \mu_{m+1-j} = n ~\forall j=1,\dots,m, \\ 
    0 & \text{otherwise,}
\end{cases}
\end{equation}
where $\lambda,\mu\in \Lambda_{m,n}$ and $\mathfrak{m}(S_{\lambda}(\mathbb{A})\cdot S_{\mu}(-\mathbb{E}))$ is the morphism induced by the multiplication of $S_{\lambda}(\mathbb{A})\cdot S_{\mu}(-\mathbb{E})$.
\end{lemma}

\begin{proof}
The proof of \eqref{phibar-compose-phi-eq} is nearly identical to that of \cite[Lemma 7.11]{Wu-color}. See \cite[Subsection 7.4]{Wu-color} for more details.
\end{proof}

\begin{lemma}\label{varpi-0-reduce-edge-split}
Let $\varpi_0$ be the functor given in Lemma \ref{MOY-object-of-hmf}. Then $\varpi_0(\phi):C(\Gamma_0) \rightarrow C(\Gamma_1)$ and $\varpi_0(\overline{\phi}):C_f(\Gamma_1) \rightarrow C_f(\Gamma_0)$ are the morphisms associated to edge splitting and edge merging in \cite[Definition 7.10]{Wu-color}.
\end{lemma}
\begin{proof}
From the definition of $\varpi_0$, it is clear that $\varpi_0(\id_{C_f(\Gamma_0)}) = \id_{C(\Gamma_0)}$. So, by Lemma \ref{phibar-compose-phi}, we have
\[
\varpi_0(\overline{\phi}) \circ \mathfrak{m}(S_{\lambda}(\mathbb{A})\cdot S_{\mu}(-\mathbb{E})) \circ \varpi_0(\phi) \approx 
\begin{cases}
    \id_{C(\Gamma_0)} & \text{if } \lambda_j + \mu_{m+1-j} = n ~\forall j=1,\dots,m, \\ 
    0 & \text{otherwise,}
\end{cases}
\]
Note that $\id_{C(\Gamma_0)}$ is not null-homotopic since $C(\Gamma_0)$ is not null-homotopic. So $\varpi_0(\overline{\phi})$ and $\varpi_0(\phi)$ are not null-homotopic. By \cite[Definition 7.10]{Wu-color}, they are the morphisms associated to edge splitting and edge merging for the matrix factorizations $C(\Gamma_i)$.
\end{proof}

\subsection{$\chi$-morphisms}

\begin{figure}[ht]

\setlength{\unitlength}{1pt}

\begin{picture}(360,75)(-180,-15)


\put(-120,0){\vector(1,1){20}}

\put(-100,20){\vector(1,-1){20}}

\put(-100,40){\vector(0,-1){20}}

\put(-100,40){\vector(1,1){20}}

\put(-120,60){\vector(1,-1){20}}

\put(-101,30){\line(1,0){2}}

\put(-132,45){\tiny{$_{m+n-l}$}}

\put(-115,15){\tiny{$_l$}}

\put(-90,45){\tiny{$_m$}}

\put(-90,15){\tiny{$_n$}}

\put(-95,28){\tiny{$_{n-l}$}}

\put(-130,55){\small{$\mathbb{A}$}}

\put(-75,0){\small{$\mathbb{Y}$}}

\put(-113,27){\small{$\mathbb{D}$}}

\put(-130,0){\small{$\mathbb{E}$}}

\put(-75,55){\small{$\mathbb{X}$}}

\put(-102,-15){$\Gamma_0$}


\put(-30,35){\vector(1,0){60}}

\put(30,25){\vector(-1,0){60}}

\put(-3,40){\small{$\chi^0$}}

\put(-3,15){\small{$\chi^1$}}


\put(60,10){\vector(1,1){20}}

\put(60,50){\vector(1,-1){20}}

\put(80,30){\vector(1,0){20}}

\put(100,30){\vector(1,1){20}}

\put(100,30){\vector(1,-1){20}}

\put(68,45){\tiny{$_{m+n-l}$}}

\put(70,15){\tiny{$_{l}$}}

\put(108,45){\tiny{$_m$}}

\put(106,15){\tiny{$_n$}}

\put(81,32){\tiny{$_{m+n}$}}

\put(50,45){\small{$\mathbb{A}$}}

\put(122,10){\small{$\mathbb{Y}$}}

\put(50,10){\small{$\mathbb{E}$}}

\put(122,45){\small{$\mathbb{X}$}}

\put(88,-15){$\Gamma_1$}

\end{picture}

\caption{}\label{general-general-chi-maps-figure}

\end{figure}

\begin{proposition}\label{general-general-chi-maps}
Let $\Gamma_0$ and $\Gamma_1$ be the MOY graphs in Figure \ref{general-general-chi-maps-figure}, where $1\leq l\leq n <m+n \leq N$. There exist homogeneous morphisms $\chi^0:C_f(\Gamma_0) \rightarrow C_f(\Gamma_1)$ and $\chi^1:C_f(\Gamma_1) \rightarrow C_f(\Gamma_0)$ satisfying
\begin{enumerate}[(i)]
	\item Both $\chi^0$ and $\chi^1$ have $\zed_2$-grading $0$ and total polynomial grading $ml$. $$ $$
	\item \begin{eqnarray*}
	\chi^1 \circ \chi^0 & \simeq & S_{\lambda_{l,m}}(\mathbb{E}-\mathbb{X}) \cdot \id_{C_f(\Gamma_0)} = (\sum_{\lambda\in\Lambda_{l,m}} (-1)^{|\lambda|} S_{\lambda'}(\mathbb{X}) S_{\lambda^c}(\mathbb{E})) \cdot \id_{C_f(\Gamma_0)}, \\
	\chi^0 \circ \chi^1 & \simeq & S_{\lambda_{l,m}}(\mathbb{E}-\mathbb{X}) \cdot \id_{C_f(\Gamma_1)} = (\sum_{\lambda\in\Lambda_{l,m}} (-1)^{|\lambda|} S_{\lambda'}(\mathbb{X}) S_{\lambda^c}(\mathbb{E})) \cdot \id_{C_f(\Gamma_1)}, \\
	\end{eqnarray*}
	where 
	\begin{eqnarray*}
	\lambda_{l,m} & = & (\underbrace{m\geq m \geq\cdots\geq m}_{l \text{ parts}}), \\
	\Lambda_{l,m} & = & \{\mu=(\mu_1\geq\cdots\geq\mu_l) ~|~ \mu_1 \leq m\},
	\end{eqnarray*}
	$\lambda'~(\in \Lambda_{m,l})$ is the conjugate of $\lambda ~(\in \Lambda_{l,m})$, and $\lambda^c ~(\in \Lambda_{l,m})$ is the complement of $\lambda ~(\in \Lambda_{l,m})$ in $\Lambda_{l,m}$, i.e., if $\lambda=(\lambda_1\geq\cdots\geq\lambda_l)\in \Lambda_{l,m}$, then $\lambda^c = (m-\lambda_l\geq\cdots\geq m-\lambda_1)$.
\end{enumerate}
\end{proposition}

\begin{proof}
First, by the definition of $S_{\lambda}(\mathbb{E}-\mathbb{X})$, it is easy to check that 
\[
S_{\lambda_{l,m}}(\mathbb{E}-\mathbb{X}) = \sum_{\lambda\in\Lambda_{l,m}} (-1)^{|\lambda|} S_{\lambda'}(\mathbb{X}) S_{\lambda^c}(\mathbb{E}).
\]
See \cite{Lascoux-notes} for more details.

Note that, in \cite[Subsections 7.5-7.6]{Wu-color}, only the right columns of the Koszul matrix factorizations $C(\Gamma_0)$ and $C(\Gamma_1)$ are explicitly used in the construction of the $\chi$-morphisms. But the right columns of $C_f(\Gamma_0)$ and $C_f(\Gamma_1)$ are identical to that of $C(\Gamma_0)$ and $C(\Gamma_1)$. So the construction in \cite[Subsections 7.5-7.6]{Wu-color} applies to $C_f(\Gamma_0)$ and $C_f(\Gamma_1)$ without any visible change. Thus, the morphisms $\chi^0$ and $\chi^1$ with the desired properties exist. (See \cite[Subsections 7.5-7.6]{Wu-color} for more details.)  
\end{proof}

\begin{lemma}\label{general-general-chi-maps-HMF}
Let $\Gamma_0$, $\Gamma_1$, $\chi^0$ and $\chi^1$ be as in Proposition \ref{general-general-chi-maps}. Then, up to homotopy and scaling, $\chi^0$ (resp. $\chi^1$) is the unique homotopically non-trivial homogeneous morphism of total polynomial degree $ml$ from $C(\Gamma_0)$ to $C(\Gamma_1)$ (resp. from $C(\Gamma_1)$ to $C(\Gamma_0)$.)
\end{lemma}

\begin{figure}[ht]

\setlength{\unitlength}{1pt}

\begin{picture}(360,75)(-180,-15)


\put(-130,10){\vector(1,1){10}}

\put(-120,20){\line(1,-1){10}}

\put(-120,40){\vector(0,-1){20}}

\put(-120,40){\line(1,1){10}}

\put(-130,50){\vector(1,-1){10}}

\put(-152,45){\tiny{$_{m+n-l}$}}

\put(-135,15){\tiny{$_l$}}

\put(-115,28){\tiny{$_{n-l}$}}

\qbezier(-110,10)(-100,0)(-90,10)

\qbezier(-110,50)(-100,60)(-90,50)

\qbezier(-130,10)(-140,0)(-100,0)

\qbezier(-30,10)(-20,0)(-100,0)

\qbezier(-130,50)(-140,60)(-100,60)

\qbezier(-30,50)(-30,60)(-100,60)

\put(-90,-15){$\Gamma$}

\put(-90,10){\vector(1,1){20}}

\put(-90,50){\vector(1,-1){20}}

\put(-70,30){\vector(1,0){20}}

\put(-50,30){\line(1,1){20}}

\put(-50,30){\line(1,-1){20}}

\put(-82,45){\tiny{$_{m}$}}

\put(-80,15){\tiny{$_{n}$}}

\put(-69,32){\tiny{$_{m+n}$}}


\put(30,40){\vector(3,1){30}}

\put(30,40){\line(1,1){10}}

\put(20,50){\vector(1,-1){10}}

\put(120,45){\tiny{$_{m+n-l}$}}

\put(67,45){\tiny{$_{m+n-l}$}}

\put(33,50){\tiny{$_{m}$}}

\put(120,15){\tiny{$_l$}}

\put(40,40){\tiny{$_{n-l}$}}

\qbezier(40,50)(50,60)(60,50)

\qbezier(60,10)(50,0)(90,0)

\qbezier(120,10)(130,0)(90,0)

\qbezier(20,50)(10,60)(50,60)

\qbezier(120,50)(130,60)(50,60)

\put(60,-15){$\Gamma'$}

\put(60,10){\vector(1,1){20}}

\put(60,50){\vector(1,-1){20}}

\put(80,30){\vector(1,0){20}}

\put(100,30){\line(1,1){20}}

\put(100,30){\line(1,-1){20}}

\put(81,32){\tiny{$_{m+n}$}}

\end{picture}

\caption{}\label{general-general-chi-maps-HMF-figure}

\end{figure}

\begin{proof}(Following \cite[Proposition 7.29]{Wu-color})
Similar to \cite[Lemmas 7.22 and 7.23]{Wu-color}, we can reduce $C_f(\Gamma_0)$ and $C_f(\Gamma_1)$ to Koszul matrix factorizations over the ring $\Sym(\mathbb{X}|\mathbb{Y}|\mathbb{A}|\mathbb{E}) \otimes_\C R_B$. Then it is easy to check that, as graded $R_B$-modules,
\begin{eqnarray*}
\Hom_\HMF(C_f(\Gamma_1),C_f(\Gamma_0)) & \cong & H_f(\Gamma) \left\langle m+n\right\rangle  \{q^{(m+n)(N-m-n)+mn+ml+nl-l^2}\}, \\
\Hom_\HMF(C_f(\Gamma_0),C_f(\Gamma_1)) & \cong & H_f(\overline{\Gamma}) \left\langle m+n\right\rangle \{q^{(m+n)(N-m-n)+mn+ml+nl-l^2}\}, 
\end{eqnarray*}
where $\Gamma$ is the MOY graph in Figure \ref{general-general-chi-maps-HMF-figure}, and $\overline{\Gamma}$ is $\Gamma$ with orientation reversed. Using Corollary \ref{contract-expand}, Decomposition (II) (Theorem \ref{decomp-II}) and Corollary \ref{circle-module}, we have
\[
H_f(\Gamma) \cong H_f(\overline{\Gamma}) \cong C_f(\emptyset) \left\langle m+n \right\rangle \{\qb{m+n-l}{m} \qb{m+n}{l}\qb{N}{m+n}\}.
\]
Thus, as graded $R_B$-modules,
\begin{eqnarray*}
& & \Hom_\HMF(C_f(\Gamma_1),C_f(\Gamma_0)) \\
& \cong & \Hom_\HMF(C_f(\Gamma_0),C_f(\Gamma_1)) \\
& \cong & C_f(\emptyset) \{\qb{m+n-l}{m} \qb{m+n}{l}\qb{N}{m+n} q^{(m+n)(N-m-n)+mn+ml+nl-l^2}\}. \\
\end{eqnarray*}

In particular, the lowest non-vanishing total polynomial grading of the above spaces is $ml$, and the subspaces of these spaces of homogeneous elements of total polynomial degree $ml$ are $1$-dimensional. So, to prove the proposition, we only need to show that $\chi^0$ and $\chi^1$ are homotopically non-trivial. To prove this, we use the diagram in Figure \ref{general-general-chi-maps-non-vanishing}, where $\phi_1$ and $\overline{\phi}_1$ (resp. $\phi_2$ and $\overline{\phi}_2$) are induced by the edge splitting and merging of the upper (resp. lower ) bubble, and $\chi^0$ and $\chi^1$ are the morphisms from Proposition \ref{general-general-chi-maps}.

\begin{figure}[ht]
$
\xymatrix{
\input{v-vector-m+n} \ar@<1ex>[rr]^{\phi_1\otimes\phi_2}  & &  \ar@<1ex>[ll]^{\overline{\phi}_1\otimes\overline{\phi}_2} \input{v-vector-two-bubbles}  \ar@<1ex>[rr]^{\chi^1} & & \input{v-vector-theta-shape} \ar@<1ex>[ll]^{\chi^0} \\
}
$
\caption{}\label{general-general-chi-maps-non-vanishing}

\end{figure}

Let us compute the composition 
\[
(\overline{\phi}_1\otimes\overline{\phi}_2) \circ \mathfrak{m}(S_{\lambda_{m,n}}(-\mathbb{Y}) \cdot S_{\lambda_{l,n-l}}(-\mathbb{A})) \circ \chi^0 \circ \chi^1 \circ (\phi_1\otimes\phi_2),
\]
where $\mathfrak{m}(S_{\lambda_{m,n}}(-\mathbb{Y}) \cdot S_{\lambda_{l,n-l}}(-\mathbb{A}))$ is the morphism induced by multiplication by $S_{\lambda_{m,n}}(-\mathbb{Y}) \cdot S_{\lambda_{l,n-l}}(-\mathbb{A})$. By Proposition \ref{general-general-chi-maps}, we have
\begin{eqnarray*}
&& (\overline{\phi}_1\otimes\overline{\phi}_2) \circ \mathfrak{m}(S_{\lambda_{m,n}}(-\mathbb{Y}) \cdot S_{\lambda_{l,n-l}}(-\mathbb{A})) \circ \chi^0 \circ \chi^1 \circ (\phi_1\otimes\phi_2) \\
& \simeq & (\overline{\phi}_1\otimes\overline{\phi}_2) \circ \mathfrak{m}(S_{\lambda_{m,n}}(-\mathbb{Y}) \cdot S_{\lambda_{l,n-l}}(-\mathbb{A}) \cdot (\sum_{\lambda\in\Lambda_{l,m}} (-1)^{|\lambda|} S_{\lambda'}(\mathbb{X}) S_{\lambda^c}(\mathbb{E})))  \circ (\phi_1\otimes\phi_2) \\
& = & \sum_{\lambda\in\Lambda_{l,m}} (-1)^{|\lambda|} (\overline{\phi}_1 \circ \mathfrak{m}(S_{\lambda_{m,n}}(-\mathbb{Y}) \cdot S_{\lambda'}(\mathbb{X})) \circ \phi_1) \otimes (\overline{\phi}_2 \circ \mathfrak{m}(S_{\lambda_{l,n-l}}(-\mathbb{A}) \cdot S_{\lambda^c}(\mathbb{E})) \circ \phi_2).
\end{eqnarray*}
But, by Lemma \ref{phibar-compose-phi}, we have that, for $\lambda\in\Lambda_{l,m}$,
\begin{eqnarray*}
\overline{\phi}_1 \circ \mathfrak{m}(S_{\lambda_{m,n}}(-\mathbb{Y}) \cdot S_{\lambda'}(\mathbb{X})) \circ \phi_1 & \approx & \begin{cases}
\id & \text{if } \lambda=(0\geq\cdots\geq0), \\
0 & \text{if } \lambda\neq(0\geq\cdots\geq0),
\end{cases} \\
\overline{\phi}_2 \circ \mathfrak{m}(S_{\lambda_{l,n-l}}(-\mathbb{A}) \cdot S_{\lambda^c}(\mathbb{E})) \circ \phi_2 & \approx & \begin{cases}
\id & \text{if } \lambda=(0\geq\cdots\geq0), \\
0 & \text{if } \lambda\neq(0\geq\cdots\geq0).
\end{cases} \\
\end{eqnarray*}
So, 
\begin{equation}\label{chi-non-trivial-eq}
(\overline{\phi}_1\otimes\overline{\phi}_2) \circ \mathfrak{m}(S_{\lambda_{m,n}}(-\mathbb{Y}) \cdot S_{\lambda_{l,n-l}}(-\mathbb{A})) \circ \chi^0 \circ \chi^1 \circ (\phi_1\otimes\phi_2) \approx \id,
\end{equation}
which implies that $\chi^0$ and $\chi^1$ are not homotopic to $0$.
\end{proof}

\begin{lemma}\label{varpi-0-reduce-chi}
Let $\varpi_0$ be the functor given in Lemma \ref{MOY-object-of-hmf}. Then $\varpi_0(\chi^0):C(\Gamma_0) \rightarrow C(\Gamma_1)$ and $\varpi_0(\chi^1):C(\Gamma_1) \rightarrow C(\Gamma_0)$ are the $\chi$-morphisms given in \cite[Proposition 7.20]{Wu-color}.
\end{lemma}

\begin{proof}
Note that $\varpi_0(\chi^0)$ and $\varpi_0(\chi^1)$ are homogeneous morphisms of quantum degree $ml$. Applying $\varpi_0$ to \eqref{chi-non-trivial-eq}, one can see that $\varpi_0(\chi^0)$ and $\varpi_0(\chi^1)$ are not homotopic to $0$. Then the lemma follows from \cite[Proposition 7.29]{Wu-color}.
\end{proof}

\subsection{Saddle move} We call the local change given in Figure \ref{saddle-move-figure} a saddle move. Next, we define the morphism $\eta$ induced by the saddle move.

\begin{figure}[ht]

\setlength{\unitlength}{1pt}

\begin{picture}(360,80)(-180,-15)


\put(-5,35){$\eta$}

\put(-25,30){\vector(1,0){50}}


\qbezier(-140,10)(-120,30)(-140,50)

\put(-140,50){\vector(-1,1){0}}

\qbezier(-100,10)(-120,30)(-100,50)

\put(-100,10){\vector(1,-1){0}}

\multiput(-130,30)(4.5,0){5}{\line(1,0){2}}

\put(-150,50){\small{$\mathbb{X}$}}

\put(-140,30){\tiny{$m$}}

\put(-150,10){\small{$\mathbb{A}$}}

\put(-105,30){\tiny{$m$}}

\put(-95,50){\small{$\mathbb{Y}$}}

\put(-95,10){\small{$\mathbb{E}$}}

\put(-125,-15){$\Gamma_0$}


\qbezier(100,10)(120,30)(140,10)

\put(140,10){\vector(1,-1){0}}

\qbezier(100,50)(120,30)(140,50)

\put(100,50){\vector(-1,1){0}}

\put(90,50){\small{$\mathbb{X}$}}

\put(120,45){\tiny{$m$}}

\put(90,10){\small{$\mathbb{A}$}}

\put(145,50){\small{$\mathbb{Y}$}}

\put(120,13){\tiny{$m$}}

\put(145,10){\small{$\mathbb{E}$}}

\put(115,-15){$\Gamma_1$}

\end{picture}

\caption{}\label{saddle-move-figure}

\end{figure}

\begin{lemma}\label{saddle-hmf}
Let $\Gamma_0$ and $\Gamma_1$ be the colored MOY graphs in Figure \ref{saddle-move-figure}. Then, as bigraded $R_B$-module,
\[
\Hom_{HMF}(C_f(\Gamma_0),C_f(\Gamma_1)) \cong C_f(\emptyset) \{\qb{N}{m} q^{2m(N-m)}\} \left\langle m \right\rangle.
\]
In particular, the subspace of $\Hom_{HMF}(C_f(\Gamma_0),C_f(\Gamma_1))$ of homogeneous elements of total polynomial degree $m(N-m)$ is $1$-dimensional.
\end{lemma}
\begin{proof}
Let $\bigcirc_m$ be a circle colored by $m$ (with $4$ marked points.) One can see that $\Hom(C_f(\Gamma_0),C_f(\Gamma_1))\cong C_f(\bigcirc_m) \{q^{2m(N-m)}\}$. The lemma follows from this and Proposition \ref{circle-module}.
\end{proof}

\begin{definition}
Let $\Gamma_0$ and $\Gamma_1$ be the colored MOY graphs in Figure \ref{saddle-move-figure}. Associate to the saddle move $\Gamma_0 \leadsto \Gamma_1$ a homogeneous morphism
\[
\eta : C_f(\Gamma_0) \rightarrow C_f(\Gamma_1)
\]
of total polynomial degree $m(N-m)$ not homotopic to $0$. By Lemma \ref{saddle-hmf}, $\eta$ is well defined up to homotopy and scaling, and $\deg_{\zed_2} \eta=m$. 
\end{definition}

Next we give the two composition formulas for $\eta$. 

Note that, in the proof of the First Composition Formula for $C(\Gamma)$ \cite[Proposition 7.36]{Wu-color}, only the right columns of the Koszul matrix factorizations associated to MOY graphs are explicitly used. So that proof applies to the First Composition Formula for $C_f(\Gamma)$ (Proposition \ref{creation+compose+saddle} below) without any visible changes.

\begin{figure}[ht]

\setlength{\unitlength}{1pt}

\begin{picture}(360,70)(-180,-10)
\put(-110,30){\tiny{$m$}}

\put(-100,0){\vector(0,1){60}}

\put(-65,35){$\iota$}

\put(-80,30){\vector(1,0){40}}

\put(-103,-10){$\Gamma$}


\put(-20,30){\tiny{$m$}}

\put(-10,0){\vector(0,1){60}}

\put(12,48){\tiny{$m$}}

\put(15,30){\oval(20,30)}

\put(25,35){\vector(0,1){0}}

\multiput(-10,30)(5,0){3}{\line(1,0){3}}

\put(0,-10){$\Gamma'$}

\put(45,35){$\eta$}

\put(30,30){\vector(1,0){40}}


\put(90,55){\vector(0,1){5}}

\qbezier(90,55)(90,45)(100,45)

\qbezier(100,45)(110,45)(110,30)

\qbezier(110,30)(110,15)(100,15)

\qbezier(100,15)(90,15)(90,5)

\put(90,0){\line(0,1){5}}

\put(115,30){\tiny{$m$}}

\put(90,-10){$\Gamma$}

\end{picture}

\caption{}\label{creation+saddle+figure}

\end{figure}

\begin{proposition}\label{creation+compose+saddle}
Let $\Gamma$ and $\Gamma'$ be the colored MOY graphs in Figure \ref{creation+saddle+figure}, $\iota:C_f(\Gamma)\rightarrow C_f(\Gamma')$ the morphism associated to the circle creation and $\eta:C_f(\Gamma')\rightarrow C_f(\Gamma)$ the morphism associated to the saddle move. Then $\eta\circ \iota \approx \id_{C_f(\Gamma)}$.
\end{proposition}

\begin{proof}
See the proof of \cite[Proposition 7.36]{Wu-color} in \cite[Subsection 7.9]{Wu-color}.
\end{proof}

Now we consider the Second Composition Formula. The proof of the Second Composition Formula for $C(\Gamma)$ \cite[Proposition 7.41]{Wu-color} is very complex and involves the left column of the Koszul matrix factorization $C(\Gamma)$. So, a direct generalization of that proof would be very complex and require many not-so-easy modifications. Fortunately, there is a simple proof of the Second Composition Formula for $C_f(\Gamma)$ (Proposition \ref{saddle+compose+annihilation} below) based on Lemma \ref{reduce-base-homotopic-equivalence}, Proposition \ref{creation+compose+saddle} and the Second Composition Formula for $C(\Gamma)$. We need the following lemma.

\begin{lemma}\label{varpi-0-reduce-saddle}
Let $\varpi_0$ be the functor given in Lemma \ref{MOY-object-of-hmf}, $\Gamma_0$, $\Gamma_1$ be the colored MOY graphs in Figure \ref{saddle-move-figure}, and $\eta : C_f(\Gamma_0) \rightarrow C_f(\Gamma_1)$ the morphism associated to the saddle move. Then $\varpi_0(\eta): C(\Gamma_0) \rightarrow C(\Gamma_1)$ is the morphism associated to the saddle move defined in \cite[Definition 7.34]{Wu-color}.
\end{lemma}

\begin{proof}
Note that $\varpi_0(\eta)$ is homogeneous of quantum degree $m(N-m)$. According to \cite[Lemma 7.33]{Wu-color}, to prove the lemma, we only need to show that $\varpi_0(\eta)$ is not null-homotopic. In the setup of Proposition \ref{creation+compose+saddle}, apply $\varpi_0$ to $\eta\circ \iota \approx \id_{C_f(\Gamma)}$. We get $\varpi_0(\eta)\circ \varpi_0(\iota) \approx \varpi_0(\id_{C_f(\Gamma)})=\id_{C(\Gamma)}$. This implies that $\varpi_0(\eta)$ is not null-homotopic.
\end{proof}

\begin{figure}[ht]

\setlength{\unitlength}{1pt}

\begin{picture}(360,70)(-180,-10)


\put(-110,55){\vector(0,1){5}}

\qbezier(-110,55)(-110,45)(-100,45)

\qbezier(-100,45)(-90,45)(-90,30)

\qbezier(-90,30)(-90,15)(-100,15)

\qbezier(-100,15)(-110,15)(-110,5)

\put(-110,0){\line(0,1){5}}

\multiput(-100,15)(0,5){6}{\line(0,1){3}}

\put(-85,30){\tiny{$m$}}

\put(-55,35){$\eta$}

\put(-70,30){\vector(1,0){40}}

\put(-103,-10){$\Gamma$}


\put(-20,30){\tiny{$m$}}

\put(-10,0){\vector(0,1){60}}

\put(2,48){\tiny{$m$}}

\put(5,30){\oval(20,30)}

\put(45,35){$\varepsilon$}

\put(30,30){\vector(1,0){40}}

\put(-5,-10){$\Gamma'$}


\put(90,30){\tiny{$m$}}

\put(100,0){\vector(0,1){60}}

\put(97,-10){$\Gamma$}

\end{picture}

\caption{}\label{saddle+annihilation+figure}

\end{figure}

\begin{proposition}\label{saddle+compose+annihilation}
Let $\Gamma$ and $\Gamma'$ be the MOY graphs in Figure \ref{saddle+annihilation+figure}, $\eta:C_f(\Gamma)\rightarrow C_f(\Gamma')$ the morphism associated to the saddle move and $\epsilon:C_f(\Gamma')\rightarrow C_f(\Gamma)$ the morphism associated to circle annihilation. Then $\epsilon \circ \eta \approx \id_{C_f(\Gamma)}$.
\end{proposition}

\begin{proof}
By Lemmas \ref{varpi-0-reduce-circle}, \ref{varpi-0-reduce-saddle} and \cite[Proposition 7.41]{Wu-color}, we have that 
\[
\varpi_0(\epsilon \circ \eta) = \varpi_0(\epsilon) \circ \varpi_0(\eta) \approx \id_{C(\Gamma)}.
\]
So, by Lemma \ref{reduce-base-homotopic-equivalence} $\epsilon \circ \eta: C_f(\Gamma) \rightarrow C_f(\Gamma)$ is a homotopy equivalence of matrix factorizations. Note that $\epsilon \circ \eta$ is homogeneous of total polynomial degree $0$ and $\zed_2$-degree $0$. From Lemma \ref{edge-splitting-lemma}, we have 
\[
\Hom_\HMF(C_f(\Gamma),C_f(\Gamma)) \cong C_f(\emptyset) \{q^{m(N-m)}\qb{N}{m}\}.
\]
So $\Hom_\hmf(C_f(\Gamma),C_f(\Gamma)) \cong \C$ and is spanned by $\id_{C_f(\Gamma)}$. (Otherwise, $\id_{C_f(\Gamma)} \simeq 0$, which implies $C_f(\Gamma) \simeq 0$ and $\Hom_\HMF(C_f(\Gamma),C_f(\Gamma)) \cong 0$, a contradiction.) This implies that $\epsilon \circ \eta \approx \id_{C_f(\Gamma)}$.
\end{proof}

\subsection{Summary} We call the bouquet move, circle creation and annihilation, edge splitting and merging, the saddle move and the local changes corresponding to the $\chi$-maps basic local changes of MOY graphs. Each of these induces morphisms of matrix factorizations $C_f$ and $C$. We have shown that the functor $\varpi_0$ defined in Lemma \ref{MOY-object-of-hmf} changes the morphism of $C_f$ induced by a basic local change to the morphism of $C$ induces by the same basic local change. For later reference, we state this in the following proposition.

\begin{proposition}\label{basic-changes-varpi-0}
Suppose $\Gamma$ and $\Gamma'$ are MOY graphs which differ by a basic local change, and $\alpha:C_f(\Gamma) \rightarrow C_f(\Gamma')$ is the morphism induced by this basic local change as defined in this section. Let $\varpi_0$ be the functor given in Lemma \ref{MOY-object-of-hmf}. Then $\varpi_0(\alpha) :C(\Gamma) \rightarrow C(\Gamma')$ is the morphism induced by this basic local change as defined in \cite[Section 7]{Wu-color}.
\end{proposition}

\begin{proof}
This follows from Lemmas \ref{varpi-0-reduce-bouquet}, \ref{varpi-0-reduce-circle}, \ref{varpi-0-reduce-edge-split}, \ref{varpi-0-reduce-chi} and \ref{varpi-0-reduce-saddle}.
\end{proof}

\section{Direct Sum Decompositions (I), (III), (IV), (V)}\label{sec-decomps}

Using the morphisms defined in Section \ref{sec-some-morph}, we are ready to prove Direct sum decompositions (I), (III), (IV) and (V). The proofs are mostly straightforward adaptations of those in \cite{Wu-color}.

\subsection{Direct sum decomposition (I)} 

\begin{theorem}[Direction Sum Decomposition (I)]\label{decomp-I}
\[
C_f(\setlength{\unitlength}{.75pt}
\begin{picture}(60,40)(-30,40)
\put(0,0){\vector(0,1){30}}
\put(0,30){\vector(0,1){20}}
\put(0,50){\vector(0,1){30}}

\put(-1,40){\line(1,0){2}}

\qbezier(0,30)(25,20)(25,30)
\qbezier(0,50)(25,60)(25,50)
\put(25,50){\vector(0,-1){20}}

\put(5,75){\tiny{$_{m}$}}
\put(5,5){\tiny{$_{m}$}}
\put(-30,38){\tiny{$_{m+n}$}}
\put(14,60){\tiny{$_{n}$}}
\end{picture}) \simeq C_f(\setlength{\unitlength}{.75pt}
\begin{picture}(40,40)(-20,40)
\put(0,0){\vector(0,1){80}}
\put(5,75){\tiny{$_{m}$}}
\end{picture})\{ \qb{N-m}{n}\}\left\langle n \right\rangle. \vspace{30pt}
\]
\end{theorem}

\begin{figure}[ht]
$
\xymatrix{
\input{v-vector-m} \ar@<1ex>[rr]^{\iota}  & &  \ar@<1ex>[ll]^{\epsilon} \input{v-vector-m-circle-n}  \ar@<1ex>[rr]^{\chi^0} & & \input{v-vector-m-loop-n} \ar@<1ex>[ll]^{\chi^1} \\
}
$
\caption{}\label{loop-addition-explicit}

\end{figure}

\begin{lemma}\label{decomp-I-special}
Consider the MOY graphs and morphisms in Figure \ref{loop-addition-explicit}. Define $\psi = \chi^0 \circ \iota$ and $\overline{\psi} = \epsilon \circ \chi^1$. Then $\psi$ and $\overline{\psi}$ are both homogeneous morphisms of $\zed_2$-degree $n$ and total polynomial degree $-n(N-n-m)$. Moreover, for a partition $\lambda=(\lambda_1\geq \cdots\geq \lambda_n)$, 
\[
\overline{\psi} \circ \mathfrak{m}(S_{\lambda}(\mathbb{B})) \circ \psi  \approx \begin{cases}
\id_{C_f(\Gamma_0)} & \text{if } \lambda = \lambda_{n,N-m-n} = (\underbrace{N-m-n\geq\cdots\geq N-m-n}_{n \text{ parts}}), \\
0 & \text{if } |\lambda|= \lambda_1 + \cdots + \lambda_n < n(N-n-m).
\end{cases}
\]

In particular, if $n=N-m$, then $C_f(\Gamma_1) \simeq C_f(\Gamma_0) \left\langle N-m \right\rangle$.
\end{lemma}

\begin{proof}
The homogeneity and degrees of $\psi$ and $\overline{\psi}$ follow easily from their definitions. Moreover, one can check that 
\[
\Hom_\HMF (C_f(\Gamma_0), C_f(\Gamma_0)) \cong H_f(\bigcirc_m)\{q^{m(N-m)}\}\left\langle m \right\rangle \cong C_f(\emptyset)\{\qb{N}{m}q^{m(N-m)}\}.
\]
So the lowest non-vanishing total polynomial grading of $\Hom_\HMF (C_f(\Gamma_0), C_f(\Gamma_0))$ is $0$. Note that the total polynomial degree of $\overline{\psi} \circ \mathfrak{m}(S_{\lambda}(\mathbb{B})) \circ \psi$ is $2|\lambda|-2n(N-n-m)$. So $\overline{\psi} \circ \mathfrak{m}(S_{\lambda}(\mathbb{B})) \circ \psi  \approx 0$ if $|\lambda| < n(N-n-m)$.

By Proposition \ref{general-general-chi-maps}, we have
\begin{eqnarray*}
&& \epsilon \circ \chi^1 \circ \mathfrak{m}(S_{\lambda_{n,N-m-n}}(\mathbb{B})) \circ \chi^0 \circ \iota \\
& = & \epsilon \circ \mathfrak{m}(S_{\lambda_{n,N-m-n}}(\mathbb{B})) \circ \chi^1  \circ \chi^0 \circ \iota \\ 
& = & \epsilon \circ \mathfrak{m}(S_{\lambda_{n,N-m-n}}(\mathbb{B})\cdot\sum_{\lambda\in\Lambda_{n,m}} (-1)^{|\lambda|} S_{\lambda'}(\mathbb{X}) S_{\lambda^c}(\mathbb{B})) \circ \iota \\
& = & \sum_{\lambda\in\Lambda_{n,m}} (-1)^{|\lambda|} S_{\lambda'}(\mathbb{X}) \cdot \epsilon \circ \mathfrak{m}(S_{\lambda_{n,N-m-n}}(\mathbb{B})\cdot S_{\lambda^c}(\mathbb{B})) \circ \iota
\end{eqnarray*}
where $\Lambda_{n,m}=\{\mu ~|~ \mu \leq \lambda_{n,m}\} =  \{\mu=(\mu_1\geq\cdots\geq\mu_n) ~|~ l(\mu)\leq n,~ \mu_1 \leq m\}$, $\lambda'\in \Lambda_{m,n}$ is the conjugate of $\lambda$, and $\lambda^c$ is the complement of $\lambda$ in $\Lambda_{n,m}$, i.e., if $\lambda=(\lambda_1\geq\cdots\geq\lambda_n)\in \Lambda_{n,m}$, then $\lambda^c = (m-\lambda_n\geq\cdots\geq m-\lambda_1)$. By Corollary \ref{iota-epsilon-composition}, we have, for $\lambda\in\Lambda_{n,m}$,
\[
\epsilon \circ \mathfrak{m}(S_{\lambda_{n,N-m-n}}(\mathbb{B})\cdot S_{\lambda^c}(\mathbb{B})) \circ \iota \approx \begin{cases}
\id_{C(\Gamma_0)} & \text{if } \lambda = (0\geq\cdots\geq0), \\
0 & \text{if } \lambda \neq (0\geq\cdots\geq0).
\end{cases}
\]
Thus, 
\[
\overline{\psi} \circ \mathfrak{m}(S_{\lambda_{n,N-m-n}}(\mathbb{B})) \circ \psi  \approx \id_{C_f(\Gamma_0)}.
\]

When $n=N-m$, $\psi$ and $\overline{\psi}$ are both homogeneous morphisms of $\zed_2$ degree $N-m$ and total polynomial degree $0$. And, from above, we have $\overline{\psi} \circ \psi  \approx \id_{C_f(\Gamma_0)}$. Let $\varpi_0$ be the functor given in Lemma \ref{MOY-object-of-hmf}. Then $\varpi_0(\overline{\psi}) \circ \varpi(\psi)  \approx \id_{C(\Gamma_0)}$. By \cite[Lemma 5.15]{Wu-color}, we know that $C(\Gamma_1) \simeq C(\Gamma_0) \left\langle N-m \right\rangle$. By \cite[Lemma 3.14]{Wu-color}, this implies that $\varpi_0(\overline{\psi})$ and $\varpi(\psi)$ are homotopy equivalences of matrix factorizations. By Lemma \ref{reduce-base-homotopic-equivalence}, $\psi$ and $\overline{\psi}$ are also homotopy equivalences of matrix factorizations. Thus, $C_f(\Gamma_1) \simeq C_f(\Gamma_0) \left\langle N-m \right\rangle$.
\end{proof}

\begin{remark}
In Lemma \ref{decomp-I-special}, we allow $m=0$. In this case, $\psi$ and $\overline{\psi}$ becomes $\iota$ and $\epsilon$. Note that Lemma \ref{decomp-I-special} remains true in this case. 
\end{remark}

\begin{figure}[ht]

\setlength{\unitlength}{1pt}

\begin{picture}(360,160)(-180,-100)


\put(-120,0){\vector(0,1){15}}

\put(-120,15){\vector(0,1){15}}

\put(-120,30){\line(0,1){15}}

\put(-120,45){\vector(0,1){15}}

\qbezier(-120,15)(-100,5)(-100,25)

\qbezier(-120,45)(-100,55)(-100,35)

\put(-100,35){\vector(0,-1){10}}

\put(-130,55){\tiny{$m$}}

\put(-130,0){\tiny{$m$}}

\put(-145,30){\tiny{$m+n$}}

\put(-95,30){\tiny{$n$}}

\put(-123,-10){$\Gamma_1$}


\put(0,0){\vector(0,1){60}}

\put(-10,30){\tiny{$m$}}

\put(-3,-10){$\Gamma_0$}


\put(120,0){\vector(0,1){5}}

\put(120,5){\vector(0,1){25}}

\put(120,30){\line(0,1){15}}

\put(120,45){\vector(0,1){15}}

\qbezier(120,5)(170,0)(170,25)

\qbezier(120,55)(170,60)(170,35)

\put(170,35){\vector(0,-1){10}}

\qbezier(120,15)(130,5)(130,25)

\qbezier(120,45)(130,55)(130,35)

\put(130,35){\vector(0,-1){10}}

\put(110,55){\tiny{$m$}}

\put(110,0){\tiny{$m$}}

\put(95,45){\tiny{$m+n$}}

\put(110,30){\tiny{$N$}}

\put(95,10){\tiny{$m+n$}}

\put(130,30){\tiny{$N-m-n$}}

\put(175,30){\tiny{$n$}}

\put(117,-10){$\Gamma_3$}


\put(120,-90){\vector(0,1){5}}

\put(120,-85){\vector(0,1){25}}

\put(120,-60){\line(0,1){15}}

\put(120,-45){\vector(0,1){15}}

\qbezier(150,-80)(170,-70)(170,-65)

\qbezier(150,-40)(170,-50)(170,-55)

\put(170,-55){\vector(0,-1){10}}

\qbezier(150,-80)(130,-70)(130,-65)

\qbezier(150,-40)(130,-50)(130,-55)

\put(130,-55){\vector(0,-1){10}}

\put(120,-40){\vector(1,0){30}}

\put(150,-80){\vector(-1,0){30}}

\put(110,-35){\tiny{$m$}}

\put(110,-90){\tiny{$m$}}

\put(110,-60){\tiny{$N$}}

\put(130,-60){\tiny{$N-m-n$}}

\put(125,-38){\tiny{$N-m$}}

\put(125,-87){\tiny{$N-m$}}

\put(175,-60){\tiny{$n$}}

\put(117,-100){$\Gamma_4$}


\put(0,-90){\vector(0,1){15}}

\put(0,-75){\vector(0,1){15}}

\put(0,-60){\line(0,1){15}}

\put(0,-45){\vector(0,1){15}}

\qbezier(0,-75)(20,-85)(20,-65)

\qbezier(0,-45)(20,-35)(20,-55)

\put(20,-55){\vector(0,-1){10}}

\put(-10,-35){\tiny{$m$}}

\put(-10,-60){\tiny{$N$}}

\put(-10,-90){\tiny{$m$}}

\put(25,-60){\tiny{$N-m$}}

\put(-3,-100){$\Gamma_5$}

\end{picture}

\caption{}\label{decomp-I-proof-figure}

\end{figure}

\begin{proof}[Proof of Theorem \ref{decomp-I}]
Consider the colored MOY graphs in Figure \ref{decomp-I-proof-figure}. By Lemma \ref{decomp-I-special}, $C(\Gamma_1)\simeq C(\Gamma_3)\left\langle N-m-n\right\rangle$. By Corollary \ref{contract-expand}, $C(\Gamma_3)\simeq C(\Gamma_4)$. By direct sum decomposition (II) (Theorem \ref{decomp-II},) $C(\Gamma_4)\simeq C(\Gamma_5)\qb{N-m}{n}$. And by Lemma \ref{decomp-I-special} again, $C(\Gamma_5)\simeq C(\Gamma_0)\left\langle N-m\right\rangle$. Putting everything together, we get $C(\Gamma_1) \simeq C(\Gamma_0)\{ \qb{N-m}{n}\}\left\langle n \right\rangle$.
\end{proof}

\subsection{Direct sum decomposition (III)}

\begin{theorem}[Direct Sum Decomposition (III)]\label{decomp-III} 
\[
C_f(\input{decomp-III-1-slide}) \simeq C_f(\setlength{\unitlength}{.75pt}
\begin{picture}(60,30)(-30,30)

\put(-20,0){\vector(0,1){60}}

\put(20,60){\vector(0,-1){60}}

\put(-25,30){\tiny{$_1$}}

\put(22,30){\tiny{$_m$}}
\end{picture}) ~\bigoplus~ C_f(\setlength{\unitlength}{.75pt}
\begin{picture}(60,30)(100,30)

\put(110,0){\vector(1,1){20}}

\put(130,20){\vector(1,-1){20}}

\put(130,40){\vector(0,-1){20}}

\put(130,40){\vector(-1,1){20}}

\put(150,60){\vector(-1,-1){20}}

\put(105,0){\tiny{$_1$}}

\put(105,55){\tiny{$_1$}}

\put(152,0){\tiny{$_m$}}

\put(152,55){\tiny{$_m$}}

\put(132,30){\tiny{$_{m-1}$}}

\end{picture})\{[N-m-1]\} \left\langle 1 \right\rangle. \vspace{23pt}
\]
The above homotopy equivalence of matrix factorizations remains true if the orientations of these MOY graphs are reversed.
\end{theorem}

\begin{proof}
Define
\[
\widehat{\iota}: C_f(\emptyset)  \rightarrow C_f(\input{decomp-III-pro-1}) \vspace{30pt}
\]
to be the composition
\[
C_f(\emptyset) \xrightarrow{\iota} C_f(\bigcirc_{m+1}) \xrightarrow{\phi_1 \otimes \phi_2} C_f(\input{decomp-III-pro-1}), \vspace{30pt}
\]
and 
\[
\widehat{\epsilon}:  C_f(\input{decomp-III-pro-1})  \rightarrow C_f(\emptyset)  \vspace{30pt}
\]
to be the composition
\[
C_f(\input{decomp-III-pro-1}) \xrightarrow{\overline{\phi_1} \otimes \overline{\phi_2}} C_f(\bigcirc_{m+1})  \xrightarrow{\epsilon} C_f(\emptyset), \vspace{30pt}
\]
where $\bigcirc_{m+1}$ is a circle colored by $m+1$, $\iota$ and $\epsilon$ are the morphisms induced by the creation/annihilation of $\bigcirc_{m+1}$, and $\phi_1$, $\phi_2$ (resp. $\overline{\phi_1}$, $\overline{\phi_2}$) are the morphisms induces by the two apparent edge splittings (resp. merging.) Then define
\[
\xymatrix{
C_f(\setlength{\unitlength}{.75pt}
\begin{picture}(60,30)(-30,30)

\put(-20,0){\vector(0,1){60}}

\put(20,60){\vector(0,-1){60}}

\put(-25,30){\tiny{$_1$}}

\put(22,30){\tiny{$_m$}}
\end{picture}) \ar@<-1ex>[rr]^{F} &&  C_f(\input{decomp-III-1-slide}) \ar@<3ex>[ll]^{G}
} \vspace{23pt}
\]
to be the compositions
\[
\xymatrix{
C_f(\setlength{\unitlength}{.75pt}
\begin{picture}(60,30)(-30,30)

\put(-20,0){\vector(0,1){60}}

\put(20,60){\vector(0,-1){60}}

\put(-25,30){\tiny{$_1$}}

\put(22,30){\tiny{$_m$}}
\end{picture}) \ar@<-1ex>[r]^<<<<<<{\widehat{\iota}} & \ar@<3ex>[l]^>>>>>>{\widehat{\epsilon}} C_f(\input{decomp-III-pro-2}) \ar@<-1ex>[rr]^<<<<<<<<<<{\eta_1 \otimes \eta_2} && C_f(\input{decomp-III-1-slide}) \ar@<3ex>[ll]^>>>>>>>>>>{\eta_3 \otimes \eta_4}, 
}\vspace{30pt}
\]
where $\widehat{\iota}$ and $\widehat{\epsilon}$ are the morphisms defined above and $\eta_1$, $\eta_2$ $\eta_3$ and $\eta_4$ are induced by the apparent saddle moves. Note that $F$ and $G$ are homogeneous morphisms preserving both gradings. From Proposition \ref{basic-changes-varpi-0} and \cite[Proposition 8.8]{Wu-color}, we know that $\varpi_0(G \circ F) \approx \id$, where $\varpi_0$ is the functor given in Lemma \ref{MOY-object-of-hmf} and $\id$ is the identity morphism of $C(\setlength{\unitlength}{.75pt}
\begin{picture}(60,30)(-30,30)

\put(-20,0){\vector(0,1){60}}

\put(20,60){\vector(0,-1){60}}

\put(-25,30){\tiny{$_1$}}

\put(22,30){\tiny{$_m$}}
\end{picture}) \vspace{23pt}$. By Lemma \ref{reduce-base-homotopic-equivalence}, this implies that $G\circ F$ is a homotopy equivalence of matrix factorizations preserving both gradings. An easy computation shows that $\Hom_\hmf (C_f(\setlength{\unitlength}{.75pt}
\begin{picture}(60,30)(-30,30)

\put(-20,0){\vector(0,1){60}}

\put(20,60){\vector(0,-1){60}}

\put(-25,30){\tiny{$_1$}}

\put(22,30){\tiny{$_m$}}
\end{picture}), C_f(\setlength{\unitlength}{.75pt}
\begin{picture}(60,30)(-30,30)

\put(-20,0){\vector(0,1){60}}

\put(20,60){\vector(0,-1){60}}

\put(-25,30){\tiny{$_1$}}

\put(22,30){\tiny{$_m$}}
\end{picture})) \cong \C \vspace{23pt}$ and is spanned by the identity morphism of $C_f(\setlength{\unitlength}{.75pt}
\begin{picture}(60,30)(-30,30)

\put(-20,0){\vector(0,1){60}}

\put(20,60){\vector(0,-1){60}}

\put(-25,30){\tiny{$_1$}}

\put(22,30){\tiny{$_m$}}
\end{picture}) \vspace{23pt}$. So $G \circ F \approx \id$, where $\id$ is the identity morphism of $C_f(\setlength{\unitlength}{.75pt}
\begin{picture}(60,30)(-30,30)

\put(-20,0){\vector(0,1){60}}

\put(20,60){\vector(0,-1){60}}

\put(-25,30){\tiny{$_1$}}

\put(22,30){\tiny{$_m$}}
\end{picture}) \vspace{23pt}$.

Next, define
\[
\xymatrix{
C_f(\setlength{\unitlength}{.75pt}
\begin{picture}(60,30)(100,30)

\put(110,0){\vector(1,1){20}}

\put(130,20){\vector(1,-1){20}}

\put(130,40){\vector(0,-1){20}}

\put(130,40){\vector(-1,1){20}}

\put(150,60){\vector(-1,-1){20}}

\put(105,0){\tiny{$_1$}}

\put(105,55){\tiny{$_1$}}

\put(152,0){\tiny{$_m$}}

\put(152,55){\tiny{$_m$}}

\put(132,30){\tiny{$_{m-1}$}}

\end{picture}) \ar@<-1ex>[rr]^{\alpha} &&  C_f(\input{decomp-III-1-slide}) \ar@<3ex>[ll]^{\beta}
} \vspace{23pt}
\]
to be the compositions 
\[
\xymatrix{
C_f(\input{decomp-III-3-slide-marked})  \ar@<-1ex>[r]^{\psi} & \ar@<3ex>[l]^{\overline{\psi}} C_f(\input{decomp-III-pro-3}) \ar@<-1ex>[rr]^{\chi^0_1 \otimes \chi^0_2} &&  C_f(\input{decomp-III-1-slide-marked}) \ar@<3ex>[ll]^{\chi^1_1 \otimes \chi^1_2},
} \vspace{30pt}
\]
where $\psi$ and $\overline{\psi}$ are defined in Lemma \ref{decomp-I-special}, and $\chi^0_1$, $\chi^0_2$, $\chi^1_1$, $\chi^1_2$ are the apparent $\chi$-morphisms. Then, for $j=0,1,\dots,N-m-2$, define
\[
\xymatrix{
C_f(\input{decomp-III-3-slide-marked})\{q^{N-m-2-2j}\} \left\langle 1\right\rangle \ar@<-1ex>[rr]^>>>>>>>>>>{\alpha_j} &&  C_f(\input{decomp-III-1-slide-marked}) \ar@<3ex>[ll]^<<<<<<<<<<{\beta_j}
} \vspace{30pt}
\]
by 
\begin{eqnarray*}
\alpha_j & = & \mathfrak{m}(s^{N-m-2-j}) \circ \alpha, \\
\beta_j & = & \beta \circ \mathfrak{m}(s^j).
\end{eqnarray*}
Note that $\alpha_j$ and $\beta_j$ are homogeneous morphisms preserving both gradings. Moreover, by Proposition \ref{general-general-chi-maps} and Lemma \ref{decomp-I-special}, we have
\[
\beta_i \circ \alpha_j \approx \begin{cases}
\id & \text{if } i=j, \\
0 & \text{if } i<j.
\end{cases}
\]
Let $\vec{\alpha} = (\alpha_0,\dots,\alpha_{N-m-2})$ and $\vec{\beta}=(\beta_0,\dots,\beta_{N-m-2})^T$. Then 
\[
\xymatrix{
C_f(\setlength{\unitlength}{.75pt}
\begin{picture}(60,30)(100,30)

\put(110,0){\vector(1,1){20}}

\put(130,20){\vector(1,-1){20}}

\put(130,40){\vector(0,-1){20}}

\put(130,40){\vector(-1,1){20}}

\put(150,60){\vector(-1,-1){20}}

\put(105,0){\tiny{$_1$}}

\put(105,55){\tiny{$_1$}}

\put(152,0){\tiny{$_m$}}

\put(152,55){\tiny{$_m$}}

\put(132,30){\tiny{$_{m-1}$}}

\end{picture})  \{[N-m-1]\} \left\langle 1 \right\rangle \ar@<-1ex>[rr]^>>>>>>>>>>{\vec{\alpha}} &&  C_f(\input{decomp-III-1-slide}) \ar@<3ex>[ll]^<<<<<<<<<<{\vec{\beta}}
} \vspace{23pt}
\]
are homogeneous morphisms preserving both gradings. And $\vec{\beta} \circ \vec{\alpha}$ is lower-triangular with homotopy equivalences of matrix factorizations along the diagonal. So $\vec{\beta} \circ \vec{\alpha}$ is a homotopy equivalence of matrix factorizations.

Consider the morphisms
\[
\xymatrix{
C_f(\setlength{\unitlength}{.75pt}
\begin{picture}(60,30)(-30,30)

\put(-20,0){\vector(0,1){60}}

\put(20,60){\vector(0,-1){60}}

\put(-25,30){\tiny{$_1$}}

\put(22,30){\tiny{$_m$}}
\end{picture}) \bigoplus C_f(\setlength{\unitlength}{.75pt}
\begin{picture}(60,30)(100,30)

\put(110,0){\vector(1,1){20}}

\put(130,20){\vector(1,-1){20}}

\put(130,40){\vector(0,-1){20}}

\put(130,40){\vector(-1,1){20}}

\put(150,60){\vector(-1,-1){20}}

\put(105,0){\tiny{$_1$}}

\put(105,55){\tiny{$_1$}}

\put(152,0){\tiny{$_m$}}

\put(152,55){\tiny{$_m$}}

\put(132,30){\tiny{$_{m-1}$}}

\end{picture})\{[N-m-1]\} \left\langle 1 \right\rangle \ar@<-1ex>[rr]^>>>>>>>>>>{\left(%
\begin{array}{cc}
F & \vec{\alpha}
\end{array}%
\right)} && C_f(\input{decomp-III-1-slide})  \ar@<3ex>[ll]^>>>>>>>>>>{\left(%
\begin{array}{c}
G \\
\vec{\beta}
\end{array}%
\right)}  
}. \vspace{23pt}
\]
Let $\varpi_0$ be the functor given in Lemma \ref{MOY-object-of-hmf}. By Proposition \ref{basic-changes-varpi-0} and \cite[Lemmas 8.14 and 8.15]{Wu-color}, we know that $\varpi_0(F,\vec{\alpha}) = (\varpi_0(F),\varpi_0(\vec{\alpha}))$ and $\varpi_0\left(%
\begin{array}{c}
G \\
\vec{\beta}
\end{array}%
\right) = \left(%
\begin{array}{c}
\varpi_0(G) \\
\varpi_0(\vec{\beta})
\end{array}%
\right)$ are homotopy equivalences of matrix factorizations. Thus, by Lemma \ref{reduce-base-homotopic-equivalence}, $(F,\vec{\alpha})$ and $\left(%
\begin{array}{c}
G \\
\vec{\beta}
\end{array}%
\right)$ are homotopy equivalences of matrix factorizations.
\end{proof}

\subsection{Direct sum decomposition (IV)}

\begin{theorem}[Direct Sum Decomposition (IV)]\label{decomp-IV} 
Let $l,m,n$ be integers satisfying $0\leq n \leq m \leq N$ and $0\leq l, m+l-1 \leq N$. Then
\[
C_f(\input{decomp-IV-1-slide}) \simeq C_f(\setlength{\unitlength}{.75pt}
\begin{picture}(85,45)(-30,45)

\put(-20,0){\vector(0,1){45}}

\put(-20,45){\vector(0,1){45}}

\put(20,0){\vector(0,1){45}}

\put(20,45){\vector(0,1){45}}

\put(20,45){\vector(-1,0){40}}

\put(-27,20){\tiny{$_1$}}

\put(23,20){\tiny{$_{m+l-1}$}}

\put(-27,65){\tiny{$_l$}}

\put(23,65){\tiny{$_m$}}

\put(-5,38){\tiny{$_{l-1}$}}

\end{picture}) \{\qb{m-1}{n}\} ~\bigoplus~ C_f(\setlength{\unitlength}{.75pt}
\begin{picture}(65,45)(110,45)

\put(110,0){\vector(2,3){20}}

\put(150,0){\vector(-2,3){20}}

\put(130,30){\vector(0,1){30}}

\put(130,60){\vector(-2,3){20}}

\put(130,60){\vector(2,3){20}}

\put(117,20){\tiny{$_1$}}

\put(140,20){\tiny{$_{m+l-1}$}}

\put(117,65){\tiny{$_l$}}

\put(140,65){\tiny{$_m$}}

\put(133,42){\tiny{$_{m+l}$}}
\end{picture})\{\qb{m-1}{n-1}\}. \vspace{34pt}
\]
The above homotopy equivalence of matrix factorizations remains true if the orientations of these MOY graphs are reversed.
\end{theorem}

\begin{proof}
Define morphisms 
\[
F: C_f(\setlength{\unitlength}{.75pt}
\begin{picture}(85,45)(-30,45)

\put(-20,0){\vector(0,1){45}}

\put(-20,45){\vector(0,1){45}}

\put(20,0){\vector(0,1){45}}

\put(20,45){\vector(0,1){45}}

\put(20,45){\vector(-1,0){40}}

\put(-27,20){\tiny{$_1$}}

\put(23,20){\tiny{$_{m+l-1}$}}

\put(-27,65){\tiny{$_l$}}

\put(23,65){\tiny{$_m$}}

\put(-5,38){\tiny{$_{l-1}$}}

\end{picture}) \rightarrow C_f(\input{decomp-IV-1-slide})\vspace{34pt}
\]
\[
G: C_f(\input{decomp-IV-1-slide}) \rightarrow C_f(\setlength{\unitlength}{.75pt}
\begin{picture}(85,45)(-30,45)

\put(-20,0){\vector(0,1){45}}

\put(-20,45){\vector(0,1){45}}

\put(20,0){\vector(0,1){45}}

\put(20,45){\vector(0,1){45}}

\put(20,45){\vector(-1,0){40}}

\put(-27,20){\tiny{$_1$}}

\put(23,20){\tiny{$_{m+l-1}$}}

\put(-27,65){\tiny{$_l$}}

\put(23,65){\tiny{$_m$}}

\put(-5,38){\tiny{$_{l-1}$}}

\end{picture})\vspace{34pt} 
\]
by the following diagram.
\[
\xymatrix{
\input{decomp-IV-pro-1} \ar@<1ex>[rr]^{G} \ar@<1ex>[d]^{\chi^1} && \input{decomp-IV-pro-2} \ar@<1ex>[ll]^{F} \ar@<1ex>[d]^{\phi} \\
\input{decomp-IV-pro-3} \ar@<1ex>[rr]^{h_1} \ar@<1ex>[u]^{\chi^0} && \input{decomp-IV-pro-4} \ar@<1ex>[ll]^{h_0} \ar@<1ex>[u]^{\overline{\phi}}\\
}
\]
That is, $F= \chi^0 \circ h_0 \circ \phi$ and $G = \overline{\phi} \circ h_1 \circ \chi^1$, where $\chi^0 , \chi^1, h_0 , h_1 , \phi, \overline{\phi}$ are the morphisms induced by the apparent basic changes of the MOY graphs.

Let $\Lambda=\Lambda_{n,m-n-1} =\{\lambda~|~l(\lambda)\leq n,~\lambda_1\leq m-n-1\}$. For $\lambda=(\lambda_1\geq\cdots\geq\lambda_n) \in \Lambda$, define $\lambda^c=(\lambda^c_1\geq\cdots\geq\lambda^c_n) \in \Lambda$ by $\lambda^c_j =m-n-1-\lambda_{n+1-j}$, $j=1,\dots,n$. For $\lambda \in \Lambda$, define $F_{\lambda} = \mathfrak{m}(S_{\lambda}(\mathbb{A})) \circ F$ and $G_{\lambda} = G \circ \mathfrak{m}(S_{\lambda^c}(-\mathbb{Y}))$. Then
\[
F_{\lambda}: C_f(\setlength{\unitlength}{.75pt}
\begin{picture}(85,45)(-30,45)

\put(-20,0){\vector(0,1){45}}

\put(-20,45){\vector(0,1){45}}

\put(20,0){\vector(0,1){45}}

\put(20,45){\vector(0,1){45}}

\put(20,45){\vector(-1,0){40}}

\put(-27,20){\tiny{$_1$}}

\put(23,20){\tiny{$_{m+l-1}$}}

\put(-27,65){\tiny{$_l$}}

\put(23,65){\tiny{$_m$}}

\put(-5,38){\tiny{$_{l-1}$}}

\end{picture})\{q^{2|\lambda|-n(m-n-1)}\} \rightarrow C_f(\input{decomp-IV-1-slide}),\vspace{34pt}
\]
\[
G_{\lambda}: C_f(\input{decomp-IV-1-slide}) \rightarrow C_f(\setlength{\unitlength}{.75pt}
\begin{picture}(85,45)(-30,45)

\put(-20,0){\vector(0,1){45}}

\put(-20,45){\vector(0,1){45}}

\put(20,0){\vector(0,1){45}}

\put(20,45){\vector(0,1){45}}

\put(20,45){\vector(-1,0){40}}

\put(-27,20){\tiny{$_1$}}

\put(23,20){\tiny{$_{m+l-1}$}}

\put(-27,65){\tiny{$_l$}}

\put(23,65){\tiny{$_m$}}

\put(-5,38){\tiny{$_{l-1}$}}

\end{picture})\{q^{2|\lambda|-n(m-n-1)}\}\vspace{34pt} 
\]
are homogeneous morphisms preserving both gradings. Define $\vec{F}= \sum_{\lambda \in \Lambda} F_{\lambda}$ and $\vec{G}= \sum_{\lambda \in \Lambda} G_{\lambda}$. Then 

\[
\vec{F}: C_f(\setlength{\unitlength}{.75pt}
\begin{picture}(85,45)(-30,45)

\put(-20,0){\vector(0,1){45}}

\put(-20,45){\vector(0,1){45}}

\put(20,0){\vector(0,1){45}}

\put(20,45){\vector(0,1){45}}

\put(20,45){\vector(-1,0){40}}

\put(-27,20){\tiny{$_1$}}

\put(23,20){\tiny{$_{m+l-1}$}}

\put(-27,65){\tiny{$_l$}}

\put(23,65){\tiny{$_m$}}

\put(-5,38){\tiny{$_{l-1}$}}

\end{picture})\{\qb{m-1}{n}\} \rightarrow C_f(\input{decomp-IV-1-slide}),\vspace{34pt}
\]
\[
\vec{G}: C_f(\input{decomp-IV-1-slide}) \rightarrow C_f(\setlength{\unitlength}{.75pt}
\begin{picture}(85,45)(-30,45)

\put(-20,0){\vector(0,1){45}}

\put(-20,45){\vector(0,1){45}}

\put(20,0){\vector(0,1){45}}

\put(20,45){\vector(0,1){45}}

\put(20,45){\vector(-1,0){40}}

\put(-27,20){\tiny{$_1$}}

\put(23,20){\tiny{$_{m+l-1}$}}

\put(-27,65){\tiny{$_l$}}

\put(23,65){\tiny{$_m$}}

\put(-5,38){\tiny{$_{l-1}$}}

\end{picture})\{\qb{m-1}{n}\}\vspace{34pt} 
\]
are homogeneous morphisms preserving both gradings.

Next, define morphisms 
\[
\alpha: C_f(\setlength{\unitlength}{.75pt}
\begin{picture}(65,45)(110,45)

\put(110,0){\vector(2,3){20}}

\put(150,0){\vector(-2,3){20}}

\put(130,30){\vector(0,1){30}}

\put(130,60){\vector(-2,3){20}}

\put(130,60){\vector(2,3){20}}

\put(117,20){\tiny{$_1$}}

\put(140,20){\tiny{$_{m+l-1}$}}

\put(117,65){\tiny{$_l$}}

\put(140,65){\tiny{$_m$}}

\put(133,42){\tiny{$_{m+l}$}}
\end{picture}) \rightarrow C_f(\input{decomp-IV-1-slide})\vspace{34pt}
\]
\[
\beta: C_f(\input{decomp-IV-1-slide}) \rightarrow C_f(\setlength{\unitlength}{.75pt}
\begin{picture}(65,45)(110,45)

\put(110,0){\vector(2,3){20}}

\put(150,0){\vector(-2,3){20}}

\put(130,30){\vector(0,1){30}}

\put(130,60){\vector(-2,3){20}}

\put(130,60){\vector(2,3){20}}

\put(117,20){\tiny{$_1$}}

\put(140,20){\tiny{$_{m+l-1}$}}

\put(117,65){\tiny{$_l$}}

\put(140,65){\tiny{$_m$}}

\put(133,42){\tiny{$_{m+l}$}}
\end{picture})\vspace{34pt} 
\]
by the following diagram.
\[
\xymatrix{
\input{decomp-IV-pro-1} \ar@<1ex>[rr]^{\beta} \ar@<1ex>[d]^{\chi^0} && \input{decomp-IV-pro-5} \ar@<1ex>[ll]^{\alpha} \ar@<1ex>[d]^{\phi} \\
\input{decomp-IV-pro-6} \ar@<1ex>[rr]^{h_0} \ar@<1ex>[u]^{\chi^1} && \input{decomp-IV-pro-7} \ar@<1ex>[ll]^{h_1} \ar@<1ex>[u]^{\overline{\phi}}\\
}
\]
That is, $\alpha= \chi^1 \circ h_1 \circ \phi$ and $\beta = \overline{\phi} \circ h_0 \circ \chi^0$, where $\chi^0 , \chi^1, h_0 , h_1 , \phi, \overline{\phi}$ are the morphisms induced by the apparent basic changes of the MOY graphs.

Let $\Lambda' = \Lambda_{m-n,n-1} =\{\lambda~|~l(\lambda)\leq m-n,~\lambda_1\leq n-1\}$. For $\lambda=(\lambda_1\geq\cdots\geq\lambda_{m-n}) \in \Lambda'$, define $\lambda^\ast=(\lambda^\ast_1\geq\cdots\geq\lambda^\ast_{m-n}) \in \Lambda'$ by $\lambda^\ast_j =n-1-\lambda_{m-n+1-j}$, $j=1,\dots,m-n$. For $\lambda \in \Lambda'$, define $\alpha_{\lambda} = \mathfrak{m}(S_{\lambda}(\mathbb{Y})) \circ \alpha$ and $\beta_{\lambda} = \beta \circ \mathfrak{m}(S_{\lambda^\ast}(-\mathbb{A}))$. Then
\[
\alpha_{\lambda}: C_f(\setlength{\unitlength}{.75pt}
\begin{picture}(65,45)(110,45)

\put(110,0){\vector(2,3){20}}

\put(150,0){\vector(-2,3){20}}

\put(130,30){\vector(0,1){30}}

\put(130,60){\vector(-2,3){20}}

\put(130,60){\vector(2,3){20}}

\put(117,20){\tiny{$_1$}}

\put(140,20){\tiny{$_{m+l-1}$}}

\put(117,65){\tiny{$_l$}}

\put(140,65){\tiny{$_m$}}

\put(133,42){\tiny{$_{m+l}$}}
\end{picture})\{q^{2|\lambda|-(n-1)(m-n)}\} \rightarrow C_f(\input{decomp-IV-1-slide}),\vspace{34pt}
\]
\[
\beta_{\lambda}: C_f(\input{decomp-IV-1-slide}) \rightarrow C_f(\setlength{\unitlength}{.75pt}
\begin{picture}(65,45)(110,45)

\put(110,0){\vector(2,3){20}}

\put(150,0){\vector(-2,3){20}}

\put(130,30){\vector(0,1){30}}

\put(130,60){\vector(-2,3){20}}

\put(130,60){\vector(2,3){20}}

\put(117,20){\tiny{$_1$}}

\put(140,20){\tiny{$_{m+l-1}$}}

\put(117,65){\tiny{$_l$}}

\put(140,65){\tiny{$_m$}}

\put(133,42){\tiny{$_{m+l}$}}
\end{picture})\{q^{2|\lambda|-(n-1)(m-n)}\}\vspace{34pt} 
\]
are homogeneous morphisms preserving both gradings. Define $\vec{\alpha}= \sum_{\lambda \in \Lambda'} \alpha_{\lambda}$ and $\vec{\beta}= \sum_{\lambda \in \Lambda'} \beta_{\lambda}$. Then 
\[
\vec{\alpha}: C_f(\setlength{\unitlength}{.75pt}
\begin{picture}(65,45)(110,45)

\put(110,0){\vector(2,3){20}}

\put(150,0){\vector(-2,3){20}}

\put(130,30){\vector(0,1){30}}

\put(130,60){\vector(-2,3){20}}

\put(130,60){\vector(2,3){20}}

\put(117,20){\tiny{$_1$}}

\put(140,20){\tiny{$_{m+l-1}$}}

\put(117,65){\tiny{$_l$}}

\put(140,65){\tiny{$_m$}}

\put(133,42){\tiny{$_{m+l}$}}
\end{picture})\{\qb{m-1}{n-1}\} \rightarrow C_f(\input{decomp-IV-1-slide}),\vspace{34pt}
\]
\[
\vec{\beta}: C_f(\input{decomp-IV-1-slide}) \rightarrow C_f(\setlength{\unitlength}{.75pt}
\begin{picture}(65,45)(110,45)

\put(110,0){\vector(2,3){20}}

\put(150,0){\vector(-2,3){20}}

\put(130,30){\vector(0,1){30}}

\put(130,60){\vector(-2,3){20}}

\put(130,60){\vector(2,3){20}}

\put(117,20){\tiny{$_1$}}

\put(140,20){\tiny{$_{m+l-1}$}}

\put(117,65){\tiny{$_l$}}

\put(140,65){\tiny{$_m$}}

\put(133,42){\tiny{$_{m+l}$}}
\end{picture})\{\qb{m-1}{n-1}\}\vspace{34pt} 
\]
are homogeneous morphisms preserving both gradings.

From Proposition \ref{basic-changes-varpi-0}, one can see that it is proved in \cite[Section 9]{Wu-color} that
\[
C(\setlength{\unitlength}{.75pt}
\begin{picture}(85,45)(-30,45)

\put(-20,0){\vector(0,1){45}}

\put(-20,45){\vector(0,1){45}}

\put(20,0){\vector(0,1){45}}

\put(20,45){\vector(0,1){45}}

\put(20,45){\vector(-1,0){40}}

\put(-27,20){\tiny{$_1$}}

\put(23,20){\tiny{$_{m+l-1}$}}

\put(-27,65){\tiny{$_l$}}

\put(23,65){\tiny{$_m$}}

\put(-5,38){\tiny{$_{l-1}$}}

\end{picture})\{\qb{m-1}{n}\} \bigoplus C(\setlength{\unitlength}{.75pt}
\begin{picture}(65,45)(110,45)

\put(110,0){\vector(2,3){20}}

\put(150,0){\vector(-2,3){20}}

\put(130,30){\vector(0,1){30}}

\put(130,60){\vector(-2,3){20}}

\put(130,60){\vector(2,3){20}}

\put(117,20){\tiny{$_1$}}

\put(140,20){\tiny{$_{m+l-1}$}}

\put(117,65){\tiny{$_l$}}

\put(140,65){\tiny{$_m$}}

\put(133,42){\tiny{$_{m+l}$}}
\end{picture})\{\qb{m-1}{n-1}\} \xrightarrow{(\varpi(\vec{F}),\varpi(\vec{\alpha}))} C(\input{decomp-IV-1-slide})\vspace{34pt}
\]
and 
\[
C_f(\input{decomp-IV-1-slide}) \xrightarrow{\left(%
\begin{array}{c}
\varpi(\vec{G}) \\
\varpi(\vec{\beta})
\end{array}%
\right)} C_f(\setlength{\unitlength}{.75pt}
\begin{picture}(85,45)(-30,45)

\put(-20,0){\vector(0,1){45}}

\put(-20,45){\vector(0,1){45}}

\put(20,0){\vector(0,1){45}}

\put(20,45){\vector(0,1){45}}

\put(20,45){\vector(-1,0){40}}

\put(-27,20){\tiny{$_1$}}

\put(23,20){\tiny{$_{m+l-1}$}}

\put(-27,65){\tiny{$_l$}}

\put(23,65){\tiny{$_m$}}

\put(-5,38){\tiny{$_{l-1}$}}

\end{picture})\{\qb{m-1}{n}\} \bigoplus C_f(\setlength{\unitlength}{.75pt}
\begin{picture}(65,45)(110,45)

\put(110,0){\vector(2,3){20}}

\put(150,0){\vector(-2,3){20}}

\put(130,30){\vector(0,1){30}}

\put(130,60){\vector(-2,3){20}}

\put(130,60){\vector(2,3){20}}

\put(117,20){\tiny{$_1$}}

\put(140,20){\tiny{$_{m+l-1}$}}

\put(117,65){\tiny{$_l$}}

\put(140,65){\tiny{$_m$}}

\put(133,42){\tiny{$_{m+l}$}}
\end{picture})\{\qb{m-1}{n-1}\} \vspace{34pt} 
\]
are homotopy equivalences of matrix factorizations. So, by Lemma \ref{reduce-base-homotopic-equivalence},
\[
C_f(\setlength{\unitlength}{.75pt}
\begin{picture}(85,45)(-30,45)

\put(-20,0){\vector(0,1){45}}

\put(-20,45){\vector(0,1){45}}

\put(20,0){\vector(0,1){45}}

\put(20,45){\vector(0,1){45}}

\put(20,45){\vector(-1,0){40}}

\put(-27,20){\tiny{$_1$}}

\put(23,20){\tiny{$_{m+l-1}$}}

\put(-27,65){\tiny{$_l$}}

\put(23,65){\tiny{$_m$}}

\put(-5,38){\tiny{$_{l-1}$}}

\end{picture})\{\qb{m-1}{n}\} \bigoplus C_f(\setlength{\unitlength}{.75pt}
\begin{picture}(65,45)(110,45)

\put(110,0){\vector(2,3){20}}

\put(150,0){\vector(-2,3){20}}

\put(130,30){\vector(0,1){30}}

\put(130,60){\vector(-2,3){20}}

\put(130,60){\vector(2,3){20}}

\put(117,20){\tiny{$_1$}}

\put(140,20){\tiny{$_{m+l-1}$}}

\put(117,65){\tiny{$_l$}}

\put(140,65){\tiny{$_m$}}

\put(133,42){\tiny{$_{m+l}$}}
\end{picture})\{\qb{m-1}{n-1}\} \xrightarrow{(\vec{F},\vec{\alpha})} C_f(\input{decomp-IV-1-slide})\vspace{34pt}
\]
and 
\[
C_f(\input{decomp-IV-1-slide}) \xrightarrow{\left(%
\begin{array}{c}
\vec{G} \\
\vec{\beta}
\end{array}%
\right)} C_f(\setlength{\unitlength}{.75pt}
\begin{picture}(85,45)(-30,45)

\put(-20,0){\vector(0,1){45}}

\put(-20,45){\vector(0,1){45}}

\put(20,0){\vector(0,1){45}}

\put(20,45){\vector(0,1){45}}

\put(20,45){\vector(-1,0){40}}

\put(-27,20){\tiny{$_1$}}

\put(23,20){\tiny{$_{m+l-1}$}}

\put(-27,65){\tiny{$_l$}}

\put(23,65){\tiny{$_m$}}

\put(-5,38){\tiny{$_{l-1}$}}

\end{picture})\{\qb{m-1}{n}\} \bigoplus C_f(\setlength{\unitlength}{.75pt}
\begin{picture}(65,45)(110,45)

\put(110,0){\vector(2,3){20}}

\put(150,0){\vector(-2,3){20}}

\put(130,30){\vector(0,1){30}}

\put(130,60){\vector(-2,3){20}}

\put(130,60){\vector(2,3){20}}

\put(117,20){\tiny{$_1$}}

\put(140,20){\tiny{$_{m+l-1}$}}

\put(117,65){\tiny{$_l$}}

\put(140,65){\tiny{$_m$}}

\put(133,42){\tiny{$_{m+l}$}}
\end{picture})\{\qb{m-1}{n-1}\} \vspace{34pt} 
\]
are also homotopy equivalences of matrix factorizations.
\end{proof}

\subsection{Direct sum decomposition (V)}

\begin{theorem}\label{decomp-V}
Let $m,n,l$ be non-negative integers satisfying $N\geq n+l,m+l$. Then, for $\max\{m-n,0\}\leq k \leq m+l$,
\[ 
C_f(\input{decomp-V-1-slide})  \simeq  \bigoplus_{j=\max\{m-n,0\}}^m C_f(\input{decomp-V-2-slide} ) \{\qb{l}{k-j}\}, \vspace{30pt}
\]
where we use the convention $\qb{a}{b}=0$ if $b<0$ or $b>a$. The above homotopy equivalence of matrix factorizations remains true if the orientations of all the MOY graphs are reversed.
\end{theorem}

\begin{proof}
The inductive proof of \cite[Theorem 10.1]{Wu-color} applies here without any essential changes. See \cite[Section 10]{Wu-color} for more details.
\end{proof}

\section{Chain Complexes Associated to Knotted MOY Graphs}\label{sec-chain-complex-def}

Let us first recall the definitions of knotted MOY graphs and their markings in \cite{Wu-color}.

\begin{definition}\label{knotted-MOY-def}
A knotted MOY graph is an immersion of an abstract MOY graph into $\mathbb{R}^2$ such that
\begin{itemize}
	\item the only singularities are finitely many transversal double points in the interior of edges (i.e. away from the vertices),
	\item we specify the upper edge and the lower edge at each of these transversal double points.
\end{itemize}
Each transversal double point in a knotted MOY graph is called a crossing. We follow the usual sign convention for crossings.

If there are crossings in an edge, these crossing divide the edge into several parts. We call each part a segment of the edge.
\end{definition}

Note that colored oriented link/tangle diagrams and (embedded) MOY graphs are special cases of knotted MOY graphs.

\begin{definition}\label{knotted-MOY-marking-def}
A marking of a knotted MOY graph $D$ consists the following:
\begin{enumerate}
	\item A finite collection of marked points on $D$ such that
	\begin{itemize}
	\item every segment of every edge of $D$ has at least one marked point;
	\item all the end points (vertices of valence $1$) are marked;
	\item none of the crossings and interior vertices (vertices of valence at least $2$) is marked.
  \end{itemize}
  \item An assignment of pairwise disjoint alphabets to the marked points such that the alphabet associated to a marked point on an edge of color $m$ has $m$ independent indeterminates. (Recall that an alphabet is a finite collection of homogeneous indeterminates of degree $2$.)
\end{enumerate}
\end{definition}

Given a knotted MOY graph $D$ with a marking, we cut $D$ open at the marked points. This produces a collection $\{D_1,\dots,D_m\}$ of simple knotted MOY graphs marked only at their end points. We call each $D_i$ a piece of $D$. It is easy to see that each $D_i$ is one of the following:
\begin{enumerate}[(i)]
	\item an oriented arc from one marked point to another,
	\item a star-shaped neighborhood of a vertex in an (embedded) MOY graph,
	\item a crossing with colored branches.
\end{enumerate}

\begin{definition}\label{def-complex-embedded-pieces}
If $D_i$ is of type (i) or (ii), then it is an (embedded) MOY graph, and its matrix factorization $C_f(D_i)$ is an object of $\hmf_{R_i,w_i}$. We define the chain complex associated to $D_i$, which is denoted by $\hat{C}_f(D_i)=C_f(D_i)$, to be 
\[
0 \rightarrow C_f(D_i) \rightarrow 0,
\]
where $C_f(D_i)$  has homological grading $0$.
\end{definition}

The definitions of $\hat{C}_f(D_i)$ and $C_f(D_i)$ for a colored crossing are more complex. Next, we adapt the construction in \cite{Wu-color} to define the chain complexes associated to colored crossings.

\subsection{Chain complex associated to a colored crossing}

\begin{lemma}\label{complex-computing-gamma-HMF-lemma}
Suppose that $M$ is an object of $\HMF_{\tilde{R},w}$, where 
\begin{eqnarray*}
\tilde{R} & = & \Sym(\mathbb{X}|\mathbb{Y}|\mathbb{A}|\mathbb{B})\otimes_\C R_B, \\
w & = & f(\mathbb{X}) +f(\mathbb{Y}) -f(\mathbb{A}) -f(\mathbb{B}).
\end{eqnarray*}
Then, 
\[
\Hom_\HMF (C_f(\input{complex-computing-HMF}), M) \cong H(M \otimes_{\tilde{R}} C_f(\input{complex-computing-HMF-rev})) \left\langle m+n+l \right\rangle \{q^{(l+m+n)(N-l)-m^2-n^2}\}, \vspace{30pt}
\]
where $H(M \otimes_{\tilde{R}} C_f(\input{complex-computing-HMF-rev}))\vspace{30pt}$ is the usual homology of the chain complex $M \otimes_{\tilde{R}} C_f(\input{complex-computing-HMF-rev}) \vspace{30pt}$.
\end{lemma}

\begin{proof}
The proof of \cite[Lemma 11.9]{Wu-color} applies here without change.
\end{proof}

\begin{figure}[ht]
$
\xymatrix{
\input{square-m-n-k-left} && \input{square-m-n-k-right}
}
$
\caption{}\label{decomp-V-special-1-figure}

\end{figure}

\begin{lemma}\label{decomp-V-special-1} 
Let $m,n$ be integers such that $0\leq m,n \leq N$. For $\max\{m-n,0\} \leq k \leq m$, define $\Gamma_k^L$ and $\Gamma_k^R$ to be the MOY graphs in Figure \ref{decomp-V-special-1-figure}. Then $C_f(\Gamma_k^L) \simeq C_f(\Gamma_k^R)$.
\end{lemma}

\begin{proof}
This is a special case of Theorem \ref{decomp-V}.
\end{proof}

\begin{lemma}\label{colored-crossing-res-HMF}
Let $m,n$ be integers such that $0\leq m,n \leq N$. Define $\Gamma_k^L$ and $\Gamma_k^R$ to be the MOY graphs in Figure \ref{decomp-V-special-1-figure}. For $\max\{m-n,0\} \leq j,k \leq m$, 
\begin{eqnarray*}
& & \Hom_\HMF (C_f(\Gamma_j^L), C_f(\Gamma_k^L)) \cong \Hom_\HMF (C_f(\Gamma_j^R), C_f(\Gamma_k^L)) \\
& \cong & \Hom_\HMF (C_f(\Gamma_j^L), C_f(\Gamma_k^R))  \cong \Hom_\HMF (C_f(\Gamma_j^R), C_f(\Gamma_k^R)) \\
& \cong & C_f(\emptyset) \{\qb{_{n+j+k-m}}{_k}\qb{_{n+j+k-m}}{_j} \qb{_{N+m-n-j-k}}{_{m-k}} \qb{_{N+m-n-j-k}}{_{m-j}} \qb{_N}{_{n+j+k-m}} q^{(m+n)N -n^2-m^2}\}.
\end{eqnarray*}

In particular, the lowest non-vanishing total polynomial grading of these spaces are all $(k-j)^2$. And the subspaces of homogeneous elements of total polynomial degree $(k-j)^2$ of these spaces are $1$-dimensional and have $\zed_2$ grading $0$.
\end{lemma}

\begin{proof}
This lemma follows easily from Lemma \ref{complex-computing-gamma-HMF-lemma}, Corollary \ref{contract-expand} and Decompositions (I-II) (Theorems \ref{decomp-I} and \ref{decomp-II}.) See the proof of \cite[Lemma 11.11]{Wu-color} for more details.
\end{proof}

\begin{corollary}\label{left-right-naturally-homotopic}
Let $m,n$ be integers such that $0\leq m,n \leq N$. For $\max\{m-n,0\} \leq k \leq m$, the matrix factorizations $C_f(\Gamma_k^L)$ and $C_f(\Gamma_k^R)$ are naturally homotopic in the sense that the homotopy equivalences $C_f(\Gamma_k^L) \xrightarrow{\simeq} C_f(\Gamma_k^R)$ and $C_f(\Gamma_k^R) \xrightarrow{\simeq} C_f(\Gamma_k^L)$ are unique up to homotopy and scaling.
\end{corollary}
\begin{proof}
The existence of these homotopy equivalences follows from Lemma \ref{decomp-V-special-1}. The uniqueness follows from the $j=k$ case of Lemma \ref{colored-crossing-res-HMF}.
\end{proof}

\begin{definition}\label{complex-colored-crossing-chain-maps-def}
We define $d_k^{\pm}: C_f (\Gamma_k^L) \rightarrow C_f (\Gamma_{k\mp1}^L)$ to be a homotopically non-trivial homogeneous morphism of total polynomial degree $1$ and $\zed_2$ degree $0$. By Lemma \ref{colored-crossing-res-HMF}, $d_k^{\pm}$ is uniquely defined up to homotopy and scaling.

By Corollary \ref{left-right-naturally-homotopic}, $C_f (\Gamma_k^L)$ and $C_f (\Gamma_k^R)$ are naturally homotopic (up scaling.) Abusing the notation, we let 
\begin{eqnarray*}
d_k^{\pm}&:& C_f (\Gamma_k^L) \rightarrow C_f (\Gamma_{k\mp1}^R) \\
d_k^{\pm}&:& C_f (\Gamma_k^R) \rightarrow C_f (\Gamma_{k\mp1}^L) \\
d_k^{\pm}&:& C_f (\Gamma_k^R) \rightarrow C_f (\Gamma_{k\mp1}^R)
\end{eqnarray*}
be the morphisms corresponding to $d_k^{\pm}: C_f (\Gamma_k^L) \rightarrow C_f (\Gamma_{k\mp1}^L)$ under the natural homotopy equivalence. These morphisms are all homotopically non-trivial homogeneous morphisms of total polynomial degree $1$ and $\zed_2$ degree $0$, which, by Lemma \ref{colored-crossing-res-HMF},  uniquely defines these morphisms up to homotopy and scaling.
\end{definition}

\begin{corollary}\label{left-right-naturally-homotopic-2}
Up to homotopy and scaling, every square in the following diagram commutes, where the vertical morphisms are either identity or the natural homotopy equivalence.
\[
\xymatrix{
C_f (\Gamma_k^L) \ar[rr]^{d_k^{\pm}}  \ar[d]^{\simeq} && C_f (\Gamma_{k\mp1}^L) \ar[d]^{=} \\
C_f (\Gamma_k^R) \ar[rr]^{d_k^{\pm}}  \ar[d]^{\simeq} && C_f (\Gamma_{k\mp1}^L) \ar[d]^{\simeq} \\
C_f (\Gamma_k^L) \ar[rr]^{d_k^{\pm}}  \ar[d]^{\simeq} && C_f (\Gamma_{k\mp1}^R) \ar[d]^{=} \\
C_f (\Gamma_k^R) \ar[rr]^{d_k^{\pm}}   && C_f (\Gamma_{k\mp1}^R)
}
\]
\end{corollary}

\begin{proof}
This follows easily from Definition \ref{complex-colored-crossing-chain-maps-def} and Lemma \ref{colored-crossing-res-HMF}.
\end{proof}

\begin{theorem}\label{complex-colored-crossing-well-defined}
$d_{k\mp1}^{\pm} \circ d_k^{\pm} \simeq 0$.
\end{theorem}

\begin{proof}
Note that $d_{k\mp1}^{\pm} \circ d_k^{\pm} : C_f (\Gamma_k^L) \rightarrow C_f (\Gamma_{k\mp2}^L)$ has total polynomial grading $2$. But, by Lemma \ref{colored-crossing-res-HMF}, the lowest non-vanishing total polynomial grading of $\Hom_\HMF(C(\Gamma_k^L), C(\Gamma_{k\mp2}^L))$ is $2^2=4$. So $d_{k\mp1}^{\pm} \circ d_k^{\pm} \simeq 0$.
\end{proof}

\begin{definition}\label{complex-colored-crossing-def}
Let $\tilde{R}=\Sym(\mathbb{X}|\mathbb{Y}|\mathbb{A}|\mathbb{D})\otimes R_B$, $w= f(\mathbb{X}) +f(\mathbb{Y}) -f(\mathbb{A}) - f(\mathbb{D})$, and $\Gamma^L_k$ the MOY graph in Figure \ref{decomp-V-special-1-figure}. Following \cite[Definition 2.33]{Wu-color}, denote by $\hch(\hmf_{\tilde{R},w})$ the homotopy category of bounded chain complexes over $\hmf_{\tilde{R},w}$.

We define the unnormalized chain complex $\hat{C}_f$ first.

$\hat{C}_f (\setlength{\unitlength}{1pt}
\begin{picture}(65,20)(-110,0)
\put(-100,-20){\vector(1,1){40}}

\put(-60,-20){\line(-1,1){15}}

\put(-85,5){\vector(-1,1){15}}

\put(-92,16){\tiny{$_m$}}

\put(-70,16){\tiny{$_n$}}

\put(-105,-20){\tiny{$\mathbb{A}$}}
\put(-105,15){\tiny{$\mathbb{X}$}}

\put(-55,-20){\tiny{$\mathbb{D}$}}
\put(-55,15){\tiny{$\mathbb{Y}$}}\end{picture}) \vspace{20pt}$ is defined to be the object 
\[
0\rightarrow C_f(\Gamma^L_m) \xrightarrow{d^{+}_m} C_f(\Gamma^L_{m-1})\{q^{-1}\} \xrightarrow{d^{+}_{m-1}} \cdots \xrightarrow{d^{+}_1} C_f(\Gamma^L_0)\{q^{-m}\} \rightarrow 0
\]
of $\hch(\hmf_{\tilde{R},w})$, where the homological grading on $\hat{C}_f (\setlength{\unitlength}{1pt}
\begin{picture}(65,20)(-110,0)
\put(-100,-20){\vector(1,1){40}}

\put(-60,-20){\line(-1,1){15}}

\put(-85,5){\vector(-1,1){15}}

\put(-92,16){\tiny{$_m$}}

\put(-70,16){\tiny{$_n$}}

\put(-105,-20){\tiny{$\mathbb{A}$}}
\put(-105,15){\tiny{$\mathbb{X}$}}

\put(-55,-20){\tiny{$\mathbb{D}$}}
\put(-55,15){\tiny{$\mathbb{Y}$}}\end{picture}) \vspace{20pt}$ is defined so that the term $C_f(\Gamma^L_{k})\{q^{-(m-k)}\}$ have homological grading $m-k$. Note that, if $m>n$, then the last homotopically non-trivial term in the chain complex $\hat{C}_f (\setlength{\unitlength}{1pt}
\begin{picture}(65,20)(-110,0)
\put(-100,-20){\vector(1,1){40}}

\put(-60,-20){\line(-1,1){15}}

\put(-85,5){\vector(-1,1){15}}

\put(-92,16){\tiny{$_m$}}

\put(-70,16){\tiny{$_n$}}

\put(-105,-20){\tiny{$\mathbb{A}$}}
\put(-105,15){\tiny{$\mathbb{X}$}}

\put(-55,-20){\tiny{$\mathbb{D}$}}
\put(-55,15){\tiny{$\mathbb{Y}$}}\end{picture}) \vspace{20pt}$ is $C_f(\Gamma^L_{m-n})\{q^{-n}\}$ since $C_f(\Gamma^L_{k}) \simeq 0$ if $k < \max \{0, m-n\}$.

$\hat{C}_f (\setlength{\unitlength}{1pt}
\begin{picture}(65,20)(50,0)
\put(100,-20){\vector(-1,1){40}}

\put(60,-20){\line(1,1){15}}

\put(85,5){\vector(1,1){15}}

\put(68,16){\tiny{$_m$}}

\put(90,16){\tiny{$_n$}}

\put(55,-20){\tiny{$\mathbb{A}$}}
\put(55,15){\tiny{$\mathbb{X}$}}

\put(105,-20){\tiny{$\mathbb{D}$}}
\put(105,15){\tiny{$\mathbb{Y}$}}
\end{picture}) \vspace{20pt}$ is defined to be the object 
\[
0\rightarrow C_f(\Gamma^L_0)\{q^{m}\} \xrightarrow{d^{-}_0} \cdots \xrightarrow{d^{-}_{m-2}} C_f(\Gamma^L_{m-1}) \{ q \} \xrightarrow{d^{-}_{m-1}} C_f(\Gamma^L_m) \rightarrow 0
\]
of $\hch(\hmf_{\tilde{R},w})$, where the homological grading on $\hat{C}_f (\setlength{\unitlength}{1pt}
\begin{picture}(65,20)(50,0)
\put(100,-20){\vector(-1,1){40}}

\put(60,-20){\line(1,1){15}}

\put(85,5){\vector(1,1){15}}

\put(68,16){\tiny{$_m$}}

\put(90,16){\tiny{$_n$}}

\put(55,-20){\tiny{$\mathbb{A}$}}
\put(55,15){\tiny{$\mathbb{X}$}}

\put(105,-20){\tiny{$\mathbb{D}$}}
\put(105,15){\tiny{$\mathbb{Y}$}}
\end{picture}) \vspace{20pt}$ is defined so that the term $C_f(\Gamma^L_{k})\{q^{m-k}\}$ has homological grading $k-m$. Again, note that, if $m>n$, then the first homotopically non-trivial term in the chain complex $\hat{C}_f (\setlength{\unitlength}{1pt}
\begin{picture}(65,20)(50,0)
\put(100,-20){\vector(-1,1){40}}

\put(60,-20){\line(1,1){15}}

\put(85,5){\vector(1,1){15}}

\put(68,16){\tiny{$_m$}}

\put(90,16){\tiny{$_n$}}

\put(55,-20){\tiny{$\mathbb{A}$}}
\put(55,15){\tiny{$\mathbb{X}$}}

\put(105,-20){\tiny{$\mathbb{D}$}}
\put(105,15){\tiny{$\mathbb{Y}$}}
\end{picture}) \vspace{20pt}$ is $C_f(\Gamma_{m-n})\{q^{n}\}$ since $C_f(\Gamma^L_{k}) \simeq 0$ if $k < \max \{0, m-n\}$.

By Corollaries \ref{left-right-naturally-homotopic} and \ref{left-right-naturally-homotopic-2}, changing $\Gamma^L_k$ in the above construction to $\Gamma^R_k$ does not change the isomorphism type of $\hat{C}_f (\setlength{\unitlength}{1pt}
\begin{picture}(65,20)(-110,0)
\put(-100,-20){\vector(1,1){40}}

\put(-60,-20){\line(-1,1){15}}

\put(-85,5){\vector(-1,1){15}}

\put(-92,16){\tiny{$_m$}}

\put(-70,16){\tiny{$_n$}}

\put(-105,-20){\tiny{$\mathbb{A}$}}
\put(-105,15){\tiny{$\mathbb{X}$}}

\put(-55,-20){\tiny{$\mathbb{D}$}}
\put(-55,15){\tiny{$\mathbb{Y}$}}\end{picture}) \vspace{20pt}$ and $\hat{C}_f (\setlength{\unitlength}{1pt}
\begin{picture}(65,20)(50,0)
\put(100,-20){\vector(-1,1){40}}

\put(60,-20){\line(1,1){15}}

\put(85,5){\vector(1,1){15}}

\put(68,16){\tiny{$_m$}}

\put(90,16){\tiny{$_n$}}

\put(55,-20){\tiny{$\mathbb{A}$}}
\put(55,15){\tiny{$\mathbb{X}$}}

\put(105,-20){\tiny{$\mathbb{D}$}}
\put(105,15){\tiny{$\mathbb{Y}$}}
\end{picture}) \vspace{20pt}$.

The normalized chain complexes $C_f (\setlength{\unitlength}{1pt}
\begin{picture}(65,20)(-110,0)
\put(-100,-20){\vector(1,1){40}}

\put(-60,-20){\line(-1,1){15}}

\put(-85,5){\vector(-1,1){15}}

\put(-92,16){\tiny{$_m$}}

\put(-70,16){\tiny{$_n$}}

\put(-105,-20){\tiny{$\mathbb{A}$}}
\put(-105,15){\tiny{$\mathbb{X}$}}

\put(-55,-20){\tiny{$\mathbb{D}$}}
\put(-55,15){\tiny{$\mathbb{Y}$}}\end{picture}) \vspace{20pt}$ and $C_f (\setlength{\unitlength}{1pt}
\begin{picture}(65,20)(50,0)
\put(100,-20){\vector(-1,1){40}}

\put(60,-20){\line(1,1){15}}

\put(85,5){\vector(1,1){15}}

\put(68,16){\tiny{$_m$}}

\put(90,16){\tiny{$_n$}}

\put(55,-20){\tiny{$\mathbb{A}$}}
\put(55,15){\tiny{$\mathbb{X}$}}

\put(105,-20){\tiny{$\mathbb{D}$}}
\put(105,15){\tiny{$\mathbb{Y}$}}
\end{picture}) \vspace{20pt}$ are defined to be
\begin{eqnarray*}
C_f(\setlength{\unitlength}{1pt}
\begin{picture}(65,20)(-110,0)
\put(-100,-20){\vector(1,1){40}}

\put(-60,-20){\line(-1,1){15}}

\put(-85,5){\vector(-1,1){15}}

\put(-92,16){\tiny{$_m$}}

\put(-70,16){\tiny{$_n$}}

\put(-105,-20){\tiny{$\mathbb{A}$}}
\put(-105,15){\tiny{$\mathbb{X}$}}

\put(-55,-20){\tiny{$\mathbb{D}$}}
\put(-55,15){\tiny{$\mathbb{Y}$}}\end{picture}) & = & \begin{cases}
\hat{C}_f(\setlength{\unitlength}{1pt}
\begin{picture}(65,20)(-110,0)
\put(-100,-20){\vector(1,1){40}}

\put(-60,-20){\line(-1,1){15}}

\put(-85,5){\vector(-1,1){15}}

\put(-92,16){\tiny{$_m$}}

\put(-70,16){\tiny{$_n$}}

\put(-105,-20){\tiny{$\mathbb{A}$}}
\put(-105,15){\tiny{$\mathbb{X}$}}

\put(-55,-20){\tiny{$\mathbb{D}$}}
\put(-55,15){\tiny{$\mathbb{Y}$}}\end{picture})\left\langle m \right\rangle\|-m\| \{q^{m(N+1-m)}\} & \text{if } m=n,  \vspace{20pt} \\
\hat{C}_f(\setlength{\unitlength}{1pt}
\begin{picture}(65,20)(-110,0)
\put(-100,-20){\vector(1,1){40}}

\put(-60,-20){\line(-1,1){15}}

\put(-85,5){\vector(-1,1){15}}

\put(-92,16){\tiny{$_m$}}

\put(-70,16){\tiny{$_n$}}

\put(-105,-20){\tiny{$\mathbb{A}$}}
\put(-105,15){\tiny{$\mathbb{X}$}}

\put(-55,-20){\tiny{$\mathbb{D}$}}
\put(-55,15){\tiny{$\mathbb{Y}$}}\end{picture}) & \text{if } m\neq n, \vspace{20pt}
\end{cases} \\
C_f(\setlength{\unitlength}{1pt}
\begin{picture}(65,20)(50,0)
\put(100,-20){\vector(-1,1){40}}

\put(60,-20){\line(1,1){15}}

\put(85,5){\vector(1,1){15}}

\put(68,16){\tiny{$_m$}}

\put(90,16){\tiny{$_n$}}

\put(55,-20){\tiny{$\mathbb{A}$}}
\put(55,15){\tiny{$\mathbb{X}$}}

\put(105,-20){\tiny{$\mathbb{D}$}}
\put(105,15){\tiny{$\mathbb{Y}$}}
\end{picture}) & = & \begin{cases}
\hat{C}_f(\setlength{\unitlength}{1pt}
\begin{picture}(65,20)(50,0)
\put(100,-20){\vector(-1,1){40}}

\put(60,-20){\line(1,1){15}}

\put(85,5){\vector(1,1){15}}

\put(68,16){\tiny{$_m$}}

\put(90,16){\tiny{$_n$}}

\put(55,-20){\tiny{$\mathbb{A}$}}
\put(55,15){\tiny{$\mathbb{X}$}}

\put(105,-20){\tiny{$\mathbb{D}$}}
\put(105,15){\tiny{$\mathbb{Y}$}}
\end{picture})\left\langle m \right\rangle\|m\| \{q^{-m(N+1-m)}\} & \text{if } m=n,  \vspace{20pt}\\
\hat{C}_f(\setlength{\unitlength}{1pt}
\begin{picture}(65,20)(50,0)
\put(100,-20){\vector(-1,1){40}}

\put(60,-20){\line(1,1){15}}

\put(85,5){\vector(1,1){15}}

\put(68,16){\tiny{$_m$}}

\put(90,16){\tiny{$_n$}}

\put(55,-20){\tiny{$\mathbb{A}$}}
\put(55,15){\tiny{$\mathbb{X}$}}

\put(105,-20){\tiny{$\mathbb{D}$}}
\put(105,15){\tiny{$\mathbb{Y}$}}
\end{picture}) & \text{if } m\neq n. \vspace{20pt}
\end{cases}
\end{eqnarray*}
Here, $\|m\|$ means shifting the homological grading by $m$. (See \cite[Definition 2.33]{Wu-color}.)
\end{definition}

\subsection{The chain complex associated to a knotted MOY graph} Let $D$ be a knotted MOY graph with a marking, and $D_1,\dots,D_l$ its pieces. In Definitions \ref{def-complex-embedded-pieces} and \ref{complex-colored-crossing-def}, we have defined the chain complexes $\hat{C}_f(D_i)$ and $C_f(D_i)$ for each $D_i$.

\begin{definition}\label{complex-knotted-MOY-def}
\begin{eqnarray*}
\hat{C}_f(D) & := & \bigotimes_{i=1}^m \hat{C}_f(D_i), \\
C_f(D) & := & \bigotimes_{i=1}^m C_f(D_i),
\end{eqnarray*}
where the tensor product is done over the common end points. More precisely, for two pieces $D_{i_1}$ and $D_{i_2}$ of $D$, let $\mathbb{W}_1,\dots,\mathbb{W}_j$ be the alphabets associated to their common end points. Then, in the above tensor product, 
\[
C_f(D_{i_1}) \otimes C_f(D_{i_2}) := C_f(D_{i_1}) \otimes_{\Sym(\mathbb{W}_1|\cdots|\mathbb{W}_j)\otimes_\C R_B} C_f(D_{i_2}).
\]

If $D$ is closed, i.e. has no endpoints, then $C_f(D)$ is an object of $\hch(\hmf_{R_B,0})$. Assume $D$ has endpoints. Let $\mathbb{E}_1,\dots,\mathbb{E}_n$ be the alphabets assigned to all end points of $D$, among which $\mathbb{E}_1,\dots,\mathbb{E}_k$ are assigned to exits and $\mathbb{E}_{k+1},\dots,\mathbb{E}_n$ are assigned to entrances. Let $\tilde{R}=\Sym(\mathbb{E}_1|\cdots|\mathbb{E}_n)\otimes_\C R_B$ and $w= \sum_{i=1}^k f(\mathbb{E}_i) - \sum_{j=k+1}^n f(\mathbb{E}_j)$. In this case, $C_f(D)$ is an object of $\hch(\hmf_{\tilde{R},w})$.

Note that $C_f(D)$ has a $\zed_2$-grading, a total polynomial grading and a homological grading.
\end{definition}

\begin{corollary}\label{complex-knotted-MOY-marking-independence}
The isomorphism type of the chain complexes $\hat{C}_f(D)$ and $C_f(D)$ associated to a knotted MOY graph $D$ is independent of the choice of the marking of $D$.
\end{corollary}

\begin{proof}
This follows easily from Lemma \ref{marking-independence} and Corollary \ref{left-right-naturally-homotopic-2}. See \cite[Corollary 11.18]{Wu-color} for a more detailed proof.
\end{proof}

\subsection{A null-homotopic chain complex}\label{subsec-null-chain} Next we introduce a null-homotopic chain complex that will appear in our proof of the invariance under fork sliding. The construction of this chain complex is similar to the chain complex of a colored crossing.

\begin{figure}[ht]
$
\xymatrix{
\input{square-m-n-1-right-k-low} && \input{square-m-n-1-right-j-high}
}
$
\caption{}\label{decomp-V-special-2-figure}

\end{figure}

\begin{lemma}\label{decomp-V-special-2} 
Let $m,n$ be integers such that $0\leq m,n \leq N-1$. For $\max\{m-n,0\} \leq k \leq m+1$ and $\max\{m-n,0\} \leq j \leq m$, define $\Gamma_k$ and $\Gamma_j'$ to be the MOY graphs in Figure \ref{decomp-V-special-2-figure}. Then, for $\max\{m-n,0\} \leq k \leq m+1$, 
\[
C_f(\Gamma_k) \simeq \begin{cases}
C_f(\Gamma_m') & \text{if } k=m+1, \\
C_f(\Gamma_k') \oplus C_f(\Gamma_{k-1}') & \text{if } \max\{m-n,0\}+1 \leq k \leq m, \\
C_f(\Gamma_{\max\{m-n,0\}}') & \text{if } k = \max\{m-n,0\}.
\end{cases}
\]
\end{lemma}

\begin{proof}
This is a special case of Decomposition (V) (Theorem \ref{decomp-V}.)
\end{proof}

\begin{lemma}\label{trivial-complex-lemma-1}
Let $\Gamma_k$ and $\Gamma_j'$ be as in Lemma \ref{decomp-V-special-2}. Then
\begin{eqnarray*}
&& \Hom_\HMF (C_f(\Gamma_j'), C_f(\Gamma_k)) \\
& \cong & C_f(\emptyset) \{\qb{_{n+k+j-m}}{_k} \qb{_{n+k+j-m}}{_j} \qb{_{N+m-n-k-j}}{_{m-j}} \qb{_{N+m-n-k-j}}{_{m+1-k}} \qb{_N}{_{n+k+j-m}} q^{(m+n+1)(N-1)-m^2-n^2} \}.
\end{eqnarray*}
In particular, the space is concentrated on $\zed_2$-grading $0$. The lowest non-vanishing total polynomial grading of the above space is $(j-k)(j-k+1)$. Moreover, the subspace of homogeneous elements of total polynomial degree $(j-k)(j-k+1)$ is $1$-dimensional.
\end{lemma}

\begin{proof}
This follows easily from Lemma \ref{complex-computing-gamma-HMF-lemma}, Corollary \ref{contract-expand} and Decompositions (I-II) (Theorems \ref{decomp-I} and \ref{decomp-II}.) See \cite[Lemma 11.20]{Wu-color} for a detailed proof.
\end{proof}

The next two lemmas are easy consequences of Lemma \ref{trivial-complex-lemma-1}. See \cite[Lemmas 11.21 and 11.22]{Wu-color} for their proofs.

\begin{lemma}\label{trivial-complex-lemma-2}
For $\max\{m-n,0\} \leq i,j \leq m$,
\[
\Hom_\hmf( C_f(\Gamma_i'),C_f(\Gamma_j')) \cong \begin{cases}
\C & \text{if } i=j, \\
0 & \text{if } i \neq j.
\end{cases}
\]
In the case $i=j$, $\Hom_\hmf( C_f(\Gamma_i'),C_f(\Gamma_i'))$ is spanned by $\id_{C_f(\Gamma_i')}$.
\end{lemma}

\begin{lemma}\label{trivial-complex-lemma-3}
For $\max\{m-n,0\} \leq j,k \leq m+1$,
\[
\Hom_\hmf( C_f(\Gamma_j),C_f(\Gamma_k)) \cong \begin{cases}
\C \oplus \C & \text{if } \max\{m-n,0\}+1 \leq j=k \leq m, \\
\C & \text{if } j=k=\max\{m-n,0\} \text{ or } m+1, \\
\C & \text{if } |j-k|=1,\\
0 & \text{if } |j-k|>1.
\end{cases}
\]
\end{lemma}

\begin{definition}\label{trivial-complex-differential-def}
Denote by 
\begin{eqnarray*}
J_{k,k} & : & C_f(\Gamma_k') \rightarrow C_f(\Gamma_k) \\
J_{k,k-1} & : & C_f(\Gamma_{k-1}') \rightarrow C_f(\Gamma_k) \\
P_{k,k} & : & C_f(\Gamma_k) \rightarrow C_f(\Gamma_k') \\
P_{k,k-1} & : & C_f(\Gamma_k) \rightarrow C_f(\Gamma_{k-1}') 
\end{eqnarray*}
the inclusion and projection morphisms in the decomposition 
\[
C_f(\Gamma_k) \simeq C_f(\Gamma_k') \oplus C_f(\Gamma_{k-1}').
\]
Define 
\begin{eqnarray*}
\delta_k^+ & = & J_{k-1,k-1}\circ P_{k,k-1} : C_f(\Gamma_k) \rightarrow C_f(\Gamma_{k-1}), \\
\delta_k^- & = & J_{k+1,k} \circ P_{k,k} : C_f(\Gamma_k) \rightarrow C_f(\Gamma_{k+1}).
\end{eqnarray*}
$\delta_k^+$ and $\delta_k^-$ are both homotopically non-trivial homogeneous morphisms preserving both the $\zed_2$-grading and the total polynomial grading. 
By Lemma \ref{trivial-complex-lemma-3}, up to homotopy and scaling, $\delta_k^+$ and $\delta_k^-$ are the unique morphisms with such properties.
\end{definition}

\begin{lemma}\label{trivial-complex-lemma-4}
$\delta_{k-1}^+\circ\delta_k^+ \simeq 0$, $\delta_{k+1}^-\circ\delta_k^- \simeq 0$.
\end{lemma}

\begin{proof}
By definition, $\delta_{k-1}^+\circ\delta_k^+$ and $\delta_{k+1}^-\circ\delta_k^-$ preserve both the $\zed_2$-grading and the total polynomial grading. But, from Lemma \ref{trivial-complex-lemma-3}, we have that 
\[
\Hom_\hmf( C_f(\Gamma_k),C_f(\Gamma_{k-2})) \cong \Hom_\hmf( C_f(\Gamma_k),C_f(\Gamma_{k+2})) \cong0.
\] 
So $\delta_{k-1}^+\circ\delta_k^+ \simeq 0$, $\delta_{k+1}^-\circ\delta_k^- \simeq 0$.
\end{proof}

Let $\tilde{R}=\Sym(\mathbb{X}|\mathbb{Y}|\mathbb{A}|\mathbb{D})\otimes_\C R_B$ and $w= f(\mathbb{X}) +f(\mathbb{Y}) -f(\mathbb{A}) - f(\mathbb{D})$. The above lemmas imply the following.

\begin{proposition}\label{trivial-complex-prop}
Let $k_1$ and $k_2$ be integers such that $\max\{m-n,0\}+1 \leq k_1\leq k_2 \leq m$. Then 
\[
\xymatrix{
0 \ar[r] & C_f(\Gamma_{k_2}')\ar[r]^{J_{k_2,k_2}} & C_f(\Gamma_{k_2}) \ar[r]^{\delta_{k_2}^+} & \cdots \ar[r]^{\delta_{k_1+1}^+} & C_f(\Gamma_{k_1}) \ar[r]^{P_{k_1,k_1-1}} & C_f(\Gamma_{k_1-1}') \ar[r] & 0, \\
0 \ar[r] & C_f(\Gamma_{k_1-1}')\ar[r]^{J_{k_1,k_1-1}} & C_f(\Gamma_{k_1}) \ar[r]^{\delta_{k_1}^-} & \cdots \ar[r]^{\delta_{k_2-1}^-} & C_f(\Gamma_{k_2}) \ar[r]^{P_{k_2,k_2}} & C_f(\Gamma_{k_2}') \ar[r] & 0
}
\]
are both chain complexes over $\hmf_{\tilde{R},w}$ and are isomorphic in $\ch(\hmf_{\tilde{R},w})$ to
\[
\bigoplus_{j=k_1-1}^{k_2} (\xymatrix{
0 \ar[r] & C_f(\Gamma_j') \ar[r]^{\simeq} & C_f(\Gamma_j') \ar[r] & 0,
})
\]
which is homotopic to $0$ (i.e. isomorphic in $\hch(\hmf_{\tilde{R},w})$ to $0$.)
\end{proposition}

\subsection{Explicit forms of the differential maps} In this subsection, we give an explicit construction of the differential maps of the chain complex of a colored crossing and the null-homotopic chain complex introduced above.

\begin{figure}[ht]
$
\xymatrix{
\input{square-m-n-l-right-k-low} \ar@<1ex>[rr]^{\mathsf{d}_k^+} \ar@<1ex>[d]^{\phi_{k,1}} & & \input{square-m-n-l-right-k-1-low} \ar@<1ex>[ll]^{\mathsf{d}_{k-1}^-} \ar@<1ex>[d]^{\phi_{k,2}} \\
\input{square-m-n-l-right-k-low-bubble} \ar@<1ex>[r]^{\chi^1\otimes\chi^1} \ar@<1ex>[u]^{\overline{\phi}_{k,1}} & \input{double-square-m-n-l-right-k-low} \ar@<1ex>[r]^{h_k} \ar@<1ex>[l]^{\chi^0\otimes\chi^0} & \input{square-m-n-l-right-k-1-low-bubble} \ar@<1ex>[l]^{\overline{h}_k} \ar@<1ex>[u]^{\overline{\phi}_{k,2}}
}
$
\caption{}\label{explicit-differential-general-figure1}

\end{figure}

Consider the MOY graphs and the morphisms in Figure \ref{explicit-differential-general-figure1}, where $\phi_{k,1}, ~ \overline{\phi}_{k,1},  ~\chi^0\otimes\chi^0, ~\chi^1\otimes \chi^1, h_k, ~\overline{h}_k,~\phi_{k,2}, ~ \overline{\phi}_{k,2}$ are the morphisms induced by the apparent basic local changes of MOY graphs as defined in Section \ref{sec-some-morph}. Define $\mathsf{d}_k^+$ and $\mathsf{d}_{k-1}^-$ by
\begin{eqnarray*}
\mathsf{d}_k^+ & = & \overline{\phi}_{k,2} \circ h_k \circ (\chi^1\otimes\chi^1) \circ \phi_{k,1}. \\
\mathsf{d}_{k-1}^- & = & \overline{\phi}_{k,1} \circ (\chi^0\otimes\chi^0) \circ \overline{h}_k \circ \phi_{k,2}.
\end{eqnarray*}

\begin{theorem}\label{explicit-differential-general}
$\mathsf{d}_k^+$ and $\mathsf{d}_{k-1}^-$ are homotopically non-trivial homogeneous morphisms of $\zed_2$-degree $0$ and quantum degree $1-l$. 

When $l=0$, $\mathsf{d}_k^+$ and $\mathsf{d}_{k-1}^-$ are homotopic to scalar multiples of the differential maps $d_k^+$ and $d_{k-1}^-$ of the chain complexes associated to colored crossings defined in Definition \ref{complex-colored-crossing-def}. 

When $l=1$, $\mathsf{d}_k^+$ and $\mathsf{d}_{k-1}^-$ are homotopic to scalar multiples of the differential maps $\delta_k^+$ and $\delta_{k-1}^-$ of the null-homotopic chain complexes in Proposition \ref{trivial-complex-prop}.
\end{theorem}

\begin{proof}
It is easy to check that $\mathsf{d}_k^+$ and $\mathsf{d}_{k-1}^-$ are homogeneous morphisms of $\zed_2$-degree $0$ and total polynomial degree $1-l$. Recall that the differential maps of the complexes in Definition \ref{complex-colored-crossing-def} and Proposition \ref{trivial-complex-prop} are homotopically non-trivial homogeneous morphisms uniquely determined up to homotopy and scaling by their total polynomial degrees. So, to prove Theorem \ref{explicit-differential-general}, we only need to prove that $\mathsf{d}_k^+$ and $\mathsf{d}_{k-1}^-$ are homotopically non-trivial.

Let $\varpi_0$ be the functor given in Lemma \ref{MOY-object-of-hmf}. By Proposition \ref{basic-changes-varpi-0} and \cite[Theorem 11.26]{Wu-color}, we know that $\varpi_0(\mathsf{d}_k^+)$ and $\varpi_0(\mathsf{d}_{k-1}^-)$ are homotopically non-trivial, which implies that $\mathsf{d}_k^+$ and $\mathsf{d}_{k-1}^-$ are homotopically non-trivial.
\end{proof}

\begin{corollary}\label{differentials-varpi-0}
The functor $\varpi_0$ given in Lemma \ref{MOY-object-of-hmf} induces functors of the categories $\ch(\hmf)$ and $\hch(\hmf)$, which are again denoted by $\varpi_0$.

If $D$ is a knotted MOY graph, then $\varpi_0(C_f(D)) \cong C(D)$ and $\varpi_0(\hat{C}_f(D)) \cong \hat{C}(D)$, where $C(D)$ and $\hat{C}(D)$ are the chain complexes associated to $D$ in \cite[Definition 11.4]{Wu-color}.

If we apply $\varpi_0$ to the null-homotopic chain complexes defined in Proposition \ref{trivial-complex-prop}, then we get the corresponding null-homotopic chain complexes defined in \cite[Proposition 11.25]{Wu-color}.
\end{corollary}

\begin{proof}
Since $\varpi_0$ is induced by a projection map of the base ring, it is clear that $\varpi_0$ induces functors of the categories $\ch(\hmf)$ and $\hch(\hmf)$. By Lemma \ref{MOY-object-of-hmf}, $\varpi_0(\hat{C}_f(D))$ and $\hat{C}(D)$ are chain complexes over $\hmf$ with the same terms. By Proposition \ref{basic-changes-varpi-0}, Theorem \ref{explicit-differential-general} and \cite[Theorem 11.26]{Wu-color}, we know that the chain maps of $\varpi_0(\hat{C}_f(D))$ and $\hat{C}(D)$ are the same (up to homotopy and scaling of morphisms of matrix factorizations for each piece of $D$.) So $\varpi_0(\hat{C}_f(D)) \cong \hat{C}(D)$. Note that $\varpi_0({C}_f(D))$ and ${C}(D)$ are define from $\varpi_0(\hat{C}_f(D))$ and $\hat{C}(D)$ by the same grading shifts. So $\varpi_0(C_f(D)) \cong C(D)$.

The result about chain complexes defined in Proposition \ref{trivial-complex-prop} can be proved similarly and is left to the reader.
\end{proof}

\begin{corollary}\label{explicit-differential-1-n-crossings--res}
\[
\hat{C}_f(\setlength{\unitlength}{1pt}
\begin{picture}(40,20)(-20,0)

\put(-20,-20){\vector(1,1){40}}

\put(20,-20){\line(-1,1){15}}

\put(-5,5){\vector(-1,1){15}}

\put(-11,15){\tiny{$_1$}}

\put(9,15){\tiny{$_n$}}

\end{picture}) \cong ``0 \rightarrow C_f(\setlength{\unitlength}{1pt}
\begin{picture}(50,20)(-25,20)

\put(-20,0){\vector(2,1){20}}

\put(20,0){\vector(-2,1){20}}

\put(0,30){\vector(-2,1){20}}

\put(0,30){\vector(2,1){20}}

\put(0,10){\vector(0,1){20}}

\put(-17,7){\tiny{$_n$}}

\put(15,7){\tiny{$_1$}}

\put(-17,33){\tiny{$_1$}}

\put(15,33){\tiny{$_n$}}

\put(3,20){\tiny{$_{n+1}$}}

\end{picture}) \xrightarrow{\chi^1} C_f(\setlength{\unitlength}{1pt}
\begin{picture}(50,20)(-25,20)

\put(-20,0){\vector(0,1){20}}

\put(-20,20){\vector(0,1){20}}

\put(20,0){\vector(0,1){20}}

\put(20,20){\vector(0,1){20}}

\put(-20,20){\vector(1,0){40}}

\put(-17,35){\tiny{$_1$}}

\put(15,35){\tiny{$_n$}}

\put(-17,0){\tiny{$_n$}}

\put(15,0){\tiny{$_1$}}

\put(-8,23){\tiny{$_{n-1}$}}

\end{picture})\{q^{-1}\} \rightarrow 0", \vspace{20pt} 
\]
\[
\hat{C}_f(\setlength{\unitlength}{1pt}
\begin{picture}(40,20)(-20,0)

\put(20,-20){\vector(-1,1){40}}

\put(-20,-20){\line(1,1){15}}

\put(5,5){\vector(1,1){15}}

\put(-11,15){\tiny{$_1$}}

\put(9,15){\tiny{$_n$}}

\end{picture}) \cong ``0 \rightarrow C_f(\setlength{\unitlength}{1pt}
\begin{picture}(50,20)(-25,20)

\put(-20,0){\vector(0,1){20}}

\put(-20,20){\vector(0,1){20}}

\put(20,0){\vector(0,1){20}}

\put(20,20){\vector(0,1){20}}

\put(-20,20){\vector(1,0){40}}

\put(-17,35){\tiny{$_1$}}

\put(15,35){\tiny{$_n$}}

\put(-17,0){\tiny{$_n$}}

\put(15,0){\tiny{$_1$}}

\put(-8,23){\tiny{$_{n-1}$}}

\end{picture})\{q\} \xrightarrow{\chi^0} C_f(\setlength{\unitlength}{1pt}
\begin{picture}(50,20)(-25,20)

\put(-20,0){\vector(2,1){20}}

\put(20,0){\vector(-2,1){20}}

\put(0,30){\vector(-2,1){20}}

\put(0,30){\vector(2,1){20}}

\put(0,10){\vector(0,1){20}}

\put(-17,7){\tiny{$_n$}}

\put(15,7){\tiny{$_1$}}

\put(-17,33){\tiny{$_1$}}

\put(15,33){\tiny{$_n$}}

\put(3,20){\tiny{$_{n+1}$}}

\end{picture}) \rightarrow 0",\vspace{20pt} 
\]
\[
\hat{C}_f(\setlength{\unitlength}{1pt}
\begin{picture}(40,20)(-20,0)

\put(-20,-20){\vector(1,1){40}}

\put(20,-20){\line(-1,1){15}}

\put(-5,5){\vector(-1,1){15}}

\put(-11,15){\tiny{$_m$}}

\put(9,15){\tiny{$_1$}}

\end{picture}) \cong ``0 \rightarrow C_f(\setlength{\unitlength}{1pt}
\begin{picture}(50,20)(-25,20)

\put(-20,0){\vector(2,1){20}}

\put(20,0){\vector(-2,1){20}}

\put(0,30){\vector(-2,1){20}}

\put(0,30){\vector(2,1){20}}

\put(0,10){\vector(0,1){20}}

\put(-17,7){\tiny{$_1$}}

\put(15,7){\tiny{$_m$}}

\put(-17,33){\tiny{$_m$}}

\put(15,33){\tiny{$_1$}}

\put(3,20){\tiny{$_{m+1}$}}

\end{picture}) \xrightarrow{\chi^1} C_f(\setlength{\unitlength}{1pt}
\begin{picture}(50,20)(-25,20)

\put(-20,0){\vector(0,1){20}}

\put(-20,20){\vector(0,1){20}}

\put(20,0){\vector(0,1){20}}

\put(20,20){\vector(0,1){20}}

\put(20,20){\vector(-1,0){40}}

\put(-17,35){\tiny{$_m$}}

\put(15,35){\tiny{$_1$}}

\put(-17,0){\tiny{$_1$}}

\put(15,0){\tiny{$_m$}}

\put(-8,23){\tiny{$_{m-1}$}}

\end{picture})\{q^{-1}\} \rightarrow 0", \vspace{20pt} 
\]
\[
\hat{C}_f(\setlength{\unitlength}{1pt}
\begin{picture}(40,20)(-20,0)

\put(20,-20){\vector(-1,1){40}}

\put(-20,-20){\line(1,1){15}}

\put(5,5){\vector(1,1){15}}

\put(-11,15){\tiny{$_m$}}

\put(9,15){\tiny{$_1$}}

\end{picture}) \cong ``0 \rightarrow C_f(\setlength{\unitlength}{1pt}
\begin{picture}(50,20)(-25,20)

\put(-20,0){\vector(0,1){20}}

\put(-20,20){\vector(0,1){20}}

\put(20,0){\vector(0,1){20}}

\put(20,20){\vector(0,1){20}}

\put(20,20){\vector(-1,0){40}}

\put(-17,35){\tiny{$_m$}}

\put(15,35){\tiny{$_1$}}

\put(-17,0){\tiny{$_1$}}

\put(15,0){\tiny{$_m$}}

\put(-8,23){\tiny{$_{m-1}$}}

\end{picture})\{q\} \xrightarrow{\chi^0} C_f(\setlength{\unitlength}{1pt}
\begin{picture}(50,20)(-25,20)

\put(-20,0){\vector(2,1){20}}

\put(20,0){\vector(-2,1){20}}

\put(0,30){\vector(-2,1){20}}

\put(0,30){\vector(2,1){20}}

\put(0,10){\vector(0,1){20}}

\put(-17,7){\tiny{$_1$}}

\put(15,7){\tiny{$_m$}}

\put(-17,33){\tiny{$_m$}}

\put(15,33){\tiny{$_1$}}

\put(3,20){\tiny{$_{m+1}$}}

\end{picture}) \rightarrow 0",\vspace{20pt} 
\]
where the $\chi^0$- and $\chi^1$-morphisms are defined in Proposition \ref{general-general-chi-maps} and the homological gradings are given in Definition \ref{complex-colored-crossing-def}.
\end{corollary}

\begin{proof}
We only prove
\[
\hat{C}_f(\setlength{\unitlength}{1pt}
\begin{picture}(40,20)(-20,0)

\put(-20,-20){\vector(1,1){40}}

\put(20,-20){\line(-1,1){15}}

\put(-5,5){\vector(-1,1){15}}

\put(-11,15){\tiny{$_1$}}

\put(9,15){\tiny{$_n$}}

\end{picture}) \cong ``0 \rightarrow C_f(\setlength{\unitlength}{1pt}
\begin{picture}(50,20)(-25,20)

\put(-20,0){\vector(2,1){20}}

\put(20,0){\vector(-2,1){20}}

\put(0,30){\vector(-2,1){20}}

\put(0,30){\vector(2,1){20}}

\put(0,10){\vector(0,1){20}}

\put(-17,7){\tiny{$_n$}}

\put(15,7){\tiny{$_1$}}

\put(-17,33){\tiny{$_1$}}

\put(15,33){\tiny{$_n$}}

\put(3,20){\tiny{$_{n+1}$}}

\end{picture}) \xrightarrow{\chi^1} C_f(\setlength{\unitlength}{1pt}
\begin{picture}(50,20)(-25,20)

\put(-20,0){\vector(0,1){20}}

\put(-20,20){\vector(0,1){20}}

\put(20,0){\vector(0,1){20}}

\put(20,20){\vector(0,1){20}}

\put(-20,20){\vector(1,0){40}}

\put(-17,35){\tiny{$_1$}}

\put(15,35){\tiny{$_n$}}

\put(-17,0){\tiny{$_n$}}

\put(15,0){\tiny{$_1$}}

\put(-8,23){\tiny{$_{n-1}$}}

\end{picture})\{q^{-1}\} \rightarrow 0". \vspace{20pt} 
\]
The proofs of the other isomorphisms are very similar and left to the reader. By Definition \ref{complex-colored-crossing-def}, we know that
\[
\hat{C}_f(\setlength{\unitlength}{1pt}
\begin{picture}(40,20)(-20,0)

\put(-20,-20){\vector(1,1){40}}

\put(20,-20){\line(-1,1){15}}

\put(-5,5){\vector(-1,1){15}}

\put(-11,15){\tiny{$_1$}}

\put(9,15){\tiny{$_n$}}

\end{picture}) = ``0 \rightarrow C_f(\setlength{\unitlength}{1pt}
\begin{picture}(50,20)(-25,20)

\put(-20,0){\vector(2,1){20}}

\put(20,0){\vector(-2,1){20}}

\put(0,30){\vector(-2,1){20}}

\put(0,30){\vector(2,1){20}}

\put(0,10){\vector(0,1){20}}

\put(-17,7){\tiny{$_n$}}

\put(15,7){\tiny{$_1$}}

\put(-17,33){\tiny{$_1$}}

\put(15,33){\tiny{$_n$}}

\put(3,20){\tiny{$_{n+1}$}}

\end{picture}) \xrightarrow{d_1^+} C_f(\setlength{\unitlength}{1pt}
\begin{picture}(50,20)(-25,20)

\put(-20,0){\vector(0,1){20}}

\put(-20,20){\vector(0,1){20}}

\put(20,0){\vector(0,1){20}}

\put(20,20){\vector(0,1){20}}

\put(-20,20){\vector(1,0){40}}

\put(-17,35){\tiny{$_1$}}

\put(15,35){\tiny{$_n$}}

\put(-17,0){\tiny{$_n$}}

\put(15,0){\tiny{$_1$}}

\put(-8,23){\tiny{$_{n-1}$}}

\end{picture})\{q^{-1}\} \rightarrow 0", \vspace{20pt} 
\]
where $d_1^+$ is given in Definition \ref{complex-colored-crossing-chain-maps-def}. From Definition \ref{complex-colored-crossing-chain-maps-def} and Proposition \ref{general-general-chi-maps}, we know $C_f(\setlength{\unitlength}{1pt}
\begin{picture}(50,20)(-25,20)

\put(-20,0){\vector(2,1){20}}

\put(20,0){\vector(-2,1){20}}

\put(0,30){\vector(-2,1){20}}

\put(0,30){\vector(2,1){20}}

\put(0,10){\vector(0,1){20}}

\put(-17,7){\tiny{$_n$}}

\put(15,7){\tiny{$_1$}}

\put(-17,33){\tiny{$_1$}}

\put(15,33){\tiny{$_n$}}

\put(3,20){\tiny{$_{n+1}$}}

\end{picture}) \xrightarrow{\chi^1} C_f(\setlength{\unitlength}{1pt}
\begin{picture}(50,20)(-25,20)

\put(-20,0){\vector(0,1){20}}

\put(-20,20){\vector(0,1){20}}

\put(20,0){\vector(0,1){20}}

\put(20,20){\vector(0,1){20}}

\put(-20,20){\vector(1,0){40}}

\put(-17,35){\tiny{$_1$}}

\put(15,35){\tiny{$_n$}}

\put(-17,0){\tiny{$_n$}}

\put(15,0){\tiny{$_1$}}

\put(-8,23){\tiny{$_{n-1}$}}

\end{picture}) \vspace{20pt}$ and $C_f(\setlength{\unitlength}{1pt}
\begin{picture}(50,20)(-25,20)

\put(-20,0){\vector(2,1){20}}

\put(20,0){\vector(-2,1){20}}

\put(0,30){\vector(-2,1){20}}

\put(0,30){\vector(2,1){20}}

\put(0,10){\vector(0,1){20}}

\put(-17,7){\tiny{$_n$}}

\put(15,7){\tiny{$_1$}}

\put(-17,33){\tiny{$_1$}}

\put(15,33){\tiny{$_n$}}

\put(3,20){\tiny{$_{n+1}$}}

\end{picture}) \xrightarrow{d_1^+} C_f(\setlength{\unitlength}{1pt}
\begin{picture}(50,20)(-25,20)

\put(-20,0){\vector(0,1){20}}

\put(-20,20){\vector(0,1){20}}

\put(20,0){\vector(0,1){20}}

\put(20,20){\vector(0,1){20}}

\put(-20,20){\vector(1,0){40}}

\put(-17,35){\tiny{$_1$}}

\put(15,35){\tiny{$_n$}}

\put(-17,0){\tiny{$_n$}}

\put(15,0){\tiny{$_1$}}

\put(-8,23){\tiny{$_{n-1}$}}

\end{picture})$ are both homotopically non-trivial homogeneous morphisms of total polynomial degree $1$. By Lemma \ref{colored-crossing-res-HMF}, such a morphism is unique up to homotopy and scaling. So $\chi^1 \approx d_1^+$. This shows that 
\[
\hat{C}_f(\setlength{\unitlength}{1pt}
\begin{picture}(40,20)(-20,0)

\put(-20,-20){\vector(1,1){40}}

\put(20,-20){\line(-1,1){15}}

\put(-5,5){\vector(-1,1){15}}

\put(-11,15){\tiny{$_1$}}

\put(9,15){\tiny{$_n$}}

\end{picture}) \cong ``0 \rightarrow C_f(\setlength{\unitlength}{1pt}
\begin{picture}(50,20)(-25,20)

\put(-20,0){\vector(2,1){20}}

\put(20,0){\vector(-2,1){20}}

\put(0,30){\vector(-2,1){20}}

\put(0,30){\vector(2,1){20}}

\put(0,10){\vector(0,1){20}}

\put(-17,7){\tiny{$_n$}}

\put(15,7){\tiny{$_1$}}

\put(-17,33){\tiny{$_1$}}

\put(15,33){\tiny{$_n$}}

\put(3,20){\tiny{$_{n+1}$}}

\end{picture}) \xrightarrow{\chi^1} C_f(\setlength{\unitlength}{1pt}
\begin{picture}(50,20)(-25,20)

\put(-20,0){\vector(0,1){20}}

\put(-20,20){\vector(0,1){20}}

\put(20,0){\vector(0,1){20}}

\put(20,20){\vector(0,1){20}}

\put(-20,20){\vector(1,0){40}}

\put(-17,35){\tiny{$_1$}}

\put(15,35){\tiny{$_n$}}

\put(-17,0){\tiny{$_n$}}

\put(15,0){\tiny{$_1$}}

\put(-8,23){\tiny{$_{n-1}$}}

\end{picture})\{q^{-1}\} \rightarrow 0". \vspace{20pt} 
\]
\end{proof}

\begin{remark}\label{generalizing-krasner}
Corollary \ref{explicit-differential-1-n-crossings--res} shows that the chain complex $C_f(D)$ defined in Definition \ref{complex-knotted-MOY-def} generalizes the chain complex used by Krasner in \cite{Krasner}.
\end{remark}

\section{Invariance under Fork Sliding}\label{sec-inv-fork}

The goal of this section is to prove Theorem \ref{fork-sliding-invariance-general}, which is the key step in the proof of invariance of the equivariant colored $\mathfrak{sl}(N)$-homology.

\begin{theorem}\label{fork-sliding-invariance-general}
\[
\xymatrix{
\hat{C}_f(\setlength{\unitlength}{1pt}
\begin{picture}(44,20)(-22,20)

\put(0,0){\line(0,1){8}}

\put(0,12){\vector(0,1){8}}

\put(0,20){\vector(1,1){20}}

\put(0,20){\vector(-1,1){20}}

\put(-20,10){\vector(1,0){40}}

\put(-13,35){\tiny{$_m$}}

\put(10,35){\tiny{$_l$}}

\put(3,2){\tiny{$_{m+l}$}}

\put(10,13){\tiny{$_n$}}

\end{picture})\simeq \hat{C}_f(\setlength{\unitlength}{1pt}
\begin{picture}(44,20)(-22,20)

\put(0,0){\vector(0,1){20}}

\put(0,20){\line(1,1){8}}

\put(0,20){\line(-1,1){8}}

\put(12,32){\vector(1,1){8}}

\put(-12,32){\vector(-1,1){8}}

\put(-20,30){\vector(1,0){40}}

\put(-13,35){\tiny{$_m$}}

\put(10,35){\tiny{$_l$}}

\put(3,2){\tiny{$_{m+l}$}}

\put(12,26){\tiny{$_n$}}

\end{picture}) & \hat{C}_f(\setlength{\unitlength}{1pt}
\begin{picture}(44,20)(-22,20)
\put(0,0){\vector(0,1){20}}

\put(0,20){\vector(1,1){20}}

\put(0,20){\vector(-1,1){20}}

\put(-20,10){\line(1,0){18}}

\put(2,10){\vector(1,0){18}}

\put(-13,35){\tiny{$_m$}}

\put(10,35){\tiny{$_l$}}

\put(3,2){\tiny{$_{m+l}$}}

\put(10,13){\tiny{$_n$}}

\end{picture})\simeq \hat{C}_f(\setlength{\unitlength}{1pt}
\begin{picture}(44,20)(-22,20)
\put(0,0){\vector(0,1){20}}

\put(0,20){\vector(1,1){20}}

\put(0,20){\vector(-1,1){20}}

\put(-20,30){\line(1,0){8}}

\put(-8,30){\line(1,0){16}}

\put(12,30){\vector(1,0){8}}

\put(-13,35){\tiny{$_m$}}

\put(10,35){\tiny{$_l$}}

\put(3,2){\tiny{$_{m+l}$}}

\put(12,26){\tiny{$_n$}}

\end{picture}) \vspace{20pt} \\
\hat{C}_f(\setlength{\unitlength}{1pt}
\begin{picture}(44,20)(-22,20)

\put(0,8){\vector(0,-1){8}}

\put(0,20){\line(0,-1){8}}

\put(20,40){\vector(-1,-1){20}}

\put(-20,40){\vector(1,-1){20}}

\put(20,10){\vector(-1,0){40}}

\put(-13,35){\tiny{$_m$}}

\put(10,35){\tiny{$_l$}}

\put(3,2){\tiny{$_{m+l}$}}

\put(10,13){\tiny{$_n$}}

\end{picture})\simeq \hat{C}_f(\setlength{\unitlength}{1pt}
\begin{picture}(44,20)(-22,20)

\put(0,20){\vector(0,-1){20}}

\put(8,28){\vector(-1,-1){8}}

\put(-8,28){\vector(1,-1){8}}

\put(20,40){\line(-1,-1){8}}

\put(-20,40){\line(1,-1){8}}

\put(20,30){\vector(-1,0){40}}

\put(-14,36){\tiny{$_m$}}

\put(11,36){\tiny{$_l$}}

\put(3,2){\tiny{$_{m+l}$}}

\put(12,26){\tiny{$_n$}}

\end{picture}) & \hat{C}_f(\setlength{\unitlength}{1pt}
\begin{picture}(44,20)(-22,20)

\put(0,20){\vector(0,-1){20}}

\put(20,40){\vector(-1,-1){20}}

\put(-20,40){\vector(1,-1){20}}

\put(20,10){\line(-1,0){18}}

\put(-2,10){\vector(-1,0){18}}

\put(-13,35){\tiny{$_m$}}

\put(10,35){\tiny{$_l$}}

\put(3,2){\tiny{$_{m+l}$}}

\put(10,13){\tiny{$_n$}}

\end{picture})\simeq \hat{C}_f(\setlength{\unitlength}{1pt}
\begin{picture}(44,20)(-22,20)

\put(0,20){\vector(0,-1){20}}

\put(20,40){\vector(-1,-1){20}}

\put(-20,40){\vector(1,-1){20}}

\put(20,30){\line(-1,0){8}}

\put(8,30){\line(-1,0){16}}

\put(-12,30){\vector(-1,0){8}}

\put(-14,36){\tiny{$_m$}}

\put(11,36){\tiny{$_l$}}

\put(3,2){\tiny{$_{m+l}$}}

\put(12,26){\tiny{$_n$}}

\end{picture}) \vspace{20pt} \\
\hat{C}_f(\setlength{\unitlength}{1pt}
\begin{picture}(44,20)(-22,20)

\put(0,0){\line(0,1){8}}

\put(0,12){\vector(0,1){8}}

\put(0,20){\vector(1,1){20}}

\put(0,20){\vector(-1,1){20}}

\put(20,10){\vector(-1,0){40}}

\put(-13,35){\tiny{$_m$}}

\put(10,35){\tiny{$_l$}}

\put(3,2){\tiny{$_{m+l}$}}

\put(10,13){\tiny{$_n$}}

\end{picture})\simeq \hat{C}_f(\setlength{\unitlength}{1pt}
\begin{picture}(44,20)(-22,20)

\put(0,0){\vector(0,1){20}}

\put(0,20){\line(1,1){8}}

\put(0,20){\line(-1,1){8}}

\put(12,32){\vector(1,1){8}}

\put(-12,32){\vector(-1,1){8}}

\put(20,30){\vector(-1,0){40}}

\put(-13,35){\tiny{$_m$}}

\put(10,35){\tiny{$_l$}}

\put(3,2){\tiny{$_{m+l}$}}

\put(12,26){\tiny{$_n$}}

\end{picture}) & \hat{C}_f(\setlength{\unitlength}{1pt}
\begin{picture}(44,20)(-22,20)

\put(0,0){\vector(0,1){20}}

\put(0,20){\vector(1,1){20}}

\put(0,20){\vector(-1,1){20}}

\put(20,10){\line(-1,0){18}}

\put(-2,10){\vector(-1,0){18}}

\put(-13,35){\tiny{$_m$}}

\put(10,35){\tiny{$_l$}}

\put(3,2){\tiny{$_{m+l}$}}

\put(10,13){\tiny{$_n$}}

\end{picture})\simeq \hat{C}_f(\setlength{\unitlength}{1pt}
\begin{picture}(44,20)(-22,20)

\put(0,0){\vector(0,1){20}}

\put(0,20){\vector(1,1){20}}

\put(0,20){\vector(-1,1){20}}

\put(20,30){\line(-1,0){8}}

\put(8,30){\line(-1,0){16}}

\put(-12,30){\vector(-1,0){8}}

\put(-13,35){\tiny{$_m$}}

\put(10,35){\tiny{$_l$}}

\put(3,2){\tiny{$_{m+l}$}}

\put(12,26){\tiny{$_n$}}

\end{picture}) \vspace{20pt} \\
\hat{C}_f(\setlength{\unitlength}{1pt}
\begin{picture}(44,20)(-22,20)

\put(0,8){\vector(0,-1){8}}

\put(0,20){\line(0,-1){8}}

\put(20,40){\vector(-1,-1){20}}

\put(-20,40){\vector(1,-1){20}}

\put(-20,10){\vector(1,0){40}}

\put(-13,35){\tiny{$_m$}}

\put(10,35){\tiny{$_l$}}

\put(3,2){\tiny{$_{m+l}$}}

\put(10,13){\tiny{$_n$}}

\end{picture})\simeq \hat{C}_f(\setlength{\unitlength}{1pt}
\begin{picture}(44,20)(-22,20)
\put(0,20){\vector(0,-1){20}}

\put(8,28){\vector(-1,-1){8}}

\put(-8,28){\vector(1,-1){8}}

\put(20,40){\line(-1,-1){8}}

\put(-20,40){\line(1,-1){8}}

\put(-20,30){\vector(1,0){40}}

\put(-14,36){\tiny{$_m$}}

\put(11,36){\tiny{$_l$}}

\put(3,2){\tiny{$_{m+l}$}}

\put(12,26){\tiny{$_n$}}

\end{picture}) & \hat{C}_f(\setlength{\unitlength}{1pt}
\begin{picture}(44,20)(-22,20)

\put(0,20){\vector(0,-1){20}}

\put(20,40){\vector(-1,-1){20}}

\put(-20,40){\vector(1,-1){20}}

\put(-20,10){\line(1,0){18}}

\put(2,10){\vector(1,0){18}}

\put(-13,35){\tiny{$_m$}}

\put(10,35){\tiny{$_l$}}

\put(3,2){\tiny{$_{m+l}$}}

\put(10,13){\tiny{$_n$}}

\end{picture})\simeq \hat{C}_f(\setlength{\unitlength}{1pt}
\begin{picture}(44,20)(-22,20)

\put(0,20){\vector(0,-1){20}}

\put(20,40){\vector(-1,-1){20}}

\put(-20,40){\vector(1,-1){20}}

\put(-20,30){\line(1,0){8}}

\put(-8,30){\line(1,0){16}}

\put(12,30){\vector(1,0){8}}

\put(-14,36){\tiny{$_m$}}

\put(11,36){\tiny{$_l$}}

\put(3,2){\tiny{$_{m+l}$}}

\put(12,26){\tiny{$_n$}}

\end{picture})
}
\]\vspace{20pt}
\end{theorem}

As in \cite{Wu-color}, we prove Theorem \ref{fork-sliding-invariance-general} by induction. The hard part of the induction is to prove that the initial cases are true. Here, we state these initial cases in Proposition \ref{fork-sliding-invariance-special} below.

\begin{proposition}\label{fork-sliding-invariance-special}
\[
\xymatrix{
\hat{C}_f(\setlength{\unitlength}{1pt}
\begin{picture}(44,20)(-22,20)

\put(0,0){\line(0,1){8}}

\put(0,12){\vector(0,1){8}}

\put(0,20){\vector(1,1){20}}

\put(0,20){\vector(-1,1){20}}

\put(-20,10){\vector(1,0){40}}

\put(-13,35){\tiny{$_m$}}

\put(10,35){\tiny{$_1$}}

\put(3,2){\tiny{$_{m+1}$}}

\put(10,13){\tiny{$_n$}}

\end{picture})\simeq \hat{C}_f(\setlength{\unitlength}{1pt}
\begin{picture}(44,20)(-22,20)

\put(0,0){\vector(0,1){20}}

\put(0,20){\line(1,1){8}}

\put(0,20){\line(-1,1){8}}

\put(12,32){\vector(1,1){8}}

\put(-12,32){\vector(-1,1){8}}

\put(-20,30){\vector(1,0){40}}

\put(-13,35){\tiny{$_m$}}

\put(10,35){\tiny{$_1$}}

\put(3,2){\tiny{$_{m+1}$}}

\put(12,26){\tiny{$_n$}}

\end{picture}) & \hat{C}_f(\setlength{\unitlength}{1pt}
\begin{picture}(44,20)(-22,20)
\put(0,0){\vector(0,1){20}}

\put(0,20){\vector(1,1){20}}

\put(0,20){\vector(-1,1){20}}

\put(-20,10){\line(1,0){18}}

\put(2,10){\vector(1,0){18}}

\put(-13,35){\tiny{$_m$}}

\put(10,35){\tiny{$_1$}}

\put(3,2){\tiny{$_{m+1}$}}

\put(10,13){\tiny{$_n$}}

\end{picture})\simeq \hat{C}_f(\setlength{\unitlength}{1pt}
\begin{picture}(44,20)(-22,20)
\put(0,0){\vector(0,1){20}}

\put(0,20){\vector(1,1){20}}

\put(0,20){\vector(-1,1){20}}

\put(-20,30){\line(1,0){8}}

\put(-8,30){\line(1,0){16}}

\put(12,30){\vector(1,0){8}}

\put(-13,35){\tiny{$_m$}}

\put(10,35){\tiny{$_1$}}

\put(3,2){\tiny{$_{m+1}$}}

\put(12,26){\tiny{$_n$}}

\end{picture}) \vspace{20pt} \\
\hat{C}_f(\setlength{\unitlength}{1pt}
\begin{picture}(44,20)(-22,20)

\put(0,8){\vector(0,-1){8}}

\put(0,20){\line(0,-1){8}}

\put(20,40){\vector(-1,-1){20}}

\put(-20,40){\vector(1,-1){20}}

\put(20,10){\vector(-1,0){40}}

\put(-13,35){\tiny{$_m$}}

\put(10,35){\tiny{$_1$}}

\put(3,2){\tiny{$_{m+1}$}}

\put(10,13){\tiny{$_n$}}

\end{picture})\simeq \hat{C}_f(\setlength{\unitlength}{1pt}
\begin{picture}(44,20)(-22,20)

\put(0,20){\vector(0,-1){20}}

\put(8,28){\vector(-1,-1){8}}

\put(-8,28){\vector(1,-1){8}}

\put(20,40){\line(-1,-1){8}}

\put(-20,40){\line(1,-1){8}}

\put(20,30){\vector(-1,0){40}}

\put(-14,36){\tiny{$_m$}}

\put(11,36){\tiny{$_1$}}

\put(3,2){\tiny{$_{m+1}$}}

\put(12,26){\tiny{$_n$}}

\end{picture}) & \hat{C}_f(\setlength{\unitlength}{1pt}
\begin{picture}(44,20)(-22,20)

\put(0,20){\vector(0,-1){20}}

\put(20,40){\vector(-1,-1){20}}

\put(-20,40){\vector(1,-1){20}}

\put(20,10){\line(-1,0){18}}

\put(-2,10){\vector(-1,0){18}}

\put(-13,35){\tiny{$_m$}}

\put(10,35){\tiny{$_1$}}

\put(3,2){\tiny{$_{m+1}$}}

\put(10,13){\tiny{$_n$}}

\end{picture})\simeq \hat{C}_f(\setlength{\unitlength}{1pt}
\begin{picture}(44,20)(-22,20)

\put(0,20){\vector(0,-1){20}}

\put(20,40){\vector(-1,-1){20}}

\put(-20,40){\vector(1,-1){20}}

\put(20,30){\line(-1,0){8}}

\put(8,30){\line(-1,0){16}}

\put(-12,30){\vector(-1,0){8}}

\put(-14,36){\tiny{$_m$}}

\put(11,36){\tiny{$_1$}}

\put(3,2){\tiny{$_{m+1}$}}

\put(12,26){\tiny{$_n$}}

\end{picture}) \vspace{20pt} \\
\hat{C}_f(\setlength{\unitlength}{1pt}
\begin{picture}(44,20)(-22,20)

\put(0,0){\line(0,1){8}}

\put(0,12){\vector(0,1){8}}

\put(0,20){\vector(1,1){20}}

\put(0,20){\vector(-1,1){20}}

\put(20,10){\vector(-1,0){40}}

\put(-13,35){\tiny{$_1$}}

\put(10,35){\tiny{$_l$}}

\put(3,2){\tiny{$_{l+1}$}}

\put(10,13){\tiny{$_n$}}

\end{picture})\simeq \hat{C}_f(\setlength{\unitlength}{1pt}
\begin{picture}(44,20)(-22,20)

\put(0,0){\vector(0,1){20}}

\put(0,20){\line(1,1){8}}

\put(0,20){\line(-1,1){8}}

\put(12,32){\vector(1,1){8}}

\put(-12,32){\vector(-1,1){8}}

\put(20,30){\vector(-1,0){40}}

\put(-13,35){\tiny{$_1$}}

\put(10,35){\tiny{$_l$}}

\put(3,2){\tiny{$_{l+1}$}}

\put(12,26){\tiny{$_n$}}

\end{picture}) & \hat{C}_f(\setlength{\unitlength}{1pt}
\begin{picture}(44,20)(-22,20)

\put(0,0){\vector(0,1){20}}

\put(0,20){\vector(1,1){20}}

\put(0,20){\vector(-1,1){20}}

\put(20,10){\line(-1,0){18}}

\put(-2,10){\vector(-1,0){18}}

\put(-13,35){\tiny{$_1$}}

\put(10,35){\tiny{$_l$}}

\put(3,2){\tiny{$_{l+1}$}}

\put(10,13){\tiny{$_n$}}

\end{picture})\simeq \hat{C}_f(\setlength{\unitlength}{1pt}
\begin{picture}(44,20)(-22,20)

\put(0,0){\vector(0,1){20}}

\put(0,20){\vector(1,1){20}}

\put(0,20){\vector(-1,1){20}}

\put(20,30){\line(-1,0){8}}

\put(8,30){\line(-1,0){16}}

\put(-12,30){\vector(-1,0){8}}

\put(-13,35){\tiny{$_1$}}

\put(10,35){\tiny{$_l$}}

\put(3,2){\tiny{$_{l+1}$}}

\put(12,26){\tiny{$_n$}}

\end{picture}) \vspace{20pt} \\
\hat{C}_f(\setlength{\unitlength}{1pt}
\begin{picture}(44,20)(-22,20)

\put(0,8){\vector(0,-1){8}}

\put(0,20){\line(0,-1){8}}

\put(20,40){\vector(-1,-1){20}}

\put(-20,40){\vector(1,-1){20}}

\put(-20,10){\vector(1,0){40}}

\put(-13,35){\tiny{$_1$}}

\put(10,35){\tiny{$_l$}}

\put(3,2){\tiny{$_{l+1}$}}

\put(10,13){\tiny{$_n$}}

\end{picture})\simeq \hat{C}_f(\setlength{\unitlength}{1pt}
\begin{picture}(44,20)(-22,20)
\put(0,20){\vector(0,-1){20}}

\put(8,28){\vector(-1,-1){8}}

\put(-8,28){\vector(1,-1){8}}

\put(20,40){\line(-1,-1){8}}

\put(-20,40){\line(1,-1){8}}

\put(-20,30){\vector(1,0){40}}

\put(-14,36){\tiny{$_1$}}

\put(11,36){\tiny{$_l$}}

\put(3,2){\tiny{$_{l+1}$}}

\put(12,26){\tiny{$_n$}}

\end{picture}) & \hat{C}_f(\setlength{\unitlength}{1pt}
\begin{picture}(44,20)(-22,20)

\put(0,20){\vector(0,-1){20}}

\put(20,40){\vector(-1,-1){20}}

\put(-20,40){\vector(1,-1){20}}

\put(-20,10){\line(1,0){18}}

\put(2,10){\vector(1,0){18}}

\put(-13,35){\tiny{$_1$}}

\put(10,35){\tiny{$_l$}}

\put(3,2){\tiny{$_{l+1}$}}

\put(10,13){\tiny{$_n$}}

\end{picture})\simeq \hat{C}_f(\setlength{\unitlength}{1pt}
\begin{picture}(44,20)(-22,20)

\put(0,20){\vector(0,-1){20}}

\put(20,40){\vector(-1,-1){20}}

\put(-20,40){\vector(1,-1){20}}

\put(-20,30){\line(1,0){8}}

\put(-8,30){\line(1,0){16}}

\put(12,30){\vector(1,0){8}}

\put(-14,36){\tiny{$_1$}}

\put(11,36){\tiny{$_l$}}

\put(3,2){\tiny{$_{l+1}$}}

\put(12,26){\tiny{$_n$}}

\end{picture})
}
\]\vspace{20pt}
\end{proposition}

\begin{proof}[Proof of Theorem \ref{fork-sliding-invariance-general} (assuming Proposition \ref{fork-sliding-invariance-special} is true)]
Each homotopy equivalence in Theorem \ref{fork-sliding-invariance-general} can be proved by an induction on $m$ or $l$. We only give details for the proof of 
\begin{equation}\label{fork-sliding-invariance-general-induction-1-+}
\hat{C}_f(\setlength{\unitlength}{1pt}
\begin{picture}(44,20)(-22,20)

\put(0,0){\line(0,1){8}}

\put(0,12){\vector(0,1){8}}

\put(0,20){\vector(1,1){20}}

\put(0,20){\vector(-1,1){20}}

\put(-20,10){\vector(1,0){40}}

\put(-13,35){\tiny{$_m$}}

\put(10,35){\tiny{$_l$}}

\put(3,2){\tiny{$_{m+l}$}}

\put(10,13){\tiny{$_n$}}

\end{picture})\simeq \hat{C}_f(\setlength{\unitlength}{1pt}
\begin{picture}(44,20)(-22,20)

\put(0,0){\vector(0,1){20}}

\put(0,20){\line(1,1){8}}

\put(0,20){\line(-1,1){8}}

\put(12,32){\vector(1,1){8}}

\put(-12,32){\vector(-1,1){8}}

\put(-20,30){\vector(1,0){40}}

\put(-13,35){\tiny{$_m$}}

\put(10,35){\tiny{$_l$}}

\put(3,2){\tiny{$_{m+l}$}}

\put(12,26){\tiny{$_n$}}

\end{picture}) \vspace{20pt}
\end{equation} 
The proof of the rest of Theorem \ref{fork-sliding-invariance-general} is very similar and left to the reader.

We prove \eqref{fork-sliding-invariance-general-induction-1-+} by an induction on $l$. The $l=1$ case is given in Proposition \ref{fork-sliding-invariance-special}. Assume that \eqref{fork-sliding-invariance-general-induction-1-+} is true for some $l=k\geq 1$. Consider $l=k+1$. 
By Decomposition (II) (Theorem \ref{decomp-II}), we have
\[
\hat{C}_f(\input{fork-sliding-induction-10-sli}) \simeq \hat{C}_f(\setlength{\unitlength}{1pt}
\begin{picture}(44,20)(-22,20)

\put(0,0){\line(0,1){8}}

\put(0,12){\vector(0,1){8}}

\put(0,20){\vector(1,1){20}}

\put(0,20){\vector(-1,1){20}}

\put(-20,10){\vector(1,0){40}}

\put(-13,35){\tiny{$_m$}}

\put(0,35){\tiny{$_{k+1}$}}

\put(3,2){\tiny{$_{m+k+1}$}}

\put(10,13){\tiny{$_n$}}

\end{picture})\{[k+1]\},   \vspace{30pt}
\] 
\[
\hat{C}_f(\input{fork-sliding-induction-11-sli}) \simeq \hat{C}_f(\setlength{\unitlength}{1pt}
\begin{picture}(44,20)(-22,20)

\put(0,0){\vector(0,1){20}}

\put(0,20){\line(1,1){8}}

\put(0,20){\line(-1,1){8}}

\put(12,32){\vector(1,1){8}}

\put(-12,32){\vector(-1,1){8}}

\put(-20,30){\vector(1,0){40}}

\put(-13,35){\tiny{$_m$}}

\put(0,35){\tiny{$_{k+1}$}}

\put(3,2){\tiny{$_{m+k+1}$}}

\put(12,26){\tiny{$_n$}}

\end{picture})\{[k+1]\}. \vspace{30pt}
\]
By Corollary \ref{contract-expand}, we have
\[
\hat{C}_f(\input{fork-sliding-induction-10-sli}) \simeq \hat{C}_f(\input{fork-sliding-induction-morph1-sli}), \vspace{30pt}
\]
\[
\hat{C}_f(\input{fork-sliding-induction-11-sli}) \simeq \hat{C}_f(\input{fork-sliding-induction-morph4-sli}). \vspace{30pt}
\]
By the $l=k$ case of \eqref{fork-sliding-invariance-general-induction-1-+}, we have
\[
\hat{C}_f(\input{fork-sliding-induction-morph1-sli}) \simeq \hat{C}_f(\input{fork-sliding-induction-morph2-sli}). \vspace{30pt}
\]
By the $l=1$ case of \eqref{fork-sliding-invariance-general-induction-1-+}, we have
\[
\hat{C}_f(\input{fork-sliding-induction-morph2-sli}) \simeq \hat{C}_f(\input{fork-sliding-induction-morph3-sli}). \vspace{30pt}
\]
By Proposition \ref{fork-sliding-invariance-special}, we known that
\[
\hat{C}_f(\input{fork-sliding-induction-morph3-sli}) \simeq \hat{C}_f(\input{fork-sliding-induction-morph4-sli}). \vspace{30pt}
\]
Putting everything together, we have
\[
\hat{C}_f(\setlength{\unitlength}{1pt}
\begin{picture}(44,20)(-22,20)

\put(0,0){\line(0,1){8}}

\put(0,12){\vector(0,1){8}}

\put(0,20){\vector(1,1){20}}

\put(0,20){\vector(-1,1){20}}

\put(-20,10){\vector(1,0){40}}

\put(-13,35){\tiny{$_m$}}

\put(0,35){\tiny{$_{k+1}$}}

\put(3,2){\tiny{$_{m+k+1}$}}

\put(10,13){\tiny{$_n$}}

\end{picture})\{[k+1]\} \simeq \hat{C}_f(\setlength{\unitlength}{1pt}
\begin{picture}(44,20)(-22,20)

\put(0,0){\vector(0,1){20}}

\put(0,20){\line(1,1){8}}

\put(0,20){\line(-1,1){8}}

\put(12,32){\vector(1,1){8}}

\put(-12,32){\vector(-1,1){8}}

\put(-20,30){\vector(1,0){40}}

\put(-13,35){\tiny{$_m$}}

\put(0,35){\tiny{$_{k+1}$}}

\put(3,2){\tiny{$_{m+k+1}$}}

\put(12,26){\tiny{$_n$}}

\end{picture})\{[k+1]\}. \vspace{20pt}
\]
Then it follows from \cite[Proposition 3.20]{Wu-color} that 
\[
\hat{C}_f(\setlength{\unitlength}{1pt}
\begin{picture}(44,20)(-22,20)

\put(0,0){\line(0,1){8}}

\put(0,12){\vector(0,1){8}}

\put(0,20){\vector(1,1){20}}

\put(0,20){\vector(-1,1){20}}

\put(-20,10){\vector(1,0){40}}

\put(-13,35){\tiny{$_m$}}

\put(0,35){\tiny{$_{k+1}$}}

\put(3,2){\tiny{$_{m+k+1}$}}

\put(10,13){\tiny{$_n$}}

\end{picture}) \simeq \hat{C}_f(\setlength{\unitlength}{1pt}
\begin{picture}(44,20)(-22,20)

\put(0,0){\vector(0,1){20}}

\put(0,20){\line(1,1){8}}

\put(0,20){\line(-1,1){8}}

\put(12,32){\vector(1,1){8}}

\put(-12,32){\vector(-1,1){8}}

\put(-20,30){\vector(1,0){40}}

\put(-13,35){\tiny{$_m$}}

\put(0,35){\tiny{$_{k+1}$}}

\put(3,2){\tiny{$_{m+k+1}$}}

\put(12,26){\tiny{$_n$}}

\end{picture}). \vspace{20pt}
\]
\end{proof}

In the remainder of this section, we concentrate on proving Proposition \ref{fork-sliding-invariance-special}. We only give detailed proofs of 
\begin{equation}\label{fork-sliding-invariance-special-1-+}
\hat{C}_f(\setlength{\unitlength}{1pt}
\begin{picture}(44,20)(-22,20)

\put(0,0){\line(0,1){8}}

\put(0,12){\vector(0,1){8}}

\put(0,20){\vector(1,1){20}}

\put(0,20){\vector(-1,1){20}}

\put(-20,10){\vector(1,0){40}}

\put(-13,35){\tiny{$_m$}}

\put(10,35){\tiny{$_1$}}

\put(3,2){\tiny{$_{m+1}$}}

\put(10,13){\tiny{$_n$}}

\end{picture})\simeq \hat{C}_f(\setlength{\unitlength}{1pt}
\begin{picture}(44,20)(-22,20)

\put(0,0){\vector(0,1){20}}

\put(0,20){\line(1,1){8}}

\put(0,20){\line(-1,1){8}}

\put(12,32){\vector(1,1){8}}

\put(-12,32){\vector(-1,1){8}}

\put(-20,30){\vector(1,0){40}}

\put(-13,35){\tiny{$_m$}}

\put(10,35){\tiny{$_1$}}

\put(3,2){\tiny{$_{m+1}$}}

\put(12,26){\tiny{$_n$}}

\end{picture}) \vspace{20pt} 
\end{equation}
and
\begin{equation}\label{fork-sliding-invariance-special-1--}
\hat{C}_f(\setlength{\unitlength}{1pt}
\begin{picture}(44,20)(-22,20)
\put(0,0){\vector(0,1){20}}

\put(0,20){\vector(1,1){20}}

\put(0,20){\vector(-1,1){20}}

\put(-20,10){\line(1,0){18}}

\put(2,10){\vector(1,0){18}}

\put(-13,35){\tiny{$_m$}}

\put(10,35){\tiny{$_1$}}

\put(3,2){\tiny{$_{m+1}$}}

\put(10,13){\tiny{$_n$}}

\end{picture})\simeq \hat{C}_f(\setlength{\unitlength}{1pt}
\begin{picture}(44,20)(-22,20)
\put(0,0){\vector(0,1){20}}

\put(0,20){\vector(1,1){20}}

\put(0,20){\vector(-1,1){20}}

\put(-20,30){\line(1,0){8}}

\put(-8,30){\line(1,0){16}}

\put(12,30){\vector(1,0){8}}

\put(-13,35){\tiny{$_m$}}

\put(10,35){\tiny{$_1$}}

\put(3,2){\tiny{$_{m+1}$}}

\put(12,26){\tiny{$_n$}}

\end{picture}). \vspace{20pt} 
\end{equation}
The proof of the rest of Proposition \ref{fork-sliding-invariance-special} is very similar and left to the reader.

\subsection{Chain complexes involved in the proof}\label{fork-sliding-complexes-involved} In this subsection, we list the chain complexes that will appear in the proof of \eqref{fork-sliding-invariance-special-1-+} and \eqref{fork-sliding-invariance-special-1--}.

\begin{figure}[ht]
$
\xymatrix{
\input{tilde-gamma-k} 
}
$
\caption{}\label{fork-sliding-special-complex1-fig}

\end{figure}

Denote by $\widetilde{\Gamma}_k$ the MOY graph in Figure \ref{fork-sliding-special-complex1-fig}. Then $\hat{C}_f(\setlength{\unitlength}{1pt}
\begin{picture}(44,20)(-22,20)

\put(0,0){\line(0,1){8}}

\put(0,12){\vector(0,1){8}}

\put(0,20){\vector(1,1){20}}

\put(0,20){\vector(-1,1){20}}

\put(-20,10){\vector(1,0){40}}

\put(-13,35){\tiny{$_m$}}

\put(10,35){\tiny{$_1$}}

\put(3,2){\tiny{$_{m+1}$}}

\put(10,13){\tiny{$_n$}}

\end{picture}) \vspace{20pt}$ is 
\begin{equation}\label{complex-D-10-+}
0 \rightarrow C_f(\widetilde{\Gamma}_{m+1}) \xrightarrow{\tilde{d}_{m+1}^+} C_f(\widetilde{\Gamma}_{m})\{q^{-1}\} \xrightarrow{\tilde{d}_{m}^+} \cdots \xrightarrow{\tilde{d}_{\tilde{k}_0+1}^+} C_f(\widetilde{\Gamma}_{\tilde{k}_0})\{q^{\tilde{k}_0-m-1}\} \rightarrow 0,
\end{equation}
where $\tilde{k}_0 := \max\{0,m+1-n\}$. Similarly, $\hat{C}_f(\setlength{\unitlength}{1pt}
\begin{picture}(44,20)(-22,20)
\put(0,0){\vector(0,1){20}}

\put(0,20){\vector(1,1){20}}

\put(0,20){\vector(-1,1){20}}

\put(-20,10){\line(1,0){18}}

\put(2,10){\vector(1,0){18}}

\put(-13,35){\tiny{$_m$}}

\put(10,35){\tiny{$_1$}}

\put(3,2){\tiny{$_{m+1}$}}

\put(10,13){\tiny{$_n$}}

\end{picture}) \vspace{20pt}$ is 
\begin{equation}\label{complex-D-10--}
0 \rightarrow C_f(\widetilde{\Gamma}_{\tilde{k}_0})\{q^{m+1-\tilde{k}_0}\} \xrightarrow{\tilde{d}_{\tilde{k}_0}^-} \cdots \xrightarrow{\tilde{d}_{m-1}^-} C_f(\widetilde{\Gamma}_{m})\{q\} \xrightarrow{\tilde{d}_{m}^-} C_f(\widetilde{\Gamma}_{m+1}) \rightarrow 0.
\end{equation}

\begin{figure}[ht]
$
\xymatrix{
\input{gamma-k-prime} &&& \input{gamma-k-double-prime} 
}
$
\caption{}\label{fork-sliding-special-complex2-fig}

\end{figure}

Let $\Gamma_k'$ and $\Gamma_k''$ be the MOY graphs in Figure \ref{fork-sliding-special-complex2-fig}. Let $\delta_k^\pm: C_f(\Gamma_k') \rightarrow C_f(\Gamma_{k\mp 1}')$ be the morphisms defined in Definition \ref{trivial-complex-differential-def} with explicit form given in Theorem \ref{explicit-differential-general}. Let $C^+$ be the chain complex
\begin{equation}\label{complex-D-contractible-+}
0 \rightarrow C_f(\Gamma_{m-1}'') \xrightarrow{J_{m-1,m-1}} C_f(\Gamma_{m-1}') \xrightarrow{\delta_{m-1}^+} \cdots \xrightarrow{\delta_{k_0+1}^+} C_f(\Gamma_{k_0}') \rightarrow 0,
\end{equation}
and $C^-$ the chain complex 
\begin{equation}\label{complex-D-contractible--}
0 \rightarrow C_f(\Gamma_{k_0}') \xrightarrow{\delta_{k_0}^-} \cdots \xrightarrow{\delta_{m-2}^-} C_f(\Gamma_{m-1}') \xrightarrow{P_{m-1,m-1}}  C_f(\Gamma_{m-1}'') \rightarrow 0,
\end{equation}
where $k_0 = \max \{m-n,0\}$ and $J_{m-1,m-1}$, $P_{m-1,m-1}$ are defined in Definition \ref{trivial-complex-differential-def}. Then, by Lemma \ref{decomp-V-special-2} and Proposition \ref{trivial-complex-prop}, both $C^\pm$ are isomorphic in $\ch(\hmf)$ to
\[
\bigoplus_{j=k_0}^{m-1} (\xymatrix{
0 \ar[r] & C_f(\Gamma_j'') \ar[r]^{\simeq} & C_f(\Gamma_j'') \ar[r] & 0,
})
\]
and are therefore homotopic to $0$.

\begin{figure}[ht]
$
\xymatrix{
\input{gamma-k0} && \input{gamma-k1} 
}
$
\caption{}\label{fork-sliding-special-complex3-fig}

\end{figure}

Now consider $\hat{C}_f(\setlength{\unitlength}{1pt}
\begin{picture}(44,20)(-22,20)

\put(0,0){\vector(0,1){20}}

\put(0,20){\line(1,1){8}}

\put(0,20){\line(-1,1){8}}

\put(12,32){\vector(1,1){8}}

\put(-12,32){\vector(-1,1){8}}

\put(-20,30){\vector(1,0){40}}

\put(-13,35){\tiny{$_m$}}

\put(10,35){\tiny{$_1$}}

\put(3,2){\tiny{$_{m+1}$}}

\put(12,26){\tiny{$_n$}}

\end{picture})$ and $\hat{C}_f(\setlength{\unitlength}{1pt}
\begin{picture}(44,20)(-22,20)
\put(0,0){\vector(0,1){20}}

\put(0,20){\vector(1,1){20}}

\put(0,20){\vector(-1,1){20}}

\put(-20,30){\line(1,0){8}}

\put(-8,30){\line(1,0){16}}

\put(12,30){\vector(1,0){8}}

\put(-13,35){\tiny{$_m$}}

\put(10,35){\tiny{$_1$}}

\put(3,2){\tiny{$_{m+1}$}}

\put(12,26){\tiny{$_n$}}

\end{picture}) \vspace{20pt}$. Note that each of these knotted MOY graphs has two crossings -- one $\pm(m,n)$-crossing and one $\pm (1,n)$-crossing. Denote by $d_k^\pm$ the differential map of the $\pm(m,n)$-crossing. From Corollary \ref{explicit-differential-1-n-crossings--res}, the differential map of the $+(1,n)$-crossing (resp. $-(1,n)$-crossing) is $\chi^1$ (resp. $\chi^0$.) Let $\Gamma_{k,0}$ and $\Gamma_{k,1}$ be the MOY graphs in Figure \ref{fork-sliding-special-complex3-fig}. Then $d_k^\pm$ acts on the left square in $\Gamma_{k,0}$ and $\Gamma_{k,1}$, and $\chi^0$, $\chi^1$ act on the upper right corners of $\Gamma_{k,0}$ and $\Gamma_{k,1}$. The chain complex $\hat{C}_f(\setlength{\unitlength}{1pt}
\begin{picture}(44,20)(-22,20)

\put(0,0){\vector(0,1){20}}

\put(0,20){\line(1,1){8}}

\put(0,20){\line(-1,1){8}}

\put(12,32){\vector(1,1){8}}

\put(-12,32){\vector(-1,1){8}}

\put(-20,30){\vector(1,0){40}}

\put(-13,35){\tiny{$_m$}}

\put(10,35){\tiny{$_1$}}

\put(3,2){\tiny{$_{m+1}$}}

\put(12,26){\tiny{$_n$}}

\end{picture}) \vspace{20pt}$ is 
{\tiny
\begin{equation}\label{complex-D-11-+}
0 \rightarrow C_f(\Gamma_{m,1}) \xrightarrow{\mathfrak{d}_m^+} \left.%
\begin{array}{c}
C_f(\Gamma_{m,0}) \{q^{-1}\}\\
\oplus \\
C_f(\Gamma_{m-1,1})\{q^{-1}\}
\end{array}%
\right. 
\xrightarrow{\mathfrak{d}_{m-1}^+} \cdots \xrightarrow{\mathfrak{d}_{k+1}^+} \left.%
\begin{array}{c}
C_f(\Gamma_{k+1,0}) \{q^{k-m}\}\\
\oplus \\
C_f(\Gamma_{k,1})\{q^{k-m}\}
\end{array}%
\right. 
\xrightarrow{\mathfrak{d}_{k}^+} \cdots\xrightarrow{\mathfrak{d}_{k_0}^+} C_f(\Gamma_{k_0,0}) \{q^{k_0-1-m}\} \rightarrow 0,
\end{equation}
}
where $k_0 = \max \{m-n,0\}$ as above and
\begin{eqnarray*}
\mathfrak{d}_m^+ & = & \left(%
\begin{array}{c}
\chi^1\\
-d_m^+
\end{array}%
\right),\\
\mathfrak{d}_k^+ & = & \left(%
\begin{array}{cc}
d_{k+1}^+ & \chi^1\\
0 & -d_k^+
\end{array}%
\right) ~\text{ for } k_0<k<m, \\
\mathfrak{d}_{k_0}^+ & = & \left(%
\begin{array}{cc}
d_{k_0 +1}^+, & \chi^1\\
\end{array}%
\right).
\end{eqnarray*}

Similarly, The chain complex $\hat{C}_f(\setlength{\unitlength}{1pt}
\begin{picture}(44,20)(-22,20)
\put(0,0){\vector(0,1){20}}

\put(0,20){\vector(1,1){20}}

\put(0,20){\vector(-1,1){20}}

\put(-20,30){\line(1,0){8}}

\put(-8,30){\line(1,0){16}}

\put(12,30){\vector(1,0){8}}

\put(-13,35){\tiny{$_m$}}

\put(10,35){\tiny{$_1$}}

\put(3,2){\tiny{$_{m+1}$}}

\put(12,26){\tiny{$_n$}}

\end{picture}) \vspace{20pt}$ is 
{\tiny
\begin{equation}\label{complex-D-11--}
0 \rightarrow C_f(\Gamma_{k_0,0}) \{q^{m+1-k_0}\} \xrightarrow{\mathfrak{d}_{k_0}^-} \cdots \xrightarrow{\mathfrak{d}_{k-1}^-} \left.%
\begin{array}{c}
C_f(\Gamma_{k,0}) \{q^{m+1-k}\}\\
\oplus \\
C_f(\Gamma_{k-1,1})\{q^{m+1-k}\}
\end{array}%
\right. 
\xrightarrow{\mathfrak{d}_{k}^-} \cdots  \xrightarrow{\mathfrak{d}_{m-1}^-} \left.%
\begin{array}{c}
C_f(\Gamma_{m,0}) \{q\}\\
\oplus \\
C_f(\Gamma_{m-1,1})\{q\}
\end{array}%
\right. 
\xrightarrow{\mathfrak{d}_m^-} C_f(\Gamma_{m,1}) \rightarrow 0,
\end{equation}
}
where $k_0 = \max \{m-n,0\}$ as above and
\begin{eqnarray*}
\mathfrak{d}_{k_0}^- & = & \left(%
\begin{array}{c}
d_{k_0}^-\\ 
\chi^0
\end{array}%
\right),\\
\mathfrak{d}_k^- & = & \left(%
\begin{array}{cc}
d_{k}^- & 0\\
\chi^0 & -d_{k-1}^-
\end{array}%
\right) ~\text{ for } k_0<k<m, \\
\mathfrak{d}_{m}^- & = & \left(%
\begin{array}{cc}
\chi^0 & -d_{m-1}^-\\
\end{array}%
\right).
\end{eqnarray*}

\subsection{Commutativity lemmas}

\begin{figure}[ht]
$
\xymatrix{
\input{chi-commute-chi-chi-fig2} \ar@<1ex>[rrrr]^{\chi^1_\triangle} \ar@<1ex>[d]^{h_1} &&&& \input{chi-commute-chi-chi-fig0} \ar@<1ex>[llll]^{\chi^0_\triangle} \ar@<1ex>[d]^{\overline{h}_2} \\
\input{chi-commute-chi-chi-fig1} \ar@<1ex>[u]^{\overline{h}_1} \ar@<1ex>[rr]^{\chi^1_\Box} && \input{chi-commute-chi-chi-fig3} \ar@<1ex>[ll]^{\chi^0_\Box} \ar@<1ex>[rr]^{\chi^1_\dag} &&\input{chi-commute-chi-chi-fig4} \ar@<1ex>[ll]^{\chi^0_\dag} \ar@<1ex>[u]^{h_2}
}
$
\caption{}\label{chi-commute-chi-chi-figure}

\end{figure}

\begin{lemma}\label{chi-commute-chi-chi}
Consider the diagram in Figure \ref{chi-commute-chi-chi-figure}, where the morphisms are induced by the apparent basic local changes of MOY graphs. Then $\chi^1_\triangle \approx h_2 \circ \chi^1_\dag \circ \chi^1_\Box \circ h_1$ and $\chi^0_\triangle \approx  \overline{h}_1 \circ \chi^0_\Box \circ \chi^0_\dag \circ \overline{h}_2$. That is, up to homotopy and scaling, the diagram in Figure \ref{chi-commute-chi-chi-figure} commutes in both directions.
\end{lemma}

\begin{proof}
It is easy to check that
\[
\xymatrix{
\Hom_\HMF (C_f(\input{chi-commute-chi-chi-fig2-sli}),C_f(\input{chi-commute-chi-chi-fig0-sli})) \vspace{30pt}\\
\cong H_f (\input{chi-commute-chi-chi-fig5-sli})\left\langle m+n+1 \right\rangle\{q^{(m+n+1)(N-m-n-1)+2m+2n+mn}\}.} \vspace{30pt}
\]
By Corollary \ref{contract-expand}, Theorem \ref{decomp-II} and Proposition \ref{circle-module}, we have
\[
H_f (\input{chi-commute-chi-chi-fig5-sli}) \cong C_f(\emptyset)\left\langle m+n+1 \right\rangle\{[m+1]\qb{m+n}{m+1} [m+n+1] \qb{N}{m+n+1}\}. \vspace{30pt}
\]
So 
\begin{eqnarray*}
&& \Hom_\HMF (C_f(\input{chi-commute-chi-chi-fig2-sli}),C_f(\input{chi-commute-chi-chi-fig0-sli})) \\
\\ \\
&\cong &C_f(\emptyset) \{[m+1]\qb{m+n}{m+1} [m+n+1] \qb{N}{m+n+1} q^{(m+n+1)(N-m-n-1)+2m+2n+mn}\}.
\end{eqnarray*}
Similarly,
\begin{eqnarray*}
&& \Hom_\HMF (C_f(\input{chi-commute-chi-chi-fig0-sli}),C_f(\input{chi-commute-chi-chi-fig2-sli})) \\
\\ \\
&\cong &C_f(\emptyset) \{[m+1]\qb{m+n}{m+1} [m+n+1] \qb{N}{m+n+1} q^{(m+n+1)(N-m-n-1)+2m+2n+mn}\}.
\end{eqnarray*}
So these $\Hom_\HMF$ spaces concentrate on $\zed_2$-grading $0$. The lowest non-vanishing total polynomial grading of these spaces is $m+1$. Moreover, the subspaces of homogeneous elements of total polynomial degree $m+1$ of these spaces are $1$-dimensional. 

Note that $\chi^1_\triangle$, $h_2 \circ \chi^1_\dag \circ \chi^1_\Box \circ h_1$, $\chi^0_\triangle$ and $\overline{h}_1 \circ \chi^0_\Box \circ \chi^0_\dag \circ \overline{h}_2$ are all homogeneous morphisms of total polynomial grading $m+1$. Therefore, to prove the lemma, we only need to show that all these morphisms are homotopically non-trivial. Let $\varpi_0$ be the functor given in Lemma \ref{MOY-object-of-hmf}. By Proposition \ref{basic-changes-varpi-0} and the proof of \cite[Lemma 12.3]{Wu-color}, one can see that $\varpi_0(\chi^1_\triangle)$, $\varpi_0(h_2 \circ \chi^1_\dag \circ \chi^1_\Box \circ h_1)$, $\varpi_0(\chi^0_\triangle)$ and $\varpi_0(\overline{h}_1 \circ \chi^0_\Box \circ \chi^0_\dag \circ \overline{h}_2)$ are homotopically non-trivial. So $\chi^1_\triangle$, $h_2 \circ \chi^1_\dag \circ \chi^1_\Box \circ h_1$, $\chi^0_\triangle$ and $\overline{h}_1 \circ \chi^0_\Box \circ \chi^0_\dag \circ \overline{h}_2$ are also homotopically non-trivial.
\end{proof}

\begin{figure}[ht]
$
\xymatrix{
\input{varphi-def-figure1} \ar@<1ex>[rr]^{\varphi} \ar@<1ex>[rd]^{\phi}&& \input{varphi-def-figure2} \ar@<1ex>[ll]^{\overline{\varphi}} \ar@<1ex>[ld]^{\overline{h}} \\
& \input{varphi-def-figure3} \ar@<1ex>[lu]^{\overline{\phi}} \ar@<1ex>[ru]^{h} &
}
$
\caption{}\label{varphi-def-figure}

\end{figure}

\begin{definition}\label{varphi-def}
Consider the morphisms in Figure \ref{varphi-def-figure}, where $\phi$ and $\overline{\phi}$ are the morphisms induced by the apparent edge splitting and merging, $h$ and $\overline{h}$ are induced by the apparent bouquet moves. Define $\varphi:= h \circ \phi$ and $\overline{\varphi}:= \overline{\phi} \circ \overline{h}$.

Let $\varpi_0$ be the functor given in Lemma \ref{MOY-object-of-hmf}. By Proposition \ref{basic-changes-varpi-0}, it is clear that $\varpi_0(\varphi)$ and $\varpi_0(\overline{\varphi})$ are the morphisms defined in \cite[Definition 12.4]{Wu-color}.
\end{definition}

By Corollary \ref{contract-expand}, Lemmas \ref{edge-splitting-lemma} and \ref{phibar-compose-phi}, it is easy to check that, up to homotopy and scaling, $\varphi$ and $\overline{\varphi}$ are the unique homotopically non-trivial homogeneous morphisms between $C_f(\Gamma)$ and $C_f(\widetilde{\Gamma})$ of $\zed_2$-degree $0$ and quantum degree $-mn$. And they satisfy, for $\lambda,\mu \in \Lambda_{m,n}$,
\begin{equation}\label{varphi-compose}
\overline{\varphi} \circ \mathfrak{m}(S_{\lambda}(\mathbb{X})\cdot S_{\mu}(-\mathbb{Y})) \circ \varphi \approx \left\{%
\begin{array}{ll}
    \id_{C_f(\Gamma)} & \text{if } \lambda_i + \mu_{m+1-i} = n ~\forall i=1,\dots,m, \\ 
    0 & \text{otherwise.}
\end{array}%
\right. 
\end{equation}

\begin{lemma}\label{varphis-commute}
Consider the diagram in Figure \ref{varphis-commute-figure}, where $\varphi_i$ and $\overline{\varphi}_i$ are the morphisms defined in Definition \ref{varphi-def} associated to the apparent local changes of the MOY graphs. Then $\varphi_2\circ\varphi_1 \approx \varphi_4\circ\varphi_3$ and $\overline{\varphi}_1 \circ \overline{\varphi}_2 \approx \overline{\varphi}_3 \circ \overline{\varphi}_4$. That is, the diagram in Figure \ref{varphis-commute-figure} commutes up to homotopy and scaling in both directions.
\end{lemma}

\begin{figure}[ht]
$
\xymatrix{
\input{varphis-commute-figure1} \ar@<1ex>[rr]^{\varphi_1} \ar@<1ex>[d]^{\varphi_3} && \input{varphis-commute-figure3} \ar@<1ex>[ll]^{\overline{\varphi}_1} \ar@<1ex>[d]^{\varphi_2} \\
\input{varphis-commute-figure2} \ar@<1ex>[u]^{\overline{\varphi}_3} \ar@<1ex>[rr]^{\varphi_4} && \input{varphis-commute-figure4} \ar@<1ex>[u]^{\overline{\varphi}_2}  \ar@<1ex>[ll]^{\overline{\varphi}_4}
}
$
\caption{}\label{varphis-commute-figure}

\end{figure}

\begin{proof}
As in the proof of \cite[Lemma 12.5]{Wu-color}, a straightforward computations shows that
{\tiny
\begin{eqnarray*}
&& \Hom_\HMF (C_f(\Gamma_1), C_f(\Gamma_0)) \cong \Hom_\HMF (C_f(\Gamma_0), C_f(\Gamma_1)) \\ 
& \cong & C_f(\emptyset) \{\qb{N-m-n-l}{j}\qb{N-m-n-l}{k}\qb{N}{m+n+l} \qb{m+n+l}{l}\qb{m+n}{n}q^{(m+n+l+j+k)(N-m-n-l)-j^2-k^2}\}.
\end{eqnarray*}
}
Thus, $\Hom_\HMF (C_f(\Gamma_1), C_f(\Gamma_0))$ and $\Hom_\HMF (C_f(\Gamma_0), C_f(\Gamma_1))$ are concentrated on $\zed_2$-degree $0$ and have lowest non-vanishing total polynomial grading $-mn-ml-nl$. And the subspaces of $\Hom_\HMF (C_f(\Gamma_1), C_f(\Gamma_0))$ and $\Hom_\HMF (C_f(\Gamma_0), C_f(\Gamma_1))$ of homogeneous elements of total polynomial grading $-mn-ml-nl$ are $1$-dimensional.

Note that $\varphi_2\circ\varphi_1$, $\varphi_4\circ\varphi_3$, $\overline{\varphi}_1 \circ \overline{\varphi}_2$ and $\overline{\varphi}_3 \circ \overline{\varphi}_4$ are all homogeneous of total polynomial degree $-mn-ml-nl$. So, to prove the lemma, we only need to show that these morphisms are all homotopically non-trivial. Let $\varpi_0$ be the functor given in Lemma \ref{MOY-object-of-hmf}. From Proposition \ref{basic-changes-varpi-0} and the proof of \cite[Lemma 12.5]{Wu-color}, one can see that $\varpi_0(\varphi_2\circ\varphi_1)$, $\varpi_0(\varphi_4\circ\varphi_3)$, $\varpi_0(\overline{\varphi}_1 \circ \overline{\varphi}_2)$ and $\varpi_0(\overline{\varphi}_3 \circ \overline{\varphi}_4)$ are all homotopically non-trivial. So $\varphi_2\circ\varphi_1$, $\varphi_4\circ\varphi_3$, $\overline{\varphi}_1 \circ \overline{\varphi}_2$ and $\overline{\varphi}_3 \circ \overline{\varphi}_4$ are all homotopically non-trivial.
\end{proof}

\subsection{Decomposing $C_f(\Gamma_{k,0})$} In this subsection, we review a special case of Decomposition (IV), including the construction of all the morphisms involved, which will be useful in our proof of the invariance under fork sliding.

By Decomposition (IV) (Theorem \ref{decomp-IV}), we have
\begin{equation}\label{fork-special-decomp-iv}
C_f(\input{fork-special-decomp-iv-1-sli}) \simeq C_f(\setlength{\unitlength}{.8pt}
\begin{picture}(105,30)(-55,30)

\put(-25,0){\vector(0,1){30}}
\put(-25,30){\vector(0,1){30}}

\put(25,0){\vector(0,1){30}}
\put(25,30){\vector(0,1){30}}

\put(-25,30){\vector(1,0){50}}

\put(-55,5){\tiny{$_{n+k-m}$}}
\put(-30,55){\tiny{$_{1}$}}

\put(-20,33){\tiny{$_{n+k-1-m}$}}

\put(28,5){\tiny{$_{m+1-k}$}}
\put(28,55){\tiny{$_{n}$}}
\end{picture}) ~\oplus~ C_f(\setlength{\unitlength}{.8pt}
\begin{picture}(85,30)(-45,30)

\put(-15,0){\vector(1,1){15}}
\put(0,45){\vector(-1,1){15}}
\put(15,0){\vector(-1,1){15}}
\put(0,45){\vector(1,1){15}}

\put(0,15){\vector(0,1){30}}

\put(-45,5){\tiny{$_{n+k-m}$}}
\put(-20,55){\tiny{$_{1}$}}

\put(18,5){\tiny{$_{m+1-k}$}}
\put(18,55){\tiny{$_{n}$}}

\put(3,30){\tiny{$_{n+1}$}}
\end{picture})\{[m-k]\}. \vspace{30pt}
\end{equation}

Define morphisms 
\[
\xymatrix{
C_f(\setlength{\unitlength}{.8pt}
\begin{picture}(105,30)(-55,30)

\put(-25,0){\vector(0,1){30}}
\put(-25,30){\vector(0,1){30}}

\put(25,0){\vector(0,1){30}}
\put(25,30){\vector(0,1){30}}

\put(-25,30){\vector(1,0){50}}

\put(-55,5){\tiny{$_{n+k-m}$}}
\put(-30,55){\tiny{$_{1}$}}

\put(-20,33){\tiny{$_{n+k-1-m}$}}

\put(28,5){\tiny{$_{m+1-k}$}}
\put(28,55){\tiny{$_{n}$}}
\end{picture}) \ar@<-1ex>[rr]^{F} && C_f(\input{fork-special-decomp-iv-1-sli}) \ar@<2ex>[ll]^{G}
} \vspace{30pt}
\]
by the diagram
\[
\xymatrix{
\setlength{\unitlength}{1pt}
\begin{picture}(100,60)(-50,0)

\put(-25,0){\vector(0,1){30}}
\put(-25,30){\vector(0,1){30}}

\put(25,0){\vector(0,1){30}}
\put(25,30){\vector(0,1){30}}

\put(-25,30){\vector(1,0){50}}

\put(-50,5){\tiny{$_{n+k-m}$}}
\put(-30,55){\tiny{$_{1}$}}

\put(-17,33){\tiny{$_{n+k-1-m}$}}

\put(28,5){\tiny{$_{m+1-k}$}}
\put(28,55){\tiny{$_{n}$}}
\end{picture} \ar@<1ex>[d]^{\phi_1} \ar@<1ex>[rr]^{F} && \input{fork-special-decomp-iv-1-ch} \ar@<1ex>[ll]^{G} \ar@<1ex>[d]^{\chi^1} \\
\input{fork-special-decomp-iv-4} \ar@<1ex>[u]^{\overline{\phi}_1} \ar@<1ex>[rr]^{h_1} && \input{fork-special-decomp-iv-5} \ar@<1ex>[ll]^{\overline{h}_1} \ar@<1ex>[u]^{\chi^0}
}.
\]
That is, 
\begin{eqnarray*}
F & = & \chi^0 \circ h_1 \circ \phi_1, \\
G & = & \overline{\phi}_1 \circ \overline{h}_1 \circ \chi^1,
\end{eqnarray*}
where the morphisms on the right hand side are induced by the apparent basic local changes of MOY graphs. Then $F$ and $G$ are both homogeneous morphisms preserving both gradings. By Lemma \ref{phibar-compose-phi} and Proposition \ref{general-general-chi-maps}, we have that, after possibly a scaling,
\begin{equation}\label{fork-special-decomp-iv-F-circ-G}
G\circ F \simeq \id,
\end{equation}
where $\id$ is the identity morphism of $C_f(\setlength{\unitlength}{.8pt}
\begin{picture}(105,30)(-55,30)

\put(-25,0){\vector(0,1){30}}
\put(-25,30){\vector(0,1){30}}

\put(25,0){\vector(0,1){30}}
\put(25,30){\vector(0,1){30}}

\put(-25,30){\vector(1,0){50}}

\put(-55,5){\tiny{$_{n+k-m}$}}
\put(-30,55){\tiny{$_{1}$}}

\put(-20,33){\tiny{$_{n+k-1-m}$}}

\put(28,5){\tiny{$_{m+1-k}$}}
\put(28,55){\tiny{$_{n}$}}
\end{picture}) \vspace{30pt}$.

Define morphisms 
\[
\xymatrix{
C_f(\setlength{\unitlength}{.8pt}
\begin{picture}(85,30)(-45,30)

\put(-15,0){\vector(1,1){15}}
\put(0,45){\vector(-1,1){15}}
\put(15,0){\vector(-1,1){15}}
\put(0,45){\vector(1,1){15}}

\put(0,15){\vector(0,1){30}}

\put(-45,5){\tiny{$_{n+k-m}$}}
\put(-20,55){\tiny{$_{1}$}}

\put(18,5){\tiny{$_{m+1-k}$}}
\put(18,55){\tiny{$_{n}$}}

\put(3,30){\tiny{$_{n+1}$}}
\end{picture}) \ar@<-1ex>[rr]^{\alpha} && C_f(\input{fork-special-decomp-iv-1-sli}) \ar@<2ex>[ll]^{\beta}
} \vspace{30pt}
\]
by the diagram
\[
\xymatrix{
\setlength{\unitlength}{1pt}
\begin{picture}(80,60)(-40,0)

\put(-15,0){\vector(1,1){15}}
\put(0,45){\vector(-1,1){15}}
\put(15,0){\vector(-1,1){15}}
\put(0,45){\vector(1,1){15}}

\put(0,15){\vector(0,1){30}}

\put(-40,5){\tiny{$_{n+k-m}$}}
\put(-20,55){\tiny{$_{1}$}}

\put(18,5){\tiny{$_{m+1-k}$}}
\put(18,55){\tiny{$_{n}$}}

\put(3,30){\tiny{$_{n+1}$}}
\end{picture}  \ar@<1ex>[d]^{\phi_2} \ar@<1ex>[rr]^{\alpha} && \input{fork-special-decomp-iv-1-ch} \ar@<1ex>[ll]^{\beta} \ar@<1ex>[d]^{\chi^0} \\
\input{fork-special-decomp-iv-6} \ar@<1ex>[u]^{\overline{\phi}_2} \ar@<1ex>[rr]^{h_2}&& \input{fork-special-decomp-iv-7} \ar@<1ex>[ll]^{\overline{h}_2} \ar@<1ex>[u]^{\chi^1}
}
\]
That is,
\begin{eqnarray*}
\alpha & = & \chi^1 \circ h_2 \circ \phi_2, \\
\beta & = & \overline{\phi}_2 \circ \overline{h}_2 \circ \chi^0,
\end{eqnarray*}
where the morphisms on the right hand side are induced by the apparent basic local changes of MOY graphs. Then define 
\begin{eqnarray*}
\vec{\alpha} & = & \sum_{j=0}^{m-k-1} \mathfrak{m}(r^j) \circ\alpha = (\alpha, ~\mathfrak{m}(r) \circ\alpha,~\dots, ~\mathfrak{m}(r^{m-k-1}) \circ\alpha),  \\
\vec{\beta} & = & \bigoplus_{j=0}^{m-k-1} \beta\circ\mathfrak{m}((-1)^{m-k-1-j}A_{m-k-1-j}) = \left(%
\begin{array}{c}
\beta\circ\mathfrak{m}((-1)^{m-k-1}A_{m-k-1})\\
\cdots \\
\beta\circ\mathfrak{m}(-A_1)\\
\beta
\end{array}%
\right),
\end{eqnarray*}
where $A_j$ is the $j$-th elementary symmetric polynomial in $\mathbb{A}$. Then 
\[
\xymatrix{
C_f(\setlength{\unitlength}{.8pt}
\begin{picture}(85,30)(-45,30)

\put(-15,0){\vector(1,1){15}}
\put(0,45){\vector(-1,1){15}}
\put(15,0){\vector(-1,1){15}}
\put(0,45){\vector(1,1){15}}

\put(0,15){\vector(0,1){30}}

\put(-45,5){\tiny{$_{n+k-m}$}}
\put(-20,55){\tiny{$_{1}$}}

\put(18,5){\tiny{$_{m+1-k}$}}
\put(18,55){\tiny{$_{n}$}}

\put(3,30){\tiny{$_{n+1}$}}
\end{picture})\{[m-k]\}  \ar@<-2ex>[rr]^{\vec{\alpha}} &&  C_f(\input{fork-special-decomp-iv-1-sli}), \vspace{30pt} \\
C_f(\input{fork-special-decomp-iv-1-sli}) \ar@<-2ex>[rr]^{\vec{\beta}} && C_f(\setlength{\unitlength}{.8pt}
\begin{picture}(85,30)(-45,30)

\put(-15,0){\vector(1,1){15}}
\put(0,45){\vector(-1,1){15}}
\put(15,0){\vector(-1,1){15}}
\put(0,45){\vector(1,1){15}}

\put(0,15){\vector(0,1){30}}

\put(-45,5){\tiny{$_{n+k-m}$}}
\put(-20,55){\tiny{$_{1}$}}

\put(18,5){\tiny{$_{m+1-k}$}}
\put(18,55){\tiny{$_{n}$}}

\put(3,30){\tiny{$_{n+1}$}}
\end{picture})\{[m-k]\}
}\vspace{30pt}
\]
are homogeneous morphisms preserving both gradings. Moreover, by \cite[Lemma 9.12]{Wu-color},
\[
\xymatrix{
C(\setlength{\unitlength}{.8pt}
\begin{picture}(85,30)(-45,30)

\put(-15,0){\vector(1,1){15}}
\put(0,45){\vector(-1,1){15}}
\put(15,0){\vector(-1,1){15}}
\put(0,45){\vector(1,1){15}}

\put(0,15){\vector(0,1){30}}

\put(-45,5){\tiny{$_{n+k-m}$}}
\put(-20,55){\tiny{$_{1}$}}

\put(18,5){\tiny{$_{m+1-k}$}}
\put(18,55){\tiny{$_{n}$}}

\put(3,30){\tiny{$_{n+1}$}}
\end{picture})\{[m-k]\}  \ar@<-2ex>[rr]^{\varpi_0(\vec{\beta} \circ \vec{\alpha})} && C(\setlength{\unitlength}{.8pt}
\begin{picture}(85,30)(-45,30)

\put(-15,0){\vector(1,1){15}}
\put(0,45){\vector(-1,1){15}}
\put(15,0){\vector(-1,1){15}}
\put(0,45){\vector(1,1){15}}

\put(0,15){\vector(0,1){30}}

\put(-45,5){\tiny{$_{n+k-m}$}}
\put(-20,55){\tiny{$_{1}$}}

\put(18,5){\tiny{$_{m+1-k}$}}
\put(18,55){\tiny{$_{n}$}}

\put(3,30){\tiny{$_{n+1}$}}
\end{picture})\{[m-k]\}
}\vspace{30pt}
\]
is a homotopy equivalence of matrix factorizations, where $\varpi_0$ is the functor given in Lemma \ref{MOY-object-of-hmf}. So, by Lemma \ref{reduce-base-homotopic-equivalence},
\[
\xymatrix{
C_f(\setlength{\unitlength}{.8pt}
\begin{picture}(85,30)(-45,30)

\put(-15,0){\vector(1,1){15}}
\put(0,45){\vector(-1,1){15}}
\put(15,0){\vector(-1,1){15}}
\put(0,45){\vector(1,1){15}}

\put(0,15){\vector(0,1){30}}

\put(-45,5){\tiny{$_{n+k-m}$}}
\put(-20,55){\tiny{$_{1}$}}

\put(18,5){\tiny{$_{m+1-k}$}}
\put(18,55){\tiny{$_{n}$}}

\put(3,30){\tiny{$_{n+1}$}}
\end{picture})\{[m-k]\}  \ar@<-2ex>[rr]^{\vec{\beta} \circ \vec{\alpha}} && C_f(\setlength{\unitlength}{.8pt}
\begin{picture}(85,30)(-45,30)

\put(-15,0){\vector(1,1){15}}
\put(0,45){\vector(-1,1){15}}
\put(15,0){\vector(-1,1){15}}
\put(0,45){\vector(1,1){15}}

\put(0,15){\vector(0,1){30}}

\put(-45,5){\tiny{$_{n+k-m}$}}
\put(-20,55){\tiny{$_{1}$}}

\put(18,5){\tiny{$_{m+1-k}$}}
\put(18,55){\tiny{$_{n}$}}

\put(3,30){\tiny{$_{n+1}$}}
\end{picture})\{[m-k]\}
}\vspace{30pt}
\]
is a homotopy equivalence of matrix factorizations. Therefore, there exists a homogeneous morphism
\[
\xymatrix{
C_f(\setlength{\unitlength}{.8pt}
\begin{picture}(85,30)(-45,30)

\put(-15,0){\vector(1,1){15}}
\put(0,45){\vector(-1,1){15}}
\put(15,0){\vector(-1,1){15}}
\put(0,45){\vector(1,1){15}}

\put(0,15){\vector(0,1){30}}

\put(-45,5){\tiny{$_{n+k-m}$}}
\put(-20,55){\tiny{$_{1}$}}

\put(18,5){\tiny{$_{m+1-k}$}}
\put(18,55){\tiny{$_{n}$}}

\put(3,30){\tiny{$_{n+1}$}}
\end{picture})\{[m-k]\}  \ar@<-2ex>[rr]^{\tau} && C_f(\setlength{\unitlength}{.8pt}
\begin{picture}(85,30)(-45,30)

\put(-15,0){\vector(1,1){15}}
\put(0,45){\vector(-1,1){15}}
\put(15,0){\vector(-1,1){15}}
\put(0,45){\vector(1,1){15}}

\put(0,15){\vector(0,1){30}}

\put(-45,5){\tiny{$_{n+k-m}$}}
\put(-20,55){\tiny{$_{1}$}}

\put(18,5){\tiny{$_{m+1-k}$}}
\put(18,55){\tiny{$_{n}$}}

\put(3,30){\tiny{$_{n+1}$}}
\end{picture})\{[m-k]\}
}\vspace{30pt}
\]
preserving both gradings such that 
\begin{equation}\label{fork-special-decomp-iv-alpha-circ-beta}
\tau \circ \vec{\beta} \circ \vec{\alpha} \simeq \vec{\beta} \circ \vec{\alpha} \circ \tau \simeq \id,
\end{equation} 
where $\id$ is the identity morphism of $C_f(\setlength{\unitlength}{.8pt}
\begin{picture}(85,30)(-45,30)

\put(-15,0){\vector(1,1){15}}
\put(0,45){\vector(-1,1){15}}
\put(15,0){\vector(-1,1){15}}
\put(0,45){\vector(1,1){15}}

\put(0,15){\vector(0,1){30}}

\put(-45,5){\tiny{$_{n+k-m}$}}
\put(-20,55){\tiny{$_{1}$}}

\put(18,5){\tiny{$_{m+1-k}$}}
\put(18,55){\tiny{$_{n}$}}

\put(3,30){\tiny{$_{n+1}$}}
\end{picture})\{[m-k]\} \vspace{30pt}$.

By a computation similar to that in \cite[Lemma 9.13]{Wu-color}, one can check that the lowest non-vanishing total polynomial grading of 
$\Hom_\HMF(C_f(\setlength{\unitlength}{.8pt}
\begin{picture}(85,30)(-45,30)

\put(-15,0){\vector(1,1){15}}
\put(0,45){\vector(-1,1){15}}
\put(15,0){\vector(-1,1){15}}
\put(0,45){\vector(1,1){15}}

\put(0,15){\vector(0,1){30}}

\put(-45,5){\tiny{$_{n+k-m}$}}
\put(-20,55){\tiny{$_{1}$}}

\put(18,5){\tiny{$_{m+1-k}$}}
\put(18,55){\tiny{$_{n}$}}

\put(3,30){\tiny{$_{n+1}$}}
\end{picture}), C_f(\setlength{\unitlength}{.8pt}
\begin{picture}(105,30)(-55,30)

\put(-25,0){\vector(0,1){30}}
\put(-25,30){\vector(0,1){30}}

\put(25,0){\vector(0,1){30}}
\put(25,30){\vector(0,1){30}}

\put(-25,30){\vector(1,0){50}}

\put(-55,5){\tiny{$_{n+k-m}$}}
\put(-30,55){\tiny{$_{1}$}}

\put(-20,33){\tiny{$_{n+k-1-m}$}}

\put(28,5){\tiny{$_{m+1-k}$}}
\put(28,55){\tiny{$_{n}$}}
\end{picture})) \vspace{30pt}$ and $\Hom_\HMF(C_f(\setlength{\unitlength}{.8pt}
\begin{picture}(105,30)(-55,30)

\put(-25,0){\vector(0,1){30}}
\put(-25,30){\vector(0,1){30}}

\put(25,0){\vector(0,1){30}}
\put(25,30){\vector(0,1){30}}

\put(-25,30){\vector(1,0){50}}

\put(-55,5){\tiny{$_{n+k-m}$}}
\put(-30,55){\tiny{$_{1}$}}

\put(-20,33){\tiny{$_{n+k-1-m}$}}

\put(28,5){\tiny{$_{m+1-k}$}}
\put(28,55){\tiny{$_{n}$}}
\end{picture}), C_f(\setlength{\unitlength}{.8pt}
\begin{picture}(85,30)(-45,30)

\put(-15,0){\vector(1,1){15}}
\put(0,45){\vector(-1,1){15}}
\put(15,0){\vector(-1,1){15}}
\put(0,45){\vector(1,1){15}}

\put(0,15){\vector(0,1){30}}

\put(-45,5){\tiny{$_{n+k-m}$}}
\put(-20,55){\tiny{$_{1}$}}

\put(18,5){\tiny{$_{m+1-k}$}}
\put(18,55){\tiny{$_{n}$}}

\put(3,30){\tiny{$_{n+1}$}}
\end{picture}))\vspace{30pt}$ is $m+1-k$, which implies that
\[
\xymatrix{
\Hom_\hmf(C_f(\setlength{\unitlength}{.8pt}
\begin{picture}(85,30)(-45,30)

\put(-15,0){\vector(1,1){15}}
\put(0,45){\vector(-1,1){15}}
\put(15,0){\vector(-1,1){15}}
\put(0,45){\vector(1,1){15}}

\put(0,15){\vector(0,1){30}}

\put(-45,5){\tiny{$_{n+k-m}$}}
\put(-20,55){\tiny{$_{1}$}}

\put(18,5){\tiny{$_{m+1-k}$}}
\put(18,55){\tiny{$_{n}$}}

\put(3,30){\tiny{$_{n+1}$}}
\end{picture})\{[m-k]\}, C_f(\setlength{\unitlength}{.8pt}
\begin{picture}(105,30)(-55,30)

\put(-25,0){\vector(0,1){30}}
\put(-25,30){\vector(0,1){30}}

\put(25,0){\vector(0,1){30}}
\put(25,30){\vector(0,1){30}}

\put(-25,30){\vector(1,0){50}}

\put(-55,5){\tiny{$_{n+k-m}$}}
\put(-30,55){\tiny{$_{1}$}}

\put(-20,33){\tiny{$_{n+k-1-m}$}}

\put(28,5){\tiny{$_{m+1-k}$}}
\put(28,55){\tiny{$_{n}$}}
\end{picture})) \cong 0, \vspace{30pt}\\ \Hom_\hmf(C_f(\setlength{\unitlength}{.8pt}
\begin{picture}(105,30)(-55,30)

\put(-25,0){\vector(0,1){30}}
\put(-25,30){\vector(0,1){30}}

\put(25,0){\vector(0,1){30}}
\put(25,30){\vector(0,1){30}}

\put(-25,30){\vector(1,0){50}}

\put(-55,5){\tiny{$_{n+k-m}$}}
\put(-30,55){\tiny{$_{1}$}}

\put(-20,33){\tiny{$_{n+k-1-m}$}}

\put(28,5){\tiny{$_{m+1-k}$}}
\put(28,55){\tiny{$_{n}$}}
\end{picture}), C_f(\setlength{\unitlength}{.8pt}
\begin{picture}(85,30)(-45,30)

\put(-15,0){\vector(1,1){15}}
\put(0,45){\vector(-1,1){15}}
\put(15,0){\vector(-1,1){15}}
\put(0,45){\vector(1,1){15}}

\put(0,15){\vector(0,1){30}}

\put(-45,5){\tiny{$_{n+k-m}$}}
\put(-20,55){\tiny{$_{1}$}}

\put(18,5){\tiny{$_{m+1-k}$}}
\put(18,55){\tiny{$_{n}$}}

\put(3,30){\tiny{$_{n+1}$}}
\end{picture})\{[m-k]\}) \cong 0.
}\vspace{30pt}
\]
Therefore, 
\begin{eqnarray}
\label{fork-special-decomp-iv-F-circ-beta} \vec{\beta} \circ F & \simeq & 0, \\
\label{fork-special-decomp-iv-alpha-circ-G} G \circ \vec{\alpha}& \simeq & 0.
\end{eqnarray}

From the proof of Decomposition (IV) (Theorem \ref{decomp-IV},) we know that the morphisms
\begin{equation}\label{fork-special-decomp-iv-lemma-1}
\xymatrix{
C_f(\input{fork-special-decomp-iv-1-sli}) \ar@<-1ex>[rr]^<<<<<<<<<<{\left(%
\begin{array}{c}
G\\
\tau\circ \vec{\beta}
\end{array}%
\right)} &&
C_f(\setlength{\unitlength}{.8pt}
\begin{picture}(105,30)(-55,30)

\put(-25,0){\vector(0,1){30}}
\put(-25,30){\vector(0,1){30}}

\put(25,0){\vector(0,1){30}}
\put(25,30){\vector(0,1){30}}

\put(-25,30){\vector(1,0){50}}

\put(-55,5){\tiny{$_{n+k-m}$}}
\put(-30,55){\tiny{$_{1}$}}

\put(-20,33){\tiny{$_{n+k-1-m}$}}

\put(28,5){\tiny{$_{m+1-k}$}}
\put(28,55){\tiny{$_{n}$}}
\end{picture})  
\bigoplus 
C_f(\setlength{\unitlength}{.8pt}
\begin{picture}(85,30)(-45,30)

\put(-15,0){\vector(1,1){15}}
\put(0,45){\vector(-1,1){15}}
\put(15,0){\vector(-1,1){15}}
\put(0,45){\vector(1,1){15}}

\put(0,15){\vector(0,1){30}}

\put(-45,5){\tiny{$_{n+k-m}$}}
\put(-20,55){\tiny{$_{1}$}}

\put(18,5){\tiny{$_{m+1-k}$}}
\put(18,55){\tiny{$_{n}$}}

\put(3,30){\tiny{$_{n+1}$}}
\end{picture})\{[m-k]\} \ar@<2ex>[ll]^<<<<<<<<<<{\left(%
\begin{array}{cc}
F, & \vec{\alpha}
\end{array}%
\right)}
}\vspace{30pt}
\end{equation}
and
\begin{equation}\label{fork-special-decomp-iv-lemma-2}
\xymatrix{
C_f(\input{fork-special-decomp-iv-1-sli}) \ar@<-1ex>[rr]^<<<<<<<<<<{\left(%
\begin{array}{c}
G\\
\vec{\beta}
\end{array}%
\right)} &&
C_f(\setlength{\unitlength}{.8pt}
\begin{picture}(105,30)(-55,30)

\put(-25,0){\vector(0,1){30}}
\put(-25,30){\vector(0,1){30}}

\put(25,0){\vector(0,1){30}}
\put(25,30){\vector(0,1){30}}

\put(-25,30){\vector(1,0){50}}

\put(-55,5){\tiny{$_{n+k-m}$}}
\put(-30,55){\tiny{$_{1}$}}

\put(-20,33){\tiny{$_{n+k-1-m}$}}

\put(28,5){\tiny{$_{m+1-k}$}}
\put(28,55){\tiny{$_{n}$}}
\end{picture})  
\bigoplus 
C_f(\setlength{\unitlength}{.8pt}
\begin{picture}(85,30)(-45,30)

\put(-15,0){\vector(1,1){15}}
\put(0,45){\vector(-1,1){15}}
\put(15,0){\vector(-1,1){15}}
\put(0,45){\vector(1,1){15}}

\put(0,15){\vector(0,1){30}}

\put(-45,5){\tiny{$_{n+k-m}$}}
\put(-20,55){\tiny{$_{1}$}}

\put(18,5){\tiny{$_{m+1-k}$}}
\put(18,55){\tiny{$_{n}$}}

\put(3,30){\tiny{$_{n+1}$}}
\end{picture})\{[m-k]\} \ar@<2ex>[ll]^<<<<<<<<<<{\left(%
\begin{array}{cc}
F, & \vec{\alpha} \circ \tau
\end{array}%
\right)}
}\vspace{30pt}
\end{equation}
are two pairs of homotopy equivalences of matrix factorizations that preserve both gradings and are inverses of each other.

\begin{figure}[ht]
$
\xymatrix{
\input{gamma-k0} & \input{tilde-gamma-k} & \input{gamma-k-prime} \\
\input{gamma-k0-bouquet} & \input{gamma-k3} 
}
$
\caption{}\label{fork-special-decomp-iv-fig4}

\end{figure}

Next, we apply the above discussion to MOY graphs that appear in the chain complexes in Subsection \ref{fork-sliding-complexes-involved}. Consider the MOY graphs in Figure \ref{fork-special-decomp-iv-fig4}. By Corollary \ref{contract-expand}, we have $C_f(\Gamma_{k,0}) \simeq C_f(\Gamma_{k,2})$ and $C_f(\widetilde{\Gamma}_k) \simeq C_f(\Gamma_{k,3})$. By \eqref{fork-special-decomp-iv}, $C_f(\Gamma_{k,2}) \simeq C_f(\Gamma_{k,3}) \oplus C_f(\Gamma_k')\{[m-k]\}$. Altogether, we have
\begin{equation}\label{fork-special-decomp-iv-eq}
C_f(\Gamma_{k,0}) \simeq C_f(\widetilde{\Gamma}_k) \oplus C_f(\Gamma_k')\{[m-k]\}.
\end{equation}

\begin{figure}[ht]
$
\xymatrix{
\input{tilde-gamma-k} \ar@<1ex>[rr]^{F_k} \ar@<1ex>[d]^{h^{(k)}} && \input{gamma-k0} \ar@<1ex>[ll]^{G_k} \ar@<1ex>[d]^{\chi^1}  \\
\input{gamma-k3} \ar@<1ex>[u]^{\overline{h}^{(k)}} \ar@<1ex>[rr]^{\varphi_1} && \input{gamma-k3-3-squares} \ar@<1ex>[ll]^{\overline{\varphi}_1} \ar@<1ex>[u]^{\chi^0}
}
$
\caption{}\label{fork-special-decomp-iv-fig5}

\end{figure}

In Figure \ref{fork-special-decomp-iv-fig5}, the morphisms $F_k$ and $G_k$ are defined by
\begin{eqnarray*}
F_k & = & \chi^0 \circ \varphi_1 \circ h^{(k)}, \\
G_k & = & \overline{h}^{(k)} \circ \overline{\varphi}_1 \circ \chi^1,
\end{eqnarray*}
where the morphisms on the right hand side are induced by the apparent basic local changes of MOY graphs. Then $F_k$ and $G_k$ are homogeneous morphisms preserving both gradings and satisfy
\begin{equation}\label{fork-special-decomp-iv-f-g-k}
G_k \circ F_k \simeq  \id_{C_f(\widetilde{\Gamma}_k)}.
\end{equation}

\begin{figure}[ht]
$
\xymatrix{
\input{gamma-k-prime} \ar@<1ex>[rr]^{\alpha_k} \ar@<1ex>[dr]^{\varphi_2} && \input{gamma-k0-marked} \ar@<1ex>[ll]^{\beta_k} \ar@<1ex>[dl]^{\chi^0} \\
& \input{gamma-k-prime-2-squares} \ar@<1ex>[ul]^{\overline{\varphi}_2} \ar@<1ex>[ur]^{\chi^1} &
}
$
\caption{}\label{fork-special-decomp-iv-fig6}

\end{figure}

In Figure \ref{fork-special-decomp-iv-fig6}, the morphisms $\alpha_k$ and $\beta_k$ are defined by 
\begin{eqnarray*}
\alpha_k & = & \chi^1 \circ \varphi_2, \\
\beta_k & = & \overline{\varphi}_2 \circ \chi^0,
\end{eqnarray*}
where the morphisms on the right hand side are induced by the apparent basic local changes of MOY graphs. Define
\begin{eqnarray*}
\vec{\alpha}_k & = & \sum_{j=0}^{m-k-1} \mathfrak{m}(r^j) \circ\alpha_k = (\alpha_k, ~\mathfrak{m}(r) \circ\alpha_k,~\dots, ~\mathfrak{m}(r^{m-k-1}) \circ\alpha_k),  \\
\vec{\beta}_k & = & \bigoplus_{j=0}^{m-k-1} \beta_k \circ\mathfrak{m}((-1)^{m-k-1-j}A_{m-k-1-j}) = \left(%
\begin{array}{c}
\beta_k \circ\mathfrak{m}((-1)^{m-k-1}A_{m-k-1})\\
\cdots \\
\beta_k \circ\mathfrak{m}(-A_1)\\
\beta_k
\end{array}%
\right).
\end{eqnarray*}
Then there is a homogeneous morphism  $\tau_k:C_f(\Gamma_k')\{[m-k]\} \rightarrow C_f(\Gamma_k')\{[m-k]\}$ preserving both gradings such that
\begin{equation}\label{fork-special-decomp-iv-alpha-beta-k}
\tau_k \circ \vec{\beta}_k \circ \vec{\alpha}_k \simeq \vec{\beta}_k \circ \vec{\alpha}_k \circ \tau_k \simeq \id_{C_f(\Gamma_k')\{[m-k]\}}.
\end{equation}

We also have
\begin{equation}\label{fork-special-decomp-iv-alpha-f-k}
G_k\circ \vec{\alpha}_k\simeq 0 \text{ and } \vec{\beta}_k\circ F_k \simeq 0.
\end{equation}

\begin{corollary}\label{fork-special-decomp-iv-corollary}
The morphisms
\begin{equation}\label{fork-special-decomp-iv-cor-1}
\xymatrix{
C_f(\Gamma_{k,0}) \ar@<1ex>[rrr]^<<<<<<<<<<<<<<<{\left(%
\begin{array}{c}
G_k\\
\tau_k \circ \vec{\beta}_k
\end{array}%
\right)} &&& 
C_f(\widetilde{\Gamma}_k) \oplus C_f(\Gamma_k')\{[m-k]\} \ar@<1ex>[lll]^<<<<<<<<<<<<<<<{\left(%
\begin{array}{cc}
F_k, & \vec{\alpha}_k
\end{array}%
\right)}
}\end{equation}
and
\begin{equation}\label{fork-special-decomp-iv-cor-2}
\xymatrix{
C_f(\Gamma_{k,0}) \ar@<1ex>[rrr]^<<<<<<<<<<<<<<<{\left(%
\begin{array}{c}
G_k\\
\vec{\beta}_k
\end{array}%
\right)} &&& 
C_f(\widetilde{\Gamma}_k) \oplus C_f(\Gamma_k')\{[m-k]\} \ar@<1ex>[lll]^<<<<<<<<<<<<<<<{\left(%
\begin{array}{cc}
F_k, & \vec{\alpha}_k \circ \tau_k
\end{array}%
\right)}
}
\end{equation}
are two pairs of homotopy equivalences of matrix factorizations that preserve both gradings and are inverses of each other.
\end{corollary}

\subsection{Relating the differential maps} Consider the diagram in Figure \ref{relating-C-D-11-fig}, where $d_k^+$, $d_{k-1}^-$ act on the left square, and $\varphi_i$, $\overline{\varphi}_i$ are induced by the apparent local changes of MOY graphs. We have the following lemma, which relates the differential maps of $C^\pm$ to that of $\hat{C}_f(\setlength{\unitlength}{1pt}
\begin{picture}(44,20)(-22,20)

\put(0,0){\vector(0,1){20}}

\put(0,20){\line(1,1){8}}

\put(0,20){\line(-1,1){8}}

\put(12,32){\vector(1,1){8}}

\put(-12,32){\vector(-1,1){8}}

\put(-20,30){\vector(1,0){40}}

\put(-13,35){\tiny{$_m$}}

\put(10,35){\tiny{$_1$}}

\put(3,2){\tiny{$_{m+1}$}}

\put(12,26){\tiny{$_n$}}

\end{picture})$ and $\hat{C}_f(\setlength{\unitlength}{1pt}
\begin{picture}(44,20)(-22,20)
\put(0,0){\vector(0,1){20}}

\put(0,20){\vector(1,1){20}}

\put(0,20){\vector(-1,1){20}}

\put(-20,30){\line(1,0){8}}

\put(-8,30){\line(1,0){16}}

\put(12,30){\vector(1,0){8}}

\put(-13,35){\tiny{$_m$}}

\put(10,35){\tiny{$_1$}}

\put(3,2){\tiny{$_{m+1}$}}

\put(12,26){\tiny{$_n$}}

\end{picture}) \vspace{20pt}$.

\begin{figure}[ht]
$
\xymatrix{
\input{square-m-n-1-right-k-low-new-mark} \ar@<1ex>[rr]^{\delta_k^+} \ar@<1ex>[d]^{\mathfrak{m}(r^{m-k})\circ \varphi_1} && \input{square-m-n-1-right-k-1-low-new-mark} \ar@<1ex>[ll]^{\delta_{k-1}^-} \ar@<1ex>[d]^{\varphi_2} \\
\input{double-square-m-n-1-right-k-m-low} \ar@<1ex>[u]^{\overline{\varphi}_1 \circ \mathfrak{m}(r^{m-k})} \ar@<1ex>[rr]^{d_k^+} && \input{double-square-m-n-1-right-k-1-m-low} \ar@<1ex>[u]^{\overline{\varphi}_2} \ar@<1ex>[ll]^{d_{k-1}^-}
}
$
\caption{}\label{relating-C-D-11-fig}

\end{figure}

\begin{lemma}\label{relating-C-D-11-lemma}
$\delta_k^+ \approx \overline{\varphi}_2 \circ d_k^+ \circ \mathfrak{m}(r^{m-k})\circ \varphi_1$ and $\delta_{k-1}^- \approx \overline{\varphi}_1 \circ \mathfrak{m}(r^{m-k}) \circ d_{k-1}^- \circ \varphi_2$. That is, the diagram in Figure \ref{relating-C-D-11-fig} commutes up to homotopy and scaling in both directions.
\end{lemma}

\begin{proof}
By Lemma \ref{trivial-complex-lemma-3}, 
\[
\xymatrix{
\Hom_\hmf(C_f(\input{square-m-n-1-right-k-low-new-mark-sli}), C_f(\input{square-m-n-1-right-k-1-low-new-mark-sli})) \cong \C, \vspace{30pt}\\ \Hom_\hmf(C_f(\input{square-m-n-1-right-k-1-low-new-mark-sli}), C_f(\input{square-m-n-1-right-k-low-new-mark-sli})) \cong \C.
}\vspace{30pt}
\]
Note that $\delta_k^+$, $\overline{\varphi}_2 \circ d_k^+ \circ \mathfrak{m}(r^{m-k})\circ \varphi_1$, $\delta_{k-1}^-$ and $\overline{\varphi}_1 \circ \mathfrak{m}(r^{m-k}) \circ d_{k-1}^- \circ \varphi_2$ are all homogeneous morphisms preserving both gradings. So, to prove the lemma, one only needs to prove that these morphisms are all homotopically non-trivial.

By their construction in Subsection \ref{subsec-null-chain}, it is easy to see that $\delta_k^+$ and $\delta_{k-1}^-$ are homotopically non-trivial. Let $\varpi_0$ be the functor given in Lemma \ref{MOY-object-of-hmf}. From Corollary \ref{differentials-varpi-0} and \cite[Proposition 11.25]{Wu-color}, one can see that $\varpi_0(\delta_k^+)$ and $\varpi_0(\delta_{k-1}^-)$ are also homotopically non-trivial. 

By Proposition \ref{basic-changes-varpi-0}, Corollary \ref{differentials-varpi-0} and \cite[Lemma 12.8]{Wu-color}, we have
\begin{eqnarray*}
\varpi_0(\delta_k^+) & \approx & \varpi_0(\overline{\varphi}_2 \circ d_k^+ \circ \mathfrak{m}(r^{m-k})\circ \varphi_1), \\
\varpi_0(\delta_{k-1}^-) & \approx & \varpi_0(\overline{\varphi}_1 \circ \mathfrak{m}(r^{m-k}) \circ d_{k-1}^- \circ \varphi_2),
\end{eqnarray*}
which implies that $\varpi_0(\overline{\varphi}_2 \circ d_k^+ \circ \mathfrak{m}(r^{m-k})\circ \varphi_1)$ and $\varpi_0(\overline{\varphi}_1 \circ \mathfrak{m}(r^{m-k}) \circ d_{k-1}^- \circ \varphi_2)$ are homotopically non-trivial. Therefore, $\overline{\varphi}_2 \circ d_k^+ \circ \mathfrak{m}(r^{m-k})\circ \varphi_1$ and $\overline{\varphi}_1 \circ \mathfrak{m}(r^{m-k}) \circ d_{k-1}^- \circ \varphi_2$ are also homotopically non-trivial.
\end{proof}

\begin{figure}[ht]
$
\xymatrix{
\input{tilde-gamma-k} \ar@<1ex>[rr]^{\tilde{d}_k^+} \ar@<1ex>[d]^{F_k} && \input{tilde-gamma-k-1} \ar@<1ex>[ll]^{\tilde{d}_{k-1}^-} \ar@<1ex>[d]^{F_{k-1}} \\
\input{gamma-k0} \ar@<1ex>[rr]^{d_k^+} \ar@<1ex>[u]^{G_k} && \input{gamma-k-1-0-no-tri} \ar@<1ex>[ll]^{d_{k-1}^-} \ar@<1ex>[u]^{G_{k-1}}
} 
$
\caption{}\label{relating-D-0-1-fig3}

\end{figure}

\begin{lemma}\label{relating-D-0-1-lemma}
In Figure \ref{relating-D-0-1-fig3}, $\tilde{d}_k^+ \approx G_{k-1} \circ d_k^+\circ F_k$ and $\tilde{d}_{k-1}^- \approx G_k \circ d_{k-1}^- \circ F_{k-1}$. That is, the diagram in Figure \ref{relating-D-0-1-fig3} commutes in both directions up to homotopy and scaling.
\end{lemma}

\begin{proof}
From Lemma \ref{complex-computing-gamma-HMF-lemma} and Decomposition (II) (Theorem \ref{decomp-II},) we know that
\[
\Hom_\HMF(C_f(\tilde{\Gamma}_k), C_f(\tilde{\Gamma}_{k-1})) \cong \Hom_\HMF (C_f(\input{tilde-gamma-k-merged-sli}), C_f(\input{tilde-gamma-k-1-merged-sli})) \{[m+1] q^m\},
\vspace{30pt}
\]
\[ \Hom_\HMF(C_f(\tilde{\Gamma}_{k-1}), C_f(\tilde{\Gamma}_{k}))  \cong  \Hom_\HMF (C_f(\input{tilde-gamma-k-1-merged-sli}), C_f(\input{tilde-gamma-k-merged-sli})) \{[m+1] q^m\}. \vspace{30pt}
\]
By Lemma \ref{colored-crossing-res-HMF}, this means that the lowest non-vanishing total polynomial grading of $\Hom_\HMF(C_f(\tilde{\Gamma}_k), C_f(\tilde{\Gamma}_{k-1}))$ and $\Hom_\HMF(C_f(\tilde{\Gamma}_{k-1}), C_f(\tilde{\Gamma}_{k}))$ is $1$, and the subspaces of these spaces of homogeneous elements of total polynomial grading $1$ is $1$-dimensional. Note that $\tilde{d}_k^+$, $G_{k-1} \circ d_k^+\circ F_k$, $\tilde{d}_{k-1}^-$ and $G_k \circ d_{k-1}^- \circ F_{k-1}$ are all homogeneous morphisms of total polynomial degree $1$. So, to prove the lemma, we only need to check that $\tilde{d}_k^+$, $G_{k-1} \circ d_k^+\circ F_k$, $\tilde{d}_{k-1}^-$ and $G_k \circ d_{k-1}^- \circ F_{k-1}$ are homotopically non-trivial.

Let $\check{d}_k^\pm$ be the differential map associated to a $\pm(m+1,n)$-crossing. From their definitions, we know that the morphisms
\[
\xymatrix{
\input{tilde-gamma-k-merged-sli} \ar@<-1ex>[rr]^{\check{d}_k^+} && \input{tilde-gamma-k-1-merged-sli} \ar@<3ex>[ll]^{\check{d}_{k-1}^-}
}\vspace{30pt}
\]
are homotopically non-trivial. We need to show that 
\[
\xymatrix{
\input{tilde-gamma-k} \ar@<1ex>[rr]^{\tilde{d}_k^+} && \input{tilde-gamma-k-1} \ar@<1ex>[ll]^{\tilde{d}_{k-1}^-}
}
\]
are homotopically non-trivial. To do this, we only need to show that 
\[
\xymatrix{
\input{tilde-gamma-k-extended} \ar@<1ex>[rr]^{\tilde{d}_k^+} && \input{tilde-gamma-k-1-extended} \ar@<1ex>[ll]^{\tilde{d}_{k-1}^-}
}
\]
are homotopically non-trivial.

Consider the diagram
\[
\xymatrix{
\input{tilde-gamma-k-extended} \ar@<1ex>[rr]^{\tilde{d}_k^+} \ar@<1ex>[d]^>>>>>>{\overline{\phi}} && \input{tilde-gamma-k-1-extended} \ar@<1ex>[d]^>>>>>>{\overline{\phi}} \ar@<1ex>[ll]^{\tilde{d}_{k-1}^-} \\
\input{tilde-gamma-k-merged-sli} \ar@<-1ex>[rr]^{\check{d}_k^+} \ar@<1ex>[u]^>>>>{\phi} && \input{tilde-gamma-k-1-merged-sli} \ar@<3ex>[ll]^{\check{d}_{k-1}^-} \ar@<1ex>[u]^>>>>{\phi} 
} \vspace{30pt}
\]
where $\phi$ and $\overline{\phi}$ are morphisms induced by the apparent edge splitting/merging. This diagram commutes in both directs up to homotopy and scaling since the horizontal and vertical morphisms act on different parts of the MOY graphs. Thus, by Lemma \ref{phibar-compose-phi},  we have that 
\begin{eqnarray*}
\check{d}_k^+ & \approx & \check{d}_k^+ \circ \overline{\phi} \circ \mathfrak{m}(r^m) \circ \phi \approx \overline{\phi} \circ \tilde{d}_k^+ \circ \mathfrak{m}(r^m) \circ \phi, \\
\check{d}_{k-1}^- & \approx & \check{d}_{k-1}^- \circ \overline{\phi} \circ \mathfrak{m}(r^m) \circ \phi \approx \overline{\phi} \circ \tilde{d}_{k-1}^- \circ \mathfrak{m}(r^m) \circ \phi.
\end{eqnarray*}
Since $\check{d}_k^+$ and $\check{d}_{k-1}^-$ are homotopically non-trivial, this implies that $\tilde{d}_k^+$ and $\tilde{d}_{k-1}^-$ are homotopically non-trivial.

Let $\varpi_0$ be the functor given in Lemma \ref{MOY-object-of-hmf}. A similar argument shows that $\varpi_0(\tilde{d}_k^+)$ and $\varpi_0(\tilde{d}_{k-1}^-)$ are homotopically non-trivial. By Proposition \ref{basic-changes-varpi-0}, Corollary \ref{differentials-varpi-0} and \cite[Lemma 12.11]{Wu-color}, we know that
\begin{eqnarray*}
\varpi_0(\tilde{d}_k^+) & \approx &  \varpi_0(G_{k-1} \circ d_k^+\circ F_k), \\ 
\varpi_0(\tilde{d}_{k-1}^-) & \approx & \varpi_0(G_k \circ d_{k-1}^- \circ F_{k-1}).
\end{eqnarray*}
So $\varpi_0(G_{k-1} \circ d_k^+\circ F_k)$ and $\varpi_0(G_k \circ d_{k-1}^- \circ F_{k-1})$ are homotopically non-trivial, which implies that $G_{k-1} \circ d_k^+\circ F_k$ and $G_k \circ d_{k-1}^- \circ F_{k-1}$ are homotopically non-trivial.
\end{proof}

\begin{figure}[ht]
$
\xymatrix{
\input{gamma-m-prime} & \input{gamma-m-1-prime} & \input{gamma-m-1-double-prime} \\ 
\input{tilde-gamma-m+1-prime} & \input{tilde-gamma-m+1}  
} 
$
\caption{}\label{gamma-m-1-fig} 

\end{figure}

\subsection{Decomposing $C_f(\Gamma_{m,1})=C_f(\Gamma_m')$} Note that $\Gamma_{m,1}$ coincide with $\Gamma_m'$. Consider the MOY graphs in Figure \ref{gamma-m-1-fig}. By Corollary \ref{contract-expand}, $C_f(\Gamma_{m}'') \simeq C_f(\widetilde{\Gamma}_{m+1})$. By Decomposition (V) (Theorem \ref{decomp-V}), $C_f(\Gamma_m') \simeq C_f(\Gamma_{m-1}'') \oplus C_f(\Gamma_{m}'')$. So
\begin{equation}\label{gamma-m-1-decomp-V}
C_f(\Gamma_{m,1}) \simeq C_f(\Gamma_{m-1}'') \oplus C_f(\widetilde{\Gamma}_{m+1}).
\end{equation}

\begin{lemma}\label{hmf-tilde-gamma-prime-m}
\[
\Hom_\hmf(C_f(\widetilde{\Gamma}_{m+1}),C_f(\Gamma_{m-1}')) \cong \Hom_\hmf(C_f(\Gamma_{m-1}'), C_f(\widetilde{\Gamma}_{m+1})) \cong 0.
\]
\end{lemma}

\begin{proof}
The proof of \cite[Lemma 12.12]{Wu-color} applies here without change.
\end{proof}

\begin{corollary}\label{hmf-tilde-gamma-double-prime-m}
\[
\Hom_\hmf(C_f(\widetilde{\Gamma}_{m+1}),C_f(\Gamma_{m-1}'')) \cong \Hom_\hmf(C_f(\Gamma_{m-1}''), C_f(\widetilde{\Gamma}_{m+1})) \cong 0.
\]
\end{corollary}

\begin{proof}
By Decomposition (V) (Theorem \ref{decomp-V}), $C_f(\Gamma_{m-1}') \simeq C_f(\Gamma_{m-1}'') \oplus C_f(\Gamma_{m-2}'')$. So the corollary follows from Lemma \ref{hmf-tilde-gamma-prime-m}.
\end{proof}

\begin{lemma}\label{hmf-tilde-gamma-m} 
$\Hom_\hmf(C_f(\widetilde{\Gamma}_{m+1}),C_f(\widetilde{\Gamma}_{m+1})) \cong \C$.
\end{lemma}
\begin{proof}
The proof of \cite[Lemma 12.14]{Wu-color} applies here without change.
\end{proof}

\begin{corollary}\label{hmf-tilde-gamma-gamma-m1} 
\[
\Hom_\hmf(C_f(\widetilde{\Gamma}_{m+1}),C_f(\Gamma_{m,1})) \cong \Hom_\hmf(C_f(\Gamma_{m,1}),C_f(\widetilde{\Gamma}_{m+1})) \cong \C.
\]
\end{corollary}

\begin{proof}
This follows easily from \eqref{gamma-m-1-decomp-V}, Corollary \ref{hmf-tilde-gamma-double-prime-m} and Lemma \ref{hmf-tilde-gamma-m}.
\end{proof}

\begin{figure}[ht]
$
\xymatrix{
\input{gamma-m-prime} \ar@<1ex>[rr]^{\tilde{p}} \ar@<1ex>[d]^{\chi^0 \otimes \chi^0} & & \input{tilde-gamma-m+1} \ar@<1ex>[ll]^{\tilde{\jmath}} \ar@<1ex>[d]^{h} \\
\input{tilde-gamma-m+1-bubble} \ar@<1ex>[u]^{\chi^1 \otimes \chi^1} \ar@<1ex>[rr]^{\overline{\phi}} && \input{tilde-gamma-m+1-prime} \ar@<1ex>[ll]^{\phi} \ar@<1ex>[u]^{\overline{h}}
} 
$
\caption{}\label{gamma-m-1-p-j-def} 

\end{figure}

Consider the diagram in Figure \ref{gamma-m-1-p-j-def}, where 
\begin{eqnarray*}
\tilde{p} & := & \overline{h} \circ \overline{\phi} \circ (\chi^0 \otimes \chi^0), \\
\tilde{\jmath} & := & (\chi^1 \otimes \chi^1) \circ \phi \circ h,
\end{eqnarray*}
and morphisms on the right hand side are induced by the apparent basic local changes of the MOY graphs and 

\begin{lemma}\label{gamma-m-1-p-j-tilde}
Up to homotopy and scaling, $C_f(\widetilde{\Gamma}_{m+1}) \xrightarrow{\tilde{\jmath}} C_f(\Gamma_{m,1})$ and $C_f(\Gamma_{m,1}) \xrightarrow{\tilde{p}} C_f(\widetilde{\Gamma}_{m+1})$ are the inclusion and projection of the component $C_f(\widetilde{\Gamma}_{m+1})$ in decomposition \eqref{gamma-m-1-decomp-V}.
\end{lemma}

\begin{proof}
Note that $\tilde{\jmath}$ and $\tilde{p}$ are homogeneous morphisms preserving both gradings. So, by Corollary \ref{hmf-tilde-gamma-gamma-m1}, to prove the lemma, we only need to show that $\tilde{\jmath}$ and $\tilde{p}$ are homotopically non-trivial. Let $\varpi_0$ be the functor given in Lemma \ref{MOY-object-of-hmf}. From \cite[Lemma 12.16]{Wu-color}, we know that $\varpi_0(\tilde{\jmath})$ and $\varpi_0(\tilde{p})$ are the inclusion and projection of the component $C(\widetilde{\Gamma}_{m+1})$ in the decomposition
\[
C(\Gamma_{m,1}) \simeq C(\Gamma_{m-1}'') \oplus C(\widetilde{\Gamma}_{m+1}).
\]
So $\varpi_0(\tilde{\jmath})$ and $\varpi_0(\tilde{p})$ are homotopically non-trivial, which implies that $\tilde{\jmath}$ and $\tilde{p}$ are homotopically non-trivial.
\end{proof}

\begin{figure}[ht]
$
\xymatrix{
\input{tilde-gamma-m+1} \ar@<1ex>[d]^{\tilde{\jmath}} \ar@<1ex>[rr]^{\tilde{d}_{m+1}^+} && \input{tilde-gamma-m} \ar@<1ex>[ll]^{\tilde{d}_{m}^-} \ar@<1ex>[d]^{h^{(m)}} \\
\input{gamma-m-prime} \ar@<1ex>[u]^{\tilde{p}} \ar@<1ex>[rr]^{\chi^1} && \input{gamma-m3} \ar@<1ex>[ll]^{\chi^0} \ar@<1ex>[u]^{\overline{h}^{(m)}}
} 
$
\caption{}\label{gamma-m-1-p-j-differential-fig} 

\end{figure}

\begin{lemma}\label{gamma-m-1-p-j-differential} 
Consider the diagram in Figure \ref{gamma-m-1-p-j-differential-fig}, where $\chi^0$, $\chi^1$, $h^{(m)}$ and $\overline{h}^{(m)}$ are induced by the apparent basic local changes of MOY graphs. Then $\tilde{d}_{m+1}^+ \approx \overline{h}^{(m)} \circ \chi^1 \circ \tilde{\jmath}$ and $\tilde{d}_{m}^- \approx \tilde{p} \circ \chi^0 \circ h^{(m)}$. That is, the diagram in Figure \ref{gamma-m-1-p-j-differential-fig} commutes in both directions up to homotopy and scaling.
\end{lemma}

\begin{proof}
This follows easily from the definitions of $\tilde{d}_{m+1}^+$, $\tilde{d}_{m}^-$, $\tilde{\jmath}$, $\tilde{p}$ and Lemma \ref{chi-commute-chi-chi}.
\end{proof}

\begin{figure}[ht]
$
\xymatrix{
\input{gamma-m-1-double-prime} \ar@<1ex>[rd]^{\jmath''} \ar@<1ex>[rr]^{J_{m-1,m-1}} &  & \input{gamma-m-1-prime} \ar@<1ex>[ld]^{\delta_{m-1}^-} \ar@<1ex>[ll]^{P_{m-1,m-1}} \\
& \input{gamma-m-prime} \ar@<1ex>[lu]^{p''} \ar@<1ex>[ru]^{\delta_m^+} &
} 
$
\caption{}\label{gamma-m-1-p-j-prime-dif-fig} 

\end{figure}

Denote by $\jmath'':C_f(\Gamma_{m-1}'') \rightarrow C_f(\Gamma_{m,1})$ and $p'':C_f(\Gamma_{m,1}) \rightarrow C_f(\Gamma_{m-1}'')$ the inclusion and projection between $C_f(\Gamma_{m-1}'')$ and $C_f(\Gamma_{m,1})$ in \eqref{gamma-m-1-decomp-V}. Consider the diagram in Figure \ref{gamma-m-1-p-j-prime-dif-fig}, where $J_{m-1,m-1}$, $P_{m-1,m-1}$ are defined in Definition \ref{trivial-complex-differential-def}. We have the following lemma.

\begin{lemma}\label{gamma-m-1-p-j-prime-dif}
$\delta_m^+ \circ \jmath'' \approx J_{m-1,m-1}$ and $p'' \circ \delta_{m-1}^- \approx P_{m-1,m-1}$. That is, the diagram in Figure \ref{gamma-m-1-p-j-prime-dif-fig} commutes in both directions up to homotopy and scaling.
\end{lemma}

\begin{proof}
The proof of \cite[Lemma 12.18]{Wu-color} applies here without change.
\end{proof}

\subsection{Proof of Proposition \ref{fork-sliding-invariance-special}} In this section, we prove \eqref{fork-sliding-invariance-special-1-+} and \eqref{fork-sliding-invariance-special-1--}. As mentioned above, the proof of the rest of Proposition \ref{fork-sliding-invariance-special} is similar and left to the reader. The proof below is a step by step adaptation of the proof of \cite[Proposition 12.2]{Wu-color}.

\begin{lemma}\cite[Lemma 4.2 -- Gaussian Elimination]{Bar-fast}\label{gaussian-elimination}
Let $\mathcal{C}$ be an additive category, and
\[
\mathtt{I}=``\cdots\rightarrow C\xrightarrow{\left(%
\begin{array}{c}
  \alpha\\
  \beta \\
\end{array}%
\right)}
\left.%
\begin{array}{c}
  A\\
  \oplus \\
  D
\end{array}%
\right.
\xrightarrow{
\left(%
\begin{array}{cc}
  \phi & \delta\\
  \gamma & \varepsilon \\
\end{array}%
\right)}
\left.%
\begin{array}{c}
  B\\
  \oplus \\
  E
\end{array}%
\right.
\xrightarrow{
\left(%
\begin{array}{cc}
  \mu & \nu\\
\end{array}%
\right)} F \rightarrow \cdots"
\]
an object of $\ch(\mathcal{C})$, that is, a bounded chain complex over $\mathcal{C}$. Assume that $A\xrightarrow{\phi} B$ is an isomorphism in $\mathcal{C}$ with inverse $\phi^{-1}$. Then $\mathtt{I}$ is homotopic to (i.e. isomorphic in $\hch(\mathcal{C})$ to)
\[
\mathtt{II}=
``\cdots\rightarrow C \xrightarrow{\beta} D
\xrightarrow{\varepsilon-\gamma\phi^{-1}\delta} E\xrightarrow{\nu} F \rightarrow \cdots".
\]
In particular, if $\delta$ or $\gamma$ is $0$, then $\mathtt{I}$ is homotopic to 
\[
\mathtt{II}=
``\cdots\rightarrow C \xrightarrow{\beta} D
\xrightarrow{\varepsilon} E\xrightarrow{\nu} F \rightarrow \cdots".
\]
\end{lemma}

\begin{figure}[ht]
$
\xymatrix{
\input{gamma-k-prime} && \input{gamma-k-1-prime}
} 
$
\caption{}\label{hmf-gamma-prime-k-k-1-fig} 

\end{figure}

\begin{lemma}\label{hmf-gamma-prime-k-k-1}
\begin{eqnarray*}
\Hom_\hmf (C_f(\Gamma_k')\{[m-k]q^{k-1-m}\}, C_f(\Gamma_{k-1}')) & \cong & 0, \\
\Hom_\hmf (C_f(\Gamma_{k-1}'), C_f(\Gamma_k')\{[m-k]q^{m+1-k}\}) & \cong & 0.
\end{eqnarray*}
\end{lemma}

\begin{proof}
The proof of \cite[Lemma 12.20]{Wu-color} applies here without change.
\end{proof}

\begin{proof}[Proof of homotopy equivalence \eqref{fork-sliding-invariance-special-1-+}]
We are trying to prove
\[
\hat{C}_f(\setlength{\unitlength}{1pt}
\begin{picture}(44,20)(-22,20)

\put(0,0){\line(0,1){8}}

\put(0,12){\vector(0,1){8}}

\put(0,20){\vector(1,1){20}}

\put(0,20){\vector(-1,1){20}}

\put(-20,10){\vector(1,0){40}}

\put(-13,35){\tiny{$_m$}}

\put(10,35){\tiny{$_1$}}

\put(3,2){\tiny{$_{m+1}$}}

\put(10,13){\tiny{$_n$}}

\end{picture})~\simeq~ \hat{C}_f(\setlength{\unitlength}{1pt}
\begin{picture}(44,20)(-22,20)

\put(0,0){\vector(0,1){20}}

\put(0,20){\line(1,1){8}}

\put(0,20){\line(-1,1){8}}

\put(12,32){\vector(1,1){8}}

\put(-12,32){\vector(-1,1){8}}

\put(-20,30){\vector(1,0){40}}

\put(-13,35){\tiny{$_m$}}

\put(10,35){\tiny{$_1$}}

\put(3,2){\tiny{$_{m+1}$}}

\put(12,26){\tiny{$_n$}}

\end{picture}). \vspace{20pt} 
\]
Recall that the chain complex $\hat{C}_f(\setlength{\unitlength}{1pt}
\begin{picture}(44,20)(-22,20)

\put(0,0){\vector(0,1){20}}

\put(0,20){\line(1,1){8}}

\put(0,20){\line(-1,1){8}}

\put(12,32){\vector(1,1){8}}

\put(-12,32){\vector(-1,1){8}}

\put(-20,30){\vector(1,0){40}}

\put(-13,35){\tiny{$_m$}}

\put(10,35){\tiny{$_1$}}

\put(3,2){\tiny{$_{m+1}$}}

\put(12,26){\tiny{$_n$}}

\end{picture}) \vspace{20pt}$ is 
{\tiny
\[
0 \rightarrow C_f(\Gamma_{m,1}) \xrightarrow{\mathfrak{d}_m^+} \left.%
\begin{array}{c}
C_f(\Gamma_{m,0}) \{q^{-1}\}\\
\oplus \\
C_f(\Gamma_{m-1,1})\{q^{-1}\}
\end{array}%
\right. 
\xrightarrow{\mathfrak{d}_{m-1}^+} \cdots \xrightarrow{\mathfrak{d}_{k+1}^+} \left.%
\begin{array}{c}
C_f(\Gamma_{k+1,0}) \{q^{k-m}\}\\
\oplus \\
C_f(\Gamma_{k,1})\{q^{k-m}\}
\end{array}%
\right. 
\xrightarrow{\mathfrak{d}_{k}^+} \cdots\xrightarrow{\mathfrak{d}_{k_0}^+} C_f(\Gamma_{k_0,0}) \{q^{k_0-1-m}\} \rightarrow 0,
\]
}

\noindent where $k_0 = \max \{m-n,0\}$ as above and
\begin{eqnarray*}
\mathfrak{d}_m^+ & = & \left(%
\begin{array}{c}
\chi^1\\
-d_m^+
\end{array}%
\right),\\
\mathfrak{d}_k^+ & = & \left(%
\begin{array}{cc}
d_{k+1}^+ & \chi^1\\
0 & -d_k^+
\end{array}%
\right) ~\text{ for } k_0<k<m, \\
\mathfrak{d}_{k_0}^+ & = & \left(%
\begin{array}{cc}
d_{k_0 +1}^+, & \chi^1\\
\end{array}%
\right).
\end{eqnarray*}

From \eqref{fork-special-decomp-iv-eq}, we have 
\[
C_f(\Gamma_{k,0}) \simeq C_f(\widetilde{\Gamma}_k) \oplus C_f(\Gamma_k')\{[m-k]\}.
\] 
(Here, we identify both sides by the homotopy equivalence given in \eqref{fork-special-decomp-iv-cor-1} in Corollary \ref{fork-special-decomp-iv-corollary}.) By Corollary \ref{contract-expand} and Decomposition (II) (Theorem \ref{decomp-II}), we have 
\[
C_f(\Gamma_{k,1}) \simeq C_f(\Gamma_k')\{[m+1-k]\} \cong C_f(\Gamma_k')\{q^{m-k}\} \oplus C_f(\Gamma_k')\{[m-k]q^{-1}\}.
\]
Altogether, we have
\[
\left.%
\begin{array}{c}
C_f(\Gamma_{k+1,0}) \{q^{k-m}\}\\
\oplus \\
C_f(\Gamma_{k,1})\{q^{k-m}\}
\end{array}%
\right. \simeq 
\left.%
\begin{array}{c}
C_f(\widetilde{\Gamma}_{k+1}) \{q^{k-m}\}\\
\oplus \\
C_f(\Gamma_{k+1}')\{[m-k-1]q^{k-m}\} \\
\oplus \\
C_f(\Gamma_k') \\
\oplus \\
C_f(\Gamma_k')\{[m-k]q^{k-m-1}\}
\end{array}%
\right. \text{ for } k_0<k<m,
\]
and
\[
C_f(\Gamma_{k_0,0}) \{q^{k_0-1-m}\} \simeq 
\left.%
\begin{array}{c}
C_f(\widetilde{\Gamma}_{k_0}) \{q^{k_0-1-m}\}\\
\oplus \\
C_f(\Gamma_{k_0}')\{[m-k_0]q^{k_0-1-m}\} \\
\end{array}%
\right..
\]
So, $\hat{C}_f(\setlength{\unitlength}{1pt}
\begin{picture}(44,20)(-22,20)

\put(0,0){\vector(0,1){20}}

\put(0,20){\line(1,1){8}}

\put(0,20){\line(-1,1){8}}

\put(12,32){\vector(1,1){8}}

\put(-12,32){\vector(-1,1){8}}

\put(-20,30){\vector(1,0){40}}

\put(-13,35){\tiny{$_m$}}

\put(10,35){\tiny{$_1$}}

\put(3,2){\tiny{$_{m+1}$}}

\put(12,26){\tiny{$_n$}}

\end{picture}) \vspace{20pt}$ is isomorphic to 
{\tiny
\[
_{0 \rightarrow C_f(\Gamma_{m,1}) \xrightarrow{\mathfrak{d}_m^+} \left.%
\begin{array}{c}
C_f(\widetilde{\Gamma}_{m}) \{q^{-1}\}\\
\oplus \\
C_f(\Gamma_{m-1}') \\
\oplus \\
C_f(\Gamma_{m-1}')\{q^{-2}\}
\end{array}%
\right.
\xrightarrow{\mathfrak{d}_{m-1}^+} \cdots \xrightarrow{\mathfrak{d}_{k+1}^+} \left.%
\begin{array}{c}
C_f(\widetilde{\Gamma}_{k+1}) \{q^{k-m}\}\\
\oplus \\
C_f(\Gamma_{k+1}')\{[m-k-1]q^{k-m}\} \\
\oplus \\
C_f(\Gamma_k') \\
\oplus \\
C_f(\Gamma_k')\{[m-k]q^{k-m-1}\}
\end{array}%
\right.
\xrightarrow{\mathfrak{d}_{k}^+} \cdots\xrightarrow{\mathfrak{d}_{k_0}^+} \left.%
\begin{array}{c}
C_f(\widetilde{\Gamma}_{k_0}) \{q^{k_0-1-m}\}\\
\oplus \\
C_f(\Gamma_{k_0}')\{[m-k_0]q^{k_0-1-m}\} \\
\end{array}%
\right. \rightarrow 0.}
\]
}

\noindent In this form, for $k_0<k<m-1$, $\mathfrak{d}_k^+$ is given by a $4\times 4$ matrix $(\mathfrak{d}_{k;i,j}^+)_{4\times4}$. Clearly, 
\[
\mathfrak{d}_{k;i,j}^+=0 \text{ for } (i,j)=(3,1),~(3,2),~(4,1),~(4,2). 
\]
By Lemma \ref{relating-D-0-1-lemma}, 
\[
\mathfrak{d}_{k;1,1}^+ \approx \tilde{d}_{k+1}^+.
\]
By Lemma \ref{relating-C-D-11-lemma},
\[
\mathfrak{d}_{k;3,3}^+ \approx \delta_{k}^+.
\]
Using the homotopy equivalence given in \eqref{fork-special-decomp-iv-cor-1}, the definition of $\alpha_k$ and that $G_k \circ\vec{\alpha}_k\simeq 0$, we get
\begin{eqnarray*}
\mathfrak{d}_{k;1,4}^+ & \simeq & 0,\\
\mathfrak{d}_{k;2,4}^+ & \approx & \id_{C_f(\Gamma_k')\{[m-k]q^{k-m-1}\}}.
\end{eqnarray*}
By Lemma \ref{hmf-gamma-prime-k-k-1}, we have
\[
\mathfrak{d}_{k;3,4}^+ \simeq 0.
\]
Altogether, we have that, for $k_0<k<m-1$,
\[
\mathfrak{d}_k^+ \simeq \left(%
\begin{array}{cccc}
c_k \tilde{d}_{k+1}^+ & \ast & \ast & 0 \\
\ast & \ast & \ast & c_k''\id_{C_f(\Gamma_k')\{[m-k]q^{k-m-1}\}}\\
0 & 0 & c_k' \delta_{k}^+ & 0 \\
0 &  0 & \ast & \ast
\end{array}%
\right),
\]
where $c_k$, $c_k'$ and $c_k''$ are non-zero scalars and $\ast$ means morphisms we have not determined. Similarly,
\begin{eqnarray*}
\mathfrak{d}_{k_0}^+ & \simeq & \left(%
\begin{array}{cccc}
c_{k_0} \tilde{d}_{k_0+1}^+ & \ast & \ast & 0 \\
\ast & \ast & \ast & c_{k_0}''\id_{C_f(\Gamma_{k_0}')\{[m-k+0]q^{k_0-1-m}\}}
\end{array}%
\right), \\
\mathfrak{d}_{m-1}^+ & \simeq & \left(%
\begin{array}{ccc}
c_{m-1} \tilde{d}_{m}^+ & \ast & 0 \\
\ast  & \ast & c_{m-1}''\id_{C_f(\Gamma_{m-1}')\{q^{-2}\}}\\
0  & c_{m-1}' \delta_{m-1}^+ & 0 \\
0  & \ast & \ast
\end{array}%
\right),
\end{eqnarray*}
where $c_{k_0}$, $c_{k_0}''$, $c_{m-1}$, $c_{m-1}'$ and $c_{m-1}''$ are non-zero scalars.

Now apply Gaussian Elimination (Lemma \ref{gaussian-elimination}) to $c_k''\id_{C_f(\Gamma_k')\{[m-k]q^{k-m-1}\}}$ in $\mathfrak{d}_{k}^+$ for $k=k_0,k_0+1,\dots,m-1$ in that order. We get that $\hat{C}_f(\setlength{\unitlength}{1pt}
\begin{picture}(44,20)(-22,20)

\put(0,0){\vector(0,1){20}}

\put(0,20){\line(1,1){8}}

\put(0,20){\line(-1,1){8}}

\put(12,32){\vector(1,1){8}}

\put(-12,32){\vector(-1,1){8}}

\put(-20,30){\vector(1,0){40}}

\put(-13,35){\tiny{$_m$}}

\put(10,35){\tiny{$_1$}}

\put(3,2){\tiny{$_{m+1}$}}

\put(12,26){\tiny{$_n$}}

\end{picture}) \vspace{20pt}$ is homotopic to 
{\tiny
\[
0 \rightarrow C_f(\Gamma_{m,1}) \xrightarrow{\hat{\mathfrak{d}}_m^+} \left.%
\begin{array}{c}
C_f(\widetilde{\Gamma}_{m}) \{q^{-1}\}\\
\oplus \\
C_f(\Gamma_{m-1}') \\
\end{array}%
\right. 
\xrightarrow{\hat{\mathfrak{d}}_{m-1}^+} \cdots \xrightarrow{\hat{\mathfrak{d}}_{k+1}^+} \left.%
\begin{array}{c}
C_f(\widetilde{\Gamma}_{k+1}) \{q^{k-m}\}\\
\oplus \\
C_f(\Gamma_k') \\
\end{array}%
\right. 
\xrightarrow{\hat{\mathfrak{d}}_{k}^+} \cdots\xrightarrow{\hat{\mathfrak{d}}_{k_0}^+} C_f(\widetilde{\Gamma}_{k_0}) \{q^{k_0-1-m}\} \rightarrow 0,
\]
}

\noindent where
\begin{eqnarray}
\label{hat-d-k}\hat{\mathfrak{d}}_k^+ & \simeq & \left(%
\begin{array}{cc}
c_k \tilde{d}_{k+1}^+ & \ast\\
0 & c_k' \delta_{k}^+
\end{array}%
\right) ~\text{ for } k_0<k<m, \\
\label{hat-d-k0}\hat{\mathfrak{d}}_{k_0}^+ & \simeq & \left(%
\begin{array}{cc}
c_{k_0} \tilde{d}_{k_0+1}^+, & \ast
\end{array}%
\right).
\end{eqnarray}
Next we determine $\hat{\mathfrak{d}}_m^+$. From \eqref{gamma-m-1-decomp-V}, we have
\[
C_f(\Gamma_{m,1}) \simeq \left.%
\begin{array}{c}
C_f(\widetilde{\Gamma}_{m+1})\\
\oplus \\
C_f(\Gamma_{m-1}'')\\
\end{array}%
\right. .
\]
Under this decomposition, $\hat{\mathfrak{d}}_m^+$ is represented by a $2\times2$ matrix. By Lemmas \ref{hmf-tilde-gamma-prime-m}, \ref{gamma-m-1-p-j-differential} and \ref{gamma-m-1-p-j-prime-dif}, we know that
\begin{equation}\label{hat-d-m}
\hat{\mathfrak{d}}_m^+ \simeq \left(%
\begin{array}{cc}
c_{m} \tilde{d}_{m+1}^+ & \ast \\
0 & c_{m}' J_{m-1,m-1}
\end{array}%
\right),
\end{equation}
where $c_m$ and $c_m'$ are non-zero scalars. So $\hat{C}_f(\setlength{\unitlength}{1pt}
\begin{picture}(44,20)(-22,20)

\put(0,0){\vector(0,1){20}}

\put(0,20){\line(1,1){8}}

\put(0,20){\line(-1,1){8}}

\put(12,32){\vector(1,1){8}}

\put(-12,32){\vector(-1,1){8}}

\put(-20,30){\vector(1,0){40}}

\put(-13,35){\tiny{$_m$}}

\put(10,35){\tiny{$_1$}}

\put(3,2){\tiny{$_{m+1}$}}

\put(12,26){\tiny{$_n$}}

\end{picture}) \vspace{20pt}$ is homotopic to 
{\tiny
\[
0 \rightarrow \left.%
\begin{array}{c}
C_f(\widetilde{\Gamma}_{m+1})\\
\oplus \\
C_f(\Gamma_{m-1}'')\\
\end{array}%
\right. \xrightarrow{\hat{\mathfrak{d}}_m^+} \left.%
\begin{array}{c}
C_f(\widetilde{\Gamma}_{m}) \{q^{-1}\}\\
\oplus \\
C_f(\Gamma_{m-1}') \\
\end{array}%
\right. 
\xrightarrow{\hat{\mathfrak{d}}_{m-1}^+} \cdots \xrightarrow{\hat{\mathfrak{d}}_{k+1}^+} \left.%
\begin{array}{c}
C_f(\widetilde{\Gamma}_{k+1}) \{q^{k-m}\}\\
\oplus \\
C_f(\Gamma_k') \\
\end{array}%
\right. 
\xrightarrow{\hat{\mathfrak{d}}_{k}^+} \cdots\xrightarrow{\hat{\mathfrak{d}}_{k_0}^+} C_f(\widetilde{\Gamma}_{k_0}) \{q^{k_0-1-m}\} \rightarrow 0,
\]
}

\noindent where $\hat{\mathfrak{d}}_m^+,\dots,\hat{\mathfrak{d}}_{k_0}^+$ are given in \eqref{hat-d-k}, \eqref{hat-d-k0} and \eqref{hat-d-m}.

Recall that, by Lemma \ref{decomp-V-special-2}, 
\[
C_f(\Gamma_k') \simeq \begin{cases}
C_f(\Gamma_k'') \oplus C_f(\Gamma_{k-1}'') & \text{if } k_0+1 \leq l \leq m-1 ,\\
C_f(\Gamma_k'') & \text{if } k=k_0.
\end{cases}
\]
By Proposition \ref{trivial-complex-prop}, under the decomposition
\[
\left.%
\begin{array}{c}
C_f(\widetilde{\Gamma}_{k+1}) \{q^{k-m}\}\\
\oplus \\
C_f(\Gamma_k') \\
\end{array}%
\right. \simeq 
\left.%
\begin{array}{c}
C_f(\widetilde{\Gamma}_{k+1}) \{q^{k-m}\}\\
\oplus \\
C_f(\Gamma_k'') \\
\oplus \\
C_f(\Gamma_{k-1}'')
\end{array}%
\right. ,
\]
we have
\begin{eqnarray}
\label{hat-d-m-1}\hat{\mathfrak{d}}_m^+ & \simeq &
\left(%
\begin{array}{cc}
c_{m} \tilde{d}_{m+1}^+ & \ast \\
0 & c_{m}''' \id_{C_f(\Gamma_{m-1}'')} \\
0 & 0
\end{array}%
\right), \\
\label{hat-d-k-1}\hat{\mathfrak{d}}_k^+ & \simeq & \left(%
\begin{array}{ccc}
c_k \tilde{d}_{k+1}^+ & \ast& \ast \\
0 & 0& c_k''' \id_{C_f(\Gamma_{k-1}'')} \\
0 & 0& 0
\end{array}%
\right) ~\text{ for } k_0+1<k<m,
\end{eqnarray}
where $c_k'''$ is a non-zero scalar for $k_0+1<k\leq m$. Since $C_f(\Gamma_{k_0}') \simeq C_f(\Gamma_{k_0}'')$, we have 
\[
\left.%
\begin{array}{c}
C_f(\widetilde{\Gamma}_{k_0+1}) \{q^{k_0-m}\}\\
\oplus \\
C_f(\Gamma_{k_0}') \\
\end{array}%
\right. \simeq 
\left.%
\begin{array}{c}
C_f(\widetilde{\Gamma}_{k_0+1}) \{q^{k_0-m}\}\\
\oplus \\
C_f(\Gamma_{k_0}'') \\
\end{array}%
\right. 
\]
and 
\begin{eqnarray}
\label{hat-d-k0+1-1} \hat{\mathfrak{d}}_{k_0+1}^+ & \simeq & \left(%
\begin{array}{ccc}
c_{k_0+1} \tilde{d}_{k_0+2}^+ & \ast& \ast \\
0 & 0 & c_{k_0+1}''' \id_{C_f(\Gamma_{k_0}'')}
\end{array}%
\right), \\
\label{hat-d-k0-1} \hat{\mathfrak{d}}_{k_0}^+ & \simeq & \left(%
\begin{array}{cc}
c_{k_0} \tilde{d}_{k_0+1}^+, & \ast
\end{array}%
\right),
\end{eqnarray}
where $c_{k_0+1}'''$ is a non-zero scalar.

Putting these together, we know that  $\hat{C}_f(\setlength{\unitlength}{1pt}
\begin{picture}(44,20)(-22,20)

\put(0,0){\vector(0,1){20}}

\put(0,20){\line(1,1){8}}

\put(0,20){\line(-1,1){8}}

\put(12,32){\vector(1,1){8}}

\put(-12,32){\vector(-1,1){8}}

\put(-20,30){\vector(1,0){40}}

\put(-13,35){\tiny{$_m$}}

\put(10,35){\tiny{$_1$}}

\put(3,2){\tiny{$_{m+1}$}}

\put(12,26){\tiny{$_n$}}

\end{picture}) \vspace{20pt}$ is homotopic to 
{\tiny
\[_{
0 \rightarrow \left.%
\begin{array}{c}
C_f(\widetilde{\Gamma}_{m+1})\\
\oplus \\
C_f(\Gamma_{m-1}'')\\
\end{array}%
\right. \xrightarrow{\hat{\mathfrak{d}}_m^+} \left.%
\begin{array}{c}
C_f(\widetilde{\Gamma}_{m}) \{q^{-1}\}\\
\oplus \\
C_f(\Gamma_{m-1}'') \\
\oplus \\
C_f(\Gamma_{m-2}'')
\end{array}%
\right. 
\xrightarrow{\hat{\mathfrak{d}}_{m-1}^+}  \cdots \xrightarrow{\hat{\mathfrak{d}}_{k_0+1}^+} \left.%
\begin{array}{c}
C_f(\widetilde{\Gamma}_{k_0+1}) \{q^{k_0-m}\}\\
\oplus \\
C_f(\Gamma_{k_0}'') \\
\end{array}%
\right. \xrightarrow{\hat{\mathfrak{d}}_{k_0}^+} C_f(\widetilde{\Gamma}_{k_0}) \{q^{k_0-1-m}\} \rightarrow 0,}
\]
}

\noindent where $\hat{\mathfrak{d}}_m^+,\dots,\hat{\mathfrak{d}}_{k_0}^+$ are given in \eqref{hat-d-m-1},\eqref{hat-d-k-1}, \eqref{hat-d-k0+1-1} and \eqref{hat-d-k0-1}.

Applying Gaussian Elimination (Lemma \ref{gaussian-elimination}) to $c_k''' \id_{C_f(\Gamma_{k-1}'')}$ in $\hat{\mathfrak{d}}_k^+$ for $k=m,m-1,\dots,k_0+1$, we get that $\hat{C}_f(\setlength{\unitlength}{1pt}
\begin{picture}(44,20)(-22,20)

\put(0,0){\vector(0,1){20}}

\put(0,20){\line(1,1){8}}

\put(0,20){\line(-1,1){8}}

\put(12,32){\vector(1,1){8}}

\put(-12,32){\vector(-1,1){8}}

\put(-20,30){\vector(1,0){40}}

\put(-13,35){\tiny{$_m$}}

\put(10,35){\tiny{$_1$}}

\put(3,2){\tiny{$_{m+1}$}}

\put(12,26){\tiny{$_n$}}

\end{picture}) \vspace{20pt}$ is homotopic to
{\tiny
\[
0 \rightarrow 
C_f(\widetilde{\Gamma}_{m+1}) \xrightarrow{\check{\mathfrak{d}}_m^+} 
C_f(\widetilde{\Gamma}_{m}) \{q^{-1}\}
\xrightarrow{\check{\mathfrak{d}}_{m-1}^+} \cdots \xrightarrow{\check{\mathfrak{d}}_{k+1}^+} 
C_f(\widetilde{\Gamma}_{k+1}) \{q^{k-m}\}
\xrightarrow{\check{\mathfrak{d}}_{k}^+} \cdots\xrightarrow{\check{\mathfrak{d}}_{k_0}^+} C_f(\widetilde{\Gamma}_{k_0}) \{q^{k_0-1-m}\} \rightarrow 0,
\]
}

\noindent where $\check{\mathfrak{d}}_{k}^+ \simeq c_k \tilde{d}_{k+1}^+$ for $k=m,m-1,\dots,k_0$. Recall that $c_k\neq 0$ for $k=m,\dots,k_0$. So this last chain complex is isomorphic to $\hat{C}_f(\setlength{\unitlength}{1pt}
\begin{picture}(44,20)(-22,20)

\put(0,0){\line(0,1){8}}

\put(0,12){\vector(0,1){8}}

\put(0,20){\vector(1,1){20}}

\put(0,20){\vector(-1,1){20}}

\put(-20,10){\vector(1,0){40}}

\put(-13,35){\tiny{$_m$}}

\put(10,35){\tiny{$_1$}}

\put(3,2){\tiny{$_{m+1}$}}

\put(10,13){\tiny{$_n$}}

\end{picture}) \vspace{20pt}$ in $\ch(\hmf)$. Therefore, 
\[
\hat{C}_f(\setlength{\unitlength}{1pt}
\begin{picture}(44,20)(-22,20)

\put(0,0){\line(0,1){8}}

\put(0,12){\vector(0,1){8}}

\put(0,20){\vector(1,1){20}}

\put(0,20){\vector(-1,1){20}}

\put(-20,10){\vector(1,0){40}}

\put(-13,35){\tiny{$_m$}}

\put(10,35){\tiny{$_1$}}

\put(3,2){\tiny{$_{m+1}$}}

\put(10,13){\tiny{$_n$}}

\end{picture})~\simeq~ \hat{C}_f(\setlength{\unitlength}{1pt}
\begin{picture}(44,20)(-22,20)

\put(0,0){\vector(0,1){20}}

\put(0,20){\line(1,1){8}}

\put(0,20){\line(-1,1){8}}

\put(12,32){\vector(1,1){8}}

\put(-12,32){\vector(-1,1){8}}

\put(-20,30){\vector(1,0){40}}

\put(-13,35){\tiny{$_m$}}

\put(10,35){\tiny{$_1$}}

\put(3,2){\tiny{$_{m+1}$}}

\put(12,26){\tiny{$_n$}}

\end{picture}). \vspace{20pt} 
\]
\end{proof}

\begin{proof}[Proof of homotopy equivalence \eqref{fork-sliding-invariance-special-1--}]
We are trying to prove that
\[
\hat{C}_f(\setlength{\unitlength}{1pt}
\begin{picture}(44,20)(-22,20)
\put(0,0){\vector(0,1){20}}

\put(0,20){\vector(1,1){20}}

\put(0,20){\vector(-1,1){20}}

\put(-20,10){\line(1,0){18}}

\put(2,10){\vector(1,0){18}}

\put(-13,35){\tiny{$_m$}}

\put(10,35){\tiny{$_1$}}

\put(3,2){\tiny{$_{m+1}$}}

\put(10,13){\tiny{$_n$}}

\end{picture})~\simeq~ \hat{C}_f(\setlength{\unitlength}{1pt}
\begin{picture}(44,20)(-22,20)
\put(0,0){\vector(0,1){20}}

\put(0,20){\vector(1,1){20}}

\put(0,20){\vector(-1,1){20}}

\put(-20,30){\line(1,0){8}}

\put(-8,30){\line(1,0){16}}

\put(12,30){\vector(1,0){8}}

\put(-13,35){\tiny{$_m$}}

\put(10,35){\tiny{$_1$}}

\put(3,2){\tiny{$_{m+1}$}}

\put(12,26){\tiny{$_n$}}

\end{picture}). \vspace{20pt} 
\]

Recall that the chain complex $\hat{C}_f(\setlength{\unitlength}{1pt}
\begin{picture}(44,20)(-22,20)
\put(0,0){\vector(0,1){20}}

\put(0,20){\vector(1,1){20}}

\put(0,20){\vector(-1,1){20}}

\put(-20,30){\line(1,0){8}}

\put(-8,30){\line(1,0){16}}

\put(12,30){\vector(1,0){8}}

\put(-13,35){\tiny{$_m$}}

\put(10,35){\tiny{$_1$}}

\put(3,2){\tiny{$_{m+1}$}}

\put(12,26){\tiny{$_n$}}

\end{picture}) \vspace{20pt}$ is
{\tiny
\[
0 \rightarrow C_f(\Gamma_{k_0,0}) \{q^{m+1-k_0}\} \xrightarrow{\mathfrak{d}_{k_0}^-} \cdots \xrightarrow{\mathfrak{d}_{k-1}^-} \left.%
\begin{array}{c}
C_f(\Gamma_{k,0}) \{q^{m+1-k}\}\\
\oplus \\
C_f(\Gamma_{k-1,1})\{q^{m+1-k}\}
\end{array}%
\right. 
\xrightarrow{\mathfrak{d}_{k}^-} \cdots  \xrightarrow{\mathfrak{d}_{m-1}^-} \left.%
\begin{array}{c}
C_f(\Gamma_{m,0}) \{q\}\\
\oplus \\
C_f(\Gamma_{m-1,1})\{q\}
\end{array}%
\right. 
\xrightarrow{\mathfrak{d}_m^-} C_f(\Gamma_{m,1}) \rightarrow 0,
\]
}
where $k_0 = \max \{m-n,0\}$ as above and
\begin{eqnarray*}
\mathfrak{d}_{k_0}^- & = & \left(%
\begin{array}{c}
d_{k_0}^-\\ 
\chi^0
\end{array}%
\right),\\
\mathfrak{d}_k^- & = & \left(%
\begin{array}{cc}
d_{k}^- & 0\\
\chi^0 & -d_{k-1}^-
\end{array}%
\right) ~\text{ for } k_0<k<m, \\
\mathfrak{d}_{m}^- & = & \left(%
\begin{array}{cc}
\chi^0 & -d_{m-1}^-\\
\end{array}%
\right).
\end{eqnarray*}

From \eqref{fork-special-decomp-iv-eq}, we have 
\[
C_f(\Gamma_{k,0}) \simeq C_f(\widetilde{\Gamma}_k) \oplus C_f(\Gamma_k')\{[m-k]\}.
\] 
(Here, we identify both sides by the homotopy equivalence given in \eqref{fork-special-decomp-iv-cor-2} in Corollary \ref{fork-special-decomp-iv-corollary}.) By Corollary \ref{contract-expand} and Decomposition (II) (Theorem \ref{decomp-II}), we have 
\[
C_f(\Gamma_{k,1}) \simeq C_f(\Gamma_k')\{[m+1-k]\} \cong C_f(\Gamma_k')\{q^{k-m}\} \oplus C_f(\Gamma_k')\{[m-k]\cdot q\}.
\]
Therefore,
\[
\left.%
\begin{array}{c}
C_f(\Gamma_{k,0}) \{q^{m+1-k}\}\\
\oplus \\
C_f(\Gamma_{k-1,1})\{q^{m+1-k}\}
\end{array}%
\right. \simeq 
\left.%
\begin{array}{c}
C_f(\widetilde{\Gamma}_{k}) \{q^{m+1-k}\}\\
\oplus \\
C_f(\Gamma_{k}')\{[m-k]q^{m+1-k}\} \\
\oplus \\
C_f(\Gamma_{k-1}') \\
\oplus \\
C_f(\Gamma_{k-1}')\{[m+1-k]q^{m+2-k}\}
\end{array}%
\right. \text{ for } k_0<k<m,
\]
and 
\[
C_f(\Gamma_{k_0,0}) \{q^{m+1-k_0}\} \simeq 
\left.%
\begin{array}{c}
C_f(\widetilde{\Gamma}_{k_0}) \{q^{m+1-k_0}\}\\
\oplus \\
C_f(\Gamma_{k_0}')\{[m-k_0]q^{m+1-k_0}\} \\
\end{array}%
\right..
\]
So, $\hat{C}_f(\setlength{\unitlength}{1pt}
\begin{picture}(44,20)(-22,20)
\put(0,0){\vector(0,1){20}}

\put(0,20){\vector(1,1){20}}

\put(0,20){\vector(-1,1){20}}

\put(-20,30){\line(1,0){8}}

\put(-8,30){\line(1,0){16}}

\put(12,30){\vector(1,0){8}}

\put(-13,35){\tiny{$_m$}}

\put(10,35){\tiny{$_1$}}

\put(3,2){\tiny{$_{m+1}$}}

\put(12,26){\tiny{$_n$}}

\end{picture}) \vspace{20pt}$ is isomorphic to 
{\tiny
\[
_{0 \rightarrow \left.%
\begin{array}{c}
C_f(\widetilde{\Gamma}_{k_0}) \{q^{m+1-k_0}\}\\
\oplus \\
C_f(\Gamma_{k_0}')\{[m-k_0]q^{m+1-k_0}\} \\
\end{array}%
\right. \xrightarrow{\mathfrak{d}_{k_0}^-} \cdots \xrightarrow{\mathfrak{d}_{k-1}^-} \left.%
\begin{array}{c}
C_f(\widetilde{\Gamma}_{k}) \{q^{m+1-k}\}\\
\oplus \\
C_f(\Gamma_{k}')\{[m-k]q^{m+1-k}\} \\
\oplus \\
C_f(\Gamma_{k-1}') \\
\oplus \\
C_f(\Gamma_{k-1}')\{[m+1-k]q^{m+2-k}\}
\end{array}%
\right.
\xrightarrow{\mathfrak{d}_{k}^-} \cdots  \xrightarrow{\mathfrak{d}_{m-1}^-} \left.%
\begin{array}{c}
C_f(\widetilde{\Gamma}_{m}) \{q\}\\
\oplus \\
C_f(\Gamma_{m-1}') \\
\oplus \\
C_f(\Gamma_{m-1}')\{q^{2}\}
\end{array}%
\right.
\xrightarrow{\mathfrak{d}_m^-} C_f(\Gamma_{m,1}) \rightarrow 0.}
\]
}

\noindent In this form, for $k_0<k<m-1$, $\mathfrak{d}_k^-$ is given by a $4\times 4$ matrix $(\mathfrak{d}_{k;i,j}^-)_{4\times4}$. Clearly, 
\[
\mathfrak{d}_{k;i,j}^-=0 \text{ for } (i,j)=(1,3),~(1,4),~(2,3),~(2,4). 
\]
By Lemma \ref{relating-D-0-1-lemma}, 
\[
\mathfrak{d}_{k;1,1}^- \approx \tilde{d}_{k}^-.
\]
By Lemma \ref{relating-C-D-11-lemma},
\[
\mathfrak{d}_{k;3,3}^- \approx \delta_{k-1}^-.
\]
Using the homotopy equivalence given in \eqref{fork-special-decomp-iv-cor-2}, the definition of $\beta_k$ and that $\vec{\beta}_k \circ F_k\simeq 0$, we know that
\begin{eqnarray*}
\mathfrak{d}_{k;4,1}^- & \simeq & 0,\\
\mathfrak{d}_{k;4,2}^- & \approx & \id_{C_f(\Gamma_{k}')\{[m-k]q^{m+1-k}\}}.
\end{eqnarray*}
By Lemma \ref{hmf-gamma-prime-k-k-1}, we have
\[
\mathfrak{d}_{k;4,3}^- \simeq 0.
\]
Altogether, we have that, for $k_0<k<m-1$,
\[
\mathfrak{d}_k^- \simeq \left(%
\begin{array}{cccc}
c_k \tilde{d}_{k}^- & \ast & 0 & 0 \\
\ast & \ast & 0 & 0\\
\ast & \ast & c_k'\delta_{k-1}^- & \ast \\
0 &  c_k''\id_{C_f(\Gamma_{k}')\{[m-k]q^{m+1-k}\}} & 0 & \ast
\end{array}%
\right),
\]
where $c_k$, $c_k'$ and $c_k''$ are non-zero scalars and $\ast$ means morphisms we have not determined. Similarly,
\begin{eqnarray*}
\mathfrak{d}_{k_0}^- & \simeq & \left(%
\begin{array}{cc}
c_{k_0} \tilde{d}_{k_0}^- & \ast  \\
\ast & \ast \\
\ast & \ast  \\
0 &  c_{k_0}''\id_{C_f(\Gamma_{k_0}')\{[m-k_0]q^{m+1-k_0}\}} 
\end{array}%
\right), \\
\mathfrak{d}_{m-1}^- & \simeq & \left(%
\begin{array}{cccc}
c_{m-1} \tilde{d}_{m-1}^- & \ast & 0 & 0 \\
\ast & \ast & c_{m-1}'\delta_{m-2}^- & \ast \\
0 &  c_{m-1}''\id_{C_f(\Gamma_{m-1}')\{q^{2}\}} & 0 & \ast
\end{array}%
\right),
\end{eqnarray*}
where $c_{k_0}$, $c_{k_0}''$, $c_{m-1}$, $c_{m-1}'$ and $c_{m-1}''$ are non-zero scalars.

Now apply Gaussian Elimination (Lemma \ref{gaussian-elimination}) to $c_{k}''\id_{C_f(\Gamma_{k}')\{[m-k]q^{m+1-k}\}}$ in $\mathfrak{d}_{k}^-$ for $k=k_0,k_0+1,\dots,m-1$ in that order. We get that $\hat{C}_f(\setlength{\unitlength}{1pt}
\begin{picture}(44,20)(-22,20)
\put(0,0){\vector(0,1){20}}

\put(0,20){\vector(1,1){20}}

\put(0,20){\vector(-1,1){20}}

\put(-20,30){\line(1,0){8}}

\put(-8,30){\line(1,0){16}}

\put(12,30){\vector(1,0){8}}

\put(-13,35){\tiny{$_m$}}

\put(10,35){\tiny{$_1$}}

\put(3,2){\tiny{$_{m+1}$}}

\put(12,26){\tiny{$_n$}}

\end{picture}) \vspace{20pt}$ is homotopic to
{\tiny
\[
0 \rightarrow 
C_f(\widetilde{\Gamma}_{k_0}) \{q^{m+1-k_0}\} \xrightarrow{\hat{\mathfrak{d}}_{k_0}^-} \cdots \xrightarrow{\hat{\mathfrak{d}}_{k-1}^-} \left.%
\begin{array}{c}
C_f(\widetilde{\Gamma}_{k}) \{q^{m+1-k}\}\\
\oplus \\
C_f(\Gamma_{k-1}') \\
\end{array}%
\right.
\xrightarrow{\hat{\mathfrak{d}}_{k}^-} \cdots  \xrightarrow{\hat{\mathfrak{d}}_{m-1}^-} \left.%
\begin{array}{c}
C_f(\widetilde{\Gamma}_{m}) \{q\}\\
\oplus \\
C_f(\Gamma_{m-1}') \\
\end{array}%
\right.
\xrightarrow{\hat{\mathfrak{d}}_m^-} C_f(\Gamma_{m,1}) \rightarrow 0,
\]
}

\noindent where
\begin{eqnarray}
\label{hat-d-k-}\hat{\mathfrak{d}}_k^- & \simeq & \left(%
\begin{array}{cc}
c_k \tilde{d}_{k}^-  & 0 \\
\ast & c_k'\delta_{k-1}^- 
\end{array}%
\right) ~\text{ for } k_0<k<m, \\
\label{hat-d-k0-}\hat{\mathfrak{d}}_{k_0}^- & \simeq & \left(%
\begin{array}{c}
c_{k_0} \tilde{d}_{k_0}^-  \\
\ast  \\
\ast  
\end{array}%
\right).
\end{eqnarray}
Next we determine $\hat{\mathfrak{d}}_m^-$. By Decomposition (V) (more specifically, \eqref{gamma-m-1-decomp-V}), we have
\[
C_f(\Gamma_{m,1}) \simeq \left.%
\begin{array}{c}
C_f(\widetilde{\Gamma}_{m+1})\\
\oplus \\
C_f(\Gamma_{m-1}'')
\end{array}%
\right. .
\]
Under this decomposition, $\hat{\mathfrak{d}}_m^-$ is represented by a $2\times2$ matrix. By Lemmas \ref{hmf-tilde-gamma-prime-m}, \ref{gamma-m-1-p-j-differential} and \ref{gamma-m-1-p-j-prime-dif}, we know that
\begin{equation}\label{hat-d-m-}
\hat{\mathfrak{d}}_m^- \simeq \left(%
\begin{array}{cc}
c_{m} \tilde{d}_{m}^- & 0 \\
\ast & c_{m}' P_{m-1,m-1}
\end{array}%
\right),
\end{equation}
where $c_m$ and $c_m'$ are non-zero scalars. So $\hat{C}_f(\setlength{\unitlength}{1pt}
\begin{picture}(44,20)(-22,20)
\put(0,0){\vector(0,1){20}}

\put(0,20){\vector(1,1){20}}

\put(0,20){\vector(-1,1){20}}

\put(-20,30){\line(1,0){8}}

\put(-8,30){\line(1,0){16}}

\put(12,30){\vector(1,0){8}}

\put(-13,35){\tiny{$_m$}}

\put(10,35){\tiny{$_1$}}

\put(3,2){\tiny{$_{m+1}$}}

\put(12,26){\tiny{$_n$}}

\end{picture}) \vspace{20pt}$ is homotopic to 
{\tiny
\[
0 \rightarrow 
C(\widetilde{\Gamma}_{k_0}) \{q^{m+1-k_0}\} \xrightarrow{\hat{\mathfrak{d}}_{k_0}^-} \cdots \xrightarrow{\hat{\mathfrak{d}}_{k-1}^-} \left.%
\begin{array}{c}
C(\widetilde{\Gamma}_{k}) \{q^{m+1-k}\}\\
\oplus \\
C(\Gamma_{k-1}') \\
\end{array}%
\right.
\xrightarrow{\hat{\mathfrak{d}}_{k}^-} \cdots  \xrightarrow{\hat{\mathfrak{d}}_{m-1}^-} \left.%
\begin{array}{c}
C(\widetilde{\Gamma}_{m}) \{q\}\\
\oplus \\
C(\Gamma_{m-1}') \\
\end{array}%
\right.
\xrightarrow{\hat{\mathfrak{d}}_m^-} \left.%
\begin{array}{c}
C(\widetilde{\Gamma}_{m+1})\\
\oplus \\
C(\Gamma_{m-1}'')
\end{array}%
\right. \rightarrow 0,
\]
}
where $\hat{\mathfrak{d}}_m^-,\dots,\hat{\mathfrak{d}}_{k_0}^-$ are given in \eqref{hat-d-k-}, \eqref{hat-d-k0-} and \eqref{hat-d-m-}.

Recall that, by  Lemma \ref{decomp-V-special-2}, 
\[
C_f(\Gamma_k') \simeq \begin{cases}
C_f(\Gamma_k'') \oplus C_f(\Gamma_{k-1}'') & \text{if } k_0+1 \leq l \leq m-1 ,\\
C_f(\Gamma_k'') & \text{if } k=k_0.
\end{cases}
\]
By Proposition \ref{trivial-complex-prop}, under the decomposition
\[
\left.%
\begin{array}{c}
C_f(\widetilde{\Gamma}_{k}) \{q^{m+1-k}\}\\
\oplus \\
C_f(\Gamma_{k-1}') \\
\end{array}%
\right. \simeq 
\left.%
\begin{array}{c}
C_f(\widetilde{\Gamma}_{k}) \{q^{m+1-k}\}\\
\oplus \\
C_f(\Gamma_{k-1}'') \\
\oplus \\
C_f(\Gamma_{k-2}'')
\end{array}%
\right. ,
\]
we have
\begin{eqnarray}
\label{hat-d-m-1-}\hat{\mathfrak{d}}_m^- & \simeq & \left(%
\begin{array}{ccc}
c_{m} \tilde{d}_{m}^- & 0 & 0\\
\ast & c_{m}''' \id_{C_f(\Gamma_{m-1}'')} & 0
\end{array}%
\right), \\
\label{hat-d-k-1-}\hat{\mathfrak{d}}_k^- & \simeq & \left(%
\begin{array}{ccc}
c_k \tilde{d}_{k}^-  & 0 & 0\\
\ast & 0 & 0 \\
\ast & c_k'''\id_{C_f(\Gamma_{k-1}'')} & 0
\end{array}%
\right) ~\text{ for } k_0+1<k<m,
\end{eqnarray}
where $c_k'''$ is a non-zero scalar for $k_0+1<k\leq m$. Since $C_f(\Gamma_{k_0}') \simeq C_f(\Gamma_{k_0}'')$, we have 
\[
\left.%
\begin{array}{c}
C_f(\widetilde{\Gamma}_{k_0+1}) \{q^{m-k_0}\}\\
\oplus \\
C_f(\Gamma_{k_0}') \\
\end{array}%
\right. \simeq 
\left.%
\begin{array}{c}
C_f(\widetilde{\Gamma}_{k_0+1}) \{q^{m-k_0}\}\\
\oplus \\
C_f(\Gamma_{k_0}'') \\
\end{array}%
\right. 
\]
and 
\begin{eqnarray}
\label{hat-d-k0+1-1-} \hat{\mathfrak{d}}_{k_0+1}^- & \simeq & \left(%
\begin{array}{cc}
c_{k_0+1} \tilde{d}_{k_0+1}^-  & 0 \\
\ast & 0 \\
\ast & c_{k_0+1}'''\id_{C_f(\Gamma_{k_0}'')}
\end{array}%
\right), \\
\label{hat-d-k0-1-} \hat{\mathfrak{d}}_{k_0}^- & \simeq & \left(%
\begin{array}{c}
c_{k_0} \tilde{d}_{k_0+1}^+ \\ 
\ast
\end{array}%
\right),
\end{eqnarray}
where $c_{k_0+1}'''$ is a non-zero scalar. Putting these together, we know that $\hat{C}_f(\setlength{\unitlength}{1pt}
\begin{picture}(44,20)(-22,20)
\put(0,0){\vector(0,1){20}}

\put(0,20){\vector(1,1){20}}

\put(0,20){\vector(-1,1){20}}

\put(-20,30){\line(1,0){8}}

\put(-8,30){\line(1,0){16}}

\put(12,30){\vector(1,0){8}}

\put(-13,35){\tiny{$_m$}}

\put(10,35){\tiny{$_1$}}

\put(3,2){\tiny{$_{m+1}$}}

\put(12,26){\tiny{$_n$}}

\end{picture}) \vspace{20pt}$ is homotopic to 
{\tiny
\[
_{0 \rightarrow 
C_f(\widetilde{\Gamma}_{k_0}) \{q^{m+1-k_0}\} \xrightarrow{\hat{\mathfrak{d}}_{k_0}^-} \left.%
\begin{array}{c}
C_f(\widetilde{\Gamma}_{k_0+1}) \{q^{m-k_0}\}\\
\oplus \\
C_f(\Gamma_{k_0}'') \\
\end{array}%
\right. \xrightarrow{\hat{\mathfrak{d}}_{k_0+1}^-} \cdots \xrightarrow{\hat{\mathfrak{d}}_{k-1}^-} \left.%
\begin{array}{c}
C_f(\widetilde{\Gamma}_{k}) \{q^{m+1-k}\}\\
\oplus \\
C_f(\Gamma_{k-1}'') \\
\oplus \\
C_f(\Gamma_{k-2}'')
\end{array}%
\right.
\xrightarrow{\hat{\mathfrak{d}}_{k}^-} \cdots 
\xrightarrow{\hat{\mathfrak{d}}_m^-} \left.%
\begin{array}{c}
C_f(\widetilde{\Gamma}_{m+1})\\
\oplus \\
C_f(\Gamma_{m-1}'')
\end{array}%
\right. \rightarrow 0,}
\]
}

\noindent where $\hat{\mathfrak{d}}_m^-,\dots,\hat{\mathfrak{d}}_{k_0}^-$ are given in \eqref{hat-d-m-1-},\eqref{hat-d-k-1-}, \eqref{hat-d-k0+1-1-} and \eqref{hat-d-k0-1-}.

Applying Gaussian Elimination (Lemma \ref{gaussian-elimination}) to $c_k''' \id_{C_f(\Gamma_{k-1}'')}$ in $\hat{\mathfrak{d}}_k^-$ for $k=m,m-1,\dots,k_0+1$, we get that $\hat{C}_f(\setlength{\unitlength}{1pt}
\begin{picture}(44,20)(-22,20)
\put(0,0){\vector(0,1){20}}

\put(0,20){\vector(1,1){20}}

\put(0,20){\vector(-1,1){20}}

\put(-20,30){\line(1,0){8}}

\put(-8,30){\line(1,0){16}}

\put(12,30){\vector(1,0){8}}

\put(-13,35){\tiny{$_m$}}

\put(10,35){\tiny{$_1$}}

\put(3,2){\tiny{$_{m+1}$}}

\put(12,26){\tiny{$_n$}}

\end{picture}) \vspace{20pt}$ is homotopic to
\[
0 \rightarrow 
C_f(\widetilde{\Gamma}_{k_0}) \{q^{m+1-k_0}\} \xrightarrow{\check{\mathfrak{d}}_{k_0}^-}  \cdots \xrightarrow{\check{\mathfrak{d}}_{k-1}^-} 
C_f(\widetilde{\Gamma}_{k}) \{q^{m+1-k}\}
\xrightarrow{\check{\mathfrak{d}}_{k}^-} \cdots  
\xrightarrow{\check{\mathfrak{d}}_m^-} 
C_f(\widetilde{\Gamma}_{m+1})\rightarrow 0,
\]
where $\check{\mathfrak{d}}_{k}^- \simeq c_k \tilde{d}_{k}^-$ for $k=m,\dots,k_0$. Recall that $c_k\neq 0$ for $k=m,\dots,k_0$. So this last chain complex is isomorphic to $\hat{C}_f(\setlength{\unitlength}{1pt}
\begin{picture}(44,20)(-22,20)
\put(0,0){\vector(0,1){20}}

\put(0,20){\vector(1,1){20}}

\put(0,20){\vector(-1,1){20}}

\put(-20,10){\line(1,0){18}}

\put(2,10){\vector(1,0){18}}

\put(-13,35){\tiny{$_m$}}

\put(10,35){\tiny{$_1$}}

\put(3,2){\tiny{$_{m+1}$}}

\put(10,13){\tiny{$_n$}}

\end{picture}) \vspace{20pt}$ in $\ch(\hmf)$. Therefore, 
\[
\hat{C}_f(\setlength{\unitlength}{1pt}
\begin{picture}(44,20)(-22,20)
\put(0,0){\vector(0,1){20}}

\put(0,20){\vector(1,1){20}}

\put(0,20){\vector(-1,1){20}}

\put(-20,10){\line(1,0){18}}

\put(2,10){\vector(1,0){18}}

\put(-13,35){\tiny{$_m$}}

\put(10,35){\tiny{$_1$}}

\put(3,2){\tiny{$_{m+1}$}}

\put(10,13){\tiny{$_n$}}

\end{picture})~\simeq~ \hat{C}_f(\setlength{\unitlength}{1pt}
\begin{picture}(44,20)(-22,20)
\put(0,0){\vector(0,1){20}}

\put(0,20){\vector(1,1){20}}

\put(0,20){\vector(-1,1){20}}

\put(-20,30){\line(1,0){8}}

\put(-8,30){\line(1,0){16}}

\put(12,30){\vector(1,0){8}}

\put(-13,35){\tiny{$_m$}}

\put(10,35){\tiny{$_1$}}

\put(3,2){\tiny{$_{m+1}$}}

\put(12,26){\tiny{$_n$}}

\end{picture}). \vspace{20pt} 
\]
\end{proof}

\section{Invariance under Reidemeister Moves}\label{sec-inv-reidemeister}

In this section, we prove that the homotopy type of the chain complex associated to a knotted MOY graph is invariant under Reidemeister moves. The main result of this section is Theorem \ref{invariance-reidemeister-all} below. Note that Theorem \ref{thm-inv-main} is a special case of Theorem \ref{invariance-reidemeister-all}.

\begin{theorem}\label{invariance-reidemeister-all}
Let $D_0$ and $D_1$ be two knotted MOY graphs. Assume that there is a finite sequence of Reidemeister moves that changes $D_0$ into $D_1$. Then $C_f(D_0) \simeq C_f(D_1)$, that is, they are isomorphic as objects of $\hch(\hmf)$.
\end{theorem}

We prove Theorem \ref{invariance-reidemeister-all} by an induction on the colors of edges involved in the Reidemeister moves. The starting point of this induction is the following theorem by Krasner \cite{Krasner}.

\begin{theorem}\cite[Theorem 14]{Krasner}\label{invariance-reidemeister-all-color=1}
Let $D_0$ and $D_1$ be two knotted MOY graphs. Assume that there is a Reidemeister move changing $D_0$ into $D_1$ that involves only edges colored by $1$. Then $C_f(D_0) \simeq C_f(D_1)$, that is, they are isomorphic as objects of $\hch(\hmf)$. 
\end{theorem}

\begin{proof}
Krasner \cite{Krasner} proved the invariance under Reidemeister moves I, II$_a$ and III. We only need to prove the invariance under Reidemeister Move II$_b$, which we prove using the argument in \cite{KR1}. 

Recall that
\[
C_f(\xygraph{
!{0;/r2pc/:}
[u]
!{\vcrossneg=>|{1}}
!{\vcross<{1}}})
=
`` 0\rightarrow C_f(\input{RM-III-pro5}) \xrightarrow{d_{-1}} C_f(\setlength{\unitlength}{1pt}
\begin{picture}(46,25)(-23,25)
\put(-15,50){\vector(1,0){30}}
\put(18,45){\tiny{$_{1}$}}

\qbezier(15,35)(15,40)(0,40)
\qbezier(-15,35)(-15,40)(0,40)
\put(15,35){\line(0,-1){20}}
\qbezier(15,15)(15,10)(0,10)
\qbezier(-15,15)(-15,10)(0,10)
\put(-15,15){\vector(0,1){20}}
\put(18,25){\tiny{$_{1}$}}

\put(15,0){\vector(-1,0){30}}
\put(-20,5){\tiny{$_{1}$}}

\end{picture}) \oplus C_f(\input{RM-III-pro2}) \xrightarrow{d_0} C_f(\input{RM-III-pro4}) \rightarrow 0,"
\]
where 
\begin{eqnarray*}
d_{-1} & = & \left(%
\begin{array}{l}
\chi^1_{lower}\\
\chi^0_{upper}
\end{array}%
\right), \\
d_0 & = & \left(%
\begin{array}{ll}
\chi^0_{upper}, & \chi^1_{lower}
\end{array}%
\right),
\end{eqnarray*}
and $\chi^1_{lower}$, $\chi^0_{upper}$ are the apparent $\chi$-morphisms.

From the proof of Decomposition (III) (Theorem \ref{decomp-III},) we know that there is a morphism 
\[
C_f(\setlength{\unitlength}{1pt}
\begin{picture}(46,25)(-23,25)
\put(-15,50){\vector(1,0){30}}
\put(18,45){\tiny{$_{1}$}}

\qbezier(15,35)(15,40)(0,40)
\qbezier(-15,35)(-15,40)(0,40)
\put(15,35){\line(0,-1){20}}
\qbezier(15,15)(15,10)(0,10)
\qbezier(-15,15)(-15,10)(0,10)
\put(-15,15){\vector(0,1){20}}
\put(18,25){\tiny{$_{1}$}}

\put(15,0){\vector(-1,0){30}}
\put(-20,5){\tiny{$_{1}$}}

\end{picture}) \xrightarrow{\rho} C_f(\input{RM-III-pro5}) \vspace{25pt}
\]
such that $\rho \circ \chi^1_{lower} \simeq \id$ and the kernel of $\rho$ is the image of
\[
C_f(\setlength{\unitlength}{1pt}
\begin{picture}(46,25)(-23,25)
\put(-15,40){\vector(1,0){30}}
\put(0,45){\tiny{$_{1}$}}

\put(15,10){\vector(-1,0){30}}
\put(0,5){\tiny{$_{1}$}}

\end{picture}) \xrightarrow{\iota} C_f(\setlength{\unitlength}{1pt}
\begin{picture}(46,25)(-23,25)
\put(-15,50){\vector(1,0){30}}
\put(18,45){\tiny{$_{1}$}}

\qbezier(15,35)(15,40)(0,40)
\qbezier(-15,35)(-15,40)(0,40)
\put(15,35){\line(0,-1){20}}
\qbezier(15,15)(15,10)(0,10)
\qbezier(-15,15)(-15,10)(0,10)
\put(-15,15){\vector(0,1){20}}
\put(18,25){\tiny{$_{1}$}}

\put(15,0){\vector(-1,0){30}}
\put(-20,5){\tiny{$_{1}$}}

\end{picture}), \vspace{25pt}
\]
where $\iota$ is the morphism induced by the apparent circle creation. Thus, applying Gaussian Elimination (Lemma \ref{gaussian-elimination},) we get
\[
C_f(\xygraph{
!{0;/r2pc/:}
[u]
!{\vcrossneg=>|{1}}
!{\vcross<{1}}})
\simeq
`` 0\rightarrow  C_f(\setlength{\unitlength}{1pt}
\begin{picture}(46,25)(-23,25)
\put(-15,40){\vector(1,0){30}}
\put(0,45){\tiny{$_{1}$}}

\put(15,10){\vector(-1,0){30}}
\put(0,5){\tiny{$_{1}$}}

\end{picture})\{q^{1-N}\} \left\langle 1 \right\rangle \oplus C_f(\input{RM-III-pro2}) \xrightarrow{\hat{d}_0} C_f(\input{RM-III-pro4}) \rightarrow 0,"
\]
where $\hat{d}_0 = (\chi^0_{upper} \circ \iota,  \chi^1_{lower})$.

By Decomposition (III) (Theorem \ref{decomp-III},) there is a homotopy equivalence
\[
C_f(\setlength{\unitlength}{1pt}
\begin{picture}(40,25)(-20,25)
\put(15,0){\vector(0,1){50}}
\put(20,45){\tiny{$_{1}$}}

\put(-15,50){\vector(0,-1){50}}
\put(-20,5){\tiny{$_{1}$}}

\end{picture}) \oplus C_f(\setlength{\unitlength}{1pt}
\begin{picture}(46,25)(-23,25)
\put(-15,40){\vector(1,0){30}}
\put(0,45){\tiny{$_{1}$}}

\put(15,10){\vector(-1,0){30}}
\put(0,5){\tiny{$_{1}$}}

\end{picture})\{[N-2]\} \left\langle 1 \right\rangle \xrightarrow{(F,\vec{\alpha})} C_f(\input{RM-III-pro2}). \vspace{25pt}
\]
From the proof of Decomposition (III) (Theorem \ref{decomp-III},) one can see that 
\[
C_f(\setlength{\unitlength}{1pt}
\begin{picture}(46,25)(-23,25)
\put(-15,40){\vector(1,0){30}}
\put(0,45){\tiny{$_{1}$}}

\put(15,10){\vector(-1,0){30}}
\put(0,5){\tiny{$_{1}$}}

\end{picture})\{q^{1-N}\} \oplus C_f(\setlength{\unitlength}{1pt}
\begin{picture}(46,25)(-23,25)
\put(-15,40){\vector(1,0){30}}
\put(0,45){\tiny{$_{1}$}}

\put(15,10){\vector(-1,0){30}}
\put(0,5){\tiny{$_{1}$}}

\end{picture})\{[N-2]\} \left\langle 1 \right\rangle \xrightarrow{(\chi^0_{upper} \circ \iota, ~\chi^1_{lower} \circ \vec{\alpha})} C_f(\input{RM-III-pro4}) \vspace{25pt}
\]
is a homotopy equivalence of matrix factorizations. Applying Gaussian Elimination (Lemma \ref{gaussian-elimination}) to this homotopy equivalence, we get
\[
C_f(\xygraph{
!{0;/r2pc/:}
[u]
!{\vcrossneg=>|{1}}
!{\vcross<{1}}})
\simeq
`` 0\rightarrow C_f(\setlength{\unitlength}{1pt}
\begin{picture}(40,25)(-20,25)
\put(15,0){\vector(0,1){50}}
\put(20,45){\tiny{$_{1}$}}

\put(-15,50){\vector(0,-1){50}}
\put(-20,5){\tiny{$_{1}$}}

\end{picture})  \rightarrow 0."
\]
Thus, $C_f$ is invariant under Reidemeister move II$_b$ of edges colored by $1$.
\end{proof}

In the rest of this section, we prove Theorem \ref{invariance-reidemeister-all} by an induction based on Theorem \ref{invariance-reidemeister-all-color=1}, which is a straightforward generalization of the proof of \cite[Theorem 13.1]{Wu-color}. 

\subsection{Invariance under Reidemeister moves II$_a$, II$_b$ and III}

\begin{lemma}\label{invariance-reidemeister-II-III}
Let $D_0$ and $D_1$ be two knotted MOY graphs. Assume that there is a Reidemeister move of type II$_a$, II$_b$ or III that changes $D_0$ into $D_1$. Then $C_f(D_0) \simeq C_f(D_1)$, that is, they are isomorphic as objects of $\hch(\hmf)$.
\end{lemma}

\begin{proof}
The proofs for Reidemeister moves II$_a$, II$_b$ and III are quite similar. We only give details for Reidemeister move II$_a$ here and leave the other two moves to the reader. 

For Reidemeister move II$_a$, we need to show that
\begin{equation}\label{RM-II-a-pro}
C_f(\xygraph{
!{0;/r2pc/:}
[u]
!{\xcapv[-1]@(0)=<>{m}}
!{\xcapv@(0)}
[ruu]!{\xcapv@(0)=><{n}}
!{\xcapv@(0)}
})  \simeq
C_f\xygraph{
!{0;/r2pc/:}
[u]
!{\vtwist=<>{m}|{n}}
!{\vtwistneg}
}).
\end{equation}

By Theorem \ref{invariance-reidemeister-all-color=1}, \eqref{RM-II-a-pro} is true if $m=n=1$. Assume there is a $k \in \nat$ such that \eqref{RM-II-a-pro} is true for $1\leq m,n \leq k$. Consider that case $1\leq m,n \leq k+1$. By the assumption and Theorem \ref{fork-sliding-invariance-general}, we know that
\[
C_f(\xygraph{
!{0;/r3pc/:}
[u(1.25)]
!{\xcapv[-0.75]@(0)=<<{m}}
[u(0.25)]!{\xcapv=>|{1}}
[u]!{\xcapv[-1]=<|{m-1}}
!{\xcapv[-0.75]@(0)=<>{m}}
[r][u(2.75)]!{\xcapv[-0.75]@(0)=<<{n}}
[u(0.25)]!{\xcapv=>|{n-1}}
[u]!{\xcapv[-1]=<|{1}}
!{\xcapv[-0.75]@(0)=<>{n}}
})
\simeq C_f(\xygraph{
!{0;/r1pc/:}
[uuuu]
!{\xcapv[-1]@(0)=<<{m}}
[ur]!{\xcapv@(0)=>>{n}}
[dll]!{\zbendv<{m-1}}
!{\vtwist}
[ru]!{\sbendv[-1]<{n-1}}
[lll]!{\vtwist}
[urr]!{\vtwist}
[l]!{\vtwist=<}
[lu]!{\xcapv@(0)=>}
!{\xcapv[-1]@(0)<{1}}
[rrruu]!{\xcapv@(0)=>}
!{\xcapv@(0)>{1}}
[llu]!{\vtwistneg}
[l]!{\vtwistneg}
[urr]!{\vtwistneg}
[l]!{\vtwistneg}
[lu]!{\sbendv}
[r]!{\zbendv}
[dl]!{\xcapv@(0)=><{n}}
[lu]!{\xcapv[-1]@(0)=<>{m}}
})
\simeq C_f(\xygraph{
!{0;/r2pc/:}
[uu]
!{\xcapv[-0.5]@(0)=<<{m}}
[u(0.5)]!{\xcapv=>>{1}}
[u]!{\xcapv[-1]=<|{m-1}}
[r][u(1.5)]!{\xcapv[-0.5]@(0)=<<{n}}
[u(0.5)]!{\xcapv=>|{n-1}}
[u]!{\xcapv[-1]=<>{1}}
[l]!{\vtwist}
!{\vtwistneg}
!{\xcapv[-0.5]@(0)=<>{m}}
[ru]!{\xcapv[0.5]@(0)=><{n}}
}).
\]
According to Decomposition II (Theorem \ref{decomp-II},) this means that, for $1\leq m,n \leq k+1$,
\[
C_f(\xygraph{
!{0;/r2pc/:}
[u]
!{\xcapv[-1]@(0)=<>{m}}
!{\xcapv@(0)}
[ruu]!{\xcapv@(0)=><{n}}
!{\xcapv@(0)}
})\{[m][n]\}  \simeq
C_f(\xygraph{
!{0;/r2pc/:}
[u]
!{\vtwist=<>{m}|{n}}
!{\vtwistneg}
})\{[m][n]\}.
\]
Then, by \cite[Proposition 3.20]{Wu-color}, we know that \eqref{RM-II-a-pro} is true for $1\leq m,n \leq k+1$.
\end{proof}

\subsection{Invariance under Reidemeister move I} This proof of invariance under Reidemeister move I is slightly more complex and requires the following lemma.

\begin{lemma}\label{twisted-forks}
\[
\xymatrix{
\hat{C}_f(\setlength{\unitlength}{1pt}
\begin{picture}(60,30)(-30,30)
\put(-20,0){\vector(1,1){30}}
\qbezier(10,30)(15,35)(0,40)

\put(20,0){\line(-1,1){18}}
\put(-2,22){\vector(-1,1){8}}
\qbezier(-10,30)(-15,35)(0,40)

\put(0,40){\vector(0,1){20}}
\put(5,55){\tiny{$_{1+n}$}}
\put(-26,0){\tiny{$_{1}$}}
\put(23,0){\tiny{$_{n}$}}
\end{picture}) \simeq \hat{C}_f(\setlength{\unitlength}{1pt}
\begin{picture}(60,20)(-30,20)
\put(-20,0){\vector(1,1){20}}
\put(20,0){\vector(-1,1){20}}
\put(0,20){\vector(0,1){20}}
\put(5,35){\tiny{$_{1+n}$}}
\put(-26,0){\tiny{$_{1}$}}
\put(23,0){\tiny{$_{n}$}}

\end{picture})\{q^{n}\}, & \hat{C}_f(\setlength{\unitlength}{1pt}
\begin{picture}(60,30)(-30,30)
\put(20,0){\vector(-1,1){30}}
\qbezier(10,30)(15,35)(0,40)

\put(-20,0){\line(1,1){18}}
\put(2,22){\vector(1,1){8}}
\qbezier(-10,30)(-15,35)(0,40)

\put(0,40){\vector(0,1){20}}
\put(5,55){\tiny{$_{1+n}$}}
\put(-26,0){\tiny{$_{1}$}}
\put(23,0){\tiny{$_{n}$}}

\end{picture} ) \simeq \hat{C}_f(\setlength{\unitlength}{1pt}
\begin{picture}(60,20)(-30,20)
\put(-20,0){\vector(1,1){20}}
\put(20,0){\vector(-1,1){20}}
\put(0,20){\vector(0,1){20}}
\put(5,35){\tiny{$_{1+n}$}}
\put(-26,0){\tiny{$_{1}$}}
\put(23,0){\tiny{$_{n}$}}

\end{picture})\{q^{-n}\}, \vspace{30pt} \\
\hat{C}_f(\setlength{\unitlength}{1pt}
\begin{picture}(60,30)(-30,30)
\put(-20,0){\vector(1,1){30}}
\qbezier(10,30)(15,35)(0,40)

\put(20,0){\line(-1,1){18}}
\put(-2,22){\vector(-1,1){8}}
\qbezier(-10,30)(-15,35)(0,40)

\put(0,40){\vector(0,1){20}}
\put(5,55){\tiny{$_{m+1}$}}
\put(-26,0){\tiny{$_{m}$}}
\put(23,0){\tiny{$_{1}$}}

\end{picture}) \simeq \hat{C}_f(\setlength{\unitlength}{1pt}
\begin{picture}(60,20)(-30,20)
\put(-20,0){\vector(1,1){20}}
\put(20,0){\vector(-1,1){20}}
\put(0,20){\vector(0,1){20}}
\put(5,35){\tiny{$_{m+1}$}}
\put(-26,0){\tiny{$_{m}$}}
\put(23,0){\tiny{$_{1}$}}

\end{picture})\{q^{m}\}, & \hat{C}_f(\setlength{\unitlength}{1pt}
\begin{picture}(60,30)(-30,30)
\put(20,0){\vector(-1,1){30}}
\qbezier(10,30)(15,35)(0,40)

\put(-20,0){\line(1,1){18}}
\put(2,22){\vector(1,1){8}}
\qbezier(-10,30)(-15,35)(0,40)

\put(0,40){\vector(0,1){20}}
\put(5,55){\tiny{$_{m+1}$}}
\put(-26,0){\tiny{$_{m}$}}
\put(23,0){\tiny{$_{1}$}}
\end{picture}) \simeq \hat{C}_f(\setlength{\unitlength}{1pt}
\begin{picture}(60,20)(-30,20)
\put(-20,0){\vector(1,1){20}}
\put(20,0){\vector(-1,1){20}}
\put(0,20){\vector(0,1){20}}
\put(5,35){\tiny{$_{m+1}$}}
\put(-26,0){\tiny{$_{m}$}}
\put(23,0){\tiny{$_{1}$}}

\end{picture})\{q^{-m}\},
} \vspace{30pt}
\]
where ``$\simeq$" is the isomorphism in $\hch(\hmf)$.
\end{lemma}

\begin{proof}
The proof of \cite[Lemma 13.6]{Wu-color} applies here without change.
\end{proof}

\begin{lemma}\label{invariance-reidemeister-I-unnormal}
\begin{eqnarray}
\label{invariance-reidemeister-I-unnormal+}\hat{C}_f(\xygraph{
!{0;/r1pc/:}
[u]
!{\xcapv@(0)=>>{m}}
!{\hover}
!{\hcap}
[ld]!{\xcapv@(0)}
}) & \simeq & \hat{C}_f(\xygraph{
!{0;/r1pc/:}
[u]
!{\xcapv@(0)=>>{m}}
!{\xcapv@(0)}
!{\xcapv@(0)}
})\left\langle m \right\rangle\|m\| \{q^{-m(N+1-m)}\}, \\
\label{invariance-reidemeister-I-unnormal-}\hat{C}_f(\xygraph{
!{0;/r1pc/:}
[u]
!{\xcapv@(0)=>>{m}}
!{\hunder}
!{\hcap}
[ld]!{\xcapv@(0)}
}) & \simeq & \hat{C}_f(\xygraph{
!{0;/r1pc/:}
[u]
!{\xcapv@(0)=>>{m}}
!{\xcapv@(0)}
!{\xcapv@(0)}
})\left\langle m \right\rangle\|-m\| \{q^{m(N+1-m)}\},
\end{eqnarray}
where $\|\ast\|$ means shifting the homological grading. 
Therefore,
\[
C_f(\xygraph{
!{0;/r1pc/:}
[u]
!{\xcapv@(0)=>>{m}}
!{\hover}
!{\hcap}
[ld]!{\xcapv@(0)}
})  \simeq  C_f(\xygraph{
!{0;/r1pc/:}
[u]
!{\xcapv@(0)=>>{m}}
!{\xcapv@(0)}
!{\xcapv@(0)}
})
\simeq C_f(\xygraph{
!{0;/r1pc/:}
[u]
!{\xcapv@(0)=>>{m}}
!{\hunder}
!{\hcap}
[ld]!{\xcapv@(0)}
}).
\]
\end{lemma}

\begin{proof}
We prove \eqref{invariance-reidemeister-I-unnormal+} by an induction on $m$. The proof of \eqref{invariance-reidemeister-I-unnormal-} is similar and left to the reader. If $m=1$, then \eqref{invariance-reidemeister-I-unnormal+} follows from Theorem \ref{invariance-reidemeister-all-color=1}. Assume \eqref{invariance-reidemeister-I-unnormal+} is true for $m$. Consider $m+1$. 

By Theorem \ref{fork-sliding-invariance-general}, we have
\[
\hat{C}_f(\xygraph{
!{0;/r2pc/:}
[uu]
!{\xcapv@(0)=>>{m+1}}
!{\hover}
!{\hcap}
[ld]!{\xcapv=><{1}}
[u]!{\xcapv[-1]=<>{m}}
!{\xcapv@(0)=><{m+1}}
}) \simeq \hat{C}_f(\xygraph{
!{0;/r1pc/:}
[uuuu]
!{\xcapv@(0)=>>{m+1}}
!{\zbendh}
!{\hcross}
[d]!{\hcross}
!{\hcap=<}
[lld]!{\hcross}
[llu]!{\hover}
[uul]!{\xcapv[2]@(0)}
[dd]!{\xcapv[-2]@(0)|{m}}
[d]!{\sbendh|{1}}
[dl]!{\xcapv@(0)=><{m+1}}
[uuuuurr]!{\hloop[3]=<}
}).
\]
Since \eqref{invariance-reidemeister-I-unnormal+} is true for $1$, we know that 
\[
\hat{C}_f(\xygraph{
!{0;/r1pc/:}
[uuuu]
!{\xcapv@(0)=>>{m+1}}
!{\zbendh}
!{\hcross}
[d]!{\hcross}
!{\hcap=<}
[lld]!{\hcross}
[llu]!{\hover}
[uul]!{\xcapv[2]@(0)}
[dd]!{\xcapv[-2]@(0)|{m}}
[d]!{\sbendh|{1}}
[dl]!{\xcapv@(0)=><{m+1}}
[uuuuurr]!{\hloop[3]=<}
}) \simeq \hat{C}_f(\xygraph{
!{0;/r1pc/:}
[uuuu]
!{\xcapv@(0)=>>{m+1}}
!{\zbendh}
!{\hcross}
[d]!{\hcap=>}
[ld]!{\hcross}
[llu]!{\hover}
[uul]!{\xcapv[2]@(0)}
[dd]!{\xcapv[-2]@(0)<{m}}
[d]!{\sbendh|{1}}
[dl]!{\xcapv@(0)=><{m+1}}
[uuuuurr]!{\hcap[3]=<}
})\left\langle 1 \right\rangle\|1\| \{q^{-N}\}.
\]
From Lemma \ref{invariance-reidemeister-II-III}, one can see that 
\[
\hat{C}_f(\xygraph{
!{0;/r1pc/:}
[uuuu]
!{\xcapv@(0)=>>{m+1}}
!{\zbendh}
!{\hcross}
[d]!{\hcap=>}
[ld]!{\hcross}
[llu]!{\hover}
[uul]!{\xcapv[2]@(0)}
[dd]!{\xcapv[-2]@(0)<{m}}
[d]!{\sbendh|{1}}
[dl]!{\xcapv@(0)=><{m+1}}
[uuuuurr]!{\hcap[3]=<}
}) \simeq \hat{C}_f(\xygraph{
!{0;/r1.5pc/:}
[uuu]
!{\xcapv[0.5]@(0)=>>{m+1}}
[u(0.5)]!{\xcapv@(0)}
[u]!{\sbendv}
[l]!{\vtwist}
!{\xcapv[-1]@(0)=<<{1}}
[ur]!{\hover}
!{\hcap}
[lld]!{\vtwist}
!{\xcapv[-1]@(0)=<<{m}}
!{\zbendv}
[dl]!{\xcapv[0.5]@(0)=>>{m+1}}
}).
\]
Since \eqref{invariance-reidemeister-I-unnormal+} is true for $m$, we know that 
\[
\hat{C}_f(\xygraph{
!{0;/r1.5pc/:}
[uuu]
!{\xcapv[0.5]@(0)=>>{m+1}}
[u(0.5)]!{\xcapv@(0)}
[u]!{\sbendv}
[l]!{\vtwist}
!{\xcapv[-1]@(0)=<<{1}}
[ur]!{\hover}
!{\hcap}
[lld]!{\vtwist}
!{\xcapv[-1]@(0)=<<{m}}
!{\zbendv}
[dl]!{\xcapv[0.5]@(0)=>>{m+1}}
}) \simeq \hat{C}_f(\xygraph{
!{0;/r2pc/:}
[uu]
!{\xcapv[0.5]@(0)=>>{m+1}}
[u(0.5)]!{\xcapv@(0)}
[u]!{\sbendv}
[l]!{\vtwist}
!{\vtwist}
!{\xcapv@(0)=>>{m}}
!{\zbendv[-1]=<<{1}}
[dl]!{\xcapv[0.5]@(0)=>>{m+1}}
})\left\langle m \right\rangle\|m\| \{q^{-m(N+1-m)}\}.
\]
By Lemma \ref{twisted-forks}, we know that 
\[
\hat{C}_f(\xygraph{
!{0;/r2pc/:}
[uu]
!{\xcapv[0.5]@(0)=>>{m+1}}
[u(0.5)]!{\xcapv@(0)}
[u]!{\sbendv}
[l]!{\vtwist}
!{\vtwist}
!{\xcapv@(0)=>>{m}}
!{\zbendv[-1]=<<{1}}
[dl]!{\xcapv[0.5]@(0)=>>{m+1}}
}) \simeq \hat{C}_f(\xygraph{
!{0;/r2pc/:}
[u]
!{\xcapv[0.5]@(0)=>>{m+1}}
[u(0.5)]!{\xcapv@(0)}
[u]!{\sbendv}
[l]!{\xcapv@(0)=>>{m}}
!{\zbendv[-1]=<<{1}}
[dl]!{\xcapv[0.5]@(0)=>>{m+1}}
})\{q^{2m}\}
\]
Putting these together, we get that
\[
\hat{C}_f(\xygraph{
!{0;/r2pc/:}
[uu]
!{\xcapv@(0)=>>{m+1}}
!{\hover}
!{\hcap}
[ld]!{\xcapv=><{1}}
[u]!{\xcapv[-1]=<>{m}}
!{\xcapv@(0)=><{m+1}}
}) \simeq \hat{C}_f(\xygraph{
!{0;/r2pc/:}
[uu]
!{\xcapv[0.5]@(0)=>>{m+1}}
[u(0.5)]!{\xcapv@(0)}
[u]!{\sbendv}
[l]!{\xcapv@(0)=>>{m}}
!{\zbendv[-1]=<<{1}}
[dl]!{\xcapv[0.5]@(0)=>>{m+1}}
})\left\langle m+1 \right\rangle\|m+1\| \{q^{-(m+1)(N-m)}\}.
\]
By Decomposition II (Theorem \ref{decomp-II}) and \cite[Proposition 3.20]{Wu-color}, it follows that \eqref{invariance-reidemeister-I-unnormal+} is true for $m+1$. This completes the induction. So \eqref{invariance-reidemeister-I-unnormal+} is true for all $m$. 

By \eqref{invariance-reidemeister-I-unnormal+}, \eqref{invariance-reidemeister-I-unnormal-} and the normalizations in Definition \ref{complex-colored-crossing-def}, one can easily see that 
\[
C_f(\xygraph{
!{0;/r1pc/:}
[u]
!{\xcapv@(0)=>>{m}}
!{\hover}
!{\hcap}
[ld]!{\xcapv@(0)}
})  \simeq  C_f(\xygraph{
!{0;/r1pc/:}
[u]
!{\xcapv@(0)=>>{m}}
!{\xcapv@(0)}
!{\xcapv@(0)}
})
\simeq C_f(\xygraph{
!{0;/r1pc/:}
[u]
!{\xcapv@(0)=>>{m}}
!{\hunder}
!{\hcap}
[ld]!{\xcapv@(0)}
}).
\]
\end{proof}

We have now proved Theorem \ref{invariance-reidemeister-all} and, therefore, Theorem \ref{thm-inv-main}.

\section{Deformations of the Colored $\mathfrak{sl}(N)$-Homology over $\C$}\label{sec-deform-over-C}

In this section, we study properties of deformations of the colored $\mathfrak{sl}(N)$-homology over $\C$ and prove Theorem \ref{thm-inv-deformation}. 
 
\subsection{Definition and invariance} First, let us recall the definition of filtered matrix factorizations. For simplicity, we only define filtered matrix factorizations over a graded ring with graded underlying modules. See for example \cite{Wu7} for a more detailed discussion.

\begin{definition}\label{filtered-mf-def}
Let $R$ be a graded commutative unital $\C$-algebra. Fix an integer $N\geq2$. Let $w$ be a (not necessarily homogeneous) element of $R$ with $\deg{w}\leq 2N+2$. A filtered matrix factorization over $R$ with potential $w$ is a collection of two graded-free $R$-modules $M^0$, $M^1$ and two $R$-module homomorphisms $d^0:M^0\rightarrow M^1$, $d^1:M^1\rightarrow M^0$, called differential maps, such that
\[
\xymatrix{
d^1 \circ d^0=w\cdot\id_{M^0}, && d^0 \circ d^1=w\cdot\id_{M^1}.
}
\]
and
\[
\xymatrix{
d^0\in\fil^{N+1}\Hom_R(M^0,M^1), &&  d^1\in\fil^{N+1}\Hom_R(M^1,M^0),
}
\]
where $\fil$ denotes the filtrations on $\Hom_R(M^0,M^1)$ and $\Hom_R(M^1,M^0)$ induced by the gradings of $M^0$ and $M^1$.

We usually write such a matrix factorization $M$ as
\[
M^0 \xrightarrow{d_0} M^1 \xrightarrow{d_1} M^0.
\]
\end{definition}

Let $\Gamma$ be a MOY graph with a marking, and $\mathbb{X}_1,\dots,\mathbb{X}_m$ the alphabets associated to the internal marked points on $\Gamma$ and $\mathbb{E}_1,\dots,\mathbb{E}_l$ the alphabet associated to end points of $\Gamma$. Recall that $C_f(\Gamma)$ is a Koszul matrix factorization over the ring $\tilde{R} = \Sym(\mathbb{X}_1|\dots|\mathbb{X}_m|\mathbb{E}_1|\dots|\mathbb{E}_l) \otimes_\C R_B$, where $R_B=\C[B_1,\dots,B_N]$. The total polynomial grading of $\tilde{R}$ is given by $\deg B_k=2k$ and the grading of each alphabet (that is, each indeterminate in $\mathbb{X}_i$ or $\mathbb{E}_j$ has degree $2$.) Under the total polynomial grading of $\tilde{R}$, $C_f(\Gamma)$ is graded. We view $C_f(\Gamma)$ as a graded matrix factorization over $\tilde{R}_\partial = \Sym(\mathbb{E}_1|\dots|\mathbb{E}_l) \otimes_\C R_B$, where $\tilde{R}_\partial$ is given the total polynomial grading, that is, the grading inherited from $\tilde{R}$ by viewing $\tilde{R}_\partial$ as a subring of $\tilde{R}$.

Note that there is a grading of $\tilde{R}$ given by $\deg B_k=0$ and that every indeterminate in each $\mathbb{X}_i$ has degree $2$. We call this grading the quantum grading of $\tilde{R}$. $\tilde{R}_\partial$ inherits this quantum grading. Note that $C_f(\Gamma)$ is a filtered matrix factorization over $\tilde{R}_\partial$ under the quantum grading. 

In the rest of this section, we denote by $\deg_T$ the degree associated to the total polynomial grading and by $\deg_Q$ the degree associated to the quantum grading/filtration. It is clear that 
\begin{equation}\label{element-degs-comp}
\deg_Q v \leq \deg_T v \text{ for any } v \in C_f(\Gamma). 
\end{equation}
Let $\Gamma'$ be an MOY graph that has the same boundary condition (and boundary marking) as $\Gamma$. Then, for any $\tilde{R}_\partial$-module homomorphism $F: C_f(\Gamma) \rightarrow C_f(\Gamma')$, one can see that 
\begin{equation}\label{morphism-degs-comp}
\deg_Q F \leq \deg_T F.
\end{equation}

For $b_1,\dots, b_N\in \C$, let 
\[
\pi:\tilde{R}_\partial \rightarrow \tilde{R}_\partial/(B_1-b_1,\dots,B_N-b_N) \cong R_\partial := \Sym(\mathbb{E}_1|\dots|\mathbb{E}_l)
\] 
be the standard projection. Note that $R_\partial$ comes with a natural grading given by that every indeterminate in each $\mathbb{E}_j$ has degree $2$. We call this grading the quantum grading of $R_\partial$. Then $\pi$ preserves the quantum grading.

Denote by $\tilde{R}_\partial-\Mod$ and $R_\partial-\Mod$ the categories of graded-free modules over $\tilde{R}_\partial$ and $R_\partial$. Then $\pi$ induces a functor $\tilde{R}_\partial-\Mod \xrightarrow{\varpi} R_\partial-\Mod$ which preserves the quantum grading. Let $w$ be an element of $\tilde{R}_\partial$ homogeneous under the total polynomial grading with $\deg_T w = 2N+2$. 

Denote by $\hmf^\fil_{R_\partial,\pi(w)}$ the homotopy category of homotopically finite matrix factorizations over $R_\partial$ with potential $\pi(w)$ whose quantum filtration is bounded below. (The morphisms of $\hmf^\fil_{R_\partial,\pi(w)}$ are homotopy classes of morphisms of matrix factorizations preserving the $\zed_2$-grading and of filtered degree $\leq 0$.)

Then $\tilde{R}_\partial-\Mod \xrightarrow{\varpi} R_\partial-\Mod$ induces a functor $\hmf_{\tilde{R}_\partial,w} \xrightarrow{\varpi} \hmf^\fil_{R_\partial,\pi(w)}$ that preserves the $\zed_2$-grading and the quantum filtration, which, in turn, induces a functor $\hch(\hmf_{\tilde{R}_\partial,w}) \xrightarrow{\varpi} \hch(\hmf^\fil_{R_\partial,\pi(w)})$ that preserves the $\zed_2$-grading, the quantum filtration and the homological grading.

\begin{definition}\label{C-f-pi-def}
Let $D$ be a knotted MOY graph. Define $C_{f,\pi}(D) = \varpi(C_f(D))$.
\end{definition}

\begin{theorem}\label{invariance-reidemeister-deformation}
Let $D$ and $D'$ be two knotted MOY graphs. Assume that there is a finite sequence of Reidemeister moves that changes $D$ into $D'$. Then $C_{f,\pi}(D) \simeq C_{f,\pi}(D')$ as objects of $\hch(\hmf_{R_\partial,\pi(w)})$. That is, the homotopy type of $C_{f,\pi}(D)$, with its $\zed_2$-grading, quantum filtration and homological grading, is invariant under Reidemeister moves.
\end{theorem}

\begin{proof}
By Theorem \ref{invariance-reidemeister-all}, $C_{f}(D) \simeq C_{f}(D')$. That is, there are chain maps $F: C_{f}(D) \rightarrow C_{f}(D')$ and $G: C_{f}(D') \rightarrow C_{f}(D)$ preserving the $\zed_2$-grading, the total polynomial grading and the homological grading that satisfy
\[
\xymatrix{
G \circ F \simeq \id_{C_{f}(D)} & \text{ and } & F \circ G \simeq \id_{C_{f}(D')}.
}
\]

Consider the chain maps $\varpi(F): C_{f,\pi}(D) \rightarrow C_{f,\pi}(D')$ and $\varpi(G): C_{f,\pi}(D') \rightarrow C_{f,\pi}(D)$. Clearly, they preserve the $\zed_2$- and homological gradings. By \eqref{morphism-degs-comp}, they also preserve the quantum filtration. Moreover, we have 
\[
\varpi(G) \circ \varpi(F) \simeq \varpi(\id_{C_{f}(D)})=\id_{C_{f,\pi}(D)} \text{ and } \varpi(F) \circ \varpi(G) \simeq \varpi(\id_{C_{f}(D')}) = \id_{C_{f,\pi}(D')}.
\]
Thus, $C_{f,\pi}(D) \simeq C_{f,\pi}(D')$.
\end{proof}

\subsection{The $\zed_2$-grading and the spectral sequence} To prove the purity of the $\zed_2$-grading of $C_{f,\pi}(D)$ and construct the spectral sequence connecting $H(D)$ to $H_{f,\pi}(D)$, we need to first establish a simple relation between $C_{f,\pi}(D)$ and $C(D)$.

First, note that the functor $\varpi_0$ defined in Corollary \ref{reduce-base-functor} is the special case of the functor $\varpi$ when $b_1=\cdots=b_N=0$.

Second, defined $\kappa: R_\partial \rightarrow R_\partial$ by letting $\kappa(r)=$ the top homogeneous component of $r$. Note that $\kappa(r_1r_2)=\kappa(r_1)\kappa(r_2)$ for all $r_1,r_2\in R_\partial$. If $M$ is a finitely generated graded-free $R_\partial$-module, then $\kappa$ induces a map $\kappa: M \rightarrow M$. Let $M$ and $M'$ be graded-free $R_\partial$-modules with gradings bounded below, and $\fil$ the filtration on $\Hom_{R_\partial}(M,M')$ induced by the gradings of $M$ and $M'$. Then $\kappa$ induces a map $\kappa: \bigcup_{i\in \zed}\fil^i\Hom_{R_\partial}(M,M') \rightarrow \Hom_{R_\partial}(M,M')$, which preserves composition and tensor product of homomorphisms. 

\begin{lemma}\label{varpi-0-kappa-varpi}
Suppose that $\tilde{M}$ and $\tilde{M}'$ are graded-free $\tilde{R}_\partial$-modules. If $F\in \bigcup_{i\in \zed}\fil^i\Hom_{R_\partial}(M,M')$ is a $\tilde{R}_\partial$-module homomorphism homogeneous under the total polynomial grading with $\deg_T F = \deg_Q F$, then $\varpi_0 (F) = \kappa (\varpi(F))$.

In particular, if $w \in \tilde{R}_\partial$ is an element homogeneous under the total polynomial grading such that $\deg_T w= \deg_Q w =2N+2$, then, for every object $M$ of $\hmf^\fil_{R_\partial,\pi(w)}$, there is a uniquely defined object $\kappa(M)$ of $\hmf_{R_\partial,\pi_0(w)}$. And for every object $\tilde{M}$ of $\hmf_{\tilde{R}_\partial,w}$, we have $\varpi_0(\tilde{M}) = \kappa (\varpi(\tilde{M}))$. 

Moreover, if $\tilde{M}$ and $\tilde{M}'$ are objects of $\hmf_{\tilde{R}_\partial,w}$ and $F: \tilde{M} \rightarrow \tilde{M}'$ is a morphism of matrix factorizations homogeneous under the total polynomial grading with $\deg_T F = \deg_Q F$, then $\varpi_0 (F) = \kappa (\varpi(F))$.
\end{lemma}

\begin{proof}
The first part of the lemma follows easily from the definitions of $\kappa$, $\varpi$ and $\varpi_0$. The second part of a straightforward application of the first part.
\end{proof}

\begin{remark}
The condition ``$\deg_T F = \deg_Q F$" is necessary in Lemma \ref{varpi-0-kappa-varpi}. Also, $\kappa$ is not a functor from $\hmf^\fil_{R_\partial,\pi(w)}$ to $\hmf_{R_\partial,\pi_0(w)}$.
\end{remark}

\begin{corollary}\label{relate-C-C-f-pi}
Let $D$ be a knotted MOY graph. Then $\kappa(C_{f,\pi}(D)) \cong C(D)$ as objects of $\ch(\hmf)$. 
\end{corollary}

\begin{proof}
Put a marking on $D$ and let $D'$ be any piece of $D$ from this marking. If $D'$ is an embedded MOY graph, then $\kappa(C_{f,\pi}(D')) \cong C(D')$ by Lemma \ref{varpi-0-kappa-varpi}. If $D'$ is a colored crossing, then $\kappa(C_{f,\pi}(D'))$ is a chain complex with the same terms as $C(D')$. By the definitions of the differential maps of $C_{f}(D')$ and $C(D')$ and Corollary \ref{differentials-varpi-0}, we know that the total polynomial degree and the quantum degree of the differential map of $C_{f}(D')$ are both $0$. So, by the last part of Lemma \ref{varpi-0-kappa-varpi}, we know that the differential maps of $\kappa(C_{f,\pi}(D'))$ and $C(D')$ are the same. Thus, $\kappa(C_{f,\pi}(D')) \cong C(D')$ in this case too. Since $\kappa$ preserve tensor product, we have $\kappa(C_{f,\pi}(D)) \cong C(D)$.
\end{proof}

\begin{figure}[ht]
$
\xymatrix{
\input{vertex-split}
} 
$
\caption{}\label{vertex-split} 

\end{figure}

Let $\Gamma$ be a closed (embedded) MOY graph. Replace each edge of $\Gamma$ of color $m$ by $m$ parallel edges colored by $1$ and, at each vertex of $\Gamma$, match the incoming and outgoing edges in linear order without creating any intersection. That is replace each vertex of $\Gamma$ by the configuration in Figure \ref{vertex-split}, where each strand is an edge colored by $1$. This changes $\Gamma$ into a collection of disjoint embedded circles colored by $1$. As in \cite{Wu-color}, we define the colored rotation number $\mathrm{cr}(\Gamma)$ of $\Gamma$ to be the sum of the rotation numbers of these circles. Denote by $H^{\ve,i}(\Gamma)$ subspace of $H(\Gamma)$ of homogeneous elements of $\zed_2$-degree $\ve$ and quantum degree $i$. The following is part of \cite[Theorem 14.7]{Wu-color}.

\begin{lemma}\cite[Theorem 14.7]{Wu-color}\label{MOY-C-cr}
If $\Gamma$ is a closed (embedded) MOY graph, then $H^{\mathrm{cr}(\Gamma)-1,i}(\Gamma)=0$ for all $i\in \zed$.
\end{lemma}

Denote by $H_{f,\pi}^{\ve}(\Gamma)$ the subspace of $H_{f,\pi}(\Gamma)$ of elements of $\zed_2$-degree $\ve$. The following is a generalization of \cite[Proposition 2.19]{Wu7}.

\begin{proposition}\label{MOY-H-H-f-pi-ismorphism}
Let $\Gamma$ be a closed (embedded) MOY graph. Then
\begin{eqnarray}
\label{MOY-H-H-f-pi-ismorphism-1} H_{f,\pi}^{\mathrm{cr}(\Gamma)-1}(\Gamma)& = & 0, \\
\label{MOY-H-H-f-pi-ismorphism-2} \fil^k H_{f,\pi}^{\mathrm{cr}(\Gamma)}(\Gamma) / \fil^{k-1} H_{f,\pi}^{\mathrm{cr}(\Gamma)}(\Gamma) & \cong & H^{\mathrm{cr}(\Gamma),k}(\Gamma) \text{ for all } k\in\zed,
\end{eqnarray}
where $\fil$ is the quantum filtration on $H_{f,\pi}(\Gamma)$.

Moreover, the isomorphism \eqref{MOY-H-H-f-pi-ismorphism-2} is natural. That is, if $\Gamma'$ is also a closed (embedded) MOY graph and $G:C_{f}(\Gamma) \rightarrow  C_{f}(\Gamma')$ is a morphism of matrix factorizations of $\zed_2$-degree $\mathrm{cr}(\Gamma)-\mathrm{cr}(\Gamma')$ that is homogeneous under the total polynomial grading with $\deg_T G = \deg_Q G = i$, then $(\varpi_0(G))_\ast:H(\Gamma) \rightarrow  H(\Gamma')$ is compatible with $(\varpi(G))_\ast:H_{f,\pi}(\Gamma) \rightarrow  H_{f,\pi}(\Gamma')$ under the isomorphism \eqref{MOY-H-H-f-pi-ismorphism-2}.
\end{proposition}

\begin{proof} (The proof is identical to that of \cite[Proposition 2.19]{Wu7}.)
The matrix factorization $C_{f,\pi}(\Gamma)$ has potential $0$ and is therefore a chain complex. Using the usual notation for matrix factorizations, we write $C_{f,\pi}(\Gamma)$ as
\[
M^0\xrightarrow{d_0}M^1\xrightarrow{d_1}M^0.
\]
Note that the quantum grading of $M^0 \oplus M^1$ is bounded from below. Since the quantum filtration $\fil$ of $C_{f,\pi}(\Gamma)$ is induced by the quantum grading of $M^0 \oplus M^1$, we know that $\fil$ is bounded below and exhaustive.

Also one can see that the quantum degrees of elements of $M^0$ have the same parity, and the quantum degrees of elements of $M^1$ have the same parity, while these two parities differ by $N+1$. 

The differential maps $d_0$ and $d_1$ have quantum degree $N+1$. Recall that $C_{f,\pi}(\Gamma)$ is a Koszul matrix factorization. From the definition of $C_{f,\pi}(\Gamma)$, one can see that the lowest homogeneous components of $d_0$ and $d_1$ have quantum degree $\geq 1-N$. So, for $\ve\in\zed_2$, $d_\ve$ has a decomposition
\[
d_\ve=\sum_{l=0}^{N}d_\ve^{(l)},
\]
where $d_\ve^{(l)}$ is a homogeneous linear map of quantum degree $N+1-2l$. Consider the homogeneous parts of $d_\ve\circ d_{\ve-1}=0$ and $d_{\ve-1}\circ d_\ve=0$ of degree $2(N+1-l)$. It is easy to see that
\begin{equation}\label{MOY-H-H-f-pi-ismorphism-pro-0}
\sum_{j+k=l,~ 0\leq j,k \leq N} d_\ve^{(j)}\circ d_{\ve-1}^{(k)}=0 \text{ and } \sum_{j+k=l,~ 0\leq j,k \leq N} d_{\ve-1}^{(j)}\circ d_\ve^{(k)}=0.
\end{equation}

It is clear that $\kappa (d_\ve) = d_\ve^{(0)}$ and therefore the matrix factorization $C(\Gamma)$ of potential $0$ is given by 
\[
M^0\xrightarrow{d_0^{(0)}}M^1\xrightarrow{d_1^{(0)}}M^0.
\]

We prove \eqref{MOY-H-H-f-pi-ismorphism-1} first. To simplify notations, we set $\ve_0 =\mathrm{cr}(\Gamma)$. Let $\alpha\in M^{\ve_0-1}$ be a cycle of quantum degree $g$. That is, $d_{\ve_0-1}\alpha=0$ and  
\[
\alpha=\sum_{k=-\infty}^{\infty}\alpha_k,
\]
where $\alpha_k$ is a homogeneous element of degree $g-2k$ satisfying $\alpha_k=0$ for $k<0$ and $\alpha_k=0$ for $k>>1$ (since the quantum grading is bounded from below.) 

We construct by induction a sequence $\{\beta_k\}_{-\infty}^{\infty}\subset M^{\ve_0}$, such that $\beta_k$ is a homogeneous element of degree $g-2k-N-1$, $\beta_k=0$ for $k<0$, and 
\begin{equation}\label{MOY-H-H-f-pi-ismorphism-pro-1}
\alpha_k=\sum_{l=0}^{N}d_{\ve_0}^{(l)}\beta_{k-l}.
\end{equation}
Of course, we have $\beta_k=0$ for $k>>1$ since the quantum grading is bounded from below.

By definition of $\alpha_k$ and $\beta_k$, \eqref{MOY-H-H-f-pi-ismorphism-pro-1} is true for $k<0$. Assume that we have found $\{\beta_k\}_{-\infty}^{m-1}\subset M^{\ve_0}$ so that \eqref{MOY-H-H-f-pi-ismorphism-pro-1} is true for $k<m$. Let us find a $\beta_m$. Consider the homogeneous part of $d_{\ve_0-1}\alpha=0$ of degree $N+1+g-2m$. We have that
\begin{eqnarray*}
0 & = & \sum_{l=0}^{N}d_{\ve_0-1}^{(l)}\alpha_{m-l} ~=~ d_{\ve_0-1}^{(0)}\alpha_{m}+\sum_{l=1}^{N}d_{\ve_0-1}^{(l)}\alpha_{m-l} \\
  & = & d_{\ve_0-1}^{(0)}\alpha_{m}+\sum_{l=1}^{N}d_{\ve_0-1}^{(l)}\sum_{j=0}^N d_{\ve_0}^j\beta_{m-l-j} \\
  & = & d_{\ve_0-1}^{(0)}\alpha_{m} + \sum_{k=1}^{2N}(\sum_{l+j=k,~1\leq l\leq N,~ 0\leq j\leq N}d_{\ve_0-1}^{(l)}d_{\ve_0}^{(j)}\beta_{m-k}) \\
  & = & d_{\ve_0-1}^{(0)}\alpha_{m} - \sum_{k=1}^N d_{\ve_0-1}^{(0)}d_{\ve_0}^{(k)}\beta_{m-k} ~=~ d_{\ve_0-1}^{(0)}(\alpha_{m} - \sum_{k=1}^N d_{\ve_0}^{(k)}\beta_{m-k}).    
\end{eqnarray*}
By Lemma \ref{MOY-C-cr}, $H^{\ve_0-1}(\Gamma)=0$, that is, $\ker{d_{\ve_0-1}^{(0)}}=\im{d_{\ve_0}^{(0)}}$. So there exists a homogeneous element $\beta_m$ of $M^{\ve_0}$ of degree $g-2m-N-1$ such that 
\[
\alpha_{m} - \sum_{k=1}^N d_{\ve_0}^{(k)}\beta_{m-k} = d_{\ve_0}^{(0)}\beta_m.
\] 
This completes the induction.

Now , we have 
\begin{eqnarray*}
\alpha & = & \sum_{k=-\infty}^{\infty}\alpha_k = \sum_{k=-\infty}^{\infty}\sum_{l=0}^{N}d_{\ve_0}^{(l)}\beta_{k-l} \\
       & = & \sum_{l=0}^{N}d_{\ve_0}^{(l)}(\sum_{k=-\infty}^{\infty}\beta_{k-l}) = \sum_{l=0}^{N}d_{\ve_0}^{(l)}(\sum_{j=-\infty}^{\infty}\beta_{j}) \\
       & = & d_{\ve_0}(\sum_{j=-\infty}^{\infty}\beta_{j}).
\end{eqnarray*}
This shows that $\ker{d_{\ve_0-1}}=\im{d_{\ve_0}}$ and therefore $H_{f,\pi}^{\ve_0-1}(\Gamma)=0$.

Next we prove \eqref{MOY-H-H-f-pi-ismorphism-2}. Let $(\ker{d_{\ve_0}^{(0)}})_k$ be the subspace of $\ker{d_{\ve_0}^{(0)}}$ consisting of homogeneous elements of degree $k$. Define $\tilde{\theta}_k:(\ker{d_{\ve_0}^{(0)}})_k \rightarrow \fil^k H_{f,\pi}^{\ve_0}(\Gamma)/\fil^{k-1} H_{f,\pi}^{\ve_0}(\Gamma)$ as following.

For $\alpha\in (\ker{d_{\ve_0}^{(0)}})_k$, construct by induction a sequence $\{\alpha_l\}_{0}^{\infty}\subset M^{\ve_0}$, such that $\alpha_0=\alpha$, $\alpha_l$ is a homogeneous element of quantum degree $k-2l$ and, for $l\geq0$,
\begin{equation}\label{MOY-H-H-f-pi-ismorphism-pro-2}
\sum_{j=0}^{l}d_{\ve_0}^{(j)}\alpha_{l-j}=0, ~\forall ~l\in\zed,
\end{equation}
where we use the convention that $d_{\ve_0}^{(j)}=0$ for $j>N$. Note that $\alpha_l=0$ for $l>>1$ since the quantum grading is bounded from below. Clearly, \eqref{MOY-H-H-f-pi-ismorphism-pro-2} is true with $\alpha_0=\alpha$. Next, assume that, for some $m\geq1$, we have found $\{\alpha_l\}_{0}^{m-1}$ such that  \eqref{MOY-H-H-f-pi-ismorphism-pro-2} is true for $0\leq l<m$. Let us find an $\alpha_m$. 

Consider
\begin{eqnarray*}
d_{\ve_0-1}^{(0)}(\sum_{j=1}^{m}d_{\ve_0}^{(j)}\alpha_{m-j}) & = & \sum_{j=1}^{m}d_{\ve_0-1}^{(0)}d_{\ve_0}^{(j)}\alpha_{m-j} \\
(\text{by \eqref{MOY-H-H-f-pi-ismorphism-pro-0}})   & = & -\sum_{j=1}^{m}\sum_{l=0}^{j-1}d_{\ve_0-1}^{(j-l)}d_{\ve_0}^{(l)}\alpha_{m-j} \\
                      (\text{set }p=m-j+l)         & = & -\sum_{p=0}^{m-1}\sum_{l=0}^p d_{\ve_0-1}^{(m-p)}d_{\ve_0}^{(l)}\alpha_{p-l} \\
                                                   & = & -\sum_{p=0}^{m-1}d_{\ve_0-1}^{(m-p)} (\sum_{l=0}^p d_{\ve_0}^{(l)}\alpha_{p-l}) \\
             (\text{by induction hypothesis})      & = & 0
\end{eqnarray*}
But $H^{\ve_0-1}(\Gamma)=0$, i.e. $\ker{d_{\ve_0-1}^{(0)}}=\im{d_{\ve_0}^{(0)}}$. So there exists a homogeneous element $\alpha_m \in M^{\ve_0}$ of degree $k-2m$ such that 
\[
d_{\ve_0}^{(0)}\alpha_m = -\sum_{j=1}^{m}d_{\ve_0}^{(j)}\alpha_{m-j}.
\]
This completes the induction. 

It is clear that 
\[
d_{\ve_0}(\sum_{l=0}^{\infty}\alpha_l)=0.
\]
Note the sum here is in fact a finite sum. Define 
\[
\tilde{\theta}_k(\alpha)=[\sum_{l=0}^{\infty}\alpha_l]\in\fil^k H_p^{\ve_0}(\Gamma)/\fil^{k-1} H_p^{\ve_0}(\Gamma).
\] 
We need to check that $\tilde{\theta}_k(\alpha)$ is independent of the choice of the sequence $\{\alpha_l\}_{0}^{\infty}$. This is easy. Let $\{\alpha'_l\}_{0}^{\infty}$ another such sequence. Then
\[
(\sum_{l=0}^{\infty}\alpha_l)-(\sum_{l=0}^{\infty}\alpha'_l) = (\sum_{l=1}^{\infty}\alpha_l)-(\sum_{l=1}^{\infty}\alpha'_l) \in \fil^{k-1} \ker d_{\ve_0}.
\]
So
\[
[\sum_{l=0}^{\infty}\alpha_l] = [\sum_{l=0}^{\infty}\alpha'_l] \in \fil^k H_p^{\ve_0}(\Gamma)/\fil^{k-1} H_p^{\ve_0}(\Gamma).
\]
Thus, $\tilde{\theta}_k:(\ker{d_{\ve_0}^{(0)}})_k \rightarrow \fil^k H_{f,\pi}^{\ve_0}(\Gamma)/\fil^{k-1} H_{f,\pi}^{\ve_0}(\Gamma)$ is well defined. It is straightforward to check that $\tilde{\theta}_k$ is $\C$-linear.

Note that any element of $\fil^k H_{f,\pi}^{\ve_0}(\Gamma)$ is represented by a cycle of the form $\sum_{j=0}^\infty \alpha_j$, where $\alpha_j \in M^{\ve_0}$ is a homogeneous element of quantum degree $k-2j$. Considering the top homogeneous part of $d_{\ve_0}(\sum_{j=0}^\infty \alpha_j)=0$, one can see that $d_{\ve_0}^{(0)}(\alpha_0) =0$. It is then easy to check that $\tilde{\theta}_k(\alpha_0) = [\sum_{j=0}^\infty \alpha_j] \in \fil^k H_{f,\pi}^{\ve_0}(\Gamma)/\fil^{k-1} H_{f,\pi}^{\ve_0}(\Gamma)$. So $\tilde{\theta}_k$ is surjective.

Next, we compute $\ker{\tilde{\theta}_k}$. Let $\alpha\in \ker{\tilde{\theta}_k}$. That is, for the above constructed sequence $\{\alpha_l\}_{0}^{\infty}$, we have 
\begin{equation}\label{MOY-H-H-f-pi-ismorphism-pro-3}
\sum_{l=0}^{\infty}\alpha_l = d_{\ve_0-1}\beta+\gamma,
\end{equation}
where $\gamma$ is a cycle in $\fil^{k-1}M^{\ve_0}$, and $\beta\in M^{\ve_0-1}$.  This equation implies that $d_{\ve_0-1}\beta\in\fil^kM^{\ve_0}$. We claim that we can choose $\beta$ so that $\deg_Q\beta\leq k-N-1$. Assume that $\deg_Q\beta=g> k-N-1$. Let $\beta_0$ be the top homogeneous part of $\beta$. Comparing the top homogeneous part in \eqref{MOY-H-H-f-pi-ismorphism-pro-3}, we have $d_{\ve_0-1}^{(0)}\beta_0=0$. So there exists a homogeneous element $\xi\in M^i$ of degree $g-N-1$ such that $d_{\ve_0}^{(0)}\xi=\beta_0$.  Let $\beta'=\beta-d_{\ve_0}\xi$. Then $\beta'$ also satisfies \eqref{MOY-H-H-f-pi-ismorphism-pro-3}, and $\deg_Q \beta' \leq \deg_Q \beta -2$. Repeat this process. Within finite steps, we can find a $\beta$ with $\deg_Q\beta\leq k-N-1$ that satisfies \eqref{MOY-H-H-f-pi-ismorphism-pro-3}. Now let $\hat{\beta}$ be the homogeneous part of $\beta$ of quantum degree $k-N-1$. By \eqref{MOY-H-H-f-pi-ismorphism-pro-3}, one can see that $\alpha=\alpha_0=d_{\ve_0-1}^{(0)}\hat{\beta}$. This shows that $\ker{\tilde{\theta}_k}\subset(\im{d_{\ve_0-1}^{(0)}})_k$, where $(\im{d_{\ve_0-1}^{(0)}})_k$ is the subspace of $\im{d_{\ve_0-1}^{(0)}}$ of homogeneous elements of quantum degree $k$. 

On the other hand, if $\alpha\in(\im{d_{\ve_0-1}^{(0)}})_k$, then there is a homogeneous element $\beta \in M^{\ve_0-1}$ of quantum degree $k-N-1$ with $d_{\ve_0-1}^{(0)}\beta=\alpha$. Note that 
\[
\sum_{l=0}^{\infty}\alpha_l = d_{\ve_0-1}\beta+\gamma,
\]
where $\gamma=-d_{\ve_0-1}\beta+\sum_{l=0}^{\infty}\alpha_l$ is a cycle in $\fil^{k-1}M^{\ve_0}$. This shows that $\alpha\in \ker{\tilde{\theta}_k}$. Thus, $\ker{\tilde{\theta}_k}=(\im{d_{\ve_0-1}^{(0)}})_k$. Therefore, $\tilde{\theta}_k$ induces an isomorphism
\[
\theta_k: H^{\ve_0,k}(\Gamma)=(\ker{d_{\ve_0}^{(0)}})_k/(\im{d_{\ve_0-1}^{(0)}})_k \rightarrow \fil^k H_{f,\pi}^{\ve_0}(\Gamma)/\fil^{k-1} H_{f,\pi}^{\ve_0}(\Gamma).
\]
This proves \eqref{MOY-H-H-f-pi-ismorphism-2}.

It remains to prove the naturality. Let $\alpha$ be a homogeneous cycle in $C(\Gamma)$ of $\zed_2$-degree $\mathrm{cr}(\Gamma)$ and quantum degree $k$. Then, from above, we know that there is a sequence $\{\alpha_l\}_{0}^{\infty} \subset C(\Gamma)$, such that $\alpha_0=\alpha$, $\alpha_l$ is a homogeneous element of $\zed_2$-degree $\mathrm{cr}(\Gamma)$ and quantum degree $k-2l$ and equation \eqref{MOY-H-H-f-pi-ismorphism-pro-2} is true. Also, there is a sequence $\{\alpha_l'\}_{0}^{\infty} \subset C(\Gamma')$ satisfying that $\alpha_0'=(\varpi_0(G))(\alpha)$, $\alpha_l'$ is a homogeneous element of $\zed_2$-degree $\mathrm{cr}(\Gamma')$ and quantum degree $k+i-2l$ and equation \eqref{MOY-H-H-f-pi-ismorphism-pro-2} is true for $\{\alpha_l'\}_{0}^{\infty}$. So $\sum_{l=0}^\infty \alpha_l'$ is a cycle in $\fil^{k+i}C_{f,\pi}(\Gamma')$ of $\zed_2$-degree $\mathrm{cr}(\Gamma')$. Moreover, it is easy to see that $(\varpi(G))(\sum_{l=0}^\infty \alpha_l) - \sum_{l=0}^\infty \alpha_l' \in \fil^{k+i-1}C_{f,\pi}(\Gamma')$. 

Denote by $\tilde{\theta}_k'$ and $\theta_k'$ the homomorphisms defined for $\Gamma'$ corresponding to $\tilde{\theta}_k$ and $\theta_k$. Then, as elements of $\fil^{k+i} H_{f,\pi}^{\mathrm{cr}(\Gamma')}(\Gamma')/\fil^{k+i-1} H_{f,\pi}^{\mathrm{cr}(\Gamma')}(\Gamma')$,
\begin{eqnarray*}
\tilde{\theta}_{k+i}'((\varpi_0(G))(\alpha)) & = & [\sum_{l=0}^\infty \alpha_l'] = [(\varpi(G))(\sum_{l=0}^\infty \alpha_l)] \\
& = & (\varpi(G))_\ast [\sum_{l=0}^\infty \alpha_l] = (\varpi(G))_\ast(\tilde{\theta}_k(\alpha)).
\end{eqnarray*}
This shows that the diagram
\[
\xymatrix{
H^{\mathrm{cr}(\Gamma),k}(\Gamma) \ar@<0ex>[rr]^{\theta_k} \ar@<0ex>[d]^{(\varpi_0(G))_\ast} && \fil^k H_{f,\pi}^{\ve_0}(\Gamma)/\fil^{k-1} H_{f,\pi}^{\ve_0}(\Gamma) \ar@<0ex>[d]^{(\varpi(G))_\ast} \\
H^{\mathrm{cr}(\Gamma'),k+i}(\Gamma') \ar@<0ex>[rr]^<<<<<<<<<<{\theta_{k+i}'} && \fil^{k+i} H_{f,\pi}^{\mathrm{cr}(\Gamma')}(\Gamma')/\fil^{k+i-1} H_{f,\pi}^{\mathrm{cr}(\Gamma')}(\Gamma')
}
\]
commutes.
\end{proof}

Let $D$ be a closed knotted MOY graph and $\Gamma$ any complete resolution of $D$. Note that $\mathrm{cr}(\Gamma)$ does not depend on the choice of $\Gamma$. To each crossing $c$ of $D$, associate an adjustment term $\mathsf{a}(c) \in \zed_2$ defined by 
\[
\mathsf{a}\left(\setlength{\unitlength}{1pt}
\begin{picture}(40,40)(-20,0)

\put(-20,-20){\vector(1,1){40}}

\put(20,-20){\line(-1,1){15}}

\put(-5,5){\vector(-1,1){15}}

\put(-11,15){\tiny{$_m$}}

\put(9,15){\tiny{$_n$}}

\end{picture}\right) = \mathsf{a}\left(\setlength{\unitlength}{1pt}
\begin{picture}(40,40)(-20,0)

\put(20,-20){\vector(-1,1){40}}

\put(-20,-20){\line(1,1){15}}

\put(5,5){\vector(1,1){15}}

\put(-11,15){\tiny{$_m$}}

\put(9,15){\tiny{$_n$}}

\end{picture}\right) =
\begin{cases}
m & \text{if } m=n,\\
0 & \text{if } m \neq n.
\end{cases}
\]
Define $\hat{\tc}(D) := \mathrm{cr}(\Gamma) + \sum_c \mathsf{a}(c) \in \zed_2$, where the sum is taken over all crossings of $D$. One can see that $\hat{\tc}(D)$ is invariant under Reidemeister moves and unknotting (switching the above- and below- strands at a crossing.) Therefore, if $D$ is a link diagram, it is easy to check that $\hat{\tc}(D) = \tc(D) \in \zed_2$, where $\tc(D)$ is the total color of $D$, that is, the sum of the colors of all components of $D$.

\begin{theorem}\label{thm-MOY-spectral}
Let $D$ be a closed knotted MOY graph. Then the subspace of $H_{f,\pi}(D)$ of elements of $\zed_2$-degree $\hat{\tc}(D)+1$ vanishes. 

Moreover, the quantum filtration of $C_{f,\pi}(D)$ induces a spectral sequence converging to $H_{f,\pi}(D)$ with $E_1$-term isomorphic to the colored $\mathfrak{sl}(N)$-homology $H(D)$ defined in \cite{Wu-color}.
\end{theorem}

\begin{proof}
The fact that the subspace of $H_{f,\pi}(D)$ of elements of $\zed_2$-degree $\hat{\tc}(D)+1$ vanishes follows from \eqref{MOY-H-H-f-pi-ismorphism-1} in Proposition \ref{MOY-H-H-f-pi-ismorphism} and the normalization in Definition \ref{complex-colored-crossing-def}. 

Recall that, when $D$ is closed, all complete resolutions of $D$ are also closed and their matrix factorizations are in fact chain complexes. Denote by $d_{mf}$ the differential map of these matrix factorizations/chain complexes. Then $H_{f,\pi}(D)$ is defined to be the homology of the filtered chain complex $H(C_{f,\pi}(D), d_{mf})$. It is easy to see that filtration on $H(C_{f,\pi}(D), d_{mf})$ is bounded below and exhaustive. So it induces a spectral sequence that converges to $H_{f,\pi}(D)$. By Proposition \ref{MOY-H-H-f-pi-ismorphism}, the $E_0$-term of this spectral sequence is isomorphic to the graded chain complex $H(C(D), d_{mf})$. So its $E_1$-term is isomorphic to $H(D)$.
\end{proof}

It is clear that Theorem \ref{thm-inv-deformation} follows from Theorem \ref{invariance-reidemeister-deformation} and Theorem \ref{thm-MOY-spectral}.

\end{document}